\documentclass[12pt]{book}

\usepackage{amsmath,amsthm,verbatim,amssymb,amsfonts,amscd, graphicx}
\usepackage{graphics}
\usepackage{tablefootnote}
\usepackage[export]{adjustbox}
\usepackage{tikz-cd}
\usepackage{multirow}
\usepackage{tikz}
\usepackage{environ}
\usepackage{adjustbox}
\usepackage{xfrac}
\usepackage{fancyhdr}
\usetikzlibrary{calc}
\usepackage{euscript}
\usepackage{setspace}
\usepackage{comment}
\usepackage[a4paper, width=160mm, top=30mm, bottom=30mm, left=25mm, right=25mm]{geometry}
\usepackage{mathtools}
\usepackage{tensor}
\usepackage[all]{xy}
\usepackage{enumitem}
\usepackage{caption}
\usepackage{subcaption}
\usepackage[many]{tcolorbox}
\usepackage{mdframed}
\usepackage{graphicx}
\usepackage{titlesec}
\usepackage{lipsum}
\usepackage{xcolor}
\usepackage{soul}
\usepackage{wasysym}
\usepackage{url}
\usepackage{hyperref}
\usepackage{pdfpages}
\usepackage[style=alphabetic,backend=biber]{biblatex}
\graphicspath{ {./images/} }
\theoremstyle{plain}
\newtheorem{theorem}{Theorem}[section]
\newtheorem*{theorem*}{Theorem}
\newtheorem{corollary}[theorem]{Corollary}
\newtheorem{lemma}[theorem]{Lemma}
\newtheorem{proposition}[theorem]{Proposition}
\newtheorem{definition}[theorem]{Definition}

\newtheorem{conjecture}{Conjecture}

\theoremstyle{definition}
\newtheorem{remark}[theorem]{Remark}
\AtBeginEnvironment{remark}{%
  \pushQED{\qed}%
}
\AtEndEnvironment{remark}{\popQED\endexample}

\newtheorem{example}[theorem]{Example}
\AtBeginEnvironment{example}{%
  \pushQED{\qed}%
}
\AtEndEnvironment{example}{\popQED\endexample}
\newtheorem{construction}[theorem]{Construction}

\newcommand{\adj}[1][]{\def\ArgI{#1}\adjRelayI}
\newcommand{\adjRelayI}[1][]{\def\ArgII{#1}\adjRelayII}
\newcommand{\adjRelayII}[3][2.2em]{\ensuremath{\SelectTips{cm}{10}\xymatrix@C=#1@1{{#2} \ar@<1ex>[r]^-{\ArgI}^-{}="1" & {#3} \ar@<1ex>[l]^-{\ArgII}^-{}="2" \ar@{}"1";"2"|(.3){\hbox{}}="7" \ar@{}"1";"2"|(.7){\hbox{}}="8" \ar@{|-} "8" ;"7"}}}

\newcommand{\radjRelayII}[3][2.2em]{\ensuremath{\SelectTips{cm}{10}\xymatrix@C=#1@1{{#2} \ar@<-1ex>[r]_-{\ArgI}^-{}="1" & {#3} \ar@<-1ex>[l]_-{\ArgII}^-{}="2" \ar@{}"1";"2"|(.3){\hbox{}}="7" \ar@{}"1";"2"|(.7){\hbox{}}="8" \ar@{|-} "7" ;"8"}}}

\newcommand{\cat}[1]{\EuScript{#1}}
\newcommand{\nl}{
~ \newline}

\usepackage[T1]{fontenc}
\DeclareUnicodeCharacter{0229}{\k{e}}

\titleformat{\chapter}[hang]
    {\normalfont\huge\bfseries}
    {\thechapter}{20pt}{\huge}

\begin{document}
\frontmatter
        \pagestyle{plain}

        \newgeometry{width=170mm, top=20mm, bottom=20mm}
        \begin{titlepage}
    \begin{center}

        \begin{figure}[t]
            \includegraphics[height=2.5cm, left]{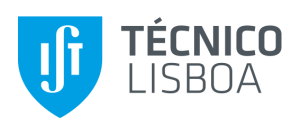}
        \end{figure}

        \vspace*{-1cm}

        \begin{center}
            
            \textbf{UNIVERSIDADE DE LISBOA} 

            \vspace{0.5cm}

            \textbf{INSTITUTO SUPERIOR TÉCNICO}
        \end{center}

        \vspace{2cm}

        \huge
        \textbf{Differential Cohomology as Diffeological Homotopy Theory}

        \vspace{1.5cm}
        \large
        \text{Nino Tiziano Scalbi}

        \vspace{0.5cm}
        \large
        \begin{align*}
            \textbf{Supervisor: } & \text{Doctor Roger Francis Picken} \\
            \textbf{Co-Supervisor: } & \text{Doctor John Gmerek Huerta}
        \end{align*}

        \vspace{1cm}

        \normalsize
        \textbf{Thesis approved in public session to obtain the PhD Degree in}

        \vspace{0.25cm}

        \large
        \textit{Mathematics}

        \vspace{1.5cm}

        \normalsize
        \textbf{Jury final classification:} \text{Pass with Distinction}

        \vfill
        \normalsize
        \textbf{2024}
    \end{center}

    \newpage 
    \thispagestyle{empty}
\end{titlepage}

\begin{titlepage}
    \begin{center}

        \begin{figure}[t]
            \includegraphics[height=2.5cm, left]{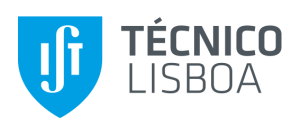}
        \end{figure}

        \vspace*{-1cm}

        \begin{center}
            
            \normalsize
            \textbf{UNIVERSIDADE DE LISBOA} 

            \vspace{0.5cm}

            \textbf{INSTITUTO SUPERIOR TÉCNICO}
        \end{center}

        \vspace{1cm}

        \Large
        \textbf{Differential Cohomology as Diffeological Homotopy Theory}

        \vspace{0.75cm}
        \large
        \text{Nino Tiziano Scalbi}

        \vspace{-0.25cm}
        \large
        \begin{align*}
            \textbf{Supervisor: } & \text{Doctor Roger Francis Picken} \\
            \textbf{Co-Supervisor: } & \text{Doctor John Gmerek Huerta}
        \end{align*}

        \vspace{0.5cm}

        \normalsize
        \textbf{Thesis approved in public session to obtain the PhD Degree in}

        \vspace{0.25cm}

        \large
        \textit{Mathematics}

        \vspace{0.5cm}

        \normalsize
        \textbf{Jury final classification:} \text{Pass with Distinction}

    \vspace{0.5cm}
    \normalsize \textbf{Jury}
    \vspace{0.25cm}
    \end{center}

    \noindent\makebox[\textwidth][c]{\begin{minipage}[t]{.8\textwidth}
            \raggedright
            \normalsize
            \textbf{Chairperson:} \\
            \vspace{2mm} 
            Doctor Pedro Manuel Agostinho Resende, Instituto Superior Técnico, Universidade de Lisboa \\
            \vspace{3mm}
            \textbf{Members of the Committee:} \\
            \vspace{2mm}
            Doctor John Daniel Christensen, University of Western Ontario, Canad\'a \\
            \vspace{1.5mm}
            Doctor Daniel Grady, Wichita State University, EUA \\
            \vspace{1.5mm}
            Doctor Pedro Boavida de Brito, Instituto Superior Técnico, Universidade de Lisboa \\
            \vspace{1.5mm}
            Doctor John Gmerek Huerta, Instituto Superior Técnico, Universidade de Lisboa \\
        \end{minipage}}

\begin{center}
    \vspace{0.25cm}
    \normalsize \textbf{Funding Institution} \\
    Fundação para a Ciência e a Tecnologia (FCT)
    \vspace{0.25cm}
\end{center}

\begin{center}
    \vspace{0.25cm}
    \normalsize
    \textbf{2024}
\end{center}
\newpage 
\thispagestyle{empty}
\end{titlepage}

        \restoregeometry


        \chapter{Abstract}

\begin{spacing}{1.5}
Thin homotopies have been introduced by Caetano--Picken \cite{Caetano-Picken} to axiomatize the holonomy of connections on principal bundles. This approach has been generalized to higher non-abelian bundles with connection through transport functors and higher holonomies, at least in dimension two and partially in dimension three. In the abelian case, that is for higher circle bundles with connection, there are classical models, such as smooth Deligne cohomology and Cheeger--Simons differential characters, which do not rely on thin homotopies. However, they only work in the abelian case. \\

We introduce a new variant of thin homotopy, based on the definition of skeletal diffeologies introduced recently by Kihara \cite{KiharaSk}. We show that, in the abelian setting, ordinary differential cohomology can be completely recovered in terms of the homotopy theory of skeletal diffeological spaces. Specifically, to any smooth manifold $M$ and non-negative integer $k$, we associate a simplicial presheaf $\mathbb{M}_k$ such that its $k$th cohomology with values in the circle group $U(1)$ is isomorphic to the abelian group of Cheeger--Simons differential characters
\[
H^k_{\infty}(\mathbb{M}_k,U(1)) \cong \hat{H}^{k+1}(M;\mathbb{Z}).
\]
\newpage
To prove this, we show that differential characters are in one-to-one correspondence with smooth higher holonomy morphisms 
\[
h \colon H_{k}(\mathbb{M}_k) \rightarrow U(1).
\]
Here $H_k(\mathbb{M}_k)$ denotes the homology sheaf of $\mathbb{M}_k$ introduced by Jardine \cite{JardineBook}. This proof relies on a recent variant of the Riemann integral based on simplicial partitions.

Further, we compare the higher holonomies $h$ to the only other existing model in the literature, which was developed by Gajer \cite{GajerHigher} based on the geometric loop group $G(M,x_0)$. This is achieved by constructing a surjective homomorphism $G(M,x_0) \rightarrow \pi^D_1(M_1,x_0)$ from the geometric loop group into the $1$-skeletal diffeological fundamental group of the manifold $M$. \\

The main motivation for us is not to merely formulate differential cohomology in terms of the homotopy theory of simplicial presheaves, but to provide a model where the differential refinement is established by refining the space $M$, rather than the coefficient object. This has the advantage of possible generalizations to non-abelian differential cohomology, which is of particular interest since the existing approaches to non-abelian differential cohomology via transport functors are challenging to generalize to higher dimensions.
\end{spacing}
        \chapter{Resumo}

\begin{spacing}{1.5}
Homotopias finas foram introduzidas por Caetano--Picken \cite{Caetano-Picken} para axiomatizar a holonomia de conexões em fibrados principais. Esta abordagem foi generalizada para fibrados superiores não abelianos com conexões por meio de funtores de transporte e holonomias superiores, pelo menos em dimensão dois e parcialmente em dimensão três. No caso abeliano, ou seja, para fibrados circulares superiores com conexão, existem modelos clássicos, como a cohomologia suave de Deligne e os caracteres diferenciais de Cheeger--Simons, que não dependem de homotopias finas. No entanto, eles funcionam apenas no caso abeliano. \\

Apresentamos uma nova variante de homotopia fina, baseada na definição de difeologias esqueléticas introduzida recentemente por Kihara \cite{KiharaSk}. Mostramos que no caso abeliano, a cohomologia diferencial ordinária pode ser completamente recuperada em termos da teoria de homotopia de espaços difeológicos esqueléticos. Especificamente, para qualquer variedade suave $M$ e número inteiro não negativo $k$, associamos um pré-feixe simplicial $\mathbb{M}_k$, de modo que a sua cohomologia $k$-ésima com valores no grupo de círculo $U(1)$ é isomórfa ao grupo abeliano de caracteres diferenciais de Cheeger--Simons
\[
H^k_{\infty}(\mathbb{M}_k,U(1)) \cong \hat{H}^{k+1}(M;\mathbb{Z}).
\]
\newpage 
Para fazer a prova, mostramos que os caracteres diferenciais estão em correspondência biunívoca com morfismos suaves de holonomia superior
\[
h \colon H_{k}(\mathbb{M}_k) \rightarrow U(1).
\]
Aqui, $H_k(\mathbb{M}_k)$ denota o feixe de homologia de $\mathbb{M}_k$ introduzido por Jardine \cite{JardineBook}. Esta prova baseia-se numa variante recente da integral de Riemann que utiliza partições simpliciais.

Além disso, comparamos as holonomias superiores $h$ com o único outro modelo existente na literatura, desenvolvido por Gajer \cite{GajerHigher}, com base no grupo geométrico de laço $G(M,x_0)$. Isso é alcançado construindo um homomorfismo sobrejectivo $G(M,x_0) \rightarrow \pi^D_1(M_1,x_0)$ do grupo geométrico de laço no grupo fundamental difeológico $1$-esquelético da variedade $M$. \\

A nossa principal motivação não é apenas formular a cohomologia diferencial em termos da teoria de homotopia de pré-feixes simpliciais, mas fornecer um modelo onde o refinamento diferencial é estabelecido refinando o espaço $M$, em vez do coeficiente. Esta abordagem tem a vantagem de poder ser generalizada para a cohomologia diferencial não abeliana. Esta generalização é de particular interesse, uma vez que as abordagens existentes para a cohomologia diferencial não abeliana por meio de funtores de transporte são difíceis de generalizar para dimensões superiores.
\end{spacing}
        \chapter{Keywords}

\begin{itemize}
    \item differential cohomology
    \item local homotopy theory
    \item diffeological spaces
    \item higher gauge theory
    \item differential characters
\end{itemize}

\vspace{10mm}

{\let\clearpage\relax\chapter*{Palavras-Chave}}

\begin{itemize}
    \item cohomologia diferencial
    \item teoria de homotopia local
    \item espaços difeológicos
    \item teoria de gauge superior 
    \item caracteres diferenciais
\end{itemize}
        \chapter{Acknowledgments}

\begin{spacing}{1.5}
To the following people and organizations, whose support and guidance have been crucial in the completion of my Ph.D.: \\

First, I would like to express my gratitude to my thesis supervisors, Doctor John Huerta and Doctor Roger Picken, for their mentorship, constant availability, and insights throughout my PhD. It was great luck to have the opportunity to learn from both of you. \\

I would like to thank Doctor Miguel Abreu, the Program Director of the Lisbon Mathematics Ph.D. program (LisMath), and Doctor Gabriel Lopes Cardoso, the Coordinator of the Doctoral Program in Mathematics at Instituto Superior Técnico, for always being available and coordinating various academic activities. Also, I acknowledge the Lisbon Mathematics Ph.D. program for sponsoring my research visit at the Department of Mathematics and Statistics at Texas Tech University. Further, I acknowledge the Fundação para a Ciência e a Tecnologia (FCT) for their financial support during my PhD studies. I also thank the Centre for Mathematical Analysis, Geometry, and Dynamical Systems (CAMGSD) for sponsoring my attendance at the Barcelona Conference on Higher Structures, and at the Leeds Research School on Bicategories, Categorification and Quantum Theory, providing me with opportunities to engage with the international academic community. \\

To Doctor Dmitri Pavlov at Texas Tech University, Lubbock, for hosting me and engaging in many insightful discussions, which had a big impact on my thesis, as well as to his students for their warm welcome. My time in Texas has been one of the highlights of my Ph.D. journey. My deep gratitude to Doctor Pedro Boavida de Brito, Doctor Manuel Araújo, and Doctor Giorgio Trentinaglia for the stimulating discussions and exchange. A big thank you also to the ``old guard'' – Vicente, Marti, Carlos, Salvatore, Augusto, Max, Paolo, Roberto, and Miguel – without you my first years would not have been the same. \\

I am particularly grateful for the many friends I have made throughout these years at Tecnico – my academic brother Arber for the many pointless debates, for always listening to my countless questions, and for introducing me to the politics of Kosovo, \textit{faleminderit}; my academic mother Rodrigo for all the philosophical, scientific and mathematical discussions and especially for believing in me, even when I did not; Fred for the many unnecessary Portuguese words and for trying to teach me to talk like a real lisboeta; Ana for the shared adventure in the housing market in Lisbon; Javi, Robbie, and Diogo, whose shared passion for mathematics has been always inspiring. \\

To all my friends in Lisbon and back home in Switzerland for the shared moments of joy and to my family for connecting me with home even when far away.

Last but not least, to Nadia, my constant source of joy and strength throughout the highs and lows of this journey.
\end{spacing}

        \tableofcontents

\mainmatter

\pagestyle{fancy}
\fancyhead[RO, LE]{\thepage}
\fancyhead[LO]{\slshape\nouppercase{\rightmark}}
\fancyhead[RE]{\slshape\nouppercase{\leftmark}}
\fancyfoot{} 

\chapter{Introduction}

When asked to draw a smooth homotopy between two paths $\gamma_1$ and $\gamma_2$
in the plane, both sharing the same start and endpoints, most of us would
probably come up with a picture like this:
\begin{center}
\begin{tikzpicture}[x=0.75pt,y=0.75pt,yscale=-1,xscale=1]

\draw [fill={rgb, 255:red, 208; green, 2; blue, 27 }  ,fill opacity=0.37 ]   (66,132) .. controls (106,102) and (171.33,81.67) .. (242,132) ;
\draw [shift={(242,132)}, rotate = 35.46] [color={rgb, 255:red, 0; green, 0; blue, 0 }  ][fill={rgb, 255:red, 0; green, 0; blue, 0 }  ][line width=0.75]      (0, 0) circle [x radius= 2.01, y radius= 2.01]   ;
\draw [shift={(158.21,101.49)}, rotate = 180.57] [color={rgb, 255:red, 0; green, 0; blue, 0 }  ][line width=0.75]    (6.56,-2.94) .. controls (4.17,-1.38) and (1.99,-0.4) .. (0,0) .. controls (1.99,0.4) and (4.17,1.38) .. (6.56,2.94)   ;
\draw [shift={(66,132)}, rotate = 323.13] [color={rgb, 255:red, 0; green, 0; blue, 0 }  ][fill={rgb, 255:red, 0; green, 0; blue, 0 }  ][line width=0.75]      (0, 0) circle [x radius= 2.01, y radius= 2.01]   ;
\draw [fill={rgb, 255:red, 208; green, 2; blue, 27 }  ,fill opacity=0.37 ]   (66,132) .. controls (118,194.33) and (186.67,192.33) .. (242,132) ;
\draw [shift={(242,132)}, rotate = 312.52] [color={rgb, 255:red, 0; green, 0; blue, 0 }  ][fill={rgb, 255:red, 0; green, 0; blue, 0 }  ][line width=0.75]      (0, 0) circle [x radius= 2.01, y radius= 2.01]   ;
\draw [shift={(157.34,177.85)}, rotate = 178.38] [color={rgb, 255:red, 0; green, 0; blue, 0 }  ][line width=0.75]    (6.56,-2.94) .. controls (4.17,-1.38) and (1.99,-0.4) .. (0,0) .. controls (1.99,0.4) and (4.17,1.38) .. (6.56,2.94)   ;
\draw [shift={(66,132)}, rotate = 50.16] [color={rgb, 255:red, 0; green, 0; blue, 0 }  ][fill={rgb, 255:red, 0; green, 0; blue, 0 }  ][line width=0.75]      (0, 0) circle [x radius= 2.01, y radius= 2.01]   ;

\draw (144,68.4) node [anchor=north west][inner sep=0.75pt]    {$\gamma _{1}$};
\draw (145,189.4) node [anchor=north west][inner sep=0.75pt]    {$\gamma _{2}$};
\draw (144,125.4) node [anchor=north west][inner sep=0.75pt]    {$H$};

\end{tikzpicture}
\end{center}
Less familiar is the idea of a \textit{thin homotopy}, which, informally, is a
homotopy not allowed to sweep out any area. Consider the case of a path
$\gamma_1$ from $x$ to $y$. Along this path lies the point $z$. The second path
$\gamma_2$ also starts at $x$ but at point $z$ we decide to take a detour to
point $z'$ and back, before continuing along the usual journey to $y$.
\begin{center}
\begin{tikzpicture}[x=0.75pt,y=0.75pt,yscale=-1,xscale=1]

\draw    (77,143) .. controls (132,117.8) and (167.2,127.8) .. (209,143) .. controls (250.8,158.2) and (294,124) .. (319,99) ;
\draw [shift={(319,99)}, rotate = 315] [color={rgb, 255:red, 0; green, 0; blue, 0 }  ][fill={rgb, 255:red, 0; green, 0; blue, 0 }  ][line width=0.75]      (0, 0) circle [x radius= 2.01, y radius= 2.01]   ;
\draw [shift={(77,143)}, rotate = 335.38] [color={rgb, 255:red, 0; green, 0; blue, 0 }  ][fill={rgb, 255:red, 0; green, 0; blue, 0 }  ][line width=0.75]      (0, 0) circle [x radius= 2.01, y radius= 2.01]   ;
\draw    (209,143) .. controls (249,113) and (180,63) .. (220,33) ;
\draw [shift={(220,33)}, rotate = 323.13] [color={rgb, 255:red, 0; green, 0; blue, 0 }  ][fill={rgb, 255:red, 0; green, 0; blue, 0 }  ][line width=0.75]      (0, 0) circle [x radius= 2.01, y radius= 2.01]   ;
\draw [shift={(209,143)}, rotate = 323.13] [color={rgb, 255:red, 0; green, 0; blue, 0 }  ][fill={rgb, 255:red, 0; green, 0; blue, 0 }  ][line width=0.75]      (0, 0) circle [x radius= 2.01, y radius= 2.01]   ;

\draw (82.8,145.4) node [anchor=north west][inner sep=0.75pt]    {$x$};
\draw (327.2,81.4) node [anchor=north west][inner sep=0.75pt]    {$y$};
\draw (228.6,16.4) node [anchor=north west][inner sep=0.75pt]    {$z'$};
\draw (205.8,147.8) node [anchor=north west][inner sep=0.75pt]    {$z$};
\end{tikzpicture}
\end{center}
These two paths are now thin homotopic, where the thin homotopy $H$ is smoothly
contracting the detour segment $[z,z']$ down to the point $z$. \newpage

Various formal definitions of thin homotopies have been proposed over time. For
instance a homotopy $H \colon [0,1]^2 \rightarrow \mathbb{R}^2$ is called
\text{thin},
\begin{enumerate}[label={(\arabic*)}]
    \item if the rank of the derivative of $H$ is strictly smaller than 2, considered as the standard definition of thin homotopy in the literature,
    \item if there is a smooth map $h \colon M \rightarrow \mathbb{R}^2$ such that $H([0,1]^2) \subset h(M)$ for $M$ some smooth manifold of $\mathrm{dim} \, M < 2$,
    \item if the homotopy $H \colon [0,1]^2 \rightarrow \mathbb{R}^2$ factors through a finite tree $T$,
    \item \label{Item: Skeletal Diffeology Thinness} if the homotopy $H$ factors locally through some open subset of $\mathbb{R}^1$.
\end{enumerate}
Most of these definitions of thinness have been introduced to study bundles with
connections, where they are essential to formalize the associated holonomy or
parallel transport. That is, the holonomy, or more generally the parallel
transport along a path $\gamma$ in some manifold $M$ is known to satisfy a list
of axioms. Most prominently, invariance under reparametrizations and that the
parallel transport along the reversed path $\gamma^{-1}$ agrees with the inverse
of the transport along $\gamma$. Further, parallel transport is also functorial
with respect to concatenation of paths. All these axioms suggest that the
holonomy of a principal $G$-bundle with connection is inherently captured by a
group homomorphism
\[
\mathrm{hol} \colon \Omega(M,x_0) \rightarrow G,
\]
from the ``group of based loops" in $M$ into the structure group $G$. Here, thin
homotopies come into play, transforming the space of loops $\Omega(M,x_0)$ into
an honest group by considering the thin homotopy classes of loops
$\pi_1^1(M,x_0)$, turning holonomy into a smooth homomorphism
\[
\mathrm{hol} \colon \pi_1^1(M,x_0) \rightarrow G.
\]

In this thesis, we study a new notion of thin homotopy based on the homotopy
theory of \textit{skeletal diffeological spaces}, which corresponds to example
\ref{Item: Skeletal Diffeology Thinness} from the previous list. This new take
on thin homotopies emerged from the current research on \textit{geometric field
theories} initiated by Grady--Pavlov \cite{Grady-Pavlov-Local,
Grady-Pavlov-GCH}. To every diffeological space $X$ we associate the
$d$-skeletal diffeological space $X_d$ where $d$ is a non-negative integer. The
diffeology on $X_d$ is the one generated from the plots of $X$ of dimension
$\leq d$. That means, a function $p \colon U \rightarrow X_d$ for $U \subset
\mathbb{R}^n$ some open subset defines a plot for $X_d$ if and only if it
locally factors trough some plot $ q \colon W \rightarrow X$ of $X$ where $W$ is
of dimension $\leq d$. For $d = 1$, the 1-skeletal diffeology introduced by
Souriau is known as the wire diffeology or spaghetti diffeology. Their
generalization to arbitrary $d$ is attributed to Kihara \cite{KiharaSk}. \\

Baez--Hoffnung \cite{BaezHoffnung} showed that diffeological spaces are
equivalent to \textit{concrete sheaves on the site of Cartesian spaces}, which
then embed into the much more general category of \textit{simplicial presheaves
on Cartesian spaces}. Along this route, given any smooth manifold $M$, the
associated $k$-skeletal diffeological space $M_k$ generalizes and defines a
simplicial presheaf on Cartesian spaces denoted $\mathbb{M}_k$. This simplicial presheaf assigns to some Cartesian space
$U$ the simplicial set $S_e(D(U,M)_k)$ given by the extended smooth singular complex of the diffeological function space
$D(U,M)_k$ endowed with the $k$-skeletal diffeology. The main result
of this thesis relates the cohomology of $\mathbb{M}_k$ with the differential
cohomology of $M$.

\begin{theorem*}[Theorem \ref{Theorem: The Main Theorem}]
Let $M$ be a smooth manifold without boundary. There is an isomorphism of abelian groups
\begin{equation} \label{Equation: The Main isomorphism}
\hat{H}^{k+1}(M;\mathbb{Z}) \xrightarrow{\cong} H^k_{\infty}(\mathbb{M}_k,U(1))
\end{equation}
identifying the Cheeger--Simons differential characters with the $k$th stack
cohomology of the simplicial presheaf $\mathbb{M}_k$ with coefficients in
$U(1)$.
\end{theorem*}

The proof of this theorem splits into two parts. The cohomology of the
simplicial presheaf $\mathbb{M}_k$ formalized by the local homotopy theory of
simplicial presheaves on Cartesian spaces can be computed in this particular
case using a generalized version of the universal coefficient spectral sequence.
This involves computing higher extensions in the category of sheaves of abelian
groups on Cartesian spaces. In the case of the circle group, this leads to the
recognition of $H^k_{\infty}(\mathbb{M}_k,U(1))$ as the abelian group of
morphisms of presheaves $\mathrm{Hom}_{\mathrm{PSh}(\mathbf{Ab})} \left( H_k
\left( \mathbb{M}_k \right), U(1) \right)$.

In the second part, we focus on constructing an isomorphism
\begin{equation} \label{Equation: The higher holonomy isomorphism}
\hat{H}^{k+1}(M;\mathbb{Z}) \cong \mathrm{Hom}_{\mathrm{PSh}(\mathbf{Ab})} \left( H_k \left( \mathbb{M}_k \right), U(1) \right)
\end{equation}
which maps a differential character to the corresponding smooth higher holonomy
homomorphism. The essential challenge here lies in assigning a curvature form
$A_h$ to any morphism of presheaves $h$. To do so, we adapt the approach of
Schreiber--Waldorf \cite{Schreiber-Waldorf-Smooth} to the simplicial setting and
use a new take on the Riemann integral, based on triangulations instead of the
classical cubic partitions.

A similar isomorphism as \eqref{Equation: The higher holonomy isomorphism} has
been constructed by Gajer \cite{GajerHigher} using a geometric version of Kan's
simplicial loop group $G(M,x_0)$. This \textit{geometric loop group} is
identified as the fiber of a bundle $E(M,x_0) \rightarrow M$ called the
geometric cobar construction. While reformulating Gajer's geometric cobar
construction in diffeological terms, we confirm the assessment previously made
by Ghazel--Kallel \cite{Ghazel-Kallel} that the proof of the contractibility of
$E(M,x_0)$ is incomplete. We then show that, in the realm of \textit{thin loop
groups}, the geometric loop group carries a universal property in the sense that
there is a commutative diagram
\begin{center}
\begin{tikzcd}
\Omega(M,x_0) \arrow[d, "q" description] \arrow[rd, "q_1" description, shift left] \arrow[rrrd, "q" description, shift left=3] \arrow[rrd, "q_{\leq 1}" description, shift left=2] &                                             &                                 &          \\
G(M,x_0) \arrow[r, "\theta_1"'] & \pi^{D}_1(M_1,x_0) \arrow[r, "\theta_{\leq1}"'] & \pi_1^1(M,x_0) \arrow[r, "\theta"'] & \pi_1(M,x_0)
\end{tikzcd}
\end{center}
sending a \textit{thin class} $[\gamma]_G$ of a loop $\gamma \in \Omega(M,x_0)$ to the next stronger notion of thinness $[\gamma]_1$ and $[\gamma]_{\leq 1}$ until we arrive at the ordinary homotopy class $[\gamma]$. This insight is then used to construct a connecting homomorphism
\begin{equation}\label{Equation: Connecting Homomorphism}
\mathrm{Hom}_{\mathrm{PSh}(\mathbf{Cart}, \mathbf{Ab})} \left( H_{k}(\mathbb{M}_k), U(1)   \right) \rightarrow \mathrm{Hom}_{\mathbf{DiffGrp}} \left(  G^{k}(M) , U(1) \right)
\end{equation}
comparing the higher holonomy morphisms constructed in this thesis on the left-hand side with the existing ones presented in \cite{GajerHigher} in the case the manifold $M$ is $(k-1)$-connected. \\

\section{Motivation}

The main motivation behind \eqref{Equation: The Main isomorphism} is not to
merely formulate differential cohomology in terms of the homotopy theory of
simplicial presheaves, but to provide a model where the differential refinement
is established by refining the space $M$, rather than the coefficient object.
This has the advantage of possible generalizations to non-abelian differential
cohomology and is of particular interest since the Cheeger--Simons differential
characters, while being astonishingly simple to define, can not be generalized
to characterize higher \textit{non-abelian bundles with connection}. Non-abelian
cohomology, on the other side, in its most general form is defined as the
cohomology of a simplicial presheaf taking values in a non-abelian coefficient
object such as 2-groups. This is exactly the framework used to define
$H^k_{\infty}(\mathbb{M}_k,U(1))$. Further, the existing approaches to
non-abelian differential cohomology via transport functors are challenging to
generalize to higher dimensions, as they do not rely on simplicial
constructions.    \\

Another motive is the recent development in geometric field theory by
Grady--Pavlov that uses a similar notion of thinness and is phrased in terms of
simplicial presheaves on Cartesian spaces. This allows for the speculation that
a higher circle bundle with connection viewed as an element of
$H^k_{\infty}(\mathbb{M}_k,U(1))$ might define a geometric field theory. This
field theory assigns to a bordism endowed with a smooth map into the manifold
$M$ the higher parallel transport along the worldvolume parametrized by the
bordism in $M$. This view on (higher) bundles with connection has been
considered by Segal \cite{Segal}, Stolz--Teichner \cite{Stolz-Teichner} and more
recently by Berwick-Evans--Pavlov \cite{Berwick-Pavlov}.  \\

The importance of understanding higher bundles with connection as geometric
field theories is related to the study of $\sigma$-models whose action is not
globally defined\footnote{Such an action is usually called a topological action,
see for example \cite{Gawedski}.}. To fully understand the path-integral
quantization of such a $\sigma$-model, it is necessary to search for a geometry
where this action finds a global interpretation. This geometry is called a
pre-quantum geometry. The easiest example of this kind is given by the
\textit{Dirac charge quantization}, that is the geometric quantization of
ordinary $U(1)$ gauge theory. This is the classical field theory of an electron
propagating in an electromagnetic field on a 4-dimensional spacetime manifold
$(M,g)$. In this setting, the electric field $E$ and the magnetic field $B$ can
be joined locally to give a differential $2$-form known as the \textit{Faraday
tensor} $F = B + E \wedge dt$. Then Maxwell's equations take the form of $dF  =
0$ and $d *F = J$ where here $*$ denotes the Hodge star operator determined by
the pseudo-Riemannian metric $g$ and $J$ denotes the electric current. \\
\newpage

The propagation of a particle of mass $m$ and charge $q$ on a 4-dimensional
pseudo-Riemannian manifold $(M,g)$ is described via a classical $\sigma$-model.
That is, we fix a \textit{target space} $(M,g)$ considered to be spacetime and a
\textit{parameter space} $\Sigma = \mathbb{R}$, also called worldline. The
parameter space $\Sigma$ is thought of as the abstract trajectory of the point
particle. In the presence of an electromagnetic field, the kinetic action
functional
\[
\gamma \in C^{\infty}(\Sigma,M) \longmapsto S_{\mathrm{kin}}(\gamma) := m \int_{\Sigma} \mathrm{dvol}(\gamma^*g)
\]
needs to be extended or \textit{coupled} with an appropriate \textit{gauge action} $S_{\mathrm{gauge}}(-)$. Using the fact that $dF = 0$ we can find an open covering $\left\{ U_i \right\}$ of $M$ such that locally $dA_i = F|_{U_i}$. This allows us to define the amplitude of the gauge action locally, i.e. assuming we are given $\gamma \colon \Sigma \rightarrow U_i$ by
\[
\mathrm{exp} \left( 2\pi i S_{\mathrm{gauge}}(\gamma) \right)= \mathrm{exp} \left( 2\pi i q \int_{\Sigma} \gamma^* A_i \right),
\]
or globally in the case $\gamma$ is a contractible loop via an extension $\tilde{\gamma} \colon D^2 \rightarrow M$ by
\begin{equation} \label{Equation: Contractible loop holonomy}
\mathrm{exp} \left( 2\pi i S_{\mathrm{gauge}}(\gamma) \right)= \mathrm{exp} \left( 2\pi i q \int_{D^2} \tilde{\gamma}^* F \right).
\end{equation}

However, to extend the gauge action of a charged relativistic particle in the presence of an electromagnetic background field $F$ to a more general class of paths, the appropriate geometry is provided by a principal $U(1)$-bundle with connection $A$ and curvature $F$. Then the amplitude is defined for any closed loop $\gamma \colon S^1 \rightarrow M$ by
\begin{align*}
\mathrm{exp} \left( 2\pi i S(\gamma) \right) &= \mathrm{exp}( 2 \pi i S_{\mathrm{kin}}(\gamma))\mathrm{exp}(2 \pi i S_{\mathrm{gauge}}(\gamma)) \\
&= \mathrm{exp}(2 \pi i S_{\mathrm{kin}}(\gamma)) \mathrm{hol}(A,\gamma),
\end{align*}
where the amplitude of the gauge action equals the holonomy of the of the connection $A$ along the loop $\gamma$. The condition for $F$ to be the curvature of a connection of a $U(1)$-bundle is also known as \textit{Dirac's charge quantization condition}. It is precisely this condition that makes the amplitude \eqref{Equation: Contractible loop holonomy} well defined, i.e. independent of the choice of extension $\tilde{\gamma}$.  \\

The observation that for closed loops $\gamma$ the amplitude $\mathrm{exp}(2 \pi i S_{\mathrm{gauge}}(\gamma)) = \mathrm{hol}(A,\gamma)$ is given by the holonomy, indicates that the global gauge action assigns to an open path $\gamma \colon \Sigma \rightarrow M$ the associated parallel transport along $\gamma$. The interpretation of the exponentiated gauge action functional as parallel transport suggests that the $\sigma$-model of a charged particle propagating on a spacetime manifold $M$ in the presence of a background gauge field $A$ is fully captured by an associated \textit{geometric field theory over the target manifold} $M$ encoding the data of the circle bundle with connection $A$. \\

Looking at higher dimensional $\sigma$-models in the presence of higher background gauge fields, such as the Wess--Zumino--Witten $\sigma$-model, the global definition of the gauge action requires likewise a higher geometrical object. The background gauge field is then described by a connection on the higher circle bundle. In dimension $d = 2$ the WZW $\sigma$-model requires the notion of a bundle gerbe with connection, as observed first by Gawȩdzki in \cite{Gawedski}. The need for even higher principal bundles with connections then appeared in the study of the associated $\sigma$-model to $3$-dimensional classical Chern--Simons theory, where the action can be interpreted as the parallel transport of a connection on the Chern--Simons circle 3-bundle \cite{Gomi, CJMSW}. The vast applications of higher gauge theory in mathematical physics then motivated further research in particular on the relationship between parallel transport and higher connections and their non-abelian analogs by the program of Schreiber--Waldorf \cite{Schreiber-Waldorf-NonAb}. \\

So far, the higher bundles and gauge fields considered were always abelian. However, the worldvolume theory of several coincident five-branes in M-theory possibly contains a non-abelian higher gauge field, as argued in \cite{Aschieri, Fiorenza-Sati-Schreiber}. It is therefore crucial to also understand the role of higher dimensional non-abelian parallel transport. The appearance of higher gauge fields in string theory and supergravity always served as a major motivation behind the development of higher gauge theory. As such, it is one of the various mathematical theories that grew out of the very rich interaction between physics and mathematics.

\section{Outline}

Chapter \ref{Chapter: Higher Gauge Theory} introduces the most common models of higher circle bundles with connection. Starting with ordinary circle bundles with connection we then introduce bundle gerbes, bundle gerbe connections, and curvings following \cite{Murray}. Next, we introduce the smooth Deligne complex and show how a degree 2 \v{C}ech-Deligne cocycle encodes the data of a bundle gerbe with connection and curving. To conclude the chapter we present a conceptually much simpler model equivalent to smooth Deligne cohomology, the \textit{Cheeger--Simons differential characters} and finish by introducing Gajer's higher holonomy morphisms in the last section. \\

Chapter \ref{Chapter: Diffeological Spaces} gives a short introduction to diffeological spaces and their characterization as concrete sheaves on the site of Cartesian spaces. We follow here the outline of Baez--Hoffnung \cite{BaezHoffnung}. After a short section on diffeological principal bundles, we switch to the homotopy theory of diffeological spaces. Here the main reference is Christensen--Wu \cite{Christensen-Wu}. \\

Chapter \ref{Chapter: Local Homotpy Theory} starts with a short introduction to model categories with a particular emphasis on model structures on functor categories. Then we introduce the Local and \v{C}ech-local model structure on the category of simplicial presheaves on a site. For the local model structure we follow Jardine \cite{JardineBook}. In the remaining section of the chapter, we introduce the cohomology of simplicial presheaves together with the main computational tool later used, the universal coefficient spectral sequence. This section finishes with the computation of extensions in the abelian category of sheaves of abelian groups on Cartesian spaces. \\

The proof of the main result \eqref{Equation: The Main isomorphism} then occupies all of Chapter \ref{Chapter: Skeletal Diffeologies and Differential Characters}. First, we use the results of the previous chapter to compute the cohomology $H^k_{\infty}(\mathbb{M}_k,U(1))$ via the universal coefficient spectral sequence. Then the isomorphism \eqref{Equation: The higher holonomy isomorphism} is constructed generalizing ideas of Schreiber--Waldorf \cite{Schreiber-Waldorf-Parallel}. More precisely, we first assign to any smooth holonomy homomorphism $h$ a curvature form $A_h$. To check that the form $A_h$ is indeed a curvature form in the sense of Cheeger--Simons, we use a coordinate-free variant of the Riemann integral. This approach, due to \cite{Lackman}, defines the Riemann sums via the Van Est homomorphism for Lie groupoids. \\

In Chapter \ref{Chapter: Geometric Loop Groups} we introduce the geometric loop group $G(M,x_0)$ of Gajer \cite{GajerHigher} and show that for any diffeological space $(X,x_0)$ there is a surjective homomorphism \newline $\theta_1 \colon G(X,x_0) \rightarrow \pi_1(X_1,x_0)$. We also argue why the proof of the contractibility of $E(M,x_0)$ is incomplete, which, however, does not affect the construction of the morphism $\theta_1$. This morphism is then used iteratively to construct the connecting homomorphism \eqref{Equation: Connecting Homomorphism} in the case $M$ is an $(k-1)$-connected manifold.  \\

The thesis then concludes with an outlook \ref{Chapter: Outlook} motivating how to extend the model $H^k_{\infty}(\mathbb{M}_k,U(1))$ to the non-abelian setting. That is, we present an outline of how to find a possible differential refinement for degree 2 non-abelian cohomology $H^2_{\infty}(M,\mathfrak{G})$ for $\mathfrak{G}$ some Lie 2-group via an iterated application of skeletal diffeologies of increasing dimension.

\chapter{Higher Gauge Theory} \label{Chapter: Higher Gauge Theory}

As outlined in the introduction, higher gauge theory serves as an essential tool
to fully understand the path integral quantization of higher dimensional
classical field theories, in particular their associated notion of
\textit{higher dimensional parallel transport}. We start this chapter by first
recalling ordinary circle bundles and their consecutive higher-dimensional
counterpart called bundle gerbes. We then proceed by introducing the smooth
Deligne complex and argue that the associated \v{C}ech-cohomology
$H^*(M,\mathbb{Z}(n+1)_D^{\infty})$ classifies circle $n$-bundles with
connection over a smooth manifold $M$. Next, we introduce the Cheeger--Simons differential characters
providing an equivalent but technically much simpler model than Deligne
cohomology classes. We conclude by introducing the higher holonomy morphisms of
Gajer based on an iterative application of the geometric loop group $G(M,x_0)$.

\section{Circle \texorpdfstring{$n$}{n}-bundles with connection}

\subsection{Principal Bundles with Connection}

Let us start first with the easiest and most familiar case, the case of principal $U(1)$-bundles with connection. As we eventually wish to give a classification using \v{C}ech cohomology, we define bundles with connections in terms of local data with respect to an open covering. 

\begin{definition} \label{Definition: Local Connection Data n=1}
    Let $M$ be a smooth manifold and let $\left\{ U_i \right\}_{i \in I}$ be a good open cover of $M$. The \textbf{local data of a principal} $U(1)$\textbf{-bundle with connection over } $M$ \textbf{with respect to} $\left\{ U_i \right\}_{i \in I}$ is given by:
    \begin{itemize}
        \item a collection of smooth functions $g_{ij} \colon U_{ij} \rightarrow U(1)$ defined on every two-fold intersection $U_{ij}$ called the transition functions, 
        \item a collection of differential forms $A_i \in \Omega^1(U_i)$ called local connection 1-forms, 
    \end{itemize}
    satisfying the following relations. 
    \begin{itemize}
        \item The cocycle identity holds on every triple intersection $U_{ijk}$, that is
        \[
        g_{jk}g_{ij} = g_{ik}.
        \]
        \item The local connection 1-forms transform on every two-fold intersection $U_{ij}$
        \[
        A_j = A_i  - d\mathrm{log} g_{ij}.
        \]
    \end{itemize}
\end{definition}

\subsection{Circle 2-bundles with Connection}

What we will call a \textit{circle 2-bundle} over a smooth manifold $M$ in this section has appeared under various names in the literature. To clarify, Table \ref{Table: The Zoo of Gerbes} presents an overview of various objects we refer to as higher (nonabelian) principal bundles.

\begin{table}[h!]
\caption {Terminology for gerbes} \label{Table: The Zoo of Gerbes}
\centering
\resizebox{\textwidth}{!}{%
    \begin{tabular}{c|c|c|c|c}
    & Name      & Construction  & Cohomology  & Structure              \\[5pt] 
    \hline
    \multirow{2}{*}{ \rule{0pt}{6ex} Giraud} & \rule{0pt}{5ex} \begin{tabular}[c]{@{}c@{}}gerbe banded by $G$\\ over $M$\end{tabular}  & \begin{tabular}[c]{@{}c@{}}sheaf of groupoids on $M$\\ i.e. a stack on $M$\end{tabular}  & $\check{H}^1(M,\mathrm{AUT}(G))$  & nonabelian   \\[10pt] \cline{2-5} 
                            & \rule{0pt}{5ex} \begin{tabular}[c]{@{}c@{}}abelian gerbe banded by $A$\\ over $M$\end{tabular}       & \begin{tabular}[c]{@{}c@{}}sheaf of groupoids on $M$\\ i.e. a stack on $M$\end{tabular}                                                       & $\check{H}^2(M,A)$                                                                       & abelian    \\[10pt] 
    \hline
    
    Brylinski  & \begin{tabular}[c]{@{}c@{}}gerbe with band $A$ \\ over $M$\end{tabular}  & \rule{0pt}{5ex} \begin{tabular}[c]{@{}c@{}}sheaf of groupoids on $M$\\ i.e. a stack on $M$\end{tabular} & $\check{H}^2(M,A)$  & abelian     \\[10pt] 
    \hline
     Murray  & bundle gerbe over $M$  & \rule{0pt}{8ex}  \begin{tabular}[c]{@{}c@{}}local principal $U(1)$-bundles\\ satisfying coherent gluing \\ conditions via the  \\ bundle gerbe product\end{tabular} & \begin{tabular}[c]{@{}c@{}}$\check{H}^2(M,U(1))$ \\ $\cong H^3(M,\mathbb{Z})$\end{tabular} & abelian \\[26pt]
     \hline
     Bartels & \begin{tabular}[c]{@{}c@{}}principal $\mathfrak{G}$-2-bundle\\ over $M$\end{tabular} & \rule{0pt}{5ex}  \begin{tabular}[c]{@{}c@{}}categorified principal bundle\\ with structure Lie 2-group $\mathfrak{G}$\end{tabular}  & $\check{H}^1(M,\mathfrak{G})$  & nonabelian
    \end{tabular}
}
\end{table}
The first notion to
appear was that of a \textit{gerbe}, introduced by Giraud \cite{Giraud} to
construct nonabelian degree 2 cohomology. More precisely, considering a
nonabelian Lie group $G$ and some base manifold $M$, Giraud defined the nonabelian \v{C}ech cohomology $H^1(M, \mathrm{AUT}(G))$ to be the isomorphism classes of \textit{gerbes banded by} $G$ over $M$. A gerbe banded by $G$ is a special type of sheaf of groupoids over
$M$, i.e. what we call nowadays a \textit{stack} over $M$. The nonabelian cohomology classifying gerbes banded by $G$ has coefficients given by the automorphism Lie 2-group $\mathrm{AUT}(G)$ of $G$, which is why in \cite{Giraud} this cohomology is referred to as nonabelian degree 2 cohomology. Given $A$ an abelian Lie group, Giraud also defined \textit{abelian gerbes banded by} $A$ which are then classified by the usual degree 2 \v{C}ech cohomology $\check{H}^2(M,A)$ with coefficients in the sheaf $A$. It is important to note that for $A$ an abelian Lie group, the notion of a gerbe banded by $A$ and the notion of an abelian gerbe banded by $A$ are \textbf{not} equivalent. This is because for gerbes banded by $A$ the structure Lie 2-group is $\mathrm{AUT}(A)$, whereas for abelian gerbes banded by $A$, it is the delooping $BA$ of $A$. \\

Later, Brylinski and Deligne \cite{Brylinski} showed that degree 2 Deligne cocycles have a geometric interpretation in terms of \textit{gerbes with band} $U(1)$ endowed with \textit{a connective structure and curving}. In their terminology, gerbes with band $U(1)$ correspond to abelian gerbes banded by $U(1)$ in the sense of Giraud, and as such are classified by $\check{H}^2(M,U(1)) \cong H^3(M,\mathbb{Z})$. The integral cohomology class associated to any gerbe with band $U(1)$ is called the Dixmier--Douady class and is a generalization of the Chern class of a principal $U(1)$-bundle to one degree higher. \\

Giraud's approach to gerbes using sheaves of groupoids has the drawback of being rather involved and
obscuring the geometric intuition. In the case of the circle group $U(1)$, Murray \cite{Murray} introduced a new object also providing a geometric interpretation of the integral cohomology $H^3(M,\mathbb{Z})$ called \textit{bundle gerbes}. It is therefore no surprise that they are equivalent to abelian gerbes banded by $U(1)$. This much simpler construction offers a better geometric understanding and has been generalized to the nonabelian setting by Bartels \cite{Bartels} introducing principal $\mathfrak{G}$-2-bundles given any Lie 2-group $\mathfrak{G}$. These objects are then classified by the nonabelian \v{C}ech cohomology $\check{H}^1(M,\mathfrak{G})$ as shown by \cite{Nikolaus-Waldorf}. \\

Even if Murray's bundle gerbes provided an easier approach to abelian gerbes, it is still difficult to generalize them to higher categorical degrees. To overcome this issue, Gajer \cite{Gajer} constructed a differentiable structure on the iterated classifying space $B^nU(1)$ and showed that, in any degree, cohomology classes in $H^n(M, \mathbb{Z})$ correspond to isomorphism classes of \textit{smooth principal} $B^{n-2}U(1)$\textit{-bundles over} $M$. Further, he showed that the smooth Deligne complex provides the desired differential refinement of integral cohomology, i.e. the cohomology groups $H^{n+1}(M, \mathbb{Z}(n+1)_D^{\infty})$ classify isomorphism classes of smooth principal $B^{n-1}U(1)$-bundles with connections. \\

\begin{table}[h!]
\caption {Higher $U(1)$-bundles} \label{Table: Higher circle bundles}
\centering
\begin{tabular}{c|ccc|c}
degree  &  \multicolumn{3}{c|}{topological object}  & cohomology class  \\[5pt] 
\hline
\rule{0pt}{4ex} $n =1$  & \multicolumn{3}{c|}{principal $U(1)$-bundle} & $H^2(M,\mathbb{Z})$  \\[10pt] 
\hline
$n=2$   & \multicolumn{1}{c|}{\begin{tabular}[c]{@{}c@{}}abelian gerbe \\ banded by $U(1)$\end{tabular}}   & \multicolumn{1}{c|}{bundle gerbe}  & \begin{tabular}[c]{@{}c@{}}smooth principal \\ $BU(1)$-bundle\end{tabular}  & $H^3(M,\mathbb{Z})$   \\ 
\hline
$n = 3$ & \multicolumn{1}{c|}{\begin{tabular}[c]{@{}c@{}}abelian 2-gerbe \\ banded by $U(1)$\end{tabular}} & \multicolumn{1}{c|}{bundle 2-gerbe}  & \begin{tabular}[c]{@{}c@{}}smooth principal \\ $B^2U(1)$-bundle\end{tabular}     & $H^4(M,\mathbb{Z})$  \\ 
\hline
\multicolumn{1}{c|}{$\vdots$} & \multicolumn{1}{c|}{$\vdots$}  & \multicolumn{1}{c|}{\vdots} & $\vdots$   & \multicolumn{1}{c}{$\vdots$} \\ 
\hline
$n$ & \multicolumn{1}{c|}{\begin{tabular}[c]{@{}c@{}}abelian $(n-1)$-gerbe \\ banded by $U(1)$\end{tabular}}  & \multicolumn{1}{c|}{?} & \begin{tabular}[c]{@{}c@{}}smooth principal \\ $B^{n-1}U(1)$-bundle\end{tabular} & $H^{n+1}(M,\mathbb{Z})$                  
\end{tabular}
\end{table}

Table \ref{Table: Higher circle bundles} gives an overview of various constructions of higher circle bundles that have appeared in the literature, all allowing for classification by integral cohomology. In the first column, we have Giraud's abelian gerbes banded by $U(1)$, which have been generalized to abelian 2-gerbes banded by $U(1)$ by Breen \cite{Breen}. The full generalization for arbitrary degree is then given by the notion of an (abelian)\footnote{In \cite{Lurie} these objects are simply called $n$-gerbes. However, to avoid confusion with other versions of gerbes introduced earlier, we call them here abelian $n$-gerbes banded by $U(1)$.} $n$-gerbe banded by $U(1)$ in the $\infty$-topos $\mathrm{sSh}(M)$ of simplicial sheaves on $M$ due to Lurie \cite[Definition 7.2.2.20]{Lurie}. Equivalence classes of (abelian) $n$-gerbes on $\mathrm{sSh}(M)$ banded by $U(1)$ are classified by $H^{n+1}(M,\mathbb{Z})$ as a direct consequence of \cite[Corollary 7.2.2.27]{Lurie}. For more details on the $\infty$-topos of simplicial sheaves on $M$, we refer to \cite[Section 6.5.4]{Lurie}. The next column starts with Murray's notion of bundle gerbes, which has been generalized to bundle 2-gerbes by Stevenson \cite{StevensonArticle}. The last column features the construction of smooth principal $B^nU(1)$-bundles by Gajer. \\

Table \ref{Table: Higher circle bundles with connection} gives a similar overview, this time featuring the various notions of higher connections appearing in the literature living on the higher circle bundles of Table \ref{Table: Higher circle bundles}. These all share the property of being classified by smooth Deligne cohomology.

\begin{table}[htb]
\caption {Higher $U(1)$-bundles with connection} \label{Table: Higher circle bundles with connection}
\centering
\resizebox{\textwidth}{!}{%
\begin{tabular}{c|ccc|c}
degree  & \multicolumn{3}{c|}{geometric object} & cohomology class \\[5pt] 
\hline
\rule{0pt}{4ex} $n =1$  & \multicolumn{3}{c|}{\begin{tabular}[c]{@{}c@{}}principal $U(1)$-bundle\\ with connection\end{tabular}} & $H^2 \left( M;\mathbb{Z}(2)_D^{\infty}\right)$  \\[10pt] 
\hline
$n=2$   & \multicolumn{1}{c|}{\begin{tabular}[c]{@{}c@{}}abelian gerbe banded by $U(1)$\\ with connection and curving\end{tabular}} & \multicolumn{1}{c|}{\begin{tabular}[c]{@{}c@{}}bundle gerbe with\\ connection and curving\end{tabular}}     & \begin{tabular}[c]{@{}c@{}}smooth principal $BU(1)$-bundle\\ with $k$-connections for $k=1,2$\end{tabular} & $H^3 \left( M;\mathbb{Z}(3)_D^{\infty}\right)$  \\
 \hline
$n = 3$ & \multicolumn{1}{c|}{?} & \multicolumn{1}{c|}{\begin{tabular}[c]{@{}c@{}}bundle 2-gerbe with\\ connection and 2-curving\end{tabular}} & \begin{tabular}[c]{@{}c@{}}smooth principal $B^2U(1)$-bundle\\ with $k$-connections for $k=1,2,3$\end{tabular}       & $H^4 \left( M;\mathbb{Z}(4)_D^{\infty}\right)$       \\ 
\hline
\multicolumn{1}{c|}{$\vdots$} & \multicolumn{1}{c|}{$\vdots$} & \multicolumn{1}{c|}{$\vdots$}  & $\vdots$  & \multicolumn{1}{c}{$\vdots$}  \\ 
\hline
$n$  & \multicolumn{1}{c|}{?} & \multicolumn{1}{c|}{?} & \begin{tabular}[c]{@{}c@{}}smooth principal $B^{n-1}U(1)$-bundle\\ with $k$-connections for $k=1,...,n$\end{tabular} & $H^{n+1} \left( M;\mathbb{Z}(n+1)_D^{\infty}\right)$
\end{tabular}%
}
\end{table}

For a better understanding of circle 2-bundles, we introduce the notion of bundle gerbes as they are the easiest to work with. Then bundle gerbe connections and curving are defined and we outline how to extract the local connection data similar to the case of local connection 1-forms for principal bundles. This paves the way for the \v{C}ech--Deligne complex presented in the next section as a generalized approach to higher local connection data. 

\begin{definition}[\cite{Murray}] \label{Definition: Bundle Gerbe}
    A \textbf{bundle gerbe} on a smooth manifold $M$ is a triple $(P,\pi, M)$ where $\pi \colon Y \rightarrow M$ is a surjective submersion and $P$ a principal $U(1)$-bundle over $Y^{[2]}$ endowed with a \textbf{bundle gerbe product}. That is, an isomorphism of principal $U(1)$-bundles over $Y^{[3]}$ of the form 
    \[
    m : \pi_1^*P \otimes \pi_3^*P \rightarrow \pi_2^*P,
    \]
    such that the following associativity condition is satisfied. For every $(y_1,y_2,y_3,y_4) \in Y^{[4]}$ commutativity holds in the diagram  
    \begin{center}
    \begin{tikzcd}
{P_{(y_1,y_2)} \otimes P_{(y_2,y_3)} \otimes P_{(y_3,y_4)}} \arrow[d, "{\mathrm{id} \otimes m_{(y_2,y_3,y_4)}}"'] \arrow[r, "{m_{(y_1,y_2,y_3)} \otimes \mathrm{id}}"] & {P_{(y_1,y_3)} \otimes P_{(y_3,y_4)}} \arrow[d, "{m_{(y_1,y_3,y_4)}}"] \\
{P_{(y_1,y_2)} \otimes P_{(y_2,y_4)}} \arrow[r, "{m_{(y_1,y_2,y_4)}}"']                                                                                                & {P_{(y_1,y_4)}}                                                       
\end{tikzcd}
\end{center}
\end{definition}
\begin{remark} $ $ 
    \begin{enumerate}[label={(\arabic*)}]
        \item The monoidal structure used to define the bundle gerbe product is given by
        the contracted product of $U(1)$-bundles. More precisely, given two principal $U(1)$-bundles $P$ and $Q$ represented 
        by their transition functions $g_{ij}$ and $h_{ij}$ respectively, their \textbf{contracted product} denoted $P \otimes Q$ 
        is defined as the bundle with transition functions $g_{ij}h_{ij}$.
        The \textbf{inverse} or \textbf{dual} principal $U(1)$-bundle denoted $P^*$ is defined as the bundle with transition functions 
        $g^{-1}_{ij}$. 
        
        \item Given a surjective submersion $\pi \colon Y \rightarrow M$ we denote by $Y^{[n]}$ the $n$-fold fiber product given by
        \[
        Y^{[n]} := \underbrace{Y \times_{M} \cdots \times_{M} Y}_{n\text{-times}} 
        \]
        which is equipped with projection maps $\pi_i \colon Y^{[n]} \rightarrow Y^{[n-1]}$ for $1 \leq i \leq n$.

        \item Given $P$ a principal
        $U(1)$-bundle over $Y^{[2]}$, denote by $\delta(P)$ the $U(1)$-bundle over
        $Y^{[3]}$ given by 
        \[
        \delta(P) := \pi^*_1 P \otimes \left( \pi_2^* P \right)^* \otimes \pi_3^*P. 
        \]
        The existence of a bundle gerbe product $m$ then forces the bundle $\delta(P)$ to be trivial and as such admits a section $s$. 
    \end{enumerate}
\end{remark}

A nice diagrammatical interpretation of a bundle gerbe $(P,\pi,M)$ on the manifold $M$ is
\begin{center}
\begin{tikzcd}
P \arrow[d]                                                                   &                    \\
{Y^{[2]}} \arrow[r, "\pi_1", shift left=2] \arrow[r, "\pi_2"', shift right=2] & Y \arrow[d, "\pi"] \\
                                                                              & M                 
\end{tikzcd}
\end{center}

As indicated in Table \ref{Table: Higher circle bundles}, to any bundle gerbe $(P,\pi,M)$ we can associate a cohomology class $DD(P,\pi,M) \in H^3(M;\mathbb{Z})$ being the analog of the Chern class associated to a principal $U(1)$-bundle in degree $n=2$. This characteristic class, called the Dixmier--Douady class of the bundle gerbe, plays an important role in understanding the differential cohomology hexagon introduced in Section \ref{Section: Differential Characters}. This will also be highlighted in Lemma \ref{Lemma: Dixmier-Douady Class and Curvarture}.  \\

\begin{construction}
    Let $(P,\pi,M)$ be a bundle gerbe and let $\cat{U} = \left\{ U_i \right\}_{i \in I}$ be a good open covering such that $\pi \colon Y \rightarrow M$ admits local sections $s_i \colon U_i \rightarrow Y$. These local sections provide maps 
    \[
    \begin{array}{rcl}
    (s_i,s_j) \colon U_{ij} & \rightarrow & Y^{[2]}  \\
         x & \mapsto & (s_i(x),s_j(x)). 
    \end{array}
    \]
    The pullback of the principal $U(1)$-bundle $P$ via $(s_i,s_j)$ is denoted by $P_{ij}$. As we have chosen $\cat{U}$ to be a good open covering, the two-fold intersections all are diffeomorphic to some cartesian space $\mathbb{R}^n$, and therefore the bundles $P_{ij} \to U_{ij}$ are all trivializable. That is, there exist sections 
    \[
    \sigma_{ij} \colon U_{ij} \rightarrow P_{ij}
    \]
    Now define a smooth map $g_{ijk} \colon U_{ijk} \rightarrow U(1)$ by setting
    \[
    m(\sigma_{jk} \otimes \sigma_{ij}) = \sigma_{ik}g_{ijk}.
    \]
    The associativity condition which the bundle gerbe product $m$ satisfies further implies that the maps $g_{ijk}$ indeed satisfy the cocycle condition
    \[
    g_{jkl}g^{-1}_{ikj}g_{ijl}g^{-1}_{ijk} = 1.
    \]
    Hence we have constructed a representative for a class in the \v{C}ech cohomology group of $M$ taking values in the sheaf $U(1)$
    \[
    [g_{ijk}] \in \check{H}^2 \left( M, U(1) \right).
    \]
    Now using the classical fact that the structure sheaf $\mathbb{R}$ of a smooth manifold $M$ is soft, the exact sequence of sheaves 
    \[
    0 \rightarrow \mathbb{Z}\rightarrow \mathbb{R} \rightarrow U(1) \rightarrow 0
    \]
    induces an isomorphism
    \[
    \begin{array}{rcl}
         \check{H}^2\left( M, U(1)\right)  & \xrightarrow{\cong} & H^3(M,\mathbb{Z}) \\
         {[g_{ijk}]} & \mapsto & DD(P,\pi,M),  
    \end{array}
    \]
    whose image of the class $[g_{ijk}]$ is defined to be the Dixmier--Douady class. 
\end{construction}

\begin{definition}[\cite{Murray}]\label{Definition: Bundle Gerbe Connection}
    A \textbf{bundle gerbe connection} on a bundle gerbe $(P,\pi,M)$ consists of a connection on the principal $U(1)$-bundle $P \rightarrow Y^{[2]}$, i.e. a connection 1-form $A \in \Omega^1(P)$ such that it respects the bundle gerbe product, that is
    \begin{equation} \label{equation: bundle gerbe connection}
    s^* \delta(A) = 0,
       \end{equation}
       where here $\delta(A)$ denotes the induced connection on the bundle $\delta(P) \rightarrow Y^{[3]}$ and $s$ the section of $\delta(P)$ induced by the bundle gerbe product.  
\end{definition}

Given a bundle gerbe $(P,\pi,M)$ with connection $A \in \Omega^1(P)$, the
associated curvature form $F_A \in \Omega^2 \left( Y^{[2]} \right)$ is such that
\[
\delta(F_A) = 0
\]
as a direct consequence of equation (\ref{equation: bundle gerbe connection}).
The exactness of the fundamental complex\footnote{See Example \ref{Example:
Fundamental Complex}} implies that there exists a 2-form $f \in \Omega^2(Y)$
such that 
\[
\delta(f) = F_A.
\]
This form is not unique and a choice of such a form is called a \textit{curving}
for the bundle gerbe connection $A$. Since $F_A$ is a curvature form, it is
closed. This further implies that
\[
d(F_A) = d(\delta(f)) = \delta(df) = 0.
\]
Again by the exactness of the fundamental complex, this implies that there
exists a 3-form $\omega \in \Omega^3(M)$ such that:
\[
\pi^*(\omega) = df.
\]
This 3-form is called the 3\textit{-curvature} of the \textit{connective structure} $(A,f)$ and is a closed 3-form, since
\[
\pi^*(d\omega) = d \pi^*(\omega) = d^2f = 0.
\]
The diagram below gives an overview of how these various differential forms
arising from a bundle gerbe with connection are related to each other, summing
up the above discussion.   

\begin{center}
\begin{tikzcd}[every matrix/.append style={nodes={font=\large}}]
\omega \atop \Omega^3(M) \arrow[r, "\pi^*", maps to] & \pi^* \omega = df \atop \Omega^3(Y) \arrow[r, "\delta", maps to]         & {\delta(df) = dF_A = 0\atop \Omega^3(Y^{[2]})}                                                 &                                          \\
                                                     & f \atop \Omega^2(Y) \arrow[r, "\delta", maps to] \arrow[u, "d", maps to] & {\delta(f) = F_A \atop \Omega^2(Y^{[2]})} \arrow[r, "\delta", maps to] \arrow[u, "d", maps to] & {\delta(F_A) =0 \atop \Omega^2(Y^{[3]})}
\end{tikzcd}
\end{center}

Consider for a moment the case where the surjective submersion $\pi \colon \coprod_{i \in I} U_i \rightarrow M$ is given by an open covering $\left\{ U_i \right\}_{i \in I}$. Then the data of a bundle gerbe $(P,\pi,M)$ with connection and curving reduces to 
\begin{itemize}
    \item local connection data $A_{ij} \in \Omega^1(U_{ij})$ such that 
    \[
    (F_A)_{ij} = dA_{ij},
    \]
    \item local curving data $f_{i} \in \Omega^2(U_i)$ for all $i \in I$ such that 
    \[
    \delta(f) = f_j - f_i = (F_A)_{ij} = dA_{ij}
    \]
    for all $i,j \in I$, and
    \item cocycle data $g_{ijk} : U_{ijk} \rightarrow U(1)$ satisfying 
    \[
    g_{jkl}g^{-1}_{ikj}g_{ijl}g^{-1}_{ijk} = 1.
    \]
\end{itemize}
Moreover, the fact that the connection $A$ respects the bundle gerbe product further implies that for all $i,j,k \in I$
\[
A_{jk} - A_{ik} + A_{ij} = d\mathrm{log}g_{ijk}.
\]
Since any surjective submersion can be "refined" by a surjective submersion arising from a good open cover $\left\{ U_i \right\}$, a bundle gerbe with connection and curving boils down to the described local data.

\begin{definition} \label{Definition: Local Connection Data n=2}
    Let $M$ be a manifold and $\left\{ U_i \right\}_{i \in I}$ be a good open covering of $M$. The local data of a bundle gerbe with connection and curving over $M$ with respect to $\left\{ U_i \right\}_{i \in I}$ is given by: 
    \begin{enumerate}[label=(\alph*)]
        \item On every open set $U_i$ a 2-form $B_i \in \Omega^2(U_i)$
        \item On every two-fold intersection $U_{ij}$ a 1-form $A_{ij} \in \Omega^2(U_{ij})$
        \item On every three-fold intersection $U_{ijk}$ a smooth map $g_{ijk} : U_{ijk} \rightarrow U(1)$
    \end{enumerate}
    such that the following cocycle conditions are satisfied: 
    \begin{enumerate}
        \item On every two-fold intersection $U_{ij}$:
        \[
        dA_{ij} = B_j - B_i
        \]
        \item On every three-fold intersection $U_{ijk}$:
        \[
        A_{ij} - A_{jk} + A_{ki} = d\mathrm{log}g_{ijk}
        \]
        \item On every four-fold intersection $U_{ijkl}$:
        \[
        g_{jkl}g^{-1}_{ikj}g_{ijl}g^{-1}_{ijk} = 1
        \]
    \end{enumerate}
\end{definition}

To conclude this section, we sketch out the correlation between the topological Dixmier--Douady class $DD(P,\pi,M)$ and the geometric class given by the 3-curvature of a bundle gerbe $(P,\pi,M)$ with connective structure $(A,f)$. Recall that the 3-curvature $\omega \in \Omega^3(M)$ being a closed form has an associated real cohomology class under the de Rham theorem which we also denote by $[\omega]$ 
\[
[\omega] \in H_{\mathrm{dR}}^3(M) \mapsto [\omega] \in H^3(M,\mathbb{R}).
\]
The change of coefficients $\mathbb{Z} \hookrightarrow \mathbb{R}$
provides a morphism 
\[
r \colon H^3(M,\mathbb{Z}) \rightarrow H^3(M, \mathbb{R})
\]
and we have the following result identifying the classes $[\omega]$ and $DD(P,\pi,M)$.

\begin{lemma}[\cite{Murray}] \label{Lemma: Dixmier-Douady Class and Curvarture}
    Let $(P,\pi, M)$ be a bundle gerbe with connective structure $(A,f)$ and associated 3-curvature form $\omega$. Then the real cohomology class $[\omega] \in H^3(M,\mathbb{R})$ coincides with the image of the Dixmier--Douady class $DD(P,\pi,M)$ under the change of coefficient map $r$, i.e. 
    \[
    \begin{array}{rcl}
        H^3(M,\mathbb{Z}) & \xrightarrow{r} & H^3(M, \mathbb{R}) \\
        DD(P,\pi,M) & \mapsto & [\omega]
    \end{array}
    \]
    In particular, the class $[\omega]$ is independent of the choice of connective structure on $(P,\pi,M)$. 
\end{lemma}

As we will see in Section \ref{Section: Differential Characters}, this lemma proves that the right-hand square inside the differential cohomology hexagon is indeed commutative, at least in degree $n = 2$.

\subsection{The smooth Deligne Complex}

As already presented in Table \ref{Table: Higher circle bundles with connection} there exists a cohomology theory classifying higher circle bundles with connection data. To introduce this cohomology, recall the notion of \textit{hypercohomology} of a bounded below complex of sheaves $K^{\bullet}$ on a manifold $M$, denoted by
\[
H^*(M; K^{\bullet}).
\]
This short introduction follows \cite{Brylinski} and \cite{Stevenson}. \\

First, recall the definition of the \v{C}ech complex associated to a sheaf of abelian groups $A$ provided an open covering $\cat{U} = \left\{ U_i \right\}_{i \in I}$ of $M$.

\begin{definition}
    Given a sheaf of abelian groups $A$ on a smooth manifold $M$ and $\cat{U} = \left\{ U_i \right\}_{i \in I}$ an open covering, for $p \geq 0$ define 
    \[
    C^p(\cat{U},A) := \prod_{i_0, ..., i_p} A(U_{i_0, ..., i_p})
    \]
    to be the product ranging over $(p+1)$-tuples of elements of $I$. An element of $C^p(\cat{U},A)$ is therefore a family $\alpha = \left\{ \alpha_{i_0, ..., i_p} \in A(U_{i_0, ..., i_p})\right\}$. The \v{C}ech coboundary morphism is then defined by
    \[
    \begin{array}{rcl}
         \delta \colon C^p(\cat{U},A) & \rightarrow & C^{p+1}(\cat{U},A)  \\
         \alpha_{i_0, ..., i_p} & \mapsto & \sum^{p+1}_{j = 0} (-1)^j \alpha_{i_0, ..., i_{j-1}, i_{j+1}, ..., i_{p+1}}.
    \end{array}
    \]
    In particular, one has $\delta^2 = 0$, and the corresponding complex is called the \textbf{\v{C}ech complex of the sheaf} $A$ \textbf{with respect to the open covering} $\cat{U}$. 
\end{definition}

\begin{example} \label{Example: Fundamental Complex}
The \v{C}ech complex of the sheaf of differential $q$-forms $\Omega^q(-)$ on some smooth manifold $M$ with respect to an open covering $\cat{U}$ is also called the \textbf{fundamental complex}. 
\[
0 \rightarrow \Omega^q(M) \xrightarrow{\delta} \Omega^q\left( \coprod_{i_0} U_{i_0} \right) \xrightarrow{\delta} \Omega^q \left( \coprod_{i_0,i_1} U_{i_0,i_1} \right) \xrightarrow{\delta} \cdots
\]
Moreover, this complex is exact, which follows from the classical fact that the sheaf of differential $q$-forms $\Omega^q(-)$ on a smooth manifold is soft\footnote{For the definition of a \textit{soft sheaf} see \cite[Definition 1.4.5]{Brylinski}. The fact that the sheaf of $q$-forms on a smooth manifold is soft is due to \cite[Theorem 1.4.15]{Brylinski}.} for all $q \geq 0$.   
\end{example}

Let $K^{\bullet} = \left[ K^0 \xrightarrow{d} K^1 \xrightarrow{d} \cdots \right]$ be a bounded below complex of sheaves of abelian groups on $M$. Given an open covering $\cat{U} = \left\{ U_i \right\}_{i \in I}$ of $M$, consider the associated \v{C}ech double complex: 
\begin{center}
    \begin{tikzcd}
                           & \vdots                                                        & \vdots                                                    & \vdots                                                        &        \\
\cdots \arrow[r, "\delta"] & {C^{p-1}(\cat{U},K^{q+1})} \arrow[u, "d"] \arrow[r, "\delta"] & {C^p(\cat{U},K^{q+1})} \arrow[u, "d"] \arrow[r, "\delta"] & {C^{p+1}(\cat{U},K^{q+1})} \arrow[u, "d"] \arrow[r, "\delta"] & \cdots \\
\cdots \arrow[r, "\delta"] & {C^{p-1}(\cat{U},K^q)} \arrow[r, "\delta"] \arrow[u, "d"]     & {C^p(\cat{U},K^q)} \arrow[r, "\delta"] \arrow[u, "d"]     & {C^{p+1}(\cat{U},K^q)} \arrow[r, "\delta"] \arrow[u, "d"]     & \cdots \\
\cdots \arrow[r, "\delta"] & {C^{p-1}(\cat{U},K^{q-1})} \arrow[u, "d"] \arrow[r, "\delta"] & {C^p(\cat{U},K^{q-1})} \arrow[u, "d"] \arrow[r, "\delta"] & {C^{p+1}(\cat{U},K^{q-1})} \arrow[u, "d"] \arrow[r, "\delta"] & \cdots \\
                           & \vdots \arrow[u, "d"]                                         & \vdots \arrow[u, "d"]                                     & \vdots \arrow[u, "d"]                                         &       
\end{tikzcd}
\end{center}
This first quadrant double complex is denoted by $C^{\bullet}(\cat{U}, K^{\bullet})$. Its associated total complex $\mathrm{Tot}(C^{\bullet}(\cat{U}, K^{\bullet}))$ with differential $D = \delta + (-1)^pd$ in bidegree $(p,q)$ then gives rise to the \v{C}ech cohomology groups 
\[
\check{H}^*(\cat{U};K^{\bullet}) := H^*\left( \mathrm{Tot}(C^{\bullet}(\cat{U}, K^{\bullet})) , D \right).
\]

\begin{definition} \label{Definition: Cech Hypercohomology}
    Let $K^{\bullet}$ be a bounded below complex of sheaves of abelian groups on a smooth manifold $M$. The \textbf{\v{C}ech hypercohomology groups} $\check{H}^n(M;K^{\bullet})$ are defined as the colimit of abelian groups
    \[
    \check{H}^p(M;K^{\bullet}) := \underset{\cat{U}}{\mathrm{colim}} \; H^p(\cat{U};K^{\bullet})
    \]
    taken over the directed set of open coverings $\cat{U}$ of $M$. 
\end{definition}

The smooth Deligne cohomology is now defined as the \v{C}ech hypercohomology of a specific complex of sheaves.

\begin{definition} \label{Definition: Smooth Deligne-Beilinson Complex}
    For $M$ a smooth manifold and $n \in \mathbb{N}$, the \textbf{smooth Deligne complex of} $M$ is the complex of sheaves of abelian groups on $M$ given by
    \begin{center}
    \begin{tikzcd}[row sep=tiny, column sep =small,/tikz/column 1/.append style={anchor=base east}]
        0 & 1                                    & 2                     &                  & n                         & n+1                \\
    \mathbb{Z}(n+1)_D^{\infty}: = \scalebox{1.5}[1.8]{\normalfont{[}} \mathbb{Z} \arrow[r, hook] & {C^{\infty}(-,\mathbb{R})} \arrow[r, "d"] & \Omega^1(-) \arrow[r, "d"] & \cdots \arrow[r, "d"] & \Omega^{n-1}(-) \arrow[r, "d"] & {\Omega^n(-)} \scalebox{1.5}[1.8]{\normalfont{]}}. 
    \end{tikzcd}
    \end{center}
    The associated \textbf{smooth Deligne cohomology of} $M$ is then denoted by 
    \[
    H^* \left( M, \mathbb{Z}(n+1)^{\infty}_D \right).
    \]
\end{definition}

Note, that the smooth Deligne cohomology can also be expressed as the hypercohomology of a slightly modified complex of sheaves. This is recognized as the standard definition for smooth Deligne cohomology. 
\newpage

\begin{definition} \label{Definition: Smooth Deligne-Beilinson Complex}
    For $M$ a smooth manifold and $n \in \mathbb{N}$, consider the complex of sheaves of abelian groups on $M$ given by
    \begin{center}
    \begin{tikzcd}[row sep=tiny, column sep =scriptsize,/tikz/column 1/.append style={anchor=base east}]
        0 \hspace{1cm} & 1                                    & 2                     &                  & n-1                         & n                \\
    B^n_{\nabla}U(1) :=  \scalebox{1.5}[1.8]{\normalfont{[}} C^{\infty}\left(-,U(1)\right) \arrow[r, "d\mathrm{log}"] & \Omega^1(-) \arrow[r, "d"] & \Omega^2(-) \arrow[r, "d"] & \cdots \arrow[r, "d"] & \Omega^{n-1}(-) \arrow[r, "d"] & {\Omega^n(-)} \scalebox{1.5}[1.8]{\normalfont{]}}. 
    \end{tikzcd}
    \end{center}
    The associated cohomology groups of $M$ with values in the above complex are then denoted by 
    \[
    H^* \left( M; B^n_{\nabla}U(1) \right).
    \]
\end{definition}
\begin{remark} \label{Remark: The dlog map}
    Let us have a closer look at the morphism of sheaves
    \[
    d\mathrm{log} \colon C^{\infty}\left( -, U(1) \right) \rightarrow \Omega^1(-)
    \]
    used to define the smooth Deligne complex. Notice that any smooth function $f : U \rightarrow U(1)$ admits locally a representative in $C^{\infty}(U,\mathbb{R})$. Hence, locally, at a point $x_0 \in U$, the morphism is given by applying the usual de Rham differential to any representative in $C^{\infty}(U,\mathbb{R})_{x_0}$ of an element in $C^{\infty}\left( U, U(1) \right)_{x_0}$. Globally this amounts to assign to the smooth map $f \colon U \rightarrow U(1)$ the pullback $f^*\theta \in \Omega^1(U)$ of the Maurer--Cartan form $\theta$ on $U(1)$.   
\end{remark}

\begin{lemma}[\cite{Brylinski}] \label{Lemma: Quasi-isomorphism Deligne Complex}
    The chain map between the following chain complexes of sheaves 
    \begin{center}
        \begin{tikzcd}
\mathbb{Z} \arrow[r, hook] \arrow[d] & \mathbb{R} \arrow[r, "d"] \arrow[d, "\mathrm{exp}"] & \Omega^1(-) \arrow[r, "d"] \arrow[d, "\mathrm{id}"] & \Omega^2(-) \arrow[r, "d"] \arrow[d, "\mathrm{id}"] & \cdots \arrow[r, "d"] & \Omega^n(-) \arrow[d, "\mathrm{id}"] \\
0 \arrow[r]                          & U(1) \arrow[r, "d\mathrm{log}"]                  & \Omega^1(-) \arrow[r, "d"]                          & \Omega^2(-) \arrow[r, "d"]                          & \cdots \arrow[r, "d"] & \Omega^n(-)                         
\end{tikzcd}
    \end{center}
    is a quasi-isomorphism. In particular, there is an induced isomorphism in cohomology 
    \[
    H^{n+1}(M;\mathbb{Z}(n+1)_D^{\infty}) \cong H^{n+1}(M ; B^n_{\nabla}U(1)[-1]) = H^n(M; B^n_{\nabla}U(1)).
    \]
\end{lemma}
\begin{proof}
    This follows from an adaptation of \cite[Proposition 1.5.2]{Brylinski}. See also \cite[p.~184]{Fiorenza-Schreiber-Stasheff}.
\end{proof}

\begin{remark} \label{Remark: Definition of Deligne Cohomology}
Notice, that the terminology of the degree of a Deligne cocycle depends on the choice of the complex. In the following, the term \textit{Deligne cocycle} or \textit{smooth Deligne cohomology} will always refer to the hypercohomology of the complex $B^n_{\nabla}U(1)$. Accordingly, by a degree $n$-Deligne cocycle we always mean a representative of a class in $H^n(M,B^n_{\nabla}U(1))$.
\end{remark}

To understand what smooth Deligne cohomology classifies, we need to have a look at the associated \v{C}ech--Deligne double complex. In this double complex, the vertical direction is given by the differential of the complex $B^n_{\nabla}U(1)$, whereas the horizontal direction is given by the \v{C}ech coboundary operator $\delta$. \\

Let $\cat{U} = \left\{ U_i \right\}$ be an open covering of $M$, then the associated double complex reads
\begin{center}
\begin{tikzcd}
0 \arrow[r]                                                                                                                    & 0 \arrow[r]                                                                                                                       & 0 \arrow[r]                                                                                                                          & \cdots \\
\Omega^{n+1}\left( \underset{i}{\coprod} U_i \right) \arrow[u, "d"'] \arrow[r, "\delta"]                                       & {\Omega^{n+1}\left( \underset{i,j}{\coprod} U_{ij} \right)} \arrow[u, "d"'] \arrow[r, "\delta"]                                   & {\Omega^{n+1}\left( \underset{i,j,k}{\coprod} U_{ijk} \right)} \arrow[u, "d"'] \arrow[r, "\delta"]                                   & \cdots \\
\vdots \arrow[u, "d"']                                                                                                         & \vdots \arrow[u, "d"']                                                                                                            & \vdots \arrow[u, "d"']                                                                                                               &        \\
\Omega^{2}\left( \underset{i}{\coprod} U_{i} \right) \arrow[u, "d"'] \arrow[r, "\delta"]                                       & {\Omega^{2}\left( \underset{i,j}{\coprod} U_{ij} \right)} \arrow[u, "d"'] \arrow[r, "\delta"]                                     & {\Omega^{2}\left( \underset{i,j,k}{\coprod} U_{ijk} \right)} \arrow[u, "d"'] \arrow[r, "\delta"]                                     & \cdots \\
\Omega^{1}\left( \underset{i}{\coprod} U_{i} \right) \arrow[u, "d"'] \arrow[r, "\delta"]                                       & {\Omega^{1}\left( \underset{i,j}{\coprod} U_{ij} \right)} \arrow[u, "d"'] \arrow[r, "\delta"]                                     & {\Omega^{1}\left( \underset{i,j,k}{\coprod} U_{ijk} \right)} \arrow[u, "d"'] \arrow[r, "\delta"]                                     & \cdots \\
{C^{\infty}\left( \underset{i}{\coprod} U_{i} , U(1) \right)} \arrow[u, "d\mathrm{log}"'] \arrow[r, "\delta"] & {C^{\infty}\left( \underset{i,j}{\coprod} U_{ij} , U(1) \right)} \arrow[u, "d\mathrm{log}"'] \arrow[r, "\delta"] & {C^{\infty}\left( \underset{i,j,k}{\coprod} U_{ijk} , U(1) \right)} \arrow[u, "d\mathrm{log}"'] \arrow[r, "\delta"] & \cdots
\end{tikzcd}
\end{center}
and the total differential of this double complex is given by 
\[
D = d + (-1)^{\mathrm{deg}}\delta.
\]
Let us have a look at some low-dimensional examples. 
\begin{itemize}

\item A \v{C}ech-Deligne cocycle in degree 1 is given by the data $( \left\{ A_i \right\}, \left\{g_{ij} \right\})$ where $A_i \in \Omega^1(U_i)$ and $g_{ij} \in C^{\infty}(U_{ij}, U(1))$ such that:
\begin{center}
    \begin{tikzcd}
\left\{A_i \right\} \arrow[r, maps to, "\delta"] & A_j - A_i = d\mathrm{log}g_{ij}   &    \\
     &  \left\{ g_{ij} \right\} \arrow[u, maps to, "d\mathrm{log}"'] \arrow[r, maps to, "\delta"] &  g_{ij}g_{jk}g^{-1}_{ik} = 1
\end{tikzcd}
\end{center}
which is equivalent to the local data of a principal $U(1)$-bundle with connection introduced in Definition \ref{Definition: Local Connection Data n=1}. 
\item A \v{C}ech-Deligne cocycle in degree 2 is given by the data $( \left\{ B_i \right\}, \left\{ A_{ij} \right\}, \left\{g_{ijk} \right\})$ where $B_i \in \Omega^2(U_i)$, $A_{ij} \in \Omega^1(U_{ij})$ and $g_{ijk} \in C^{\infty}(U_{ijk}, U(1))$ such that:
\begin{center}
    \begin{tikzcd}[column sep = small]
\left\{ B_i \right\} \arrow[r, "\delta", maps to] & B_j - B_i = dA_{ij}                                                           &                                                                                           &                                            \\
                                                  & \left\{A_{ij} \right\} \arrow[u, "d"', maps to] \arrow[r, "\delta", maps to] & A_{jk} - A_{ik} + A_{ij} = d\mathrm{log}g_{ijk}                                          &                                            \\
                                                  &                                                                              & \left\{ g_{ijk} \right\} \arrow[u, "d\mathrm{log}"', maps to] \arrow[r, "\delta", maps to] & g_{jkl}g^{-1}_{ikl}g_{ijl}g^{-1}_{ijk} = 1
\end{tikzcd}
\end{center}
This data is equivalent to the local data of a bundle gerbe with connection and curving introduced in Definition \ref{Definition: Local Connection Data n=2}. 
\end{itemize}

\section{Cheeger--Simons Differential Characters}

\subsection{Differential Characters} \label{Section: Differential Characters}

The abelian group of degree $k$ differential characters $\hat{H}^k(M,\mathbb{Z})$ on a smooth manifold $M$ was introduced by Cheeger and Simons \cite{Cheeger-Simons} as an object providing a refinement for characteristic classes. Around the same time, Deligne introduces the (smooth) Deligne complex $\mathbb{Z}^{\infty}_D(q)$ whose cohomology ring is strongly related to the ring of differential characters. In the subsequent years, various other equivalent cohomology theories appear classifying the same objects as the smooth Deligne cohomology and the Cheeger--Simons differential characters. The work of Simons and Sullivan \cite{Simons-Sullivan} then shows that all these cohomology theories fit into a \textit{character diagram} and that this diagram is sufficient to uniquely characterize such cohomology theories. This \textit{character diagram} nowadays is usually referred to as the \textit{differential cohomology hexagon}. What follows is a brief introduction to differential characters based on the exposition of Bär and Becker \cite{Bär-Becker} and the book by Amabel, Debray, and Haine \cite{Amabel-Debray-Haine}.

\begin{definition} \label{Definition: Smooth singular chains }
    Let $M$ be a smooth manifold. For $k \in \mathbb{N}$, denote by $C_k(M,\mathbb{Z})$ the abelian group of smooth singular $k$-chains in $M$ with integral coefficients. That is, the free abelian group generated by smooth $k$-simplices $|\Delta^k| \rightarrow M$ in $M$. The resulting smooth singular complex is denoted by $(C_k(M,\mathbb{Z}), \partial)$ where $Z_k(M,\mathbb{Z}) = \mathrm{Ker}(\partial)$ denotes the group of $k$-cycles and $B_k(M;\mathbb{Z}) = \mathrm{Im}(\partial)$ the group of $k$-boundaries.  
\end{definition}

\begin{remark}
    For the precise terminology of what we mean by a smooth map $|\Delta^k| \rightarrow M$ defined on the standard compact $k$-simplex $|\Delta^k| \subset \mathbb{R}^k$,  we refer to Lemma \ref{Lemma: Subset diffeology of compact simplex}.  
\end{remark}

\begin{definition}[\cite{Bär-Becker}, Section 5.1] \label{Definition: Differential Character}
    Let $k \geq 1$ be an integer and $M$ a manifold. A \textbf{degree} $k$ \textbf{differential character} on $M$ is a homomorphism of abelian groups
    \[
    \chi : Z_{k-1}(M,\mathbb{Z}) \rightarrow U(1)
    \]
    such that there exists a $k$-form $\omega_{\chi} \in \Omega^k(M)$ with the property that, for every smooth chain $c \in C_k(M;\mathbb{Z})$, we have 
    \[
    \chi(\partial c) = \mathrm{exp} \left( 2\pi i \int_c \omega_{\chi} \right).  
    \]
    Write $\hat{H}^k(M;\mathbb{Z}) \subset \mathrm{Hom}_{\mathbb{Z}}\left( Z_{k-1}(M;\mathbb{Z}), U(1)  \right)$ for the abelian group of degree $k$ differential characters on $M$ with point-wise multiplication. Notice that the differential form $\omega_h$ is uniquely determined, closed, and has integral periods. That is, we have a map
    \[
    \mathrm{curv} : \hat{H}^k(M;\mathbb{Z}) \rightarrow \Omega^k(M), \hspace{4mm} \chi \mapsto \omega_{\chi}
    \]
    whose image is given exactly by the space of closed $k$-forms with integral periods
    \[
    \mathrm{Im}(\mathrm{curv}) = \Omega^k_{\mathrm{cl}}(M)_{\mathbb{Z}}.
    \]
    A differential character is called \textbf{flat} if we have that $\mathrm{curv}(h) = 0$. 
\end{definition}

Simons and Sullivan show \cite{Simons-Sullivan} that the collection of functors 
\[
\hat{H}^k(-,\mathbb{Z}) \colon \mathbf{Mfd} \rightarrow \mathbf{Ab}
\]
admits the structure of a \textit{character functor}. That is, there are natural transformations 
\[
\begin{array}{ll}
   i \colon H^{k-1}(-,U(1)) \rightarrow \hat{H}^k(-,\mathbb{Z}) & j \colon \Omega^{k-1}(-)/\Omega^{k-1}_{\mathrm{cl}}(-)_{\mathbb{Z}} \rightarrow \hat{H}^k(-,\mathbb{Z}) \\
    & \\
    cc \colon \hat{H}^k(-,\mathbb{Z}) \rightarrow H^k(-,\mathbb{Z}) & \mathrm{curv} \colon \hat{H}^k(-,\mathbb{Z}) \rightarrow \Omega^{k}_{\mathrm{cl}}(-)_{\mathbb{Z}}
\end{array}
\]
such that, for all smooth manifolds $M$, the following diagram commutes, and its diagonal sequences are exact.

 \begin{center}
     \begin{tikzcd}
0 \arrow[rd]                                   &                                                                                                         &                                                                                                                              &                                                                & 0                      \\
                                               & H^{k-1}(M,U(1)) \arrow[rd, "i" description, hook] \arrow[rr, "B" description] &                                                                                                                              & H^k(M;\mathbb{Z}) \arrow[rd, "r"] \arrow[ru]                        &                        \\
H^{k-1}_{\mathrm{dR}}(M) \arrow[ru, "\alpha"] \arrow[rd, "\beta"'] &                                                                                                         & \hat{H}^k(M; \mathbb{Z}) \arrow[ru, "\mathrm{cc}" description, two heads] \arrow[rd, "\mathrm{curv}" description, two heads] &                                                                & H^{k}_{\mathrm{dR}}(M) \\
                                               & \Omega^{k-1}(M)/\Omega^{k-1}_{\mathrm{cl}}(M)_{\mathbb{Z}} \arrow[ru, "j", hook] \arrow[rr, "d" description] &                                                                                                                              & \Omega^{k}_{\mathrm{cl}}(M)_{\mathbb{Z}} \arrow[ru, "s"'] \arrow[rd] &                        \\
0 \arrow[ru]                                   &                                                                                                         &                                                                                                                              &                                                                & 0                     
\end{tikzcd}
 \end{center}
The upper sequence of the hexagon is given by the Bockstein sequence associated to the coefficient exact sequence 
\[
0 \rightarrow \mathbb{Z} \hookrightarrow \mathbb{R} \rightarrow U(1) \rightarrow 0
\]
together with the de Rham theorem, identifying cohomology with real coefficients with de Rham cohomology. The lower sequence of the hexagon follows from the definition of de Rham cohomology together with $d$ being the usual exterior derivative. The necessary natural transformations for the differential characters to be a character functor can be constructed as follows.
\begin{construction} $ $
\begin{itemize}
    \item The \textbf{inclusion of flat classes} 
    \[
    i : H^{k-1}(M,U(1)) \hookrightarrow \hat{H}^{k}(M,\mathbb{Z}) 
    \]
    is a direct consequence of $U(1)$ being a divisible abelian group. Indeed, recall the universal coefficient sequence
    \[
    0 \rightarrow \mathrm{Ext}^1(H_{k-2}(M),U(1)) \rightarrow H^{k-1}(M,U(1)) \xrightarrow{\langle -,- \rangle } \mathrm{Hom}_{\mathbf{Ab}}(H_{k-1}(M),U(1)) \rightarrow 0.
    \]
    Since $U(1)$ is divisible, it follows that $\mathrm{Ext}^1(H_{k-2}(M),U(1)) = 0$ and therefore 
    \[
   \langle -,- \rangle \colon H^{k-1}(M,U(1)) \cong \mathrm{Hom}_{\mathbf{Ab}}(H_{k-1}(M),U(1)). 
    \]
    Denote now by $q : Z_{k-1}(M,\mathbb{Z}) \rightarrow H_{k-1}(M)$ the quotient map, sending a smooth $(k-1)$-cycle to its homology class. Then the induced pullback morphism gives
    \[
    q^* \colon \mathrm{Hom}_{\mathbf{Ab}}(H_{k-1}(M),U(1)) \hookrightarrow \mathrm{Hom}_{\mathbf{Ab}}(Z_{k-1}(M;\mathbb{Z}),U(1)). 
    \]
    Together with the universal coefficient isomorphism, this gives the desired morphism
    \[
      H^{k-1}(M,U(1)) \cong \mathrm{Hom}_{\mathbf{Ab}}(H_{k-1}(M),U(1)) \hookrightarrow \mathrm{Hom}_{\mathbf{Ab}}(Z_{k-1}(M;\mathbb{Z}),U(1)). 
    \]
    It is clear from the construction that this map factors through the subgroup $\hat{H}^k(M,\mathbb{Z})$ and its image are the flat differential characters. Geometrically, this map encodes the inclusion of the isomorphism classes of higher circle bundles with \textit{flat connections} into the isomorphism classes of higher circle bundles with arbitrary connections.   

\item   The \textbf{characteristic class map} $cc : \hat{H}^k(M, \mathbb{Z}) \rightarrow H^k(M,\mathbb{Z})$ is defined as follows. Since $Z_{k-1}(M,\mathbb{Z})$ is a free $\mathbb{Z}$-module and the exponential $\mathbb{R} \rightarrow U(1)$ is an epimorphism, we have that any differential character $\chi$ lifts to some $\tilde{\chi}$
    \begin{center}
        \begin{tikzcd}
                                                                       & \mathbb{R} \arrow[d, two heads] \\
Z_{k-1}(M, \mathbb{Z}) \arrow[r, "\chi"'] \arrow[ru, "\tilde{\chi}", dashed] & U(1)          
\end{tikzcd}
    \end{center}
    Now, define
    \[
    I(\tilde{\chi}) : C_k(M,\mathbb{Z}) \rightarrow \mathbb{Z}; \hspace{4mm} c \mapsto -\tilde{\chi}(\partial c) + \int_c \mathrm{curv}(\chi).
    \]
    To see that this is well-defined, i.e. that $I(\tilde{\chi})$ takes integral values, note that 
    \[
    \chi(\partial c) = \mathrm{exp} \left( 2\pi i  \int_c \omega_{\chi} \right) = \mathrm{exp}(2\pi i \tilde{\chi}(\partial c))
    \] 
    where the first equality follows from $\chi$ being a differential character with $\omega_{\chi} = \mathrm{curv}(\chi)$, and the second equality by construction of the lift $\tilde{\chi}$. Further, using that $\mathrm{curv}(\chi)$ is a closed form, the above map $I(\tilde{\chi})$ defines a cocycle. Its cohomology class $[I(\tilde{\chi})] \in H^k(M, \mathbb{Z})$ does not depend on the choice of lift $\tilde{\chi}$ and we define the characteristic class map 
    \[
    cc : \hat{H}^k(M,\mathbb{Z}) \rightarrow H^k(M,\mathbb{Z}); \hspace{4mm} \chi \mapsto [I(\tilde{\chi})].
    \]
    
    \item The \textbf{topological trivialization} is the inclusion 
    \[
    j \colon \Omega^{k-1}(M)/\Omega^{k-1}_{\mathrm{cl}}(M)_{\mathbb{Z}} \hookrightarrow \hat{H}^k(M,\mathbb{Z})
    \]
    which is constructed by first considering the map 
    \[
    \begin{array}{rcl}
      \iota \colon \Omega^{k-1}(M)  & \rightarrow &  \hat{H}^k(M,\mathbb{Z}) \\
         \omega & \mapsto & \iota(\omega)(c):= \mathrm{exp}(2 \pi i \int_c \omega). 
    \end{array}
    \]
    Stokes' theorem shows that the curvature of this differential character is simply given by
    \[
    \mathrm{curv}(\iota(\omega)) = d\omega.
    \]
Recall that, for the characteristic class map $cc$, a map $I(\tilde{\chi})$ was associated with $\tilde{\chi}$ being a real lift of some differential character $\chi$. In this case, the character $\iota(\omega)$ has a canonical lift given simply by 
\[
\tilde{\iota(\omega)}(c) = \int_c \omega
\]
Now the associated map $I(\tilde{\iota(\omega)})$ reads 
\begin{align*}
I(\tilde{\iota(\omega)})(c) &= - \tilde{\iota(\omega)}(\partial c) + \int_c \mathrm{curv}(\iota(\omega)) \\
&= - \int_{\partial c} \omega + \int_c d\omega \\
&= 0. 
\end{align*}
Stokes' theorem implies therefore that the characteristic class associated with $\iota(\omega)$ is \textit{topologically trivial}. In this case, one calls the form $\omega$ a \textit{topological trivialization} of the differential character $\iota(\omega)$. Let's now examine the kernel of the map $\iota$. That is, assume for $\omega \in \Omega^{k-1}(M)$ we have that for all $c \in Z_{k-1}(M;\mathbb{Z})$ 
\begin{align*}
    \iota(\omega)(c) = \mathrm{exp} \left( 2 \pi i \int_c \omega \right) = 1.
\end{align*}
It follows that the kernel is given by the closed $(k-1)$-forms with integral periods
\[
\mathrm{ker}(\iota) = \Omega^{k-1}_{\mathrm{cl}}(M)_{\mathbb{Z}}.
\]
\newpage 
By taking the quotient with respect to the kernel one gets the desired inclusion map $j$.

\item The \textbf{curvature} map is simply the map 
\[
\mathrm{curv} \colon \hat{H}^k(M;\mathbb{Z}) \to \Omega^k_{\mathrm{cl}}(M)_{\mathbb{Z}}
\]
sending a differential character $\chi$ to its unique curvature form $\omega_{\chi}$ as defined in Definition \ref{Definition: Differential Character}.
\end{itemize}
\end{construction}

\begin{theorem}[\cite{Brylinski}, Proposition 1.5.7] \label{Theorem: Differential Characters = Smooth Deligne-Beilinson Cohomology}
    Let $M$ be a smooth manifold. Then there is a canonical isomorphism between the abelian group of differential characters $\hat{H}^k(M,\mathbb{Z})$ and the smooth Deligne cohomology group $H^k(M, \mathbb{Z}(k)^{\infty}_D)$.
\end{theorem}

The classical\footnote{By ``classical" we mean without using the characterization theorem of Simons--Sullivan \cite{Simons-Sullivan}.} isomorphism between the smooth Deligne cohomology and the Cheeger--Simons differential characters unfortunately does not admit an easy geometric interpretation as it is induced from a sequence of quasi-isomorphisms of complexes of sheaves. However, the fact that smooth Deligne classes can be represented by circle $n$-bundles with connections allows for the construction of an isomorphism from $H^k(M, \mathbb{Z}(k)^{\infty}_D)$ to $\hat{H}^k(M,\mathbb{Z})$ encoding the notion of \textit{higher holonomy}.

\begin{theorem}[\cite{GajerHigher}, Theorem 3.1]
    For every smooth manifold $M$ there is an isomorphism of abelian groups for all $k \geq 0$
    \[
    \begin{array}{rcl}
         H^{k+1}(M,\mathbb{Z}(k+1)_D^{\infty})& \xrightarrow{\cong} & \hat{H}^{k+1}(M,\mathbb{Z}) \\
        c = [(g_{i_0, ..., i_k}, ..., A_{i_0i_1i_2}, B_{i_0i_1}, C_{i_0})] & \mapsto & h^c 
    \end{array}
    \]
    where $h^c$ represents the holonomy map associated with the higher connection data $c$. 
\end{theorem}

The fact that differential characters can be thought of as higher holonomy maps of circle $k$-bundles with connection will be the main motivation behind the next section.

\section{Geometric Loop Groups and Higher Holonomy}

The interpretation of differential characters in terms of higher holonomies motivated the work of Gajer \cite{GajerHigher, Gajer_Geometry} where he introduced higher holonomies as smooth group morphisms based on an iterative application of \textit{geometric loop groups}. In the case of principal $G$-bundles with connection this gives the following classification theorem.  

\begin{theorem}[\cite{GajerHigher}, Theorem 1.4]
    Given a connected and pointed smooth manifold $(M,x_0)$ and a Lie group $G$, we denote by $\mathrm{Bun}(M,G,\nabla)$ the pointed set of isomorphism classes $[P,p_0,A]$ of smooth principal $G$-bundles $P$ over $M$ endowed with a connection $A$ together with a designated point $p_0 \in P_{x_0}$ in the fiber over the basepoint. Then there is an isomorphism of pointed sets 
    \begin{align*}
        \mathrm{Bun}(M,G,\nabla) &\cong \mathrm{Hom}^{\infty}(G(M,x_0),G) \\
        [P,p_0,A] &\mapsto h^{A} : G(M,x_0) \rightarrow G,
    \end{align*}
    given by the holonomy morphism associated to a principal bundle with connection. In the case that $G$ is abelian, the above assignment becomes a group isomorphism where on the left-hand side the group structure is given by the contracted product of principal bundles with connection. 
\end{theorem}
\begin{remark}
    The basepoint $[M \times G,e,\mathrm{pr}^*\theta_G]$ in $\mathrm{Bun}(M,G,\nabla)$ is given by the class of the trivial bundle endowed with the trivial flat connection together with the choice $p_0 = e$. Here $\theta_G$ denotes the Maurer--Cartan form on $G$ and $\mathrm{pr} \colon M \times G \rightarrow G$ the projection to the second factor. 
\end{remark}

The geometric loop group is constructed as a diffeological quotient space of the space of based piecewise smooth loops in $(M,x_0)$ and will be studied in detail in Chapter \ref{Chapter: Geometric Loop Groups}. Its diffeological abelianization is denoted $A(M,x_0)$ and is a crucial ingredient to classify higher circle bundles with connection in terms of their higher holonomy morphisms.   

\begin{theorem}[\cite{GajerHigher}, Theorem 4.2] \label{Theorem: Gajer's holonomy Theorem}
    Given a connected and pointed smooth manifold $(M,x_0)$, there is an isomorphism of abelian groups
    \[
    \hat{H}^k(M; \mathbb{Z}) \cong \mathrm{Hom}^{\infty}(G^{k-1} \left( A(M,x_0) \right), U(1) ),
    \]
    where $A(M,x_0)$ is the abelianization of the geometric loop group, considered as a diffeological group. 
\end{theorem}
\begin{remark}
    \begin{enumerate}[label={(\arabic*)}] $ $
        \item     For $k \geq 1$ we denote by $G^k(M,x_0)$ the $k$-th fold loop group defined iteratively by 
        \[
        G^k(M,x_0) := G(G^{k-1}(M,x_0),[x_0]),    
        \]
        where $[x_0]$ denotes the equivalence class of the constant loop at $x_0$.
        \item It is important to point out that this theorem relies on a result whose proof contains a serious gap. This issue will be discussed in detail in section \ref*{Chapter: Geometric Loop Groups}, specifically in Remark \ref*{Remark: Differences with Gajer's Paper}. 
    \end{enumerate}
\end{remark}
\chapter{Diffeological Spaces} \label{Chapter: Diffeological Spaces}

In the introduction, we mentioned our goal of introducing a new variation of thin homotopy based on skeletal diffeological spaces. Further, these skeletal diffeologies are then used to construct for any smooth manifold $M$ a simplicial presheaf $\mathbb{M}_k$, providing us with a refinement from ordinary cohomology to differential cohomology. In this chapter, we introduce diffeological spaces, show how they relate to concrete sheaves on Cartesian spaces, and review their smooth homotopy theory. Our main references in this chapter are Iglesias-Zemmour \cite{Iglesias-Zemmour}, Baez--Hoffnung \cite{BaezHoffnung}, Christensen--Wu \cite{Christensen-Wu} and Kihara \cite{Kihara22}.

\section{Basics of Diffeological Spaces} \label{Subsection: Diffeological Spaces}

\begin{definition} \label{Definition: Diffeological Space}
    A \textbf{diffeological space} is a set of points $X$ together with a specified set $\cat{D}$ of functions $p \colon U \rightarrow X$ called \textbf{plots} for each open set $U \subset \mathbb{R}^n$ and for each $n \in \mathbb{N}$, such that the following axioms are satisfied:
    \begin{enumerate}[label={(\arabic*)}]
        \item (Covering) Every constant map $\mathbb{R}^0 \rightarrow X$ is a plot. In the following, we will write $\mathrm{pt}$ for the one-point space $\mathbb{R}^0$.
        \item (Compatibility) For every plot $p \colon U \rightarrow X$ and any smooth function $f \colon V \rightarrow U$, the composition
        \[
        p \circ f \colon V \rightarrow X
        \]
        is also a plot.
        \item (Sheaf Condition) Given a map $p \colon U \rightarrow X$ defined on some $U \subset \mathbb{R}^n$ and given an open covering $\left\{ U_i \right\}_{i \in I}$ of $U$ for which all restrictions $p|_{U_i} \colon U_i \rightarrow X$ are plots of $X$, then also $p$ is a plot.
     \end{enumerate}
     Given two diffeological spaces $X$ and $Y$ let $f \colon X \rightarrow Y$ be a map. Then we call $f$ \textbf{smooth} if, for every plot $p \colon U \rightarrow X$ of $X$, the composition $f \circ p \colon U \rightarrow Y$ is a plot of $Y$.
\end{definition}

The category of diffeological spaces together with smooth maps is denoted by $\mathbf{Diff}$. The category of smooth manifolds $\mathbf{Mfd}$ embeds fully faithfully into the category of diffeological spaces by endowing the underlying set of a smooth manifold $M$ with the canonical diffeology where a plot is simply given by a smooth map $p \colon U \rightarrow M$ in the usual sense. Contrary to the category of smooth manifolds, $\mathbf{Diff}$ is well-behaved under categorical constructions such as quotients and function spaces. In particular, it can be shown that the category $\mathbf{Diff}$ is both complete and cocomplete \cite{BaezHoffnung}. \\

Let us have a closer look at some important examples of diffeological spaces:

\begin{example} \label{Example: Diffeological Spaces} \nl
    \begin{itemize}
        \item \textbf{Subspaces:} Any subset $Y \subset X$ of a diffeological space becomes a diffeological space if we define $p \colon U \rightarrow Y$ to be a plot in $Y$ if and only if its composite $ p \colon U \rightarrow X$ with the inclusion defines a plot in $X$. By definition, this turns the inclusion $i \colon Y \hookrightarrow X$ into a smooth map.
        \item \textbf{Quotient spaces:} Given $X$ a diffeological space and $\sim$ any equivalence relation on $X$, the quotient space $Y := X/\!\sim$ becomes a diffeological space if we define a plot in $Y$ to be any function
        \[
        p \colon U \rightarrow Y
        \]
        for which there exists an open covering $\left\{ U_i \right\}_{i \in I}$ of $U$ together with plots \newline $\left\{ p_i \colon U_i \rightarrow X \right\}_{i \in I}$ of $X$ such that the following diagram commutes, where $q \colon X \rightarrow Y$ denotes the quotient map.
        \begin{center}
\begin{tikzcd}
U_i \arrow[d, hook] \arrow[r, "p_i"] & X \arrow[d, "q"] \\
U \arrow[r, "p"']                    & Y
\end{tikzcd}
        \end{center}
        Note that the quotient diffeology agrees with the pushforward diffeology $q_*(\cat{D})$ of $X$ along the quotient map $q \colon X \rightarrow Y$, see for example \cite[Art. 1.43 and Art. 1.50]{Iglesias-Zemmour}.

        \item \textbf{Product spaces:} Given diffeological spaces $X$ and $Y$, endow the Cartesian product $X \times Y$ with the diffeology where a function $p \colon U \rightarrow X \times Y$ is a plot if and only if $p$ composed with the projection maps $\mathrm{pr}_1 \colon X \times Y \rightarrow X$ and $\mathrm{pr}_2 \colon X \times Y \rightarrow Y$ defines plots for $X$ and $Y$ respectively.

        \item \textbf{Function spaces:} Given diffeological spaces $X$ and $Y$, endow the set
        \[
        C^{\infty}(X,Y) = \left\{ f : X \rightarrow Y \; | \; f \text{ is smooth} \right\}
        \]
        with the diffeology where a function $p \colon U \rightarrow C^{\infty}(X,Y)$ is declared to be a plot if and only if the corresponding function
        \begin{align*}
            U \times X &\rightarrow Y \\
            (t,x) &\mapsto p(t)(x)
        \end{align*}
        is smooth. The function diffeology turns $\mathbf{Diff}$ into a cartesian closed category. To avoid confusion we write $D(X,Y)$ for the \textbf{diffeological space of smooth maps between} $X$ \textbf{and} $Y$, and $C^{\infty}(X,Y)$ for the \textbf{set of smooth maps between} $X$ \textbf{and} $Y$.
    \end{itemize}
\end{example}

An important feature of diffeological spaces is that a diffeology induces a topology on the underlying set.

\begin{definition} \label{Definition: D-topology}
    Given a diffeological space $X$, the final topology induced by its plots, where each domain is equipped with the standard topology, is called the $D$\textbf{-topology} on $X$. That is, a subset $A$ of $X$ is open in the $D$-topology if and only if $p^{-1}(A)$ is open for all $p \in \cat{D}$.
\end{definition}
\begin{example}
    The $D$-topology on a smooth manifold endowed with the standard diffeology coincides with the usual topology on the manifold.
\end{example}

Notice that any smooth map $f \colon X \rightarrow Y$ is continuous when $X$ and $Y$ are equipped with the corresponding $D$-topologies. In particular, there is a functor
\[
\tau \colon \mathbf{Diff} \rightarrow \mathbf{Top}
\]
sending a diffeological space to the underlying topological space equipped with the $D$-topology.

\begin{remark}[\cite{Christensen-Sinnamon-Wu}] \label{Remark: Functional Topologies}
Let $X$ and $Y$ be two diffeological spaces and consider the diffeological space of smooth functions $D(X,Y)$. The associated $D$-topology we call the \textbf{functional topology}. Given $M$ and $N$ two manifolds, the set of functions $C^{\infty}(M,N)$ can be endowed with various topologies, such as the compact-open topology, the weak topology, and the strong topology, see \cite{Hirsch}. In this case, the functional topology relates to the above-mentioned other topologies as follows.
\[
\text{compact-open topology} \subseteq \text{weak topology} \subseteq \text{functional topology} \subseteq \text{strong topology}
\]
Here the coarsest topology is on the left and the finest topology is on the right-hand side. For $M$ a compact manifold, the weak, functional, and strong topologies all coincide.
\end{remark}

Diffeological spaces naturally define presheaves taking values in $\mathbf{Set}$. To understand how this works, we first define the category of Cartesian spaces.

\begin{definition}\label{Definition: Category of Cartesian Spaces}
    Let $\mathbf{Cart}$ denote the small category of open subsets $U \subseteq \mathbb{R}^{n}$ that are diffeomorphic to $\mathbb{R}^n$ for $n \in \mathbb{N}$ arbitrary. Such an open subset $U$ is called a \textbf{Cartesian space}. The morphisms are given by ordinary smooth maps.
\end{definition}

A diffeological space $X$ now defines a presheaf on the category $\mathbf{Cart}$ by
\begin{equation*}
\begin{array}{rcl}
      X \colon \mathbf{Cart}^{\mathrm{op}} & \rightarrow & \mathbf{Set}  \\
     U & \mapsto & X(U) = \left\{ \varphi \colon U \rightarrow X \; | \; \varphi \text{ is a plot in } X \right\}  \\
      f \colon U' \rightarrow U &\mapsto & X(f) \colon X(U) \rightarrow X(U')
\end{array}
\end{equation*}
where the map $X(f)$ is given by pre-composition with $f$ and is well defined by the compatibility axiom. \\

Not all presheaves on $\mathbf{Cart}$ however, represent a diffeological space. To specify presheaves of this kind, the two other axioms of Definition \ref{Definition: Diffeological Space} need to be reformulated in the language of presheaves and will lead us to the notions of \textit{coverage} and \textit{concrete site}.

\begin{definition}
A \textbf{coverage} on a category $\cat{D}$ is a function assigning to each object $D \in \cat{D}$ a collection $\cat{I}(D)$ of families $\left\{ f_i \colon D_i \rightarrow D \; | \; i \in I \right\}$ called \textbf{covering families}, with the following property:
\begin{itemize}
    \item[] Given a covering family $\left\{ f_i \colon D_i \rightarrow D \; | \; i \in I \right\}$ and a morphism $g \colon C \rightarrow D$ then there exists a covering family $\left\{ h_j \colon C_j \rightarrow C \; | \; j \in J \right\}$ such that each morphism $g \circ h_j$ factors through some $f_i$, that is there exists a dashed arrow making the following diagram commute:
    \begin{center}
    \begin{tikzcd}
    C_j \arrow[r, "h_j"] \arrow[d, dashed, "k_{ji}"'] & C \arrow[d, "g"] \\
    D_i \arrow[r, "f_i"']                  & D
    \end{tikzcd}
    \end{center}
\end{itemize}
\end{definition}
A coverage on a category is precisely what will allow us to re-phrase the sheaf condition appearing in Definition \ref{Definition: Diffeological Space} as a property of presheaves.

\begin{definition}
A small category equipped with a coverage is called a \textbf{site}.
\end{definition}

\begin{example} \label{Example: All the sites I need} $ $
\begin{enumerate}[label={(\arabic*)}]
    \item Consider the small category $\mathbf{Open}$ with objects given by open subsets $U \subseteq \mathbb{R}^{n}$ for $n \in \mathbb{N}$ arbitrary, and with morphisms given by ordinary smooth maps. To each object $U$ we assign the collection $\cat{I}(U)$ consisting of open coverings of $U$. That is, a covering family $\left\{ U_j \rightarrow U \, | \, j \in J\right\}$ is  given by an open cover $\left\{ U_j \right\}$ of $U$. It can be shown that this indeed defines a coverage on the category $\mathbf{Open}$ and we denote the according site by $(\mathbf{Open}, \tau_{\mathrm{open}})$.
    \item The same category $\mathbf{Open}$ can also be equipped with a slightly different coverage by only allowing \textit{good open covers}\footnote{An open cover $\left\{U_i \right\}$ of a smooth manifold $M$ is said to be a \textbf{good open covering} if all
    nonempty finite intersections $U_{i_0} \cap \cdots \cap U_{i_n}$ are diffeomorphic to $\mathbb{R}^n$.} as covering families. This site is denoted $(\mathbf{Open}, \tau_{\mathrm{gd}})$.
    \item The category $\mathbf{Cart}$ of Cartesian spaces can be equipped with a similar coverage as follows. To each Cartesian space $U$ we assign the collection $\cat{I}(U)$ of good open covers of $U$, that is a family in $\cat{I}(U)$ is given by a good open cover $\left\{ i_j \colon U_j \rightarrow U \; | \; j \in J \right\}$ where $i_j \colon U_j \rightarrow U$ are the inclusion maps. To show that this defines a coverage, let $\left\{ U_j \right\}$ be a good open cover of $U$ and let $g \colon V \rightarrow U$ be a morphism, i.e. a smooth map. Define the covering family on $V$ by pulling back the open cover to $V$, i.e. $\left\{ V_j \right\}$ open cover defined by $V_j = g^{-1}(U_j)$. Now refine this cover by a good open cover, which is always possible\footnote{For the fact that any open covering of a manifold can be refined by a good open cover we refer to \cite[Corollary 5.2]{Bott-Tu}}. Then, the diagram commutes by construction, and the category $\mathbf{Cart}$ equipped with the coverage of good open covers defines a site, denoted $(\mathbf{Cart},\tau_{\mathrm{gd}})$.
    \item Consider the small category $\mathbf{Mfd}$ of smooth manifolds recognized as submanifolds of $\mathbb{R}^{\infty}$ with smooth maps. To each manifold $M$ we assign the collection $\cat{I}(M)$ consisting of open coverings of $M$. That is, a covering family $\left\{ U_j \rightarrow M \, | \, j \in J \right\}$ is given by an open cover $\left\{ U_j \right\}$ of $M$. Similar to all the other examples, this also defines a coverage and the according site we denote by $(\mathbf{Mfd},\tau_{\mathrm{open}})$.
\end{enumerate}
\end{example}

\begin{definition}
Given a covering family  $\left\{ f_i \colon D_i \rightarrow D \; | \; i \in I \right\}$ in a site $\cat{D}$ and a presheaf $X \in \mathrm{PSh}(\cat{D})$, a collection of sections or plots $\left\{ \varphi_i \in X(D_i)\right\}$ is called \textbf{compatible} if whenever given arrows $g \colon C \rightarrow D_i$ and $h \colon C \rightarrow D_j$ making the following diagram commute
\begin{center}
    \begin{tikzcd}
C \arrow[r, "h"] \arrow[d, "g"'] & D_j \arrow[d, "f_j"] \\
D_i \arrow[r, "f_i"']            & D
\end{tikzcd}
\end{center}
then $X(g)(\varphi_i) = X(h)(\varphi_j)$.
\end{definition}
\begin{remark}
Think of $D_i$ being an open covering and $C$ being a double intersection $C = D_{ij}$ with its inclusion maps into $D_i$ and $D_j$. Then the collection of sections is called compatible if they agree when restricted to double intersections.
\end{remark}

\begin{definition} \label{Definition: Sheaf Property}
Given a site $\cat{D}$, a presheaf $X \in \mathrm{PSh}(\cat{D})$ is a \textbf{sheaf} if it satisfies the following condition:
\begin{itemize}
    \item[] Given a covering family $\left\{ f_i \colon D_i \rightarrow D \; | \; i \in I \right\}$ and a compatible collection of sections $\left\{ \varphi_i \in X(D_i)\right\}$, then there exists a unique section $\varphi \in X(D)$ such that $X(f_i)(\varphi) = \varphi_i$ for each $i \in I$.
\end{itemize}
Denote the full subcategory of sheaves on $\cat{D}$ by $\mathrm{Sh}(\cat{D})$.
\end{definition}

The Yoneda embedding functor generates a special class of presheaves, namely the representable ones as we will see in the next definition.

\begin{definition}
A presheaf $X \in \mathrm{PSh}(\cat{D})$ is called \textbf{representable} if it is naturally isomorphic to $\mathrm{Hom}(-,D) \in \mathrm{PSh}(\cat{D})$ for some $D \in \cat{D}$.
\end{definition}

In general representable presheaves do not need to satisfy the sheaf property, however, a site having this property is called subcanonical.

\begin{definition}
A site $\cat{D}$ is said to be \textbf{subcanonical} if every representable presheaf on this site is a sheaf. That is, the Yoneda embedding functor
\begin{align*}
\cat{Y} \colon \cat{D} &\longrightarrow \mathrm{PSh}(\cat{D}) \\
D &\longmapsto \mathrm{Hom}(-,D)
\end{align*}
factors through $\mathrm{Sh}(\cat{D})$.
\end{definition}

\begin{remark}
Why do we introduce the notion of a subcanonical site? We have seen before, that diffeological spaces define presheaves on the site $\mathbf{Cart}$. In particular, presheaves arising from a space $X$ sending an object $U$ to the set of plots $X(U)$ satisfy the sheaf property and hence define sheaves on the corresponding site. Of course, every object $U$ itself can be regarded as a diffeological space via the Yoneda embedding, and therefore should also define a sheaf and not merely a presheaf. Therefore, any site encoding some notion of generalized smooth space should be subcanonical by this observation. \\

Another important property of diffeological spaces that sheaves in general do not satisfy is \textit{concreteness}. The idea behind concreteness is that for a diffeological space $X$ functions into $X$ are completely specified by the underlying set of points of $X$. More precisely, if two smooth maps $f,f' \colon X' \rightarrow X$ agree as functions on the underlying set of points, then they agree also as smooth maps of diffeological spaces.
\end{remark}

\begin{definition}
A subcanonical site $\cat{D}$ is said to be \textbf{concrete} if it has a terminal object $1$ satisfying the following two conditions:
\begin{itemize}
    \item The functor $\mathrm{Hom}(1,-) \colon \cat{D} \rightarrow \mathbf{Set}$ is faithful.
    \item For each covering family $\left\{ f_i \colon D_i \rightarrow D \; | \; i \in I \right\}$, the resulting family of functions \newline  $\left\{ \mathrm{Hom}(1,f_i) \colon \mathrm{Hom}(1,D_i) \rightarrow \mathrm{Hom}(1,D) \; | \; i \in I \right\}$ is \textbf{jointly surjective}, that is, the function
    \[
    \coprod_{i \in I } \mathrm{Hom}(1,D_i) \rightarrow \mathrm{Hom}(1,D)
    \]
    is surjective.
\end{itemize}
\end{definition}

\begin{remark}
Let us have a closer look at the two conditions introduced above.
\begin{itemize}
    \item Requiring the functor $\mathrm{Hom}(1,-)$ to be faithful implies that for any two objects $C,D$ in $\cat{D}$ we have that the following map is injective:
    \[
    \mathrm{Hom}_{\cat{D}}(C,D) \rightarrow \mathrm{Hom}_{\mathbf{Set}}\left( \mathrm{Hom}(1,C),\mathrm{Hom}(1,D) \right)
    \]
    That is, two morphisms $f,g \colon C \rightarrow D$ in $\cat{D}$ are equal if they induce the same functions from the set of points in $C$ to the set of points in $D$, i.e.
    \[
    \mathrm{Hom}(1,g) = \mathrm{Hom}(1,f) \colon \mathrm{Hom}(1,C) \rightarrow  \mathrm{Hom}(1,D)
    \]
    This condition implies that the objects in $\cat{D}$ have enough points to distinguish morphisms.
    \item The requirement for a covering family to induce a family of jointly surjective functions \newline $\left\{ \mathrm{Hom}(1,f_i) \colon \mathrm{Hom}(1,D_i) \rightarrow \mathrm{Hom}(1,D)\right\}$ implies that given a covering family, the induced functions on the corresponding sets of points define again a covering family.
    \item All the sites introduced earlier in Example \ref{Example: All the sites I need} are concrete.
\end{itemize}
\end{remark}

Having introduced the definition of a concrete site, we now turn to the notion of a concrete sheaf on a concrete site. In principle, we want a concrete sheaf to have the fundamental properties given by sheaves induced by diffeological spaces. Let us consider a concrete site $\cat{D}$ and a sheaf $X \in \mathrm{Sh}(\cat{D})$. Then, $X$ has an underlying set of points, namely $X(1)$. In the case that $X$ is induced by a diffeological space we have
\[
X(1) = \left\{ \varphi \colon * \rightarrow X \; | \; \varphi \text{ is a plot in } X \right\} \overset{\text{covering axiom}}{=} X
\]
the underlying set of the diffeological space. \\

Moreover, for any sheaf on a concrete site, any section $\varphi \in X(D)$ induces a function on the underlying set of points
\begin{align*}
\underline{\varphi} \colon \mathrm{Hom}(1,D) &\longrightarrow X(1) \\
d \colon 1 \rightarrow D &\longmapsto X(d)(\varphi),
\end{align*}
where $X(d) \colon X(D) \rightarrow X(1)$ is the arrow induced by $d$. Hence, any section of $X$ defines an underlying function on the corresponding set of points.  In particular, for a concrete sheaf, this assignment is one-to-one.

\begin{definition}
Given a concrete site $\cat{D}$ then a sheaf $X \in \mathrm{Sh}(\cat{D})$ is said to be \textbf{concrete}, if for every object $D \in \cat{D}$ the assignment sending sections $\varphi \in X(D)$ to their underlying functions $\underline{\varphi} \colon \mathrm{Hom}(1,D) \rightarrow X(1)$ is one-to-one. That is, the following function is injective:
\begin{align*}
    X(D) &\longrightarrow \mathrm{Hom}_{\mathbf{Set}}\left( \mathrm{Hom}(1,D) , X(1) \right) \\
    \varphi &\longmapsto \underline{\varphi} \colon 1 \rightarrow d \mapsto X(d)(\varphi)
\end{align*}
\end{definition}

To conclude, we were able to reformulate the three axioms which completely characterize a diffeological space abstractly. More specifically, given a diffeological space $X$, then
\begin{align*}
    \text{compatibility axiom} &\rightsquigarrow X \text{ defines a presheaf on } \mathbf{Open}, \\
    \text{sheaf axiom} &\rightsquigarrow X \text{ defines a sheaf on the site } (\mathbf{Open},\tau_{\mathrm{open}}), \\
    \text{covering axiom} &\rightsquigarrow X \text{ defines a concrete sheaf on the  concrete site } (\mathbf{Open},\tau_{\mathrm{open}}).
\end{align*}

This observation is summed up by the following result of Baez--Hoffnung.

\begin{theorem}[\cite{BaezHoffnung}, Proposition 24] \label{Theorem: Baez-Hoffnung main result}
The category of diffeological spaces and smooth maps denoted by $\mathbf{Diff}$ is equivalent to the category of concrete sheaves on the site $(\mathbf{Open},\tau_{\mathrm{open}})$.
\end{theorem}

For purposes that will become clear later, it is more convenient for us to regard diffeological spaces as concrete sheaves on the site of Cartesian spaces $(\mathbf{Cart},\tau_{\mathrm{gd}})$.

\begin{proposition}[\cite{Minichiello}] \label{Proposition: Diffeological Spaces are Concrete Sheaves on Cart}
The category of diffeological spaces and smooth maps denoted by $\mathbf{Diff}$ is equivalent to the category of concrete sheaves on the site $(\mathbf{Cart},\tau_{\mathrm{gd}})$.
\end{proposition}

Proposition \ref{Proposition: Diffeological Spaces are Concrete Sheaves on Cart} follows from a classical result of topos theory known as the \textit{comparison Lemma}.

\begin{proposition}[\cite{Johnstone}, Theorem C.2.2.3] \label{Proposition: Comparison Lemma}
    Let $(\cat{C},\tau)$ be a locally small site with a Grothendieck coverage and $\cat{C}' \hookrightarrow \cat{C}$ a small $\tau$-dense subcategory. Then the restriction functor $\mathrm{res} \colon \mathrm{PSh}(\cat{C}) \rightarrow \mathrm{PSh}(\cat{C}')$ restricts to an equivalence of categories
    \[
    \mathrm{res} \colon \mathrm{Sh}(\cat{C},\tau) \rightarrow \mathrm{Sh}(\cat{C}',\tau|_{\cat{C}'})
    \]
\end{proposition}

An application of the comparison Lemma gives the following.

\begin{corollary} \label{Corollary: Comparison Lemma applied}
   The restriction functor induces an equivalence of categories
   \[
   \mathrm{res} \colon \mathrm{Sh}(\mathbf{Open},\tau_{\mathrm{open}}) \xrightarrow{\simeq} \mathrm{Sh}(\mathbf{Cart},\tau_{\mathrm{gd}})
   \]
   which further restricts to an equivalence between the corresponding subcategories of concrete sheaves
   \[
   \mathrm{CSh}(\mathbf{Open},\tau_{\mathrm{open}}) \simeq \mathrm{CSh}(\mathbf{Cart},\tau_{\mathrm{gd}})
   \]
\end{corollary}
\begin{proof}
First notice that since any open covering in $\mathbf{Open}$ can be refined by a good open covering, it follows that
\[
\mathrm{Sh}(\mathbf{Open},\tau_{\mathrm{open}}) = \mathrm{Sh}(\mathbf{Open},\tau_{\mathrm{gd}})
\]
such that automatically we have
\[
\mathrm{CSh}(\mathbf{Open},\tau_{\mathrm{open}}) = \mathrm{CSh}(\mathbf{Open},\tau_{\mathrm{gd}}).
\]
We now wish to apply the comparison Lemma. That is, we have to show that $\mathbf{Cart} \hookrightarrow \mathbf{Open}$ is indeed a small $\tau_{\mathrm{gd}}$-dense subcategory\footnote{See \cite[Definition 2.2.1]{Johnstone} for the precise definition of a dense subcategory.}. This however, follows from the fact that for any object $U \in \mathrm{Open}$  and any good open covering $\left\{ U_i \rightarrow U \, | \, i \in I \right\}$ the associated covering sieve $S_{U}$ is generated by a family of morphisms, in this case, $\left\{ U_i \rightarrow U \, | \, i \in I \right\}$, whose domains $U_i$ are objects in $\mathbf{Cart}$ by assumption. As such, the comparison Lemma now implies that the restriction functor
\begin{equation}\label{Equation: Comparison Lemma}
       \mathrm{res} \colon \mathrm{Sh}(\mathbf{Open},\tau_{\mathrm{gd}}) \xrightarrow{\simeq} \mathrm{Sh}(\mathbf{Cart},\tau_{\mathrm{gd}})
\end{equation}
is an equivalence of categories. \\

What is left to show is that this equivalence descends to an equivalence of concrete sheaves. This follows from an adaptation of \cite[Lemma 2.9]{Watts-Wolbert} to the present case. More precisely, one can show analogously to \cite[Lemma 2.8]{Watts-Wolbert} that given a sheaf of sets $F$ on the site $(\mathbf{Open},\tau_{\mathrm{gd}})$ and given any $U \in \mathbf{Open}$ together with a good open cover $\left\{ U_i \rightarrow U \right\}$, if $F$ is concrete at each $U_i$, then it is concrete at $U$. By assumption $U_i \in \mathbf{Cart}$ hence it follows that a sheaf $F \in \mathrm{Sh}(\mathbf{Open},\tau_{\mathrm{gd}})$ is concrete if $\mathrm{res}(F)$ is concrete as a sheaf on the cartesian site $\mathbf{Cart},\tau_{\mathrm{gd}})$. Since $\mathrm{res}$ is an equivalence of categories it follows that $\mathrm{res} \circ \mathrm{ext}$ is naturally isomorphic to the identity and therefore any concrete sheaf $F \in \mathrm{Sh}(\mathbf{Open},\tau_{\mathrm{gd}})$ is sent to a concrete sheaf $\mathrm{res} \left( \mathrm{ext}(F) \right)$. By the previous observation, it follows now that $\mathrm{ext}(F)$ is concrete. Hence we have shown that the extension functor $\mathrm{ext} \colon \mathrm{Sh}(\mathbf{Cart},\tau_{\mathrm{gd}}) \rightarrow \mathrm{Sh}(\mathbf{Open},\tau_{\mathrm{gd}})$ takes concrete sheaves to concrete sheaves. On the other hand, $\mathrm{res}$ automatically takes concrete sheaves to concrete sheaves. Therefore the equivalence (\ref{Equation: Comparison Lemma}) indeed descends to the desired equivalence on the corresponding subcategories of concrete sheaves
\[
    \mathrm{CSh}(\mathbf{Open},\tau_{\mathrm{gd}}) \simeq \mathrm{CSh}(\mathbf{Cart},\tau_{\mathrm{gd}}).
\]
This concludes the proof.
\end{proof}

\begin{proof}[Proof of Proposition  \ref{Proposition: Diffeological Spaces are Concrete Sheaves on Cart}]
This is now a direct consequence of Corollary \ref{Corollary: Comparison Lemma applied} and Theorem \ref{Theorem: Baez-Hoffnung main result}.
\end{proof}
\begin{remark}
    For more details, such as the precise notions of Grothendieck coverage and covering sieves, we refer to the exposition \cite[Appendix A]{Minichiello} which has served as the main reference for Proposition \ref{Proposition: Diffeological Spaces are Concrete Sheaves on Cart}.
\end{remark}

As already mentioned, the category $\mathbf{Diff}$ is complete and cocomplete. Let us, therefore, have a closer look at the computation of small limits and colimits. The fact that we have identified the category $\mathbf{Diff}$ with a certain category of sheaves comes in handy and one shows that the limits in $\mathbf{Diff}$ may be computed pointwise\footnote{See \cite[Proposition 39]{BaezHoffnung}.}. For colimits, however, the story is more involved. \\

Recall that the sheafification process associates a sheaf to any presheaf in a universal manner and likewise, there exists a concretization functor associating a concrete presheaf to any presheaf. In other words, sheafification and concretization are both left adjoints to the inclusion functors respectively:

\begin{center}
    \begin{tikzcd}
\mathrm{PSh}(\mathbf{Cart}) \arrow[r, "C", shift left=3] \arrow[r, "\perp", phantom] & \mathrm{Conc}(\mathbf{Cart}) \arrow[l, "i", shift left=3]
\end{tikzcd}
\end{center}
and
\begin{center}
    \begin{tikzcd}
\mathrm{PSh}(\mathbf{Cart}) \arrow[r, "(-)^{\#}", shift left=3] \arrow[r, "\perp", phantom] & \mathrm{Sh}(\mathbf{Cart}). \arrow[l, "i", shift left=3]
\end{tikzcd}
\end{center}
Given that the sheafification of a concrete presheaf is again concrete\footnote{See \cite[Lemma 49]{BaezHoffnung}} and that left adjoints preserve colimits, it follows that a colimit in $\mathbf{Diff}$ can be computed by  first computing the colimit pointwise and then applying the concretization and the sheafification functor: \\

Let $F \colon \cat{D} \rightarrow \mathbf{Diff}$ be a small diagram. Then we have
\begin{equation} \label{Equation: Colimits in Diff}
    \mathrm{colim} \, F \cong \left( C \left( \mathrm{colim} \, \tilde{F} \right) \right)^{\#}
\end{equation}
where $\tilde{F} \colon \cat{D} \rightarrow \mathrm{PSh}(\mathbf{Cart})$ represents the underlying diagram of presheaves, whose colimit is simply computed pointwise. \\

An interesting example are quotient spaces. We have already encountered them in Example \ref{Example: Diffeological Spaces} where we introduced quotient diffeology on $X/\! \sim$. Consider for the moment the equivalence relation $\sim$ as a subset $R \subset X \times X$. We endow $R$ with the subset diffeology and denote the smooth projection maps to the first and second factors by
\[
r_i \colon R \hookrightarrow X \times X \xrightarrow{\mathrm{pr}_i} X
\]
for $i = 1,2$.

\begin{lemma}[\cite{Collier-Lerman-Wolbert}, Construction A.15] \label{Lemma: Quotient Space is Coequalizer}
Let $X$ be a diffeological space and $R \subset X \times X$ an equivalence relation on $X$. Then the diffeological quotient space $X/\! \sim$ is given by the coequalizer in $\mathbf{Diff}$ of the diagram
\begin{center}
    \begin{tikzcd}
R  \arrow[r, "r_1", shift left] \arrow[r, "r_2"', shift right] & X \arrow[r, "q", dotted] & X/\sim,
\end{tikzcd}
\end{center}
where $R$ is the diffeological space endowed with the subset diffeology of $X \times X$.
\end{lemma}

According to (\ref{Equation: Colimits in Diff}) we can first compute the coequalizer in $\mathrm{PSh}(\mathbf{Cart})$, i.e. pointwise. That is, set $Y := \mathrm{coeq}_{\mathrm{PSh}(\mathbf{Cart})} \left( R \rightrightarrows X \right)$ which is the presheaf given by
\[
Y(U) = \mathrm{coeq} \left( R(U) \rightrightarrows X(U) \right) = X(U)/\sim
\]
where two plots $p \colon U \rightarrow X$ and $p' \colon U \rightarrow X$ are equivalent if and only if there is a plot $q \colon U \rightarrow R$ such that
\[
r_1 \circ q = p \text{ and } r_2 \circ q = p'
\]
Now observe that to recover the diffeological quotient space we only need to apply the sheafification functor to $Y$, since the presheaf $Y$ is already concrete.

\begin{lemma} \label{Lemma: Presheaf quotient is concrete}
    Let $X$ be a diffeological space and $R$ an equivalence relation on $X$. Endow $R \subset X \times X$ with the subset diffeology. Then the coequalizer taken in the category $\mathrm{PSh}(\mathbf{Cart})$ is a concrete presheaf. In particular, we have that
    \[
    X/\! \sim = \mathrm{coeq}_{\mathbf{Diff}}\left( R \rightrightarrows X \right) \cong \left(\mathrm{coeq}_{\mathrm{PSh}(\mathbf{Cart})} \left( R \rightrightarrows X \right) \right)^+
    \]
\end{lemma}
\begin{proof}
Regard $X$ as a concrete sheaf on $\mathbf{Cart}$. Therefore we have that for all $U \in \mathbf{Cart}$ the assignment
\begin{align*}
X(U) &\rightarrow \mathrm{Hom} \left( \mathrm{Hom}_{\mathbf{Cart}}(*,U) ,X(*) \right) \\
p &\mapsto \left\{ \underline{p} \colon f \mapsto X(f)(p) \right\}
\end{align*}
is injective. Now let $R \subset X \times X$ denote an equivalence relation on $X$. We endow $R$ with its subset diffeology. Then the coequalizer in $\mathrm{PSh}(\mathbf{Cart})$
\begin{center}
\begin{tikzcd}
R \subset X \times X \arrow[r, "r_1", shift left] \arrow[r, "r_2"', shift right] & X \arrow[r, "q", dotted] & Y
\end{tikzcd}
\end{center}
is computed pointwise. That is
\[
Y(U) = \left\{ [p] \; | p \in X(U) \right\}
\]
with $p \sim p'$ if and only if $q = (p,p') \in R(U)$. We now want to show that the presheaf $Y$ is again concrete. First notice that since $R$ is a diffeological space it is also concrete and hence the assignment
\begin{align}\label{concreteness of R}
R(U) &\rightarrow \mathrm{Hom} \left( \mathrm{Hom}_{\mathbf{Cart}}(*,U) ,R(*) \right) \\
(p,p') &\mapsto \left[ \underline{(p,p')} \colon f \mapsto R(f)(p,p') \right]
\end{align}
is injective. We now want to show that for all $U \in \mathbf{Cart}$ the assignment
\[
Y(U) \ni [p] \mapsto \underline{[p]} \colon \mathrm{Hom}_{\mathbf{Cart}}(*,U) \rightarrow Y(*)
\]
is injective. Let therefore $[p],[p']$ be given such that for all $f \colon * \rightarrow U$ we have that
\[
\underline{[p]}(f) = \underline{[p']}(f)
\]
in $Y(*)$ which implies that
\[
(\underline{p}(f),\underline{p'}(f) ) \in R(*)
\]
for all $f \colon * \rightarrow U$. Since $R$ is endowed with the subset diffeology we conclude that
\[
(p,p') \in R(U)
\]
which then implies that $[p] = [p']$. This shows that $Y$ is concrete. The fact that colimits in $\mathbf{Diff}$ may be computed via equation (\ref{Equation: Colimits in Diff}) finishes the proof.
\end{proof}

Not all colimits of diagrams taking values in concrete presheaves are again concrete, so Lemma \ref{Lemma: Presheaf quotient is concrete} can be only applied to the case of quotient spaces. To show why in general the statement can not be true we give the following counter-example:
\newpage

\begin{example} \label{Example: Non-concrete sheaf}
Let $n \in \mathbb{N}$ such that $n \geq 1$. Then consider the \textbf{sheaf of differential} $n$\textbf{-forms} on $\mathbf{Cart}$:
\begin{align*}
    \mathbf{Cart}^{\mathrm{op}} &\longrightarrow \mathbf{Set} \\
    U &\longmapsto \Omega^n(U)
\end{align*}
This sheaf is non-concrete since we have chosen $n \geq 1$. To see that, we notice that the underlying set of points of the sheaf $\Omega^n(-)$ only has one element:
\[
\Omega^n(\mathbb{R}^0) = *
\]
On the other hand, for $U = \mathbb{R}^k$ with $k > 0$ the sheaf admits a large number of plots. Therefore the map
\begin{align*}
\Omega^n(\mathbb{R}^k) &\rightarrow \mathrm{Hom} \left( \mathrm{Hom}_{\mathbf{Cart}}(*,U) , \Omega^n(\mathbb{R}^0) \right) = *\\
\omega &\mapsto \left\{ \underline{\omega} \colon f \mapsto f^*\omega \right\}
\end{align*}
clearly can not be one-to-one. \\

Regarding $\Omega^n(-)$ as a presheaf it is well known that one can represent it as a colimit of representables. More precisely, any presheaf $F \in \mathrm{PSh}(\mathbf{Cart})$ can be written as a colimit of representables
\[
F \cong \underset{\cat{Y}(X) \rightarrow F}{\mathrm{colim}} \cat{Y}(X)
\]
via the canonical diagram:
\[
\cat{Y} \downarrow F \rightarrow \mathrm{PSh}(\mathbf{Cart})
\]
Notice that $\cat{Y}: \mathbf{Cart} \rightarrow \mathrm{PSh}(\mathbf{Cart})$ denotes the Yoneda embedding functor, where for $U \in \mathbf{Cart}$ the representable presheaf $\cat{Y}(U)$ is precisely the concrete presheaf characterizing the standard diffeology on $U \cong \mathbb{R}^k$. Returning now to the non-concrete sheaf of differential forms, we conclude that it can be written as a colimit in $\mathrm{PSh}(\mathbf{Cart})$ of a diagram taking values in concrete sheaves:
\[
\Omega^n(-) \cong \underset{\cat{Y}(U) \rightarrow \Omega^n(-)}{\mathrm{colim}} \cat{Y}(U)
\]
\end{example}

\section{Diffeological Bundles}

In this short section, we introduce diffeological fiber bundles and diffeological principal $G$-bundles. These definitions are standard and were introduced by Iglesias-Zemmour \cite{Iglesias-Zemmour}. Brief introductions to diffeological bundles can also be found in \cite{Christensen-Wu} and \cite{Magnot-Watts}.

\begin{definition}
    A \textbf{diffeological group} is a group $G$ equipped with a diffeology such that the multiplication map $m \colon G \times G \rightarrow G$ and the inversion map $i \colon G \rightarrow G$ are smooth. A \textbf{diffeological group action} of $G$ on a diffeological space $X$ is a group action of $G$ on the set $X$ such that the action map
    \[
    G \times X \rightarrow X
    \]
    is smooth.
\end{definition}

\begin{definition}\label{Definition: Diffeological bundle}
    Let $F$ be a diffeological space and $\pi \colon E \rightarrow X$ a smooth surjective map of diffeological spaces.
    \begin{enumerate}[label={(\arabic*)}]
        \item The map $\pi$ is \textbf{trivial of fiber type} $F$ if there exists a diffeomorphism $\varphi \colon E \rightarrow F \times X$ such that the following diagram commutes:
        \begin{center}
        \begin{tikzcd}
E \arrow[d, "\pi"'] \arrow[r, "\varphi"] & F \times X \arrow[ld, "\mathrm{pr}_2"] \\
X                                        &
\end{tikzcd}
        \end{center}
        \item The map $\pi$ is \textbf{locally trivial of fiber type} $F$ if there exists a $D$-open cover $\left\{ U_i \right\}_{i \in I}$ of $X$ such that $\pi|_{\pi^{-1}(U_i)} \colon \pi^{-1}(U_i) \rightarrow U_i$ is trivial of fiber type $F$ for all $i \in I$.
        \item The map $\pi$ is called a \textbf{diffeological bundle of fiber type} $F$ if the pullback of $\pi$ along any plot $ p \colon U \rightarrow X$ is locally trivial of fiber type $F$.
    \end{enumerate}
\end{definition}

The following equivalent characterization of diffeological bundles will play an important role after introducing the smooth singular complex in the next section.
\begin{lemma}[\cite{Iglesias-Zemmour}]
    A smooth surjective map $\pi \colon E \rightarrow X$ between diffeological spaces is a diffeological bundle of fiber type $F$ if and only if the pullback of $\pi$ along any global plot, that is a plot of the form $p \colon \mathbb{R}^n \rightarrow X$, is trivial of fiber type $F$:
    \begin{center}
    \begin{tikzcd}
p^*E \arrow[d, "p^*(\pi)"'] \arrow[r, "\varphi"] & F \times \mathbb{R}^n \arrow[ld, "\mathrm{pr}_2"] \\
\mathbb{R}^n                                        &
    \end{tikzcd}
    \end{center}
\end{lemma}

The most important examples of diffeological fiber bundles we will encounter are \textit{diffeological principal} $G$\textit{-bundles}, for $G$ some diffeological group.

\begin{definition}[\cite{Waldorf-Transgression}] \label{Definition: Diffeological Principal Bundle}
    Let $G$ be a diffeological group. We call a map $\pi \colon P \rightarrow X$ a \textbf{diffeological principal} $G$\textbf{-bundle} if $\pi$ is a subduction, the space $P$ is endowed with a fibre-preserving right $G$-action, and the shear map
    \begin{align*}
        P \times G \rightarrow P \times_X P \\
        (p,g) \mapsto (p,p \cdot g)
    \end{align*}
    is a diffeomorphism.
\end{definition}

Definition \ref{Definition: Diffeological Principal Bundle} is the straightforward generalization of ordinary principal $G$-bundles from the category of manifolds to diffeological spaces. However, it is not completely trivial how this definition relates to Definition \ref{Definition: Diffeological bundle} of a diffeological bundle of fiber type $G$. In the following, we will show that diffeological principal $G$-bundles indeed are a special case of diffeological bundles.

\begin{lemma}[\cite{Iglesias-Zemmour}] \label{Lemma: Inductive action map}
    Let $P$ be a diffeological space endowed with a right $G$-action. We call the map
    \begin{align*}
       \alpha:  P \times G \rightarrow P \times P \\
        (p,g) \mapsto (p,p \cdot g)
    \end{align*}
    the \textbf{action map}. If $\alpha$ is an induction, then the projection map
    \[
    q \colon P \rightarrow P/G
    \]
    is a diffeological bundle of fiber type $G$.
\end{lemma}

Let us unwind Definition \ref{Definition: Diffeological Principal Bundle} to understand how it relates to diffeological bundles via Lemma \ref{Lemma: Inductive action map}. Consider a diffeological group $G$ and a diffeological $G$-space $P$. The first condition for a map $\pi \colon P \rightarrow X$ to be a diffeological principal $G$-bundle is that it is submersive. Notice that submersions can always be considered projections to a quotient. More precisely, given a submersion $\pi \colon P \rightarrow X$ we introduce the equivalence relation $p \sim p'$ if and only if $\pi(p) = \pi(p')$. Then the map sending the class $[p]$ to $\pi(p)$ defines a diffeomorphism making the following diagram commutative.
\begin{center}
\begin{tikzcd}
P \arrow[rd, "\pi"] \arrow[d, "q"'] &   \\
P/\sim \arrow[r, "\cong"']           & X
\end{tikzcd}
\end{center}
The second condition is that the shear map
    \begin{align*}
        P \times G \rightarrow P \times_X P \\
        (p,g) \mapsto (p,p \cdot g)
    \end{align*}
is a diffeomorphism. Notice that for two elements $p,p'$ such that $\pi(p) = \pi(p')$ by definition we have $(p,p') \in P \times_X P$. Since the shear map is a diffeomorphism this tells us that there exists a unique $(p,g) \in P \times G$ such that $(p,p \cdot g) = (p,p')$. Hence the equivalence relation induced by the submersion $\pi$ now translates via the shear map to $p \sim p'$ if and only if there is a $g \in G$ such that $p' = p \cdot g$. In particular, this shows that
\[
P/ \! \sim \; = P/G,
\]
and that there is a diffeomorphism $P/G \xrightarrow{\cong} X$ making the following diagram commute.
\begin{center}
\begin{tikzcd}
P \arrow[rd, "\pi"] \arrow[d, "q"'] &   \\
P/G \arrow[r, "\cong"']           & X
\end{tikzcd}
\end{center}
To conclude, notice that as the shear map is a diffeomorphism it is also inductive, which then implies that the action map
\[
\alpha \colon P \times G \rightarrow P \times_X P \hookrightarrow P \times P
\]
is inductive. Lemma \ref{Lemma: Inductive action map} now implies that $q$ is a diffeological bundle of fiber type $G$. \\

To summarize, from Definition \ref{Definition: Diffeological Principal Bundle} we were able to extract an inductive action map $\alpha$ leading to a quotient map $q : P \rightarrow P/G$ being a diffeological bundle of fiber type $G$ together with a diffeomorphism identifying $P/G$ with $X$ making the obvious diagram commute. This is precisely the definition of a \textit{diffeological principal fibration with structure group} $G$ introduced in \cite[Art. 8.11]{Iglesias-Zemmour}.

\section{Homotopy Theory of Diffeological Spaces}

In this section, we introduce the homotopy theory of diffeological spaces closely following Christensen--Wu \cite{Christensen-Wu}.

\begin{remark} \label{Remark: Cutoff-function}
    For $0 < \varepsilon < 1/2$, by an $\varepsilon$\textbf{-cut-off function} we mean a smooth function $\psi \colon \mathbb{R} \rightarrow \mathbb{R}$ such that:
    \begin{itemize}
        \item $0 \leq \psi(t) \leq 1$ for all $t \in \mathbb{R}$
        \item $\psi(t) = 0$ if $t < \varepsilon$ and
        \item $\psi(t) = 1$ if $t > 1- \varepsilon$.
    \end{itemize}
    such functions exist for all such $\varepsilon$.
\end{remark}

Before introducing the homotopy groups of diffeological spaces, we introduce some elementary vocabulary needed. Let $X$ be a diffeological space. By a \textbf{path} in $X$ we mean a smooth function $\gamma : \mathbb{R} \rightarrow X$. We say that the path $\gamma$ has \textbf{sitting instants} if there is an $\varepsilon > 0$ such that $\gamma$ is constant on $(- \infty,\varepsilon)$ and $(1-\varepsilon, \infty)$. Sitting instances are needed to concatenate smooth paths. \\

 We denote by $(X,U)$ a pair of diffeological spaces, that is $U$ is a diffeological subspace of $X$. A map of pairs $f \colon (X,U) \rightarrow (Y,V)$ is then a smooth map $f \colon X \rightarrow Y$ such that $f(U) \subset V$. We denote the set of maps of pairs by $C^{\infty}\left( (X,U) , (Y,V) \right)$. As a subset of $C^{\infty}(X,Y)$ we endow $C^{\infty}\left( (X,U) , (Y,V) \right)$ with the subset diffeology of the functional diffeology, which is then denoted by $D\left( (X,U) , (Y,V) \right)$. Given a pointed diffeological space $(X,x_0)$ we denote by
 \[
 \Omega(X,x_0) := D\left( (\mathbb{R}, \left\{ 0 ,1 \right\} ),(X,x_0) \right)
 \]
the \textbf{diffeological based loop space} of $X$ at $x_0$. The loop space itself is again pointed with the constant loop $c_{x_0}$ as basepoint. This allows for the iterative definition of the \textbf{diffeological} $n$\textbf{-fold loop space} via
\[
\Omega^n(X,x_0) := \Omega \left( \Omega^{n-1}(X,x_0), c_{x_0} \right).
\]

\begin{construction}
    Let $(X,x_0)$ be a pointed diffeological space. Define a relation on $X$ as follows: say $x \sim y$ if and only if there is a smooth path $\gamma$ in $X$ connecting $x$ and $y$, i.e. $\gamma(0) = x$ and $\gamma(1) = y$. Notice that by the existence of cut-off functions, given such a path $\gamma$ we can always replace it with a path having sitting instances. This makes sure that the relation $\sim$ defines an equivalence relation on $X$. Denote the quotient set by
    \[
    \pi^D_0(X,x_0) := (X/\! \sim, [x_0]),
    \]
    which is a pointed set with basepoint the path-component of $x_0$. We call this set the \textbf{smooth path components of} $X$.
\end{construction}

\begin{definition}
    Given a pointed diffeological space $(X,x_0)$ and $n > 0$ we define the $n$\textbf{-th smooth homotopy group} as the pointed set
    \[
    \pi_n^D(X,x_0) := \pi_0^D \left( \Omega^n(X,x_0), c_{x_0} \right)
    \]
    endowed with the group structure given by concatenation.
\end{definition}
\begin{remark}
    To show that the pointed sets $\pi_n^D(X,x_0)$ are indeed groups, observe that the existence of cut-off functions allows for the choice of a representative with sitting instances in any equivalence class $[\gamma] \in \pi_n^D(X,x_0)$. The group multiplication is now defined by choosing $\gamma' \in [\gamma]$ and $\delta' \in [\delta]$ representatives with sitting instances, which can be smoothly concatenated
    \[
    [\gamma] * [\delta] := [\gamma' * \delta'].
    \]
\end{remark}

Many classical results from algebraic topology translate into their smooth counterparts for diffeological spaces.

\begin{proposition}[\cite{Iglesias-Zemmour}] \label{Proposition: Long Exact Sequence Fibration}
    Let $ \pi \colon E \rightarrow X$ be a diffeological bundle of fiber type $F = \pi^{-1}(x)$ for some $x \in X$. Then for any $y \in F$, there is a long exact sequence of smooth homotopy groups:
    \[
    \cdots \rightarrow \pi^D_n(F,y) \xrightarrow{i_*} \pi^D_n(E,y) \rightarrow \pi_n^D(X,x) \rightarrow \pi_{n-1}^D(F,y) \rightarrow \cdots \rightarrow \pi_0^D(X,x) \rightarrow 0
    \]
\end{proposition}

The previous result already suggests that a diffeological bundle is a special case of what we call for the moment a \textit{smooth fibration}. To make this concept precise and to show that this is indeed the case, we need to introduce the smooth singular complex of a diffeological space. It will allow us to translate the smooth homotopy theory of diffeological spaces to the classical homotopy theory of simplicial sets.

Like the classical singular complex of a topological space, the smooth singular complex is an instance of the general nerve--realization construction. This machinery automatically provides us with the corresponding adjoint, the diffeological realization functor. Following the general procedure of the nerve--realization construction, we first specify a cosimplicial object in $\mathbf{Diff}$. There are two cosimplicial objects in $\mathbf{Diff}$ of special interest to us:

\begin{definition}\label{Definition: Affine Comsimplical Object}
    Denote by $\mathbb{A}^n = \left\{ (x_0, ..., x_n) \in \mathbb{R}^{n+1} \, | \, \sum x_i = 1 \right\}$ the \textbf{affine} or \textbf{extended smooth} $n$\textbf{-simplex} endowed with the subset diffeology of $\mathbb{R}^{n+1}$. The associated cosimplicial object $\mathbb{A}^{\bullet}$ in $\mathbf{Diff}$ is the functor
    \begin{align*}
        \mathbb{A}^{\bullet} \colon \Delta &\rightarrow \mathbf{Diff} \\
        [n] &\mapsto \mathbb{A}^n. \\
    \end{align*}
    On morphisms, the cosimplicial object $\mathbb{A}^{\bullet}$ associates to $f \colon [n] \rightarrow [m]$ the smooth map
    \begin{align*}
        \mathbb{A}^n &\longrightarrow \mathbb{A}^m \\
        (x_0, ..., x_n) &\longmapsto \left( \sum_{ i \in f^{-1}(0)} x_i , ...,  \sum_{ i \in f^{-1}(m)} x_i  \right).
    \end{align*}
\end{definition}

\begin{definition}\label{Definition: Compact Cosimpicial Object}
    Denote by $|\Delta^{n}| = \left\{ (x_0, ..., x_n) \in \mathbb{R}^{n+1} \, | \, \sum x_i = 1 \text{ and } x_i \geq 0 \text{ for each } i\right\}$ the \textbf{standard} or \textbf{compact smooth} $n$\textbf{-simplex} endowed with the subset diffeology of $\mathbb{R}^{n+1}$. The associated simplicial object we denote by $|\Delta^{\bullet}|$ and is constructed analog to the standard topological simplex.
\end{definition}

The extended smooth simplex and the standard smooth simplex now lead to the \textit{extended smooth singular complex} and the \textit{standard smooth singular complex}:

\begin{definition}\label{Definition: Extended Smooth Singular Complex} $ $
    \begin{enumerate}[label={(\arabic*)}]
        \item The adjoint pair associated to the extended cosimplicial object $\mathbb{A}^{\bullet}$ we call \textbf{extended smooth singular complex} $S_{e}$ and \textbf{extended realization} $\lvert - \rvert_{e}$:
        \begin{center}
\begin{tikzcd}
\lvert - \rvert_{e} : \mathbf{sSet} \arrow[r, shift left=3] \arrow[r, "\perp", phantom] & \mathbf{Diff} : S_e \arrow[l, shift left=3]
\end{tikzcd}
        \end{center}
        \item The adjoint pair associated to the standard cosimplicial object $|\Delta^{\bullet}|$ we call \textbf{smooth singular complex} $S$ and \textbf{smooth realization} $\lvert - \rvert$:
        \begin{center}
\begin{tikzcd}
\lvert - \rvert : \mathbf{sSet} \arrow[r, shift left=3] \arrow[r, "\perp", phantom] & \mathbf{Diff} : S \arrow[l, shift left=3]
\end{tikzcd}
        \end{center}
    \end{enumerate}
\end{definition}

Notice that for the smooth compact simplex $|\Delta^n|$ in Definition \ref{Definition: Compact Cosimpicial Object}, the notion of a smooth map $\sigma \colon |\Delta^n| \rightarrow M$ into some smooth manifold $M$ is a priori not identical to the standard definition of smooth simplices. Thankfully, the next Lemma shows that all these notions do coincide in the case of the compact $n$-simplex.

\begin{lemma} \label{Lemma: Subset diffeology of compact simplex}
    Denote by $|\Delta^n| = \{ (x_1, ...,x_n) \in \mathbb{R}^n \, | \, \sum_{i = 1}^n x_i \leq 1 \text{ with } x_i \geq 0  \}$ the compact standard $n$-simplex considered as a closed subset of $\mathbb{R}^n$ and let $f \colon |\Delta^n| \rightarrow \mathbb{R}$ be a function. Then the following are equivalent:
    \begin{enumerate}[label={(\arabic*)}]
        \item  For every point $p \in |\Delta^n|$ there is a neighborhood $U$ of $p$ in $\mathbb{R}^n$ and a smooth function $f_p \colon U \rightarrow \mathbb{R}$ such that $(f_p)|_{U \cap |\Delta^n|} = f|_{U \cap |\Delta^n|}$.
        \item There is an open neighborhood $U$ of $|\Delta^n|$ in $\mathbb{R}^n$ together with a smooth function $\tilde{f} \colon U \rightarrow \mathbb{R}$ such that $f = \tilde{f}|_{|\Delta^n|}$.
        \item The function $f \colon |\Delta^n| \rightarrow \mathbb{R}$ is smooth as a morphism of diffeological spaces where $|\Delta^n|$ is equipped with the subset diffeology of $\mathbb{R}^n$.
    \end{enumerate}
\end{lemma}
\begin{proof}
    It is immediate that $(2) \Rightarrow (1)$. The fact that $(1) \Rightarrow (3)$ follows from Lemma 2.1 of \cite{Karshon-Watts}. The missing step $(3) \Rightarrow (2)$ follows from the Whitney Extension Theorem, as stated in the proof of Theorem 4.1 of \cite{Karshon-Watts}.
\end{proof}

Let us briefly unwind the definition of the smooth singular simplicial sets introduced above. Let $X$ be a diffeological space, then the extended smooth singular complex $S_e(X)$ is the simplicial set with $n$-simplices given by
\[
S_e(X)_n = C^{\infty}(\mathbb{A}^n,X) \cong C^{\infty}(\mathbb{R}^n, X)
\]
Here we have used the identification of the affine $n$-simplex $\mathbb{A}^n$ with $\mathbb{R}^n$ given by the diffeomorphism via forgetting the first coordinate for example. Similarly, the standard smooth singular complex $S(X)$ has $n$-simplices given by
\[
S_n(X) = C^{\infty}(|\Delta^{n}|, X)
\]
In the case that $X = M$ is a manifold, the previous lemma allows us to identify the $n$-simplices as ordinary smooth maps of manifolds. This even holds in the case where $M$ is a smooth manifold with boundary and corners. This follows from the fact that the compact smooth $n$-simplex, which by definition can be considered as a diffeological manifold with boundary and corners, is equivalently also an ordinary smooth manifold with boundary and corners, and these two notions agree by \cite[Theorem 3]{Gurer-Iglesias-Zemmour}.  \\

Recall that both $\mathbb{A}^n$ and $|\Delta^n|$ are endowed with the subset diffeology from $\mathbb{R}^{n+1}$. Hence it is clear that the inclusion map is smooth:
\[
i \colon |\Delta^n| \hookrightarrow \mathbb{A}^n
\]
This induces a natural morphism of simplicial sets for all diffeological spaces $X$:
\[
S_e(X) \rightarrow S(X)
\]

\newpage

\begin{theorem} \label{Theorem: The Fundamental Theorem of Diffeological Homotopy}
The natural morphism of simplicial sets
\[
S_e(X) \rightarrow S(X)
\]
is a weak equivalence. Moreover given a pointed diffeological space $(X,x_0)$, both $\pi_i(S_e(X),x_0)$ and $\pi_i(S(X),x_0)$ are naturally isomorphic to the smooth homotopy groups $\pi^D_i(X,x_0)$ for all $i \geq 0$.
\end{theorem}
\begin{proof}
    This is Theorem 1.1 together with Corollary 1.2 in \cite{Kihara22}. An alternative proof for the natural isomorphism $\pi_i(S(X),x_0) \cong \pi^D_i(X,x_0)$ is also given in \cite[Proposition 2.18]{BBP}.
\end{proof}

\begin{definition} \label{Definition: fibration of diffeological spaces}
    A morphism of diffeological spaces $p \colon X \rightarrow Y$ is called a \textbf{fibration} if the induced morphisms of simplicial sets via the extended smooth singular complex functor $S_e(X) \rightarrow S_e(Y)$ is a Kan fibration.
\end{definition}

\begin{example}
Any diffeological principal $G$-bundle $\pi \colon P \rightarrow X$ for $G$ a diffeological group is a fibration of diffeological spaces which follows from \cite[Proposition 4.28]{Christensen-Wu}. Further, the morphism of simplicial sets $S_e(P) \rightarrow S_e(X)$ is even a principal $S_e(G)$-fibration by \cite[Theorem 1.3]{Kihara22}.
\end{example}

\chapter{Local Homotopy Theory} \label{Chapter: Local Homotpy Theory}

Introducing smooth homotopy groups and smooth singular complexes allowed us in the previous chapter to combine elements of classical homotopy theory with the concept of smoothness through diffeological spaces. At the heart of this unification lies the interplay between simplicial sets with their inherent homotopical structures, and sheaves on the Cartesian site, encoding smooth structures. The abstract theory combining these concepts is called \textit{local homotopy theory}, which is formulated through the language of model categories. The following two sections provide a short introduction to the topic concluding with the definition of the \textit{local} and \textit{\v{C}ech local model structures on simplicial presheaves}. The third section advances to give an introduction to the cohomology of simplicial sheaves. Our main references for this chapter include Hirschhorn \cite{Hirschhorn}, Lurie \cite{Lurie}, Dugger--Hollander--Isaksen \cite{Dugger-Hollander-Isaksen}, Jardine \cite{JardineBook} and Bunk \cite{Bunk}.

\section{A Primer on Model Categories} \label{Subsection: A Primer on Model Categories}

Model categories originated from the idea of axiomatizing the principles of homotopy theory, to provide sufficiently general methods to numerous applications. Also, model categories allowed for the interpretation of several fundamental concepts of homological algebra, such as derived categories and derived functors, as ``homotopical" in nature in a broad spectrum. As an example, we can consider the homological algebra of modules over a ring $R$ as an instance of abstract homotopy theory via the Dold--Kan correspondence. Concepts like injective and projective resolutions are then understood to be of the same nature as approximations of topological spaces by CW-complexes.

\begin{definition}
Let $\cat{C}$ be a category. Then a morphism $f \colon A \rightarrow B$ in $\cat{C}$ is a \textbf{retract} of a morphism $g$ if there is a commutative diagram
\begin{center}
    \begin{tikzcd}
            A \arrow[d, "f"] \arrow[r] & C \arrow[d, "g"] \arrow[r]  & A \arrow[d, "f"]  \\
            B  \arrow[r] & D \arrow[r] & B
    \end{tikzcd}
\end{center}
such that the horizontal composites are the identities on $A$ and $B$, respectively.
\end{definition}

\begin{remark}
Recall from ordinary homotopy theory that a space $A \hookrightarrow B$ is said to be a retract of $B$ if there is a map $r \colon B \rightarrow A$ such that $r \circ i = \mathrm{id}_{A}$. Considering the arrow category of $\cat{C}$ given by $\mathrm{Mor} \cat{C} = \mathrm{Funct}([1], \cat{C})$, the notion of a retract in $\mathrm{Mor} \cat{C}$ in the classical sense gives the above definition.
\end{remark}

\begin{definition} \label{Definition: Model Category}
A \textbf{model category} is a category $\cat{M}$ together with three distinguished classes of morphisms $\mathbf{Cof}, \mathbf{Fib}$ and $\mathbf{W}$, called cofibrations, fibrations, and weak equivalences, satisfying the following axioms.
\begin{enumerate}[label=\textbf{M\arabic*}]
    \item \underline{Limit Axiom:} The category $\cat{M}$ is complete and cocomplete i.e. has all small limits and colimits.
    \item \underline{2-of-3:} Given two composable morphisms  $f,g$ in $\cat{M}$, if two of the following morphisms $\left\{ f,g,g \circ f \right\}$ are in $W$, then so is the third.
    \item \underline{Retract Axiom:} The classes $\mathbf{Cof}, \mathbf{Fib}$ and $\mathbf{W}$ are closed under retracts.
    \item \underline{Lifting Axiom:} Given the following commutative diagram of solid arrows
\begin{center}
    \begin{tikzcd}
A \arrow[d, "i"] \arrow[r] & X \arrow[d, "p"] \\
B \arrow[r] \arrow[ur, dashed]& Y
    \end{tikzcd}
\end{center}
then the lift exists if either
\begin{itemize}
    \item $i \in \mathbf{Cof}$ and $p \in \mathbf{Fib} \cap \mathbf{W}$, or
    \item $i \in \mathbf{Cof} \cap \mathbf{W}$ and $p \in \mathbf{Fib}$.
\end{itemize}
    \item \underline{Factorization Axiom:} Every morphism $g$ in $\cat{M}$ has two factorizations:
\begin{itemize}
    \item $g = q \circ i$, with $i \in \mathbf{Cof}$ and $q \in \mathbf{Fib} \cap \mathbf{W}$, and
    \item $g = p \circ j$, with $j \in  \mathbf{Cof} \cap \mathbf{W}$ and $p \in \mathbf{Fib}$.
\end{itemize}
\end{enumerate}
\end{definition}

\begin{remark}
It is usually required for the factorization to be functorial. That is, given maps $f,g$ in $\cat{M}$ being part of a commutative diagram
\begin{center}
\begin{tikzcd}
A \arrow[d, "f"] \arrow[r, "h"] & C \arrow[d, "g"] \\
B \arrow[r, "k"'] & D
    \end{tikzcd}
\end{center}

the factorization of $f$ and $g$ gives:
\begin{center}
\begin{tikzcd}
A \arrow[d, "i/j"] \arrow[r, "h"] & C \arrow[d, "i'/j'"] \\
E \arrow[d, "q/p"] \arrow[r, dashed, "{F(h,k)}"] & F \arrow[d, "q'/p'"] \\
B \arrow[r, "k"'] & D
    \end{tikzcd}
\end{center}
Whereas the existence of the dashed map $F(h,k)$ follows from the fourth axiom, its naturality/functoriality $F(h_1 \circ h_2,k_1 \circ k_2) = F(h_1,k_1) \circ F(h_2,k_2)$ does not follow from any axiom and is usually guaranteed by the additional requirement for the factorization to be functorial. Three distinguished classes of morphisms that satisfy this property are called a \textit{functorial factorization system}.
\end{remark}

\begin{remark} \label{Remark: Overdetermined Axioms}
Notice that the axioms are overdetermined. That is, given only two classes the third one is completely determined via the lifting axiom, the factorization axiom, and the 2-of-3 axiom. Indeed,
\begin{itemize}
    \item given $\mathbf{Fib}$ and $\mathbf{W}$ then $\mathbf{Cof} := \mathrm{llp}(\mathbf{Fib} \cap \mathbf{W})$
    \item given $\mathbf{Cof}$ and $\mathbf{W}$ then $\mathbf{Fib} := \mathrm{rlp}(\mathbf{Cof} \cap \mathbf{W})$
    \item given $\mathbf{Cof}$ and $\mathbf{Fib}$  then $\mathbf{Fib} \cap \mathbf{W} := \mathrm{rlp}(\mathbf{Cof})$ and  $\mathbf{Cof} \cap \mathbf{W} := \mathrm{llp}(\mathbf{Fib})$, which then determine $\mathbf{W}$ via the factorization axiom.
\end{itemize}
\end{remark}

A model category should be thought of as a place specifically set up for abstract homotopy theory. Accordingly, one of the most important examples of a model category is constituted by the category of simplicial sets and weak homotopy equivalences.

\begin{definition} \label{Definition: Kan fibration}
A map $p \colon X \rightarrow Y$ is said to be a \textbf{Kan fibration} if it has the right lifting property with respect to the horn inclusions. That is, for every commutative square of solid arrows

\begin{center}
    \begin{tikzcd}
\Lambda^k[n] \arrow[r] \arrow[d, hook] & X \arrow[d, "p"] \\
\Delta[n] \arrow[r] \arrow[ur, dashed] & Y
    \end{tikzcd}
\end{center}
there is a lift making the diagram commute for any $0 \leq k \leq n$ and $n >0$.
\end{definition}

\begin{theorem}[\cite{Quillen67}] \label{Theorem: Kan-Quillen Structure}
The category $\mathbf{sSet}$ with
\begin{itemize}
    \item $\mathbf{Fib} = \left\{ \text{Kan fibrations} \right\}$,
    \item $\mathbf{W} = \left\{ \text{weak equivalences of simplicial sets} \right\}$, where a map $f \colon X \rightarrow Y$ is called a weak equivalence if its geometric realization $|f| \colon |X| \rightarrow |Y|$ is a weak homotopy equivalence,
    \item $\mathbf{Cof} = \left\{ \text{monomorphisms} \right\}$,
\end{itemize}
has the structure of a model category.
\end{theorem}

Another important class of model categories are functor categories. Let $\cat{M}$ be a fixed model category and consider the category of functors or diagrams $[\cat{C}, \cat{M}]$ from some small category $\cat{C}$ into $\cat{M}$. A natural transformation $\eta \colon F \rightarrow G$ in $[\cat{C}, \cat{M}]$ is said to be
\begin{itemize}
    \item an \textbf{injective cofibration} if for all $C \in \cat{C}$ the components $\eta_C \colon F(C) \rightarrow G(C)$ are cofibrations in $\cat{M}$,
    \item a \textbf{projective fibration} if for all $C \in \cat{C}$ the components $\eta_C \colon F(C) \rightarrow G(C)$ are fibrations in $\cat{M}$,
    \item a \textbf{pointwise} or \textbf{objectwise weak equivalence} if for all $C \in \cat{C}$ the components $\eta_C \colon F(C) \rightarrow G(C)$ are weak equivalences in $\cat{M}$.
\end{itemize}

In light of Remark \ref{Remark: Overdetermined Axioms} we expect the category of diagrams $[\cat{C},\cat{M}]$ to allow for two distinguished model structures.

\begin{itemize}
    \item The \textit{injective structure} with cofibrations given by injective cofibrations, weak equivalences given by pointwise weak equivalences, and fibrations determined by the lifting property.
    \item The \textit{projective structure} with fibrations given by projective fibrations, weak equivalences given by pointwise weak equivalences, and cofibrations determined by the lifting property.
\end{itemize}

The following Proposition confirms our hypothesis in the case $\cat{M}$ is sufficiently ``nice", i.e. a combinatorial model category.

\begin{proposition}[\cite{Lurie}, Proposition A.2.8.2] \label{Proposition: Existence of inj. / proj.}
    Let $\cat{M}$ be a combinatorial model category and $\cat{C}$ a small category. Then two combinatorial model structures on the category of diagrams $[\cat{C},\cat{M}]$ exist:
    \begin{itemize}
        \item The \textbf{projective model structure} given by projective fibrations, pointwise weak equivalences, and projective cofibrations being maps having the left lifting property with respect to trivial projective fibrations.
        \item The \textbf{injective model structure} given by injective cofibrations, pointwise weak equivalences, and injective fibrations being maps having the right lifting property with respect to trivial injective cofibrations.
    \end{itemize}
\end{proposition}

\begin{remark}\label{Remark: Combinatorial and Cofibrantly generated model categories}
    The notion of a combinatorial model category goes back to Jeffrey Smith \cite[Theorem 1.7]{Beke} and is as follows. A model category $\cat{M}$ is said to be \textbf{combinatorial} if the following two conditions are satisfied:
    \begin{enumerate}
        \item The category $\cat{M}$ is \textbf{locally presentable}, and
        \item the model category $\cat{M}$ is \textbf{cofibrantly generated}.
    \end{enumerate}
    The notion of a cofibrantly generated model category is due to Daniel M. Kan and is based on the observation that in the Kan--Quillen model structure on $\mathbf{sSet}$ the fibrations and trivial fibrations are completely determined by the right lifting property not merely with respect to the classes of all trivial cofibrations and cofibrations, but also with respect to a specific subset of trivial cofibrations $\left\{ \Lambda^k[n] \hookrightarrow \Delta[n] \; | \; n > 0 \text{ and } 0 \leq k \leq n \right\}$ called \textbf{generating trivial cofibrations} and a specific subset of cofibrations $\left\{ \partial \Delta[n] \hookrightarrow \Delta[n] \; | \; n \geq 0 \right\}$ called \textbf{generating cofibrations}. For a precise definition of a cofibrantly generated model category in general, we refer to \cite[Definition 11.1.2]{Hirschhorn}. \\

    It is interesting to point out that the existence of the projective model structure on the category of small diagrams $[\cat{C},\cat{M}]$ relies only on $\cat{M}$ being cofibrantly generated \cite[Theorem 11.6.1]{Hirschhorn}. This has been observed first by Dwyer--Kan in \cite{Dwyer-Kan} in the case where the diagrams take values in simplicial sets. For the injective model structure however, the additional requirement of $\cat{M}$ being locally presentable\footnote{Notice that what we call here locally presentable is denoted presentable in \cite{Lurie}.} is needed to prove its existence, see for example \cite[Proposition A.2.8.2]{Lurie}. As the category $\mathbf{sSet}$ is locally presentable it follows from the discussion in the previous paragraph that the Kan--Quillen model structure on $\mathbf{sSet}$ provides an example of a combinatorial model category.
\end{remark}

As already mentioned at the beginning of this section, local homotopy theory is about combining the homotopy theory of simplicial sets with the theory of sheaves, in our case defined on the Cartesian site $\mathbf{Cart}$. By considering the category of simplicial presheaves on $\mathbf{Cart}$, which is denoted by
\[
\mathrm{sPSh}(\mathbf{Cart}) := [\mathbf{Cart}^{\mathrm{op}}, \mathbf{sSet}],
\]
we conclude from Proposition \ref{Proposition: Existence of inj. / proj.} together with Theorem \ref{Theorem: Kan-Quillen Structure} that the injective and projective model structures on $\mathrm{sPSh}(\mathbf{Cart})$ both exist. From now on, the terms  \textit{injective} and \textit{projective model structure} always refer to the case of simplicial presheaves on the Cartesian site, if not specified otherwise. \\

The last important class of model categories we will encounter are \textit{transferred model structures}. Again, we start by fixing a model category $\cat{M}$ together with an adjoint pair of functors:
\begin{center}
\begin{tikzcd}
\cat{M} \arrow[r, "L", shift left=3] \arrow[r, "\perp", phantom] & \cat{D} \arrow[l, "R", shift left=3]
\end{tikzcd}
\end{center}
We now use the adjunction to transfer the model structure on the left-hand side along the right adjoint $R$, to define a model structure on $\cat{D}$. More precisely, we declare a morphism $f \colon D \rightarrow D'$ in $\cat{D}$ to be a
\begin{itemize}
    \item weak equivalence if its image under $R$ is a weak equivalence in $\cat{M}$.
    \item a fibration if its image under $R$ is a fibration in $\cat{M}$.
    \item a cofibration if it has the left lifting property with respect to all trivial fibrations.
\end{itemize}
Whenever the category $\cat{D}$ together with these three classes of morphisms satisfy the axioms of a model category, we call this the \textbf{right transferred model structure from} $\cat{M}$ \textbf{along} $R$. Analogous results to Proposition \ref{Proposition: Existence of inj. / proj.} guaranteeing the existence of such transferred model structures exist, such as \cite[Theorem 11.3.2]{Hirschhorn}. For our purposes, however, we only focus on one particular case: the free-forgetful adjunction between simplicial sets and simplicial groups together with its abelian counterpart.

\begin{center}
\begin{tikzcd}
\mathbf{sSet} \arrow[r, "F", shift left=3] \arrow[r, "\perp", phantom] & \mathbf{sGrp} \arrow[l, "U", shift left=3] & \mathbf{sSet} \arrow[r, "\mathbb{Z}", shift left=3] \arrow[r, "\perp", phantom] & \mathbf{sAb} \arrow[l, "U", shift left=3]
\end{tikzcd}
\end{center}

In both cases, one can show that the Kan--Quillen model structure on $\mathbf{sSet}$ of Theorem \ref{Theorem: Kan-Quillen Structure} can be right transferred along the forgetful functor as summarized in the following theorem.

\begin{theorem} \label{Theorem: model structure simplicial groups}
    The category of simplicial groups $\mathbf{sGrp}$ (simplicial abelian groups $\mathbf{sAb}$) allows for a right transferred model structure along the free-forgetful adjunction $F \dashv U$ ($\mathbb{Z} \dashv U$) where a morphism is
    \begin{itemize}
        \item a weak equivalence if the underlying map of simplicial sets is a weak equivalence.
        \item a fibration if the underlying map of simplicial sets is a Kan fibration.
        \item a cofibration if it has the left lifting property with respect to all trivial fibrations.
    \end{itemize}
    In particular, the model structure is cofibrantly generated with generating (trivial) cofibrations given by the images of the generating (trivial) cofibrations in $\mathbf{sSet}$ under the free simplicial group functor $F$ (free simplicial abelian group functor $\mathbb{Z}$).
\end{theorem}
\begin{proof}
    The existence of these two model structures has been shown in \cite{Quillen67}. The fact that they are cofibrantly generated\footnote{See Remark \ref{Remark: Combinatorial and Cofibrantly generated model categories} for more details on cofibrantly generated model categories.} follows from the fact that the Kan--Quillen structure is cofibrantly generated together with \cite[Theorem 11.3.2]{Hirschhorn}.
\end{proof}

\begin{definition}
Since any model category $\cat{M}$ is complete and cocomplete, it has an initial object $\emptyset$ and a terminal object $*$. Then we say that an object $X \in \cat{M}$ is
\begin{enumerate}[label=\roman*)]
    \item \textbf{cofibrant} if the unique map $\emptyset \rightarrow X$ is a cofibration.
    \item \textbf{fibrant} if the unique map $X \rightarrow *$ is a fibration.
\end{enumerate}
\end{definition}

\begin{remark} $ $ \\
A direct consequence of the functorial factorization axiom is that we have two functors \newline
$Q \colon \cat{M} \rightarrow \cat{M}$ and $R \colon \cat{M} \rightarrow \cat{M}$ characterized by the following properties:
\begin{itemize}
    \item $Q(X)$ is cofibrant for all $X$ and there is a weak equivalence $Q(X) \xrightarrow{\simeq} X$. \\
    \item  $R(X)$ is fibrant for all $X$ and there is a weak equivalence $X \xrightarrow{\simeq} R(X)$.
\end{itemize}
These functors, called \textbf{cofibrant} and \textbf{fibrant approximation} functors, are obtained by applying the functorial factorization to the unique maps $\emptyset \rightarrow X$ and $X \rightarrow *$.
\end{remark}

\begin{example} $ $
\begin{itemize}
    \item In the Quillen model structure on $\mathbf{sSet}$ the fibrant objects are given by the Kan complexes. On the other hand, all the objects are cofibrant since $\emptyset \rightarrow X$ is indeed a monomorphism.
    \item In the injective model structure all the objects are cofibrant, as they are pointwise cofibrant.
    \item In the projective model structure the fibrant objects are given by simplicial presheaves taking values in Kan complexes.
    \item In the transferred model structures on both $\mathbf{sGrp}$ and $\mathbf{sAb}$ all the objects are fibrant since a simplicial group is automatically a Kan complex, see for example \cite{Quillen67}.
    \item In the transferred model structures on both $\mathbf{sGrp}$ and $\mathbf{sAb}$ any free simplicial group $F(K)$ and any free simplicial abelian group $\mathbb{Z}(K)$ are cofibrant respectively.
\end{itemize}
\end{example}

\begin{definition}[\cite{Riehl}]
Let $\cat{M}$ and $\cat{N}$ be model categories. Then a functor
\begin{itemize}
    \item $F \colon \cat{M} \rightarrow \cat{N}$ is said to be  \textbf{left Quillen} if $F$ preserves colimits, cofibrations and trivial cofibrations.
    \item $G \colon \cat{N} \rightarrow \cat{M}$ is said to be \textbf{right Quillen} if $G$ preserves limits, fibrations and trivial fibrations.
\end{itemize}
If, in addition, we have that $F \dashv G$ are an adjoint pair, we call it a \textbf{Quillen adjunction}.
\end{definition}
There is a nice fact about Quillen adjunctions: if we are given adjoint functors, then the fact that it is Quillen follows if either the right or the left adjoint is Quillen.
\begin{lemma} \label{Lemma: Quillen Adj Lemma}
An adjunction $\adj{F: \cat{M} \:}{\: \cat{N}: G}$ is a Quillen adjunction if any of the following equivalent conditions hold:
\begin{enumerate}[label=\roman*)]
    \item $F$ is left Quillen.
    \item $G$ is right Quillen.
\end{enumerate}
\end{lemma}

The most important examples of Quillen adjunctions are the \textit{Quillen equivalences}. From a heuristic point of view, two model categories are said to be Quillen equivalent if their homotopy theories are ``equivalent".

\begin{definition}[\cite{Hirschhorn}] \label{Definition: Quillen Equivalence}
    Let $\cat{M}$ and $\cat{N}$ be model categories and let $F : \cat{M} \rightleftarrows \cat{N} : G$ be a Quillen adjunction. Then we call
    \begin{enumerate}
        \item $F$ a \textbf{left Quillen equivalence},
        \item $G$ a \textbf{right Quillen equivalence},
        \item $(F,G)$ a \textbf{pair of Quillen equivalences},
    \end{enumerate}
    if for every cofibrant object $B$ in $\cat{M}$ and every fibrant object $X$ in $\cat{N}$, and every morphism $f \colon B \rightarrow G(X)$ in $\cat{M}$, the morphism $f$ is a weak equivalence in $\cat{M}$ if and only if its adjoint morphism $f^* \colon F(B) \rightarrow X$ is a weak equivalence in $\cat{N}$.
\end{definition}

\begin{example} \label{Example: Equivalence between proj and inj on diagrams}
    Recall the injective and projective structures introduced earlier on the diagram category $[\cat{C}, \cat{M}]$. Here the identity functor provides a first example of a Quillen pair.
    \begin{center}
    \begin{tikzcd}
{[\cat{C},\cat{M}]_{\mathrm{proj}}} \arrow[r, "\mathrm{id}", shift left=3] \arrow[r, "\perp", phantom] & {[\cat{C},\cat{M}]_{\mathrm{inj}}} \arrow[l, "\mathrm{id}", shift left=3]
\end{tikzcd}
\end{center}
Indeed, first notice that in both model structures, the class of weak equivalences coincide. As a consequence of Lemma \ref{Lemma: Quillen Adj Lemma} it suffices to show that one of the identity functors either preserves cofibrations or fibrations. It then directly follows that the adjunction is a pair of Quillen equivalence. The fact that every projective cofibration is also an injective cofibration (or,  dually, that every injective fibration is a projective fibration) is presented in \cite[Remark A.2.8.5]{Lurie}.
\end{example}

\begin{definition} \label{Definition: Simplicial Model Category}
A \textbf{simplicial model category} is a model category $\cat{M}$ that is a simplicially enriched category, which is tensored and cotensored over $\mathbf{sSet}$ such that the following condition is satisfied: \\

For any cofibration $i \colon A \rightarrow B$ and any fibration $p \colon X \rightarrow Y$ in $\cat{M}$, the map of simplicial sets
    \[
    \mathrm{Map}(B,X) \rightarrow \mathrm{Map}(A,X) \times_{\mathrm{Map}(A,Y)} \mathrm{Map}(B,Y)
    \]
    is a fibration that is a trivial fibration if either $i$ or $p$ is a weak equivalence.
\end{definition}

All the model structures that have been introduced so far are simplicial. The most important ones for us are the model structures on simplicial presheaves, where the corresponding simplicial enrichment is defined below.

\begin{itemize}
    \item[] The category of simplicial presheaves $\mathrm{sPSh}(\mathbf{Cart})$ carries a natural simplicial enrichment. For $K$ a simplicial set and $F$ a simplicial presheaf, define the simplicial presheaves $K \otimes F$ and $F^K$ objectwise by
    \[
    (K \otimes F)(U):= K \times F(U) \text{ and } \left( F^K \right)(U) = F(U)^K
    \]
    The simplicial mapping space $\mathrm{Map}(F,G)$ for $F,G \in \mathrm{sPSh}(\mathbf{Cart})$ is then given by
    \[
    \mathrm{Map}(F,G)_n := \mathrm{Hom}(\Delta[n] \otimes F,G)
    \]
    This turns $\mathrm{sPSh}(\mathbf{Cart})$ into a simplicially enriched category which is tensored and cotensored over $\mathbf{sSet}$. Both the injective and the projective model structures are simplicial.
\end{itemize}

\begin{definition}
    Given a simplicial model category $\cat{M}$ that is also cofibrantly generated, we denote by $\mathbb{R}\mathrm{Map}(-,-)$ the functor
    \begin{align*}
        \cat{M}^{\mathrm{op}} \times \cat{M} &\longrightarrow \mathbf{sSet} \\
        (X,Y) &\longmapsto \mathrm{Map}(QX,RY)
    \end{align*}
    where $R$ and $Q$ denote the simplicial cofibrant and fibrant replacement functors. Given objects $X$ and $Y$, we call the space $\mathbb{R}\mathrm{Map}(X,Y)$ the \textbf{derived hom space} or \textbf{homotopy function complex} from $X$ to $Y$.
\end{definition}

\begin{remark}
    For the assignment
    \[
    (X,Y) \rightarrow \mathbb{R}\mathrm{Map}(X,Y)
    \]
    to be functorial, we need the cofibrant and fibrant replacement functors to be simplicial. The condition of $\cat{M}$ being cofibrantly generated assures that the replacement functors have this property, see \cite[Theorem 6.1]{Blumberg-Riehl}.
\end{remark}

The condition for a simplicially enriched model category to be a simplicial model category in Definition \ref{Definition: Simplicial Model Category} appears cryptic at first glance. From a more naive point of view, a simplicial model category $\cat{M}$ is such that the classic notion of \textit{simplicial homotopy} between two morphisms $f,g \in \mathbb{R}\mathrm{Map}(X,Y)$ relates to the abstract notion of homotopy between $f$ and $g$. To be more precise, given a model category $\cat{M}$, an abstract notion of homotopy can be introduced using either a cylinder object or a path-space object. This notion of homotopy between morphisms in $\cat{M}$, in turn, allows to define an equivalence relation $\sim$ on the set of morphisms $\mathrm{Hom}_{\cat{M}}(Q(X),R(Y))$. The \textit{homotopy category} $\mathrm{Ho} \, \cat{M}$ of $\cat{M}$ can then be defined as having the same objects as $\cat{M}$ and as morphisms the equivalence classes, i.e. abstract homotopy classes of morphisms in $\cat{M}$.
\[
\mathrm{Hom}_{\mathrm{Ho} \, \cat{M}}(X,Y) := \mathrm{Hom}_{\cat{M}}(RQ(X),RQ(Y))/\! \sim
\]

\begin{proposition}[\cite{Hirschhorn}, Theorem 17.7.2] \label{Proposition: simplicial homotopy classes and morphism in Ho-Cat}
    Let $\cat{M}$ be a simplicial model category. Then for objects $X,Y \in \cat{M}$ there is a natural isomorphism
    \[
    \pi_0 \mathbb{R}\mathrm{Map}(X,Y) \cong \mathrm{Hom}_{\mathrm{Ho} \, \cat{M}}(X,Y)
    \]
\end{proposition}

\begin{definition} \label{Definition: Left/Right derived functors}
    Let $\cat{M}$ and $\cat{N}$ be model categories and let $F \colon \cat{M} \rightarrow \cat{N}$ be a left Quillen functor. Denote by $Q(-)$ and $R(-)$ a choice of functorial cofibrant and fibrant approximation on $\cat{M}$. Then the \textbf{left derived functor} $\mathbb{L}F$ \textbf{of} $F$ is defined as $\mathbb{L}F := F(Q(-))$. Dually, if $F$ is a right Quillen functor, the \textbf{right derived functor} $\mathbb{R}F$ \textbf{of} $F$ is defined as $\mathbb{R}F := F(R(-))$.
\end{definition}
\begin{remark} \label{Remark: Derived Functors}
    First notice that the left and right derived functors both are functors $\mathbb{L}F \colon \cat{M} \rightarrow \cat{N}$ and $\mathbb{R}F \colon \cat{M} \rightarrow \cat{N}$, in contrast to the notion of \textit{total left/right derived functor} which has as domain and codomain the corresponding homotopy categories. An important aspect of the definition of left/right derived functor, as introduced above, is that it depends on a choice of cofibrant/fibrant approximation. Defining left/right derived functors like this corresponds to a choice of model for $\mathbb{L}F$ and $\mathbb{R}F$. More precisely, the concept of left/right derived functor can be defined for a much wider class of objects than model categories, such as categories with weak equivalences (also known as homotopical categories). In this setting a left derived functor $\mathbb{L}F$ of $F \colon \cat{M} \rightarrow \cat{N}$ is a pair $(\mathbb{L}F,\lambda)$ where $\mathbb{L}F$ is a functor from $\cat{M}$ to $\cat{N}$ preserving weak equivalences and $\lambda \colon \mathbb{L}F \Rightarrow F$ a natural transformation, such that $\delta\lambda \colon \delta \circ \mathbb{L}F \Rightarrow \delta \circ F$ is a right Kan extension $\mathrm{Ran}_{\gamma}\delta F$, where $\gamma \colon \cat{M} \rightarrow \mathrm{Ho} \, \cat{M}$ and $\delta \colon \cat{N} \rightarrow \mathrm{Ho} \, \cat{N}$ are the localization functors.
    \begin{center}
    \begin{tikzcd}[row sep=huge, column sep=large]
\cat{M} \arrow[r, "F", bend left, shift left=3] \arrow[d, "\gamma"'] \arrow[r, "\mathbb{L}F"', bend right, shift right] \arrow[r, "\Uparrow \lambda", phantom, shift left] & \cat{N} \arrow[d, "\delta"] \\
{\mathrm{Ho} \, \cat{M}} \arrow[r, "\mathrm{Ran}_{\gamma}\delta F"'] \arrow[r, "\Uparrow \mathrm{id}", phantom, shift left=6]                                              & {\mathrm{Ho} \, \cat{N}}
\end{tikzcd}
    \end{center}
    Consider now the case where the homotopical categories are also model categories and $F$ is a left Quillen functor. Then the pair $(F(Q(-)), F(q))$, where $q \colon Q  \Rightarrow \mathrm{id}$, is the natural transformation with components given by the weak equivalences $q_X \colon Q(X) \xrightarrow{\simeq} X$, is indeed a left derived functor of $F$. This is spelled out in detail in Riehl \cite[Chapter 2]{Riehl}.
\end{remark}

\begin{proposition} \label{Proposition: Derived Quillen adjunction is adjunction in Ho-Cat}
    Let $\cat{M}$ and $\cat{N}$ be simplicial model categories and let $F : \cat{M} \rightleftarrows \cat{N} : G$ be a Quillen adjunction. Then there are natural isomorphisms
    \[
    \pi_0 \mathbb{R}\mathrm{Map}\left( \mathbb{L}F(X),Y \right) \cong \pi_0 \mathbb{R}\mathrm{Map}\left(X,\mathbb{R}G(Y)\right)
    \]
\end{proposition}
\begin{proof}
    This follows from \cite[Theorem 8.5.18]{Hirschhorn} together with the previous observation of Proposition \ref{Proposition: simplicial homotopy classes and morphism in Ho-Cat}.
\end{proof}
\begin{remark}
    Spelled out this means that, given objects $X \in \cat{M}$ and $Y \in \cat{N}$, the left-hand side reads:
    \begin{align*}
    \pi_0 \mathbb{R}\mathrm{Map}\left( \mathbb{L}F(X),Y \right) &= \pi_0 \mathrm{Map}\left( Q_{\cat{N}}\left( F(Q_{\cat{M}}(X)) \right), R_{\cat{N}}(Y) \right) \\
    &\cong \pi_0\mathrm{Map}\left(F(Q_{\cat{M}}(X)), R_{\cat{N}}(Y) \right)
    \end{align*}
    and likewise for the right-hand side:
        \begin{align*}
    \pi_0 \mathbb{R}\mathrm{Map}\left( X,\mathbb{R}G(Y) \right) &= \pi_0 \mathrm{Map}\left(Q_{\cat{M}}(X), R_{\cat{M}}\left( F \left( R_{\cat{N}}(Y) \right) \right) \right) \\
    &\cong \pi_0 \mathrm{Map}\left(Q_{\cat{M}}(X),  F \left( R_{\cat{N}}(Y) \right) \right)
    \end{align*}
    so we conclude that there are natural isomorphisms:
    \[
    \pi_0\mathrm{Map}\left(F(Q_{\cat{M}}(X)), R_{\cat{N}}(Y) \right) \cong \pi_0 \mathrm{Map}\left(Q_{\cat{M}}(X),  F \left( R_{\cat{N}}(Y) \right) \right)
    \]
\end{remark}

\begin{definition}
Let $\cat{M}$ be a simplicial model category and $S$ a set of arrows in $\cat{M}$. Then we say that an object $X$ is $S$\textbf{-local} if, for all $f \colon A \rightarrow B$ in $S$, the induced map
\[
\mathbb{R}\mathrm{Map}(B,X) \rightarrow \mathbb{R}\mathrm{Map}(A,X)
\]
is a weak equivalence of simplicial sets.
\end{definition}

The following theorem guarantees the existence of a localized model structure having as fibrant objects the $S$-local objects.

\begin{theorem}[\cite{Barwick10}]\label{Theorem: Enriched Bousfield Localization}
Let $\cat{M}$ be a combinatorial left proper simplicial model category and let $S$ be a set of arrows in $\cat{M}$. Then there is a model structure on $\cat{M}$ denoted by $L_S\cat{M}$ satisfying the following properties:
\begin{enumerate}
    \item The cofibrations in $L_S\cat{M}$ are the cofibrations in $\cat{M}$.
    \item The fibrant objects of $L_S\cat{M}$ are the fibrant objects of $\cat{M}$ which are also $S$-local.
    \item The weak equivalences of $L_S\cat{M}$ are the maps $f \colon X \rightarrow Y$ such that the induced map
    \[
    \mathbb{R}\mathrm{Map}(Y,Z) \rightarrow \mathbb{R}\mathrm{Map}(X,Z)
    \]
    is a weak equivalence of simplicial sets for every $S$-local object $Z$.
\end{enumerate}
Moreover, $L_S\cat{M}$ is left proper and combinatorial. In particular, if $\cat{M}$ admits a set of generating cofibrations with a cofibrant source, then $L_S\cat{M}$ is also simplicial.
\end{theorem}

\section{The Local and \v{C}ech-Local model structures}

Recall that Proposition \ref{Proposition: Diffeological Spaces are Concrete Sheaves on Cart} identifies diffeological spaces as a special kind of sheaves defined on the Cartesian site $\mathbf{Cart}$. The essence of Proposition \ref{Proposition: Diffeological Spaces are Concrete Sheaves on Cart} is part of a bigger picture to be introduced within this section. \\

Let $\cat{D}$ be a site and $F$ a sheaf of sets on $\cat{D}$. Given a covering family $\left\{ f_i \colon D_i \rightarrow D \, | \, i \in I \right\}$ the \textit{sheaf property} of $F$ can be phrased as follows: for all $i_0,i_1 \in I$ consider the fiber product of representable presheaves $\cat{Y}\left( D_{i_0} \right) \times_{\cat{Y}(D)} \cat{Y}\left( D_{i_1} \right)$ which comes equipped with projection maps $\mathrm{pr}_0$ and $\mathrm{pr}_1$ to $\cat{Y}\left( D_{i_0} \right)$ and $\cat{Y}\left( D_{i_1} \right)$. Here we denote by $\cat{Y} \colon \cat{D} \rightarrow \mathrm{PSh}(\cat{D})$ the Yoneda embedding. These projection maps now induce two natural maps
\begin{center}
\begin{tikzcd}
\prod_{i \in I}F(D_i) \arrow[r, "\mathrm{pr}^*_1"', shift right=2] \arrow[r, "\mathrm{pr}^*_0", shift left=2] & {\prod_{(i_0,i_1) \in I \times I} \mathrm{Hom}_{\mathrm{PSh}(\cat{D})}\left( \cat{Y}\left( D_{i_0} \right) \times_{\cat{Y}(D)} \cat{Y} \left( D_{i_1} \right),F \right)}
\end{tikzcd}
\end{center}
where on the left-hand side we have identified $\mathrm{Hom}_{\mathrm{PSh}(\cat{D})} \left( \cat{Y}(D_i) , F \right) \cong F(D_i)$ via the Yoneda Lemma. The maps $f_i \colon D_i \rightarrow D$ themselves induce a map
\[
\begin{array}{rcl}
     F(D) & \rightarrow & \displaystyle \prod_{i \in I } F(D_i)  \\
     s & \mapsto  & (s|_{D_i})_{i \in I}
\end{array}
\]
Then the property of $F$ being a sheaf on $\cat{D}$ is equivalent to the above map being an equalizer of the projection maps $\mathrm{pr}^*_0$ and $\mathrm{pr}^*_1$. More precisely, we can give an equivalent definition of a sheaf of sets on a site.
\newpage

\begin{definition}
    Let $\cat{D}$ be a site and consider $F$ a presheaf on $\cat{D}$. Then $F$ is said to be a \textbf{sheaf on the site} $\cat{D}$ if for every covering family $\left\{ f_i \colon D_i \rightarrow D \right\}_{i \in I}$ the diagram
\begin{center}
\begin{tikzcd}
F(D) \arrow[r] & \prod_{i \in I}F(D_i) \arrow[r, "\mathrm{pr}^*_1"', shift right=2] \arrow[r, "\mathrm{pr}^*_0", shift left=2] & {\prod_{(i_0,i_1) \in I \times I} \mathrm{Hom}_{\mathrm{PSh}(\cat{D})}\left( \cat{Y}\left( D_{i_0} \right) \times_{\cat{Y}(D)} \cat{Y} \left( D_{i_1}\right),F \right)}
\end{tikzcd}
\end{center}
exhibits the left arrow as an equalizer.
\end{definition}
\begin{remark} $ $
    \begin{enumerate}[label={(\arabic*)}]
        \item In the case the site $\cat{D}$ has the property that given any covering family \newline $\left\{ f_i \colon D_i \rightarrow D \right\}_{i \in I}$ the pullbacks $D_{i_0} \times_D D_{i_1}$ exist in $\cat{D}$, the sheaf property simplifies via the Yoneda Lemma
        \[
            \mathrm{Hom}_{\mathrm{PSh}(\cat{D})}\left( \cat{Y}\left( D_{i_0} \right) \times_{\cat{Y}(D)} \cat{Y} \left( D_{i_1} \right),F \right) \cong F \left( D_{i_0} \times_{D} D_{i_1} \right).
        \]
        Therefore, a presheaf $F$ is a sheaf if and only if the diagram
        \begin{center}
            \begin{tikzcd}
            F(D) \arrow[r] & \prod_{i \in I}F(D_i) \arrow[r, "\mathrm{pr}^*_1"', shift right=2] \arrow[r, "\mathrm{pr}^*_0", shift left=2] & {\prod_{(i_0,i_1) \in I \times I} F \left( D_{i_0} \times_{D} D_{i_1} \right)}
            \end{tikzcd}
            \end{center}
            exhibits the left arrow as an equalizer.
        \item This definition certainly agrees with Definition \ref{Definition: Sheaf Property} presented in the previous chapter. Indeed, consider the site $\mathbf{Cart}$ of Cartesian spaces and good open coverings. First, we notice that in this site the pullbacks $U_i \times_U U_j$ indeed exist, since they are simply given by the intersection $U_{ij}$. Then the equalizer diagram translates into: for any Cartesian space $U$ together with a good open covering $\left\{U_i\right\}_{i \in I}$ and for any family of sections $s_i \in F(U_i)$ satisfying
        \begin{equation} \label{Equation: compatible section}
        s_i|_{U_{ij}} = s_j|_{U_{ij}}
        \end{equation}
        there exists a unique $s \in F(U)$ with $s_i = s|_{U_i}$. The condition \eqref{Equation: compatible section} exhibits the collection of sections $\left\{ s_i \in F(U_i) \right\}$ as a compatible collection of sections.
    \end{enumerate}
\end{remark}

So far, we have considered only presheaves of sets, and no homotopy theory was needed to introduce the theory of sheaves. To understand how to generalize the sheaf condition for presheaves of spaces using homotopy theory, we first present a heuristic discussion inspired by the survey of Mestrano--Simpson \cite{Mestrano-Simpson}.

The gluing property a sheaf of sets $F$ satisfies tells us that $F(U)$ is isomorphic to the set of \textit{descent data} with respect to some open covering of $U$. This set of descent data is simply the equalizer of the above diagram for some covering family $\left\{ U_i \rightarrow U \right\}_{i \in I}$ of $U$. However, considering a presheaf of spaces, i.e. a simplicial presheaf, we would rather define a \textit{space of descent data} that is well-behaved in the context of homotopy theory. \\

The idea is that given sections $s_i \in F(U_i)$, instead of asking for an equality
\[
s_i|_{U_{ij}} = s_j|_{U_{ij}},
\]
it is required to ask for a path connecting these local sections inside the space $F(U_{ij})$. A choice of such a path is then also considered part of the gluing data. This approach results in the concept of a \textit{space of descent data}
\[
\mathrm{Desc}(F, \left\{ U_i \right\}) \ni (s_i,s_j, \gamma_{ij})
\]
consisting of triples where $s_i$ and $s_j$ are sections $s_i \in F(U_i)$ and $s_j \in F(U_j)$ and $\gamma_{ij}$ a path in $F(U_{ij})$ such that $\gamma_{ij}(0) = s_i|_{U_{ij}}$ and $\gamma_{ij}(1) = s_j|_{U_{ij}}$. \\

Nevertheless, this space is conceptually not yet satisfactory. To see why, consider a covering of $U$ given by three open subsets $U_0,U_1$ and $U_2$ such that $U_0 \cap U_1 \cap U_2 \neq \emptyset$. Now assume that in the space of descent data $\mathrm{Desc}(F, \left\{ U_i \right\})$ we are given elements $(s_0,s_1,\gamma_{01})$, $(s_1,s_2,\gamma_{12})$ and $(s_0,s_2,\gamma_{02})$. By restricting all the data to the space $F(U_{012})$ one gets the following picture.
\begin{center}
    \begin{tikzpicture}[x=0.75pt,y=0.75pt,yscale=-1,xscale=1]

\draw    (105,220) .. controls (144.75,187) and (133.25,170.5) .. (145,155) .. controls (156.75,139.5) and (184.75,144) .. (195,125) ;
\draw [shift={(195,125)}, rotate = 298.35] [color={rgb, 255:red, 0; green, 0; blue, 0 }  ][fill={rgb, 255:red, 0; green, 0; blue, 0 }  ][line width=0.75]      (0, 0) circle [x radius= 3.35, y radius= 3.35]   ;
\draw [shift={(105,220)}, rotate = 320.3] [color={rgb, 255:red, 0; green, 0; blue, 0 }  ][fill={rgb, 255:red, 0; green, 0; blue, 0 }  ][line width=0.75]      (0, 0) circle [x radius= 3.35, y radius= 3.35]   ;
\draw    (195,125) .. controls (199.75,154) and (216.75,150) .. (225,175) .. controls (233.25,200) and (241.75,184.5) .. (260,220) ;
\draw [shift={(260,220)}, rotate = 62.79] [color={rgb, 255:red, 0; green, 0; blue, 0 }  ][fill={rgb, 255:red, 0; green, 0; blue, 0 }  ][line width=0.75]      (0, 0) circle [x radius= 3.35, y radius= 3.35]   ;
\draw [shift={(195,125)}, rotate = 80.7] [color={rgb, 255:red, 0; green, 0; blue, 0 }  ][fill={rgb, 255:red, 0; green, 0; blue, 0 }  ][line width=0.75]      (0, 0) circle [x radius= 3.35, y radius= 3.35]   ;
\draw    (105,220) .. controls (145,190) and (156.25,226.5) .. (185,215) .. controls (213.75,203.5) and (233.25,210.5) .. (260,220) ;
\draw [shift={(105,220)}, rotate = 323.13] [color={rgb, 255:red, 0; green, 0; blue, 0 }  ][fill={rgb, 255:red, 0; green, 0; blue, 0 }  ][line width=0.75]      (0, 0) circle [x radius= 3.35, y radius= 3.35]   ;

\draw (56,218.4) node [anchor=north west][inner sep=0.75pt]    {$s_{0} |_{U_{012}}{}$};
\draw (174,97.4) node [anchor=north west][inner sep=0.75pt]    {$s_{1} |_{U_{012}}$};
\draw (269,218.4) node [anchor=north west][inner sep=0.75pt]    {$s_{2} |_{U_{012}}$};
\draw (81,147.4) node [anchor=north west][inner sep=0.75pt]    {$\gamma _{01} |_{U}{}_{_{012}}$};
\draw (231,147.4) node [anchor=north west][inner sep=0.75pt]    {$\gamma _{12} |_{U}{}_{_{012}}$};
\draw (178,227.4) node [anchor=north west][inner sep=0.75pt]    {$\gamma _{02} |_{U}{}_{_{012}}$};
\draw (166,172.4) node [anchor=north west][inner sep=0.75pt]    {$\Downarrow H$};

\end{tikzpicture}
\end{center}
The three paths form a triangle, i.e. the boundary of a 2-simplex in the space $F(U_{012})$. Instead of asking for a cocycle condition to hold, that is
\[
\gamma_{01} * \gamma_{12} = \gamma_{02},
\]
one should rather ask for a homotopy to exist, which mighty not always be the case, connecting these two paths $H \colon \gamma_{01} * \gamma_{12} \Rightarrow \gamma_{02} $, i.e. a $2$-simplex $H \colon \Delta^2 \rightarrow F(U_{012})$ filling in the above diagram. This homotopy should therefore also be part of the descent data and our preliminary definition for the space of descent data $\mathrm{Desc}(F, \left\{ U_i \right\}) $ needs to be further enhanced. However, this argument continues further considering arbitrary good open coverings and one concludes that these homotopies themselves need to be connected by other homotopies on four-fold intersections and so forth. \\

The solution to finding an appropriate definition of the space $\mathrm{Desc}(F,\left\{ U_i \right\})$ for an arbitrary open covering is to consider the \textit{homotopy limit} of the extended diagram:
\begin{center}
\begin{tikzcd}
\mathrm{holim} \arrow[r] & \prod_{i \in I}F(U_i) \arrow[r, shift left=2] \arrow[r, shift right=2] & {\underset{(i_0,i_1) \in I \times I}{\prod} F(U_{i_0i_1})} \arrow[r, shift left=3] \arrow[r] \arrow[r, shift right=3] & {\underset{(i_0,i_1,i_2) \in I \times I \times I}{\prod} F(U_{i_0i_1i_2})} \arrow[r, shift left=3] \arrow[r, shift left] \arrow[r, shift right] \arrow[r, shift right=3] & \cdots
\end{tikzcd}
\end{center}
 This idea is made precise by the concept of \textbf{\v{C}ech descent} for simplicial presheaves on a site $\cat{D}$.

\begin{definition}[\cite{Dugger-Hollander-Isaksen}, Definition 4.3] \label{Definition: Cech descent}
Let $\cat{D}$ be a site satisfying the additional requirement that for any covering family $\left\{ f_i \colon D_i \rightarrow D \; | \; i \in I \right\}$, the pullbacks $D_{i_0 \cdots i_n} = D_{i_0 \cdots i_{n-1}} \times_D D_{i_n}$ exists in $\cat{D}$ for all $n \geq 1$. Given an objectwise fibrant simplicial presheaf $F \in \mathrm{sPSh}(\cat{D})$ on $\cat{D}$ we say that $F$ satisfies \textbf{\v{C}ech descent} if for every covering family $\left\{ f_i \colon D_i \rightarrow D \; | \; i \in I \right\}$ the natural map from $F(D)$  to the homotopy limit of the diagram
\begin{equation} \label{Equation: Cech descent natural map}
    \begin{tikzcd}
     \underset{i \in I}{\prod}F(D_i) \arrow[r, shift left=2] \arrow[r, shift right=2] & {\underset{(i_0,i_1) \in I \times I}{\prod} F(D_{i_0i_1})} \arrow[r, shift left=3] \arrow[r] \arrow[r, shift right=3] & {\underset{(i_0i_1i_2) \in I \times I \times I}{\prod} F(D_{i_0,i_1,i_2})} \arrow[r, shift left=3] \arrow[r, shift left] \arrow[r, shift right] \arrow[r, shift right=3] & \cdots
    \end{tikzcd}
    \end{equation}
    is a weak equivalence. In the case $F$ is not an objectwise fibrant simplicial presheaf, we say it satisfies \textbf{\v{C}ech descent} if some objectwise fibrant replacement for $F$ does.
\end{definition}
\begin{remark}
    Notice that all the sites introduced so far satisfy this additional requirement, in particular the Cartesian site $\mathbf{Cart}$.
\end{remark}
We want to characterize simplicial presheaves on a site $\cat{D}$ satisfying \v{C}ech descent using a certain model structure on the category of simplicial presheaves. As before, such a model structure is obtained by localization of either the injective or the projective model structure on $\mathrm{sPSh}(\cat{D})$. To localize an appropriate set of morphisms, $S$ needs to be constructed such that the $S$-local objects are exactly the presheaves satisfying \v{C}ech descent. As a way to construct this set $S$ the \v{C}ech nerve of a covering family is slightly adapted to what is called the associated \v{C}ech complex. This also allows us to define \v{C}ech descent for arbitrary sites $\cat{D}$ which do not satisfy the additional requirement of Definition \ref{Definition: Cech descent}.

\begin{definition} \label{Definition: Cech complexx}
    Let $\cat{D}$ be a site and let $\left\{ f_i \colon D_i \rightarrow D \right\}_{i \in I}$ be a covering family. There is an associated presheaf of sets on $\cat{D}$ defined as the coproduct
    \[
    U_D := \coprod_{i \in I } \cat{Y}(D_i)
    \]
    of the associated representable presheaves. The morphisms $f_i$ further induce a natural map $f \colon U \rightarrow \cat{Y}(D)$ and the associated \textbf{\v{C}ech complex} is the simplicial presheaf given by
    \[
    [n] \mapsto \underbrace{ U_D \times_{\cat{Y}(D)} \cdots \times_{\cat{Y}(D)} U_D}_{(n+1)\text{ factors}}.
    \]
    The \v{C}ech complex is denoted $\check{C}(U_D)$ and is equipped with a natural map of simplicial presheaves
    \[
    f \colon \check{C}(U_D) \rightarrow \cat{Y}(D)
    \]
    which is also denoted by $f$.
\end{definition}

Recall that any site $\cat{D}$ comes equipped with a \textbf{set} $\mathrm{Cov}(\cat{D})$ of coverings of $\cat{D}$ since by definition a site is a small category. Then the set of morphisms in $\mathrm{sPSh}(\cat{D})$, denoted by $S_{\check{C}}$, is given by
\[
S_{\check{C}} := \left\{f \colon \check{C}(U_D) \rightarrow \cat{Y}(D) \, | \, \left\{ f_i \colon D_i \rightarrow D \right\} \in \mathrm{Cov}(\cat{D})  \right\},
\]
which is the collection of \v{C}ech complexes associated to any covering family that exists in $\cat{D}$. The next lemma now characterizes the $S_{\check{C}}$-local objects.

\begin{lemma}[\cite{Dugger-Hollander-Isaksen}] \label{Lemma: S-local equals Cech descent}
    Let $\cat{D}$ be a site and consider the associated set of morphisms of \v{C}ech complexes $S_{\check{C}}$. A simplicial presheaf $F \in \mathrm{sPSh}(\cat{D})$ satisfies \v{C}ech descent if and only if for any $ f \colon \check{C}(U_D) \rightarrow \cat{Y}(V)$ in $S_{\check{C}}$ the induced map of homotopy function complexes
    \[
    \mathbb{R}\mathrm{Map}(\cat{Y}(V),F) \rightarrow \mathbb{R}\mathrm{Map}(\cat{C}(U_D),F)
    \]
    is a weak equivalence, i.e. $F$ is an $S_{\check{C}}$-local object. Notice that the homotopy function complexes can be taken to be either with respect to the projective or injective model structure on simplicial presheaves.
\end{lemma}
\begin{proof}
    This Lemma is a special case of \cite[Lemma 4.4]{Dugger-Hollander-Isaksen} where the hypercover is simply given by the \v{C}ech complex associated with a covering family.
\end{proof}
\begin{remark}
    As mentioned earlier this lemma can be used as a definition for \v{C}ech descent in the case the site $\cat{D}$ does not have the pullbacks $D_{i_0 \cdots i_n} = D_{i_0 \cdots i_{n-1}} \times_D D_{i_n}$ for all $n\geq 1$ given any covering family $\left\{ f_i \colon D_i \rightarrow D \; | \; i \in I \right\}$.
\end{remark}

The fact that the $S_{\check{C}}$-local objects in $\mathrm{sPSh}(\cat{D})$ are precisely the simplicial presheaves satisfying \v{C}ech descent allows for the definition of the \textit{\v{C}ech-local model structure} on simplicial presheaves.

\begin{proposition} \label{Proposition: Cech Local Model Structures}
Let $\cat{D}$ be a site and consider the injective and the projective model structures on the category of simplicial presheaves on $\cat{D}$ denoted by $\mathbf{sPSh}(\cat{D})_{\mathrm{proj}}$ and $\mathbf{sPSh}(\cat{D})_{\mathrm{inj}}$. Then the left Bousfield localization of $\mathbf{sPSh}(\cat{D})_{\mathrm{proj}}$ and $\mathbf{sPSh}(\cat{D})_{\mathrm{inj}}$ at the set $S_{\check{C}}$ of all \v{C}ech complexes associated to all covering families $\left\{ f_i \colon D_i \rightarrow D \; | \; i \in I \right\}$ exists. The resulting model structures, called the \textbf{\v{C}ech-local model structures on simplicial presheaves}, are denoted by $\mathbf{sPSh}(\cat{D})_{\mathrm{proj, \check{C}}}$ and $\mathbf{sPSh}(\cat{D})_{\mathrm{inj, \check{C}}}$. The fibrant objects in these \v{C}ech-local model structures are given by projective/injective fibrant objects which, in addition, satisfy \v{C}ech descent.
\end{proposition}
\begin{proof}
    The existence of the left Bousfield localizations follows from Theorem \ref{Theorem: Enriched Bousfield Localization}. The fact that the fibrant objects in the localized model structures agree with projectively/injectively fibrant simplicial presheaves satisfying \v{C}ech descent follows from Lemma \ref{Lemma: S-local equals Cech descent}.
\end{proof}

\begin{lemma}[\cite{Dugger}, Proposition 3.3.2] \label{Lemma: constant presheaf and sheafification is weak equivalence}
    Let $\cat{D}$ be a site and let $F$ be a presheaf on $\cat{D}$.
    \begin{itemize}
        \item The discrete simplicial presheaf $F^{\delta}$ is fibrant in the \v{C}ech-local model structure if and only if $F$ is a sheaf.
        \item If $F^{\#}$ denotes the sheafification of $F$, the map of discrete simplicial presheaves $F^{\delta} \rightarrow (F^{\#})^{\delta}$ is a \v{C}ech-local weak equivalence.
    \end{itemize}
    The above statement holds for both the \v{C}ech-local injective and the \v{C}ech-local projective model structures.
\end{lemma}

\begin{remark}
    Given a presheaf of sets $F \in \mathrm{PSh}(\cat{D})$ the \textbf{discrete simplicial presheaf} $F^{\delta}$ is the simplicial presheaf such that for every $U \in \cat{D}$ the simplicial set $F^{\delta}(U)$ is the \textit{constant} or \textit{discrete simplicial set} associated to the set $F(U)$. Note that there also exists the \textbf{constant simplicial presheaf} $\mathsf{c}(K)$ given any simplicial set $K \in \mathbf{sSet}$ which is simply characterized by $\mathsf{c}(K)(U) = K$ for all $U \in \cat{D}$.
\end{remark}

Besides the concept of descent with respect to \v{C}ech complexes, Jardine \cite{Jardine} proposed another approach to a \textit{local model structure} of simplicial presheaves, which focuses more on the homotopy theoretic nature of simplicial presheaves and reduces the definition of descent to the classical case of presheaves of sets on a site. However, it has been observed by Dugger--Hollander--Isaksen that Jardine's theory of descent is equivalent to descent with respect to all hypercovers, which is not equivalent to the notion of \v{C}ech-descent as pointed out in \cite[A.10]{Dugger-Hollander-Isaksen}.   \\

For the specific site of Cartesian spaces with good open coverings $\mathbf{Cart}$ however, they coincide. This follows from the fact that the site $(\mathbf{Cart},\tau_{\mathrm{gd}})$ satisfies the additional property of being \textit{hypercomplete}, see Definition \ref{Definition: Hypercomplete Site}. \\

We follow Jardine \cite{JardineFields} to introduce the local model structure on simplicial presheaves. For the rest of this section and if not otherwise specified, given a simplicial set $K$ together with a vertex $x_0$ its $n$th homotopy group for $n \geq$ is defined
\[
\pi_n(K,x_0) := \pi_n(|K|,x_0),
\]
via the homotopy group of its associated geometric realization. For $n = 0$ set
\[
\pi_0(K) := \mathrm{coeq} \left( K_1 \rightrightarrows K_0 \right).
\]
To avoid choices of base points, define a \textit{base point free} version for $n \geq 1$ by setting
\[
\pi_n(K) := \coprod_{x \in K_0} \pi_n(K,x),
\]
where here the coproduct is taken in the category of sets. The set $\pi_n(K)$ is equipped with a canonical function $\pi_n(K) \rightarrow K_0$ which gives $\pi_n(K)$ the structure of a group object over $K_0$, i.e. $\pi_n(K) \rightarrow K_0$ is a group object in the over-category $\mathbf{Set} \downarrow K_0$.

Let now $\cat{D}$ be a site. Given a simplicial presheaf $F \in \mathrm{sPSh}(\cat{D})$, consider the associated presheaves of sets
\begin{align*}
\pi_n(F)(U) &:= \pi_n(F(U)) = \coprod_{x_U \in F(U)_0 } \pi_n(F(U),x_U) \; \text{ for } n \geq 1\\
\pi_0(F)(U) &:= \pi_0(F(U)).
\end{align*}

\begin{definition} \label{Definition: Homotopy Sheaf}
    Let $\cat{D}$ be a site. The above construction defines a functor for all $n \geq 0$
    \[
    \begin{array}{rcl}
         \pi_n \colon \mathrm{sPSh}(\cat{D}) & \rightarrow & \mathrm{PSh}(\cat{D}) \\
         F & \mapsto  & \pi_n(F)
    \end{array}
    \]
    sending a simplicial presheaf $F$ to its associated \textbf{homotopy presheaf} $\pi_n(F)$. By further applying the sheafification procedure, one obtains a functor for all $n \geq 0$
    \[
    \begin{array}{rcl}
         \tilde{\pi}_n \colon \mathrm{sPSh}(\cat{D}) & \rightarrow & \mathrm{Sh}(\cat{D}) \\
         F & \mapsto  & \tilde{\pi}_n(F)
    \end{array}
    \]
    sending a simplicial presheaf $F$ to its associated \textbf{homotopy sheaf} $\tilde{\pi}_n(F)$.
\end{definition}

Before stating the definition of \textit{local weak equivalence} of simplicial presheaves it is important to reformulate the classical notion of weak equivalence of simplicial sets in terms of the base point free versions of the homotopy groups.

\begin{lemma}[\cite{JardineFields}]
    A map of simplicial sets $f \colon X \rightarrow Y$ is a weak equivalence if and only if
    \begin{enumerate}
        \item the function $\pi_0(X) \rightarrow \pi_0(Y)$ is a bijection, and
        \item the commutative diagram induced by $f$
        \begin{center}
        \begin{tikzcd}
\pi_n(X) \arrow[r] \arrow[d] & \pi_n(Y) \arrow[d] \\
X_0 \arrow[r, "f_0"]                & Y_0
        \end{tikzcd}
        \end{center}
        is a pullback diagram for $n \geq 1$.
    \end{enumerate}
\end{lemma}
\begin{proof}
    This follows from the simple observation that given $\pi_n(X)$ and $x_0$ a vertex in $X$, the classical based homotopy groups are recovered as a pullback.
    \begin{center}
    \begin{tikzcd}
{\pi_n(X,x_0)} \arrow[d] \arrow[r] & \pi_n(X) \arrow[d] \\
{*} \arrow[r, "x_0"]       & X_0
\end{tikzcd}
    \end{center}
    Now the usual pasting law of pullback diagrams implies that, if the right square is a pullback, then the total diagram is a pullback precisely if the left square is a pullback.
    \begin{center}
    \begin{tikzcd}
{\pi_n(X,x_0)} \arrow[d] \arrow[r] & \pi_n(X) \arrow[d] \arrow[r] & \pi_n(Y) \arrow[d] \\
{*} \arrow[r, "x_0"]       & X_0 \arrow[r]                & Y_0
\end{tikzcd}
    \end{center}
    This then identifies $\pi_n(X,x_0)$ with $\pi_n(Y,f(x_0))$ for all $n \geq 1$ and all vertices $x_0 \in X$, i.e. $f$ is a weak equivalence. The opposite direction follows directly.
\end{proof}

\begin{definition}[\cite{JardineFields}] \label{Definition: Local Weak Equivalence}
    Let $\cat{D}$ be a site and $f \colon F \rightarrow G$ a map of simplicial presheaves on $\cat{D}$. We say that $f$ is a \textbf{local weak equivalence} if
    \begin{enumerate}
        \item the map $\tilde{\pi}_0(F) \rightarrow \tilde{\pi}_0(G)$ is an isomorphism of sheaves, and
        \item the diagram of morphisms of sheaves
        \begin{center}
        \begin{tikzcd}
\tilde{\pi}_n(F) \arrow[d] \arrow[r] & \tilde{\pi}_n(G) \arrow[d] \\
\tilde{F}_0 \arrow[r]                & \tilde{G}_0
\end{tikzcd}
        \end{center}
        is a pullback for all $n \geq 1$ in $\mathrm{Sh}(\cat{D})$.
    \end{enumerate}
\end{definition}

Of course, there also exist base point dependent versions of the homotopy sheaves introduced above. Whereas the base point free construction allows clean constructions and definitions, it is often more convenient to use the dependent versions to prove certain statements.

\begin{definition}
    Let $F \in \mathrm{sPSh}(\cat{D})$ be a simplicial presheaf. Then, given an object $V \in \cat{D}$ and a vertex $x \in F_0(V)$, we define the homotopy presheaves $\pi_n(F,x)$ on the site $\cat{D} \downarrow V$ by the assignment:
    \[
    \pi_n(F,x) : U \longmapsto \pi_n(F(U),x|_U)
    \]
    which is a presheaf of groups for $n \geq 1$. The associated sheaves are denoted $\Tilde{\pi}_n(F,x)$.
\end{definition}

\begin{remark}
    Similar in the case of the usual homotopy groups, given the presheaves $\pi_n(F)$ one can recover the base point dependent version $\pi_n(F,x)$ via a pullback. Also, notice that the over-category $\cat{D} \downarrow V$ naturally inherits a coverage from $\cat{D}$ turning it into a site. Indeed, a collection of maps $\left\{ U_{i} \xrightarrow{f_i} U \rightarrow V \, | \, i \in I \right\}$ is a covering family of $U \rightarrow V$ if and only if the family $\left\{ U_i \rightarrow U \, | \, i \in I \right\}$ is a covering family of $U$ in $\cat{D}$.
\end{remark}

The following Lemma already gives a class of examples of local weak equivalences:

\begin{lemma}[\cite{JardineBook}]
    Let $f : F \rightarrow G$ be an objectwise weak equivalence of simplicial presheaves. Then $f$ is a local weak equivalence.
\end{lemma}
\begin{proof}
Since $f$ is a objectwise weak equivalence for any $V \in \cat{D}$ and any $x \in F(V)$, the induced map of homotopy presheaves
    \[
    \pi_n(F,x) \rightarrow \pi_n(G,f(x))
    \]
    are isomorphisms. Indeed, for any $U \in \cat{D} \downarrow V$ the induced map
    \[
    \pi_n(F(U),x|_U) \xrightarrow{\cong} \pi_n(G(U), f(x|_U) )
    \]
    is an isomorphism, as $F(U) \rightarrow G(U)$ is a weak equivalence of simplicial sets. Now we apply the sheafification functor, which is exact hence the induced maps on the associated sheaves are also isomorphisms.
\end{proof}

\begin{theorem}[\cite{JardineBook}, Theorem 2.3] \label{Theorem: Jardine local structure}
    Let $\cat{D}$ be a site. Then the category of simplicial presheaves $\mathrm{sPSh}(\cat{D})$ together with the classes of
    \begin{enumerate}
        \item cofibrations given by objectwise monomorphisms,
        \item weak equivalences given by local weak equivalences, and
        \item fibrations having the right lifting property with respect to all trivial cofibrations,
    \end{enumerate}
    satisfies the axioms for a closed model category. We call this model structure the \textbf{local injective model structure on simplicial presheaves} and denote it by
    \[
    \mathrm{sPSh}(\cat{D})_{\mathrm{inj,loc}}.
    \]
\end{theorem}

\begin{remark}
    By construction, it now follows that the local injective model structure is given by the left Bousfield localization of the injective model structure $\mathrm{sPSh}(\cat{D})_{\mathrm{inj}}$ at the \textbf{class} of local weak equivalences. It is important to note that the existence of a left Bousfield localization at a class of morphisms is not guaranteed. Therefore, it is only possible to identify the local injective model structure as a left Bousfield localization after having proven its existence via Theorem \ref{Theorem: Jardine local structure}. Hence we have:
    \[
    L_{W_{\mathrm{loc}}} \left( \mathrm{sPSh}(\cat{D})_{\mathrm{inj}} \right) = \mathrm{sPSh}(\cat{D})_{\mathrm{inj,loc}}.
    \]
\end{remark}

Recall from Example \ref{Example: Equivalence between proj and inj on diagrams} that the identity functors provide a Quillen equivalence between the projective and the injective model structure on simplicial presheaves on a site $\cat{D}$.
\begin{center}
\begin{tikzcd}
\mathrm{sPSh}(\cat{D})_{\mathrm{proj}} \arrow[r, "\perp", phantom] \arrow[r, "\mathrm{id}", shift left=3] & \mathrm{sPSh}(\cat{D})_{\mathrm{inj}} \arrow[l, "\mathrm{id}", shift left=3]
\end{tikzcd}
\end{center}

The existence of the local injective model structure as a localization at the class of local weak equivalences suggests that also its projective counterpart, \textit{the local projective model structure} should exist and act as the left Bousfield localization of the projective structure. This is indeed the case as shown by Blander \cite[Theorem 1.6]{Blander}. The \textbf{local projective model structure} is denoted by $\mathrm{sPSh}(\cat{D})_{\mathrm{proj,loc}}$ with weak equivalences given by local weak equivalences and cofibrations given by projective cofibrations. In particular, the Quillen equivalence between the projective and injective model structures descends to an equivalence of the respective local model structures.\footnote{See \cite{Dugger-Hollander-Isaksen}}
\begin{equation}\label{Equation: Local injective and local projective Quillen Equivalence}
\begin{tikzcd}
\mathrm{sPSh}(\cat{D})_{\mathrm{proj,loc}} \arrow[r, "\perp", phantom] \arrow[r, "\mathrm{id}", shift left=3] & \mathrm{sPSh}(\cat{D})_{\mathrm{inj,loc}} \arrow[l, "\mathrm{id}", shift left=3]
\end{tikzcd}
\end{equation}

The big question that now arises is how the \v{C}ech-local model structure and Jardine's local model structure on simplicial presheaves $\mathrm{sPSh}(\cat{D})$ are related to each other. It was shown by Dugger--Hollander--Isaksen \cite{Dugger-Hollander-Isaksen} that both the local injective and local projective model structure coincide with the left Bousfield localization of the injective/projective model structure at the collection of all \textit{hypercovers}\footnote{See \cite[Definition 4.2]{Dugger-Hollander-Isaksen} for a precise definition of hypercover.} in $\mathrm{sPSh}(\cat{D})$.

\begin{theorem}[\cite{Dugger-Hollander-Isaksen}] \label{Theorem: The DHI-fibrancy Theorem} $ $
\begin{itemize}
    \item The fibrant objects in the local injective model structure $\mathrm{sPSh}(\cat{D})_{\mathrm{inj,loc}}$ are the simplicial presheaves which
    \begin{enumerate}[label={(\roman*)}]
        \item are fibrant in the injective model structure $\mathrm{sPSh}(\cat{D})_{\mathrm{inj}}$, and
        \item satisfy the descent condition for all hypercovers.
    \end{enumerate}

    \item The fibrant objects in the local projective model structure $\mathrm{sPSh}(\cat{D})_{\mathrm{proj,loc}}$ are the simplicial presheaves which
    \begin{enumerate}[label={(\roman*)}]
        \item are fibrant in the projective model structure $\mathrm{sPSh}(\cat{D})_{\mathrm{proj}}$, i.e. are objectwise fibrant, and
        \item satisfy the descent condition for all hypercovers.
    \end{enumerate}
\end{itemize}
\end{theorem}
\begin{proof}
    The statement for the local injective model structure is precisely \cite[Theorem 1.1]{Dugger-Hollander-Isaksen}. The case for the local projective model structure then follows from \cite[Theorem 1.3]{Dugger-Hollander-Isaksen} together with \cite[Theorem 6.2]{Dugger-Hollander-Isaksen}.
\end{proof}

Here one should think of a \textit{hypercover} as a generalized version of a \v{C}ech-complex. For an arbitrary site $\cat{D}$, descent with respect to \textit{hypercovers} is a stronger condition than \v{C}ech descent. However, there are sites where these two descent conditions agree and the according model structures coincide. This motivates the following definition.

\begin{definition} \label{Definition: Hypercomplete Site}
Let $\cat{D}$ be a site. We say that $\cat{D}$ is a \textbf{hypercomplete site} if, in the associated category of simplicial presheaves $\mathrm{sPSh}(\cat{D})$, the notion of hypercover descent coincides with the notion of \v{C}ech descent. More precisely, if in both, the injective and the projective case, the local and \v{C}ech-local model structures coincide:
\begin{align*}
    \mathrm{sPSh}(\cat{D})_{\mathrm{inj,loc}} &= \mathrm{sPSh}(\cat{D})_{\mathrm{inj,\check{C}}} \\
    \mathrm{sPSh}(\cat{D})_{\mathrm{proj,loc}} &= \mathrm{sPSh}(\cat{D})_{\mathrm{proj,\check{C}}}.
\end{align*}
\end{definition}

\begin{remark}
The notion of \textit{hypercompleteness} is usually phrased in the language of $\infty$-topoi such as presented in \cite[Section~6.5]{Lurie}. However, to justify the definition of a hypercomplete site as stated above, notice that by \cite[Proposition~6.5.2.14]{Lurie} it follows that the local injective model structure on the category of simplicial presheaves $\mathrm{sPSh}(\cat{D})$ on some site $\cat{D}$ presents indeed a hypercomplete $\infty$-topos. Moreover, the local injective model structure is a further left Bousfield localization of the \v{C}ech-local model structure
\[
\mathrm{PSh}(\cat{D})_{\mathrm{inj,\check{C}}} \rightarrow \mathrm{PSh}(\cat{D})_{\mathrm{inj,loc}}
\]
and this localization identifies $\mathrm{PSh}(\cat{D})_{\mathrm{inj,loc}}$ as the \textit{hypercompletion} of $\mathrm{PSh}(\cat{D})_{\mathrm{inj,\check{C}}}$. Since in a \textit{hypercomplete site} the notion of hypercover descent coincides with the notion of \v{C}ech descent, it then follows that the \v{C}ech-local injective and local injective model structures share the same cofibrations and fibrant objects. Since a model category is determined by the class of cofibrations and fibrant objects the two model structures indeed coincide.
\end{remark}

\begin{theorem}[\cite{Amabel-Debray-Haine}, \cite{Schreiber}] \label{Theorem: Hypercompleteness of Cart}
    The Cartesian site $\mathbf{Cart}$ is a hypercomplete site.
\end{theorem}
\begin{proof}
    In \cite{Amabel-Debray-Haine} it is shown in Proposition A.5.3 that for the site of manifolds $\mathbf{Mfd}$ with covering families given by open covers, the associated category of simplicial presheaves is a hypercomplete $\infty$-topos. For the Cartesian site $\mathbf{Cart}$ we refer to \cite[Proposition 4.4.6]{Schreiber}.
\end{proof}

It is due to this theorem that whenever we are considering simplicial presheaves on the Cartesian site $\mathbf{Cart}$ the terms local and \v{C}ech-local coincide and we can use both terms interchangeably.

\subsection{About Homotopy (Co-)Limits}

Homotopy limits and colimits in their full generality do not play an essential role in this thesis. However, they are indispensable in the treatment of descent of simplicial presheaves. A very thorough introduction to homotopy limits and colimits can be found in the book by Riehl \cite{Riehl}. Nevertheless, for the sake of self-containment, this section provides a brief discussion of homotopy limits and colimits in the case of simplicial/cosimplicial spaces, as the diagrams arising in the definition of \v{C}ech descent are precisely of this form. \\

Recall the case of classical limits and colimits, i.e. given diagrams indexed by some small category $\cat{C}$ and taking values in simplicial sets, there are functors
\[
\begin{array}{rcl}
  [\cat{C},\mathbf{sSet}] & \rightarrow  & \mathbf{sSet} \\
     F & \mapsto & \mathrm{lim}_{\cat{C}} \, F \\
     F & \mapsto & \mathrm{colim}_{\cat{C}} \, F,
\end{array}
\]
which assign to each diagram $F$ either its limit or colimit. A homotopy limit/colimit should be understood as a universal way of approximating the classical limit/colimit functor by a functor that \textit{preserves weak equivalences}. Here the weak equivalences in the diagram category are objectwise weak equivalences. Therefore, a homotopy limit is a functor
\[
\mathrm{holim}_{\cat{C}} \colon [\cat{C},\mathbf{sSet}] \rightarrow \mathbf{sSet}
\]
that preserves weak equivalences and at the same time is a universal replacement of the classical limit functor with this property. In Remark \ref{Remark: Derived Functors} we have introduced left and right-derived functors, which incorporate exactly this feature. Therefore, define
\begin{align*}
&\mathrm{holim}_{\cat{C}} := \mathbb{R}  \mathrm{lim}_{\cat{C}}, \text{ and} \\
&\mathrm{hocolim}_{\cat{C}} := \mathbb{L}  \mathrm{colim}_{\cat{C}}.
\end{align*}

In light of Definition \ref{Definition: Left/Right derived functors}, it follows that homotopy limits/colimits being right/left derived functors can be defined in special cases by endowing the functor category $[\cat{C},\mathbf{sSet}]$ with a suitable model structure having as weak equivalences the objectwise weak equivalences and such that the classical limit/colimit functors are right/left Quillen functors.\footnote{See for example \cite[Theorem~14.3.1]{Riehl} in the case where the indexing category $\cat{D}$ is a certain Reedy category.} \\

In the special case where the indexing category is given by the simplex category $\Delta$ or $\Delta^{\mathrm{op}}$, as occurring in Definition \ref{Definition: Cech descent}, the homotopy limits/colimits can be computed as follows.
\begin{proposition} $ $
\begin{enumerate}[label={(\arabic*)}]
    \item  Let $F : \Delta \rightarrow \mathbf{sSet}$ be a cosimplicial space. Denote by $R(F)$ the fibrant replacement of $F$ in the Reedy model structure on $[\Delta,\mathbf{sSet}]$. Then the homotopy limit of $F$ is modeled by the totalization of the fibrant replacement
    \[
    \mathrm{holim} \, F \simeq \mathrm{Tot}(R(F)).
    \]
    \item Let $F : \Delta^{\mathrm{op}} \rightarrow \mathbf{sSet}$ be a simplicial space. Then the homotopy colimit of $F$ is modeled by the diagonal simplicial set
    \[
    \mathrm{hocolim} \, F \simeq d(F).
    \]
\end{enumerate}
\end{proposition}
\begin{proof}
    For $(2)$ the statement is \cite[Corollary~18.7.7]{Hirschhorn}. For $(1)$ notice first that, for $F \in [\Delta,\mathbf{sSet}]$ a Reedy fibrant cosimplicial space, the Bousfield-Kan map
    \[
    \mathrm{Tot}(F) \xrightarrow{\sim} \mathrm{holim} \, F
    \]
    is a natural weak equivalence by \cite[Theorem~18.7.4]{Hirschhorn}. As homotopy limits preserve weak equivalences, it then follows:
    \[
    \mathrm{holim} \, F \simeq \mathrm{holim} \, R(F) \simeq \mathrm{Tot}(R(F)).
    \]
\end{proof}

To see how the totalization can be used to check the \v{C}ech descent condition, notice that in the case we are given an objectwise fibrant simplicial presheaf $F \in \mathrm{sPSh}(\cat{D})$, then the associated cosimplicial space given a covering family $\left\{ D_i \rightarrow D \, | \, i \in I \right\}$ denoted $F\left( \check{C}(D_i \rightarrow D) \right)$
\begin{center}
\begin{tikzcd}
   \prod_{i \in I}F(D_i) \arrow[r, shift left=2] \arrow[r, shift right=2] & {\prod_{i_0,i_i \in I}F(D_{i_0i_1})} \arrow[l] \arrow[r] \arrow[r, shift left=4] \arrow[r, shift right=4] & {\prod_{i_0,i_i,i_2 \in I}F(D_{i_0i_1i_2})} \arrow[l, shift left=2] \arrow[l, shift right=2] \cdots
\end{tikzcd}
\end{center}
is fibrant in the Reedy model structure on cosimplicial spaces $\mathbf{sSet}^{\Delta}$ by \cite[Lemma C.5]{Glass-etAl}. Hence by the previous proposition it follows that
\[
    \mathrm{holim}_{[n] \in \Delta}  F\left( \check{C}(D_i \rightarrow D) \right)_n \simeq \mathrm{Tot} \left( F\left( \check{C}(D_i \rightarrow D) \right) \right).
\]
In Definition \ref{Definition: Cech descent} however, the homotopy limit is taken over the diagram
\begin{center}
    \begin{tikzcd}
        \underset{i \in I}{\prod}F(D_i) \arrow[r, shift left=2] \arrow[r, shift right=2] & {\underset{(i_0,i_1) \in I \times I}{\prod} F(D_{i_0i_1})} \arrow[r, shift left=3] \arrow[r] \arrow[r, shift right=3] & {\underset{(i_0i_1i_2) \in I \times I \times I}{\prod} F(D_{i_0i_1i_2})} \arrow[r, shift left=3] \arrow[r, shift left] \arrow[r, shift right] \arrow[r, shift right=3] & \cdots
       \end{tikzcd}
\end{center}
which is a semi-cosimplicial space. That is an element of $\mathbf{sSet}^{\Delta_+}$ where $\Delta_+$ is the subcategory of the simplex category containing only injective functions. It is now a consequence of the inclusion functor $j \colon \Delta_+ \hookrightarrow \Delta$ being left cofinal \cite[Section 3.17]{Dror-Dwyer} together with the fact that $F$ is objectwise fibrant, that we can apply the cofinality theorem \cite[Theorem 3.12]{Dror-Dwyer} which implies that the homotopy limit of both, the cosimplicial and the semi-cosimplicial space agree up to weak equivalence. This shows that to check \v{C}ech descent for $F$ it is equivalent to checking that the canonical map
\[
F(D) \rightarrow \mathrm{Tot} \left( F\left( \check{C}(D_i \rightarrow D) \right) \right)
\]
is a weak equivalence of simplicial sets.

\section{Sheaf Cohomology}

In this section, we introduce the cohomology for arbitrary simplicial presheaves $X$ on a site $\cat{D}$ with coefficients in an abelian presheaf $A$ following Jardine's book \cite{JardineBook} as the main reference. Therefore, in what follows, we will always consider the category of simplicial presheaves $\mathrm{sPSh}(\cat{D})$ endowed with the \textbf{local injective model structure}  of Theorem \ref{Theorem: Jardine local structure}.

\subsection{The Cohomology of a Simplicial Presheaf: Theory}

\begin{definition} \label{Definition: Cohomology of simplicial presheaf}
Let $X$ be a simplicial presheaf and $A$ an abelian presheaf on a site $\cat{D}$. Then the $n$\textbf{-th cohomology of} $X$ \textbf{with coefficients in} $A$ is given by
\[
H_{\infty}^n(X,A) := \pi_0 \mathbb{R}\mathrm{Map}(X,K(A,n)),
\]
where $K(A,n) := \Gamma(A[-n])$ is the Eilenberg--MacLane object in degree $n$ associated with the abelian presheaf $A$.
\end{definition}

Here the derived hom space is considered to be taken in the local injective model structure. Further, to clarify what we mean by the Eilenberg--MacLane object $K(A,n)$, recall the Dold--Kan correspondence being a link between homological algebra and simplicial homotopy theory.

\begin{theorem}[\cite{JardineEt}]  \label{Theorem: Dold-Kan correspondence}
    Let $\mathbf{Ab}$ denote the category of abelian groups and $\mathbf{sAb}$ the category of simplicial abelian groups. Then the normalized chain complex functor $N$ induces an adjoint equivalence of categories
    \begin{center}
\begin{tikzcd}
\mathbf{sAb} \arrow[r, "N"', shift right=3] \arrow[r, "\perp", phantom] & \mathrm{Ch}_{\geq 0}(\mathbf{Ab}). \arrow[l, "\Gamma"', shift right=3]
\end{tikzcd}
    \end{center}
\end{theorem}
\begin{remark}
    Notice that since the above is an equivalence of categories, we also have that $N \dashv \Gamma$. Moreover, this equivalence of categories induces an equivalence of presheaf categories
    \begin{equation} \label{Equation: Presheaf Dold-Kan version}
\begin{tikzcd} \
{N: \mathrm{PSh}(\cat{D},\mathbf{sAb})} \arrow[r, shift right=3] \arrow[r, "\perp", phantom] & {\mathrm{Ch}_{\geq 0}(\mathrm{PSh}(\cat{D}, \mathbf{Ab})) : \Gamma}, \arrow[l, shift right=3]
\end{tikzcd}
\end{equation}
where on the left-hand side we have the category of presheaves of simplicial abelian groups on the site $\cat{D}$ and on the right-hand side we have the category of bounded below chain complexes of presheaves of abelian groups on $\cat{D}$.
\end{remark}

Suppose we are now given a presheaf of abelian groups $A$ on a site $\cat{D}$. Then denote by $A[-n]$ the chain complex of abelian presheaves concentrated in degree $n$ which is given by
\[
(A[-n])_k = \begin{cases}
    A \text{ if } k = n \\
    0 \text{ else }
\end{cases}
\]
and differential given by $d = 0$. Then under the Dold--Kan correspondence the associated simplicial presheaf $\Gamma(A[-n]) =: K(A,n)$ is called the \textbf{Eilenberg--MacLane object of degree} $n$ associated with $A$. Notice that the Eilenberg--MacLane object can also be computed using the simplicial classifying space construction $\overline{W}$.

\begin{lemma} \label{Lemma: simplicial classifying space is Eilenberg-Mac Lane object}
    Given a presheaf of abelian groups $A$ on $\cat{D}$, then for every $n \geq 0$ there are isomorphisms of simplicial presheaves
    \[
    \overline{W}^nA \cong K(A,n) := \Gamma(A[-n])
    \]
\end{lemma}
\begin{proof}
    This can be shown by induction. The base case follows from the discussion in section 4.6 of \cite{JardineEt}. For a full proof using induction see \cite[Lemma 4.36]{Minichiello}.
\end{proof}

As a consequence, the cohomology of a simplicial presheaf $X$ with coefficients in an abelian presheaf $A$ can be computed by
\[
H^n_{\infty}(X,A) = \pi_0 \mathbb{R}\mathrm{Map}(X,\overline{W}^nA),
\]
where the derived mapping space is again taken with respect to the local injective model structure. To compute the cohomology of a simplicial presheaf $X$ with coefficients in a sheaf of abelian groups $A$ on a site $\cat{D}$, recall the free-forgetful adjunction between simplicial sets and simplicial abelian groups.

\begin{definition} \label{Definition: Presheaf free-forgetful adjunction}
    The free forgetful adjunction
    \begin{center}
\begin{tikzcd}
\mathbb{Z}(-): \mathbf{sSet} \arrow[r, shift left=2] \arrow[r, "\perp", phantom] & \mathbf{sAb} : U \arrow[l, shift left=2]
\end{tikzcd}
    \end{center}
    assigning to a simplicial set its free simplicial abelian group induces an adjunction between presheaf categories
    \begin{center}
        \begin{tikzcd}
\mathbb{Z}(-):  \mathrm{sPSh}(\cat{D}) \arrow[r, shift left=2] \arrow[r, "\perp", phantom] & {\mathrm{PSh}(\cat{D},\mathbf{sAb}) : U.} \arrow[l, shift left=2]
\end{tikzcd}
    \end{center}
\end{definition}

\begin{definition}[\cite{JardineBook}]
    Let $X \in \mathrm{PSh}(\cat{D})$ be a simplicial presheaf on $\cat{D}$. Then we define the \textbf{homology presheaf associated to} $X$ to be the presheaf of abelian groups given by
    \[
    H_n(X) := H_n \left( N(\mathbb{Z}X) \right) \in \mathrm{PSh}(\cat{D}, \mathbf{Ab}),
    \]
    where $N(\mathbb{Z}X) \in \mathrm{Ch}_{\geq 0}(\mathrm{PSh}(\cat{D}, \mathbf{Ab}))$ is the chain complex of presheaves obtained by first applying the induced free simplicial abelian group functor of Definition \ref{Definition: Presheaf free-forgetful adjunction} followed by the normalized chain complex functor of \eqref{Equation: Presheaf Dold-Kan version}.
    The \textbf{homology sheaf associated to} $X$ is then the sheafification of the homology presheaf and denoted by
    \[
    \tilde{H}_n(X) \in \mathrm{Sh}(\cat{D}, \mathbf{Ab}).
    \]
\end{definition}
\begin{remark}
    Recall that the simplicial homology for a simplicial set $K$ is defined as the homology of the abelian chain complex
    \[
    H_*(K) = H_*(N(\mathbb{Z}(K))).
    \]
    It follows therefore that the homology presheaf of $X$ is then given by the assignment
    \[
    U \in \cat{D} \longmapsto H_n(X)(U) = H_n(X(U)),
    \]
   i.e. as the simplicial homology of the simplicial set $X(U)$.
\end{remark}

The homology sheaves associated with $X$ can be used to compute the cohomology $H_{\infty}^*(X,A)$ via a generalized universal coefficients spectral sequence.

\begin{proposition}[\cite{JardineBook}, Corollary 8.28] \label{Proposition: Universal coefficient in stack cohomology}
    Let $X$ be a simplicial presheaf and $A$ a sheaf of abelian groups, regarded as a discrete presheaf of simplicial abelian groups. Then there is a spectral sequence, with
    \[
    E^{p,q}_2 = \mathrm{Ext}^q \left( \tilde{H}_p(X) , A \right) \Rightarrow H^{p+q}_{\infty}(X,A),
    \]
    where $\tilde{H}_*(X)$ denotes the homology sheaf associated to $X$.
\end{proposition}

To use Proposition \ref{Proposition: Universal coefficient in stack cohomology} it is essential to understand how the universal coefficients spectral sequence for sheaf cohomology arises. First, recall that the category of sheaves of abelian groups on a site $\cat{C}$ has enough injectives, and as such any such sheaf $A$ admits an injective resolution. That is, there exists a cochain complex $J^{\bullet}$ together with a map $A \rightarrow J^0$ of sheaves of abelian groups such that each $J^i$ is injective and the augmented complex
\[
0 \rightarrow A \rightarrow J^0 \rightarrow J^1 \rightarrow \cdots
\]
is exact.

Given a bounded below chain complex $(C_{\bullet}, \partial)$ and bounded below cochain complex $(K^{\bullet}, d)$ in the abelian category $\mathrm{PSh}(\cat{C}, \mathrm{Ab})$ we can define a first quadrant double complex by \newline $\mathrm{hom}(C,K)^{p,q} := \mathrm{Hom}(C_p,K^q)$ with the corresponding induced differentials $\partial' := (\partial)^*$ and $\partial'':= (d)_*$. To this double complex we associate the total complex 
\[
\mathrm{Tot}^n := \bigoplus_{p + q = n} \mathrm{hom}(C,K)^{p,q} 
\]
with differential $D := \partial' + (-1)^p \partial''$ whose cohomology we denote by $H^n(\mathrm{Tot})$. The universal coefficients spectral sequence is then simply the spectral sequence of this double complex $\mathrm{hom}(C,K)^{p,q}$ for specific choices of $C_{\bullet}$ and $K^{\bullet}$. \\

Before specifying the complexes $C_{\bullet}$ and $K^{\bullet}$ observe the following. The first quadrant double complex $\{ \mathrm{hom}(C,K)^{p,q}\}_{p,q \geq 0}$ can be understood equivalently as a third quadrant double complex $\mathrm{hom}'(C,K)_{p,q}$ for $p,q \leq 0$ by setting $\mathrm{hom}'(C,K)_{p,q} = \mathrm{Hom}(C_{-p},K^{-q})$. The differentials here have degrees
\begin{align*}
& \partial' \colon \mathrm{hom}'(C,K)_{p,q} \rightarrow \mathrm{hom}'(C,K)_{p-1,q} \\
& \partial'' \colon \mathrm{hom}'(C,K)_{p,q} \rightarrow \mathrm{hom}'(C,K)_{p,q-1}
\end{align*}
and the associated total complex $\mathrm{Tot}'_n$ is concentrated in negative degrees. In particular, one has that for $n \geq 0$
\[
H_{-n}(\mathrm{Tot}'_{\bullet}) = H^n(\mathrm{Tot}^{\bullet}).
\]
The reason behind this identification is the observation that bounded below chain complexes $\mathrm{Ch}_{\geq 0 }(\cat{A})$ embed into the category of unbounded cochain complexes $\mathrm{CoCh}(\cat{A})$ via the assignment sending a chain complex $(C_{\bullet},\partial)$ to the cochain complex $(K^{\bullet},d)$ specified by
\[
K^n = \begin{cases}
    0 \text{ for } n > 0 \\
    C_{-n} \text{ for } n \leq 0
\end{cases}
\]
This allows us to make sense of the set of chain homotopy classes of maps $\pi\left( C_{\bullet} , K^{\bullet} \right)$ where $C_{\bullet}$ is a bounded below chain complex and $K^{\bullet}$ a cochain complex. It follows from pure homological algebra that for all $n \geq 0$ there is a natural isomorphism\footnote{This is \cite[Lemma 8.21]{JardineBook}.}
\begin{equation} \label{Equation: Lemma 8.21}
H_{-n}(\mathrm{Tot}'_{\bullet}) \cong \pi \left( C_{\bullet}, K[-n] \right).
\end{equation}

\begin{proof}[Proof of Proposition \ref{Proposition: Universal coefficient in stack cohomology}]
As already hinted above, the spectral sequence arises as a spectral sequence of a double complex and thus we only need to make the right choices for $(C_{\bullet},\partial)$ and $(K^{\bullet},d)$. The obvious choice for the cochain complex is simply a choice of injective resolution $J^{\bullet}$  of $A$. Then \cite[ Theorem 8.25]{JardineBook} implies that there is a natural isomorphism in $X$
\[
\pi \left( N\mathbb{Z}X, \tau(J[-n])\right) \cong H_{\infty}^n(X,A).
\]
Since $N\mathbb{Z}X$ is indeed a bounded below chain complex in $\mathrm{PSh}(\cat{C},\mathbf{Ab})$ the left-hand side denotes the abelian group of chain homotopy classes of maps, regarding $N\mathbb{Z}X$ as a cochain complex concentrated in negative degrees. From observation (\ref{Equation: Lemma 8.21}) it follows that
\begin{align*}
\pi \left( N\mathbb{Z}X, \tau(J[-n])\right) &\cong \pi \left( \tau^*(N\mathbb{Z}X), J[-n]\right) \\
&\cong \pi \left(N\mathbb{Z}X , J[-n]\right) \\
&\cong   H_{-n} \left(  \mathrm{Tot'}_{\bullet} \mathrm{hom}(N\mathbb{Z}X,J) \right) \\
&= H^{n} \left(  \mathrm{Tot}_{\bullet} \mathrm{hom}(N\mathbb{Z}X,J) \right)
\end{align*}
which presents us the right choice $(C_{\bullet},\partial) =N\mathbb{Z}X$. \\

What we have shown so far is that the spectral sequence of the double complex $\mathrm{hom}(N\mathbb{Z}X,J)$, being a first quadrant double complex, converges to the cohomology of the total complex, which we argued above, indeed computes the cohomology $H_{\infty}^n(X,A)$. Therefore, we are left to show that the $E_2^{p,q}$ page is indeed given by
\[
E^{p,q}_2 = \mathrm{Ext}^q \left( \widetilde{H}_p(X) , A \right) \Rightarrow H^{p+q}_{\infty}(X,A).
\]
The spectral sequence of the double complex $\mathrm{hom}(N\mathbb{Z}X,J)$ is constructed as follows. Start by taking the cohomology with respect to $\partial'$, i.e. the horizontal differential. This then gives the first page $E_1^{p,q} := H^p(\mathrm{hom}(N\mathbb{Z}X,J)^{\bullet,q})$. That is,
\begin{align*}
E_1^{p,q} = H^p\left(\mathrm{hom} ( N\mathbb{Z}X_{\bullet}, J^q) \right) &\cong \mathrm{Hom}_{\mathrm{PSh}(\cat{C},\mathbf{Ab})} \left( H_p(X), J^q \right) \\
&\cong \mathrm{Hom}_{\mathrm{Sh}(\cat{C},\mathbf{Ab})} \left( \widetilde{H}_p(X), J^q \right)
\end{align*}
using the fact that $J^q$ is an injective sheaf for all $q \geq 0$. The second page is then obtained by taking the cohomology with respect to the induced vertical differential $d_1$ on the first page. Using that $J^{\bullet}$ is an injective resolution of $A$ then readily implies that
\[
E_2^{p,q} = H^q \left( \mathrm{Hom} \left( \widetilde{H}_p(X), J^{\bullet} \right) \right) = \mathrm{Ext}_{\mathrm{Sh}(\cat{C},\mathbf{Ab})}^q \left( \widetilde{H}_p(X) , A \right)
\]
and the differential $d_2$ on the second page has degree $(-1,2)$.
\end{proof}

\begin{remark} \label{Remark: The aboutment Remark}
One last detail to clarify is what is meant in Proposition \ref{Proposition: Universal coefficient in stack cohomology} when we say that there is a spectral sequence  $E_2^{p,q}$ \textit{converging to} $H^{p+q}_{\infty}(X,A)$. As we have seen, the universal coefficients spectral sequence arises as the spectral sequence of a double complex and as such, its limiting page converges to the associated graded of the total cohomology. More specifically, for all $n \geq 0$, there is a filtration
\[
H^n_{\infty}(X,A) = J_{n,0} \supset J_{n-1,1} \supset \cdots \supset J_{0,n} \supset J_{-1,n+1} = 0
\]
such that $J_{p,q}/J_{p-1,q+1} \cong E_{\infty}^{p,q}$.
\end{remark}

\subsection{The Cohomology of a Simplicial Presheaf: Practice}

Now that the notion of cohomology for a simplicial presheaf with coefficients in an abelian presheaf $A$ has been introduced, let us focus for a moment on how to compute these objects, at least for some special cases. By definition, we have that
\[
H^n_{\infty}(X,A) := \pi_0 \mathbb{R}\mathrm{Map}(X,K(A,n)) = \pi_0 \mathrm{Map}(X,R\left( K(A,n) \right))
\]
where here $R(-)$ denotes the fibrant replacement in the local injective model structure and using the fact that in this model structure, all the objects are cofibrant. It is therefore clear that using the local injective structure, the obstacle to computing the cohomology is completely shifted to finding a concrete fibrant replacement $R(K(A,n))$ for $K(A,n)$. Proposition \ref{Proposition: Universal coefficient in stack cohomology} addresses this problem using the fact that the simplicial presheaves $K(J,n)$ for injective resolutions $A \rightarrow J$ of $A$ serve as good candidates in approximating $R\left( K(A,n) \right)$. On the other hand, Proposition \ref{Proposition: Universal coefficient in stack cohomology} only translates the problem of computing the cohomology into computing $\mathrm{Ext}(-,-)$ groups in the abelian category of abelian sheaves on $\cat{D}$. This issue will be addressed in detail in the next section. Equivalently, we could work with the local projective structure where the fibrant objects are much easier to characterize as we have seen by Theorem \ref{Theorem: The DHI-fibrancy Theorem}. It is precisely for this reason, that we will occasionally have to make use of the local projective structure. On the other hand, the local projective structure has the downside in that one needs to find a suitable cofibrant replacement for $X$. The good news is that such cofibrant replacements have been constructed for arbitrary simplicial presheaves, such as in \cite[Proposition 2.8]{DuggerUniversal}.    \\

In some special cases, however, one can find even simpler cofibrant replacements. Consider the local projective model structure on the category of simplicial presheaves on the site $\mathbf{Cart}$ with good open coverings. The following fact shows that constant presheaves are cofibrant objects in the projective model structure.

\begin{lemma}[\cite{Bunk}, Section 2.2] \label{Lemma: Constant-Evaluation Quillen pair}
    The constant simplicial presheaf functor $\mathsf{c}(-)$ and the evaluation at the point $\mathrm{ev}_{\mathrm{pt}}$ form a Quillen adjoint pair
    \begin{center}
        \begin{tikzcd}
\mathsf{c} : \mathbf{sSet} \arrow[r, shift left=3] \arrow[r, "\perp", phantom] & \mathrm{sPSh}(\mathbf{Cart})_{\mathrm{inj/proj}} : \mathrm{ev}_{\mathrm{pt}} \arrow[l, shift left=3]
\end{tikzcd}
    \end{center}
    between the Kan--Quillen model structure on simplicial sets and both the injective and projective model structure on simplicial presheaves on Cartesian spaces. In particular, the constant presheaves $\mathsf{c}(K)$ are cofibrant objects in both the injective and the projective model structure.
\end{lemma}
\begin{remark} \label{Remark: Const-EV quillen pair local}
    Notice that by composing the above Quillen adjunction with the localization adjunction $\mathrm{sPSh}(\mathbf{Cart})_{\mathrm{proj}} \rightleftarrows \mathrm{sPSh}(\mathbf{Cart})_{\mathrm{proj,loc}}$ one obtains a Quillen adjunction
\begin{center}
\begin{tikzcd}
\mathsf{c} : \mathbf{sSet} \arrow[r, shift left=3] \arrow[r, "\perp", phantom] & \mathrm{sPSh}(\mathbf{Cart})_{\mathrm{proj,loc}} : \mathrm{ev}_{\mathrm{pt}} . \arrow[l, shift left=3]
\end{tikzcd}
\end{center}
\end{remark}

Recall that given diffeological spaces $X$ and $Y$ we denote by $D(X,Y)$ the diffeological space of smooth functions from $X$ to $Y$, see Example \ref{Example: Diffeological Spaces}. Further, any diffeological space $X$ is assigned a simplicial set which we denoted $S_e(X)$ called the extended smooth singular complex. Its $n$-simplices are given by smooth maps from the $n$-dimensional affine smooth simplex $\mathbb{A}^n$ into $X$, see Definition \ref{Definition: Extended Smooth Singular Complex}.

\begin{definition}
    Given a diffeological space $X$ denote by $\overline{X}$ the enriched simplicial presheaf on $\mathbf{Cart}$ given by
    \[
    \begin{array}{rcl}
        \mathbf{Cart}^{\mathrm{op}} & \rightarrow & \mathbf{sSet} \\
         U & \mapsto & S_e(D(U,X))
    \end{array}
    \]
\end{definition}

\begin{lemma} \label{Lemma: constant simplicial presheaf is cofibrant resolution}
    Let $X$ be a diffeological space. Then the natural morphism from the constant simplicial presheaf $\mathsf{c}(S_e(X)) \rightarrow \overline{X}$ to the enriched simplicial presheaf exhibits $\mathsf{c}(S_e(X))$ as a cofibrant replacement of $\overline{X}$ in the local projective model structure.
\end{lemma}
\begin{proof}
It follows from Corollary \ref{Corollary: Constant Homology sheaves} that the morphism of simplicial presheaves $\mathsf{c}(S_e(X)) \rightarrow \overline{X}$ is an objectwise weak equivalence and hence also a local projective weak equivalence. From Lemma \ref{Lemma: Constant-Evaluation Quillen pair} we know that $\mathsf{c}(S_e(X))$ is cofibrant in the projective model structure and therefore also in the local projective structure.
\end{proof}

\begin{proposition}
Let $X$ be a diffeological space and let $A$ be an abelian diffeological group. Then there is an isomorphism
    \[
    H_{\infty}^1\left( \overline{X}, A \right) \cong H^1(X,A(\mathrm{pt}))
    \]
    between the cohomology of $\overline{X}$ with values in the discrete simplicial sheaf $A$ and the simplicial cohomology of $X$ with values in the abelian group $A(\mathrm{pt})$, i.e. the underlying group of points by forgetting the diffeology on $A$.
\end{proposition}
\begin{remark}
    Given a diffeological space $X$ and $B$ an abelian group the \textbf{simplicial cohomology of} $X$ \textbf{with coefficients in } $B$ is defined to be the simplicial cohomology of the associated extended smooth singular complex $S_e(X)$ with coefficient in $B$. That is
    \[
    H^*(X,B) := H^*\left( S_e(X), B \right),
    \]
    where the right-hand side is isomorphic to the homotopy classes of maps $[S_e(X),K(B,n)]$ via the classical representation theorem \cite[Theorem 2.19]{GoerssJardine}.
\end{remark}
\begin{proof}
    By definition, we have that
    \[
    H_{\infty}^1\left( \overline{X}; A \right) = \pi_0 \mathbb{R}\mathrm{Hom}(\overline{X},\overline{W}A).
    \]
    Working in the local projective structure, we have that a cofibrant resolution of $\overline{X}$ is given by $\mathsf{c}(S_e(X))$ using Lemma \ref{Lemma: constant simplicial presheaf is cofibrant resolution}. Since $A$ is an abelian diffeological group it follows first by Lemma \ref{Lemma: constant presheaf and sheafification is weak equivalence} that $A$ is fibrant in the local projective structure. The functor $\mathsf{c}(-)$ agrees with its left derived functor $\mathbb{L}\mathsf{c}(-) = \mathsf{c}(Q(-)) = \mathsf{c}(-)$ since all objects in $\mathbf{sSet}$ are cofibrant. Therefore we have
    \begin{align*}
        \pi_0\mathbb{R}\mathrm{Hom}(\overline{X},\overline{W}A) &\cong \pi_0\mathbb{R}\mathrm{Map}(\mathbb{L}\mathsf{c}(S_e(X)),\overline{W}A) \\ &\cong \pi_0\mathbb{R}\mathrm{Map}\left(S_e(X),\mathbb{R}\mathrm{ev}_{\mathrm{pt}}\left(\overline{W}A\right)\right)
    \end{align*}
    using the fact that we are given a Quillen pair from remark \ref{Remark: Const-EV quillen pair local}. Now $\mathbb{R}\mathrm{ev}_{\mathrm{pt}}\left(\overline{W}A\right) = \mathrm{ev}_{\mathrm{pt}} \left(R \left( \overline{W}A \right) \right) \simeq \mathrm{ev}_{\mathrm{pt}}\left(\overline{W}A\right)$ using that $\overline{W}A$ is fibrant by \cite[Proposition 4.13]{Pavlov} and that $\mathrm{ev}_{\mathrm{pt}}$ is right Quillen. Hence we have
    \begin{align*}
\pi_0\mathbb{R}\mathrm{Map}\left(S_e(X),\mathbb{R}\mathrm{ev}_{\mathrm{pt}}\left(\overline{W}A\right)\right) &\cong \pi_0\mathbb{R}\mathrm{Map}(S_e(X),\overline{W}A(\mathrm{pt})) \\
        &\cong \pi_0 \mathrm{Map}\left(S_e(X),(\overline{W}A)(\mathrm{pt})\right).
    \end{align*}
    In the last step, we have also used the fact that
    \[
    (\overline{W}A)(\mathrm{pt}) = \overline{W}(A(\mathrm{pt})) = K(A(\mathrm{pt}),1)
    \]
    is the Eilenberg--MacLane space associated with the underlying group of points of the diffeological group $A$ and hence a Kan complex. The right-hand side denotes now the ordinary simplicial mapping space. Therefore, we have that
    \[
    \pi_0 \mathrm{Map}(S_e(X), \overline{W}A(\mathrm{pt}) ) \cong [S_e(X), K(A(\mathrm{pt}),1)] \cong H^1(X,A(\mathrm{pt})).
    \]
\end{proof}

How can we interpret this result?
\begin{itemize}
        \item First notice that the cohomology of the diffeological space $X$ when considered in the standard way as a discrete simplicial presheaf with values in $A$ is given by
        \[
        H_{\infty}^1(X,A) \cong \pi_0 \mathrm{DiffBun}_A(X)
        \]
        the \textit{isomorphism classes of diffeological principal} $A$\textit{-bundles} over $X$ for every diffeological abelian group $A$. This follows from \cite[Proposition 5.39]{Minichiello} and \cite[Theorem 5.15]{Krepski-Watts-Wolbert}. In the case where $M$ is a manifold and $A$ a Lie group, the above cohomology group coincides with the \v{C}ech cohomology of $M$ with values in $A$
        \[
        H_{\infty}^1(M,A) \cong \check{H}^1(M,A) \cong \pi_0 \mathrm{Bun}_A(M).
        \]
        \item The cohomology of the diffeological space $X$ considered now as an enriched simplicial presheaf $\overline{X}$ with values in $A$ is given by:
        \[
        H_{\infty}^1(\overline{X},A) \cong H^1(X,A(\mathrm{pt}))
        \]
        In the case where $X =M$ is a smooth manifold  and $A = U(1)$ is the circle group it thus follows that the cohomology of $\overline{X}$ with values in $U(1)$ classifies \textit{flat principal circle bundles over} $M$:
        \[
        H_{\infty}^1(\overline{M},U(1)) \cong H^1(M,U(1)(\mathrm{pt})) \cong \pi_0 \mathrm{Bun}^{\nabla_{flat}}_{U(1)}(M)
        \]
        For the latter isomorphism see \cite[Proposition 2.1.12]{Brylinski}. In the next chapter, we will then show how a specific refinement of $\overline{M}$ can be used to recover all circle bundles with connection and not only the flat ones.
    \end{itemize}

\subsection{Extensions in the category of smooth abelian groups}

As mentioned earlier, it is difficult in general to compute the cohomology of a simplicial presheaf $X$ taking values in a sheaf of abelian groups $A$. For us, the main computational tool will be the universal coefficients spectral sequence, which however requires computing extensions in the category of sheaves of abelian groups $\mathrm{Ab}(\cat{D}) := \mathrm{Sh}(\cat{D},\mathbf{Ab})$. In this section, we focus on the category of \textit{smooth abelian groups}, that is sheaves of abelian groups on the Cartesian site $\mathrm{Ab}(\mathbf{Cart})$, and show how to use local homotopy theory to compute higher extensions
\[
\mathrm{Ext}_{\mathrm{Ab}(\mathbf{Cart})}^r(B, A)
\]
in terms of derived mapping spaces. This will then in turn allow us to show that for special cases, the computation of these \textit{smooth extensions} reduces simply to the classical case of extensions in the category $\mathbf{Ab}$. \\

First notice that indeed, the category of sheaves of abelian groups is abelian, such that the Ext groups are defined as usual.

\begin{lemma}[\cite{stacks}, Lemma 18.3.1]
    Let $\cat{D}$ be a site and consider the category of abelian sheaves on $\cat{D}$ denoted by $\mathrm{Ab}(\cat{D})$. That is, the category of sheaves on the site $\cat{D}$ with values in the category of abelian groups $\mathbf{Ab}$. Let $\varphi \colon \cat{F} \rightarrow \cat{G}$ be a morphism of abelian sheaves.
    \begin{enumerate}
        \item The category $\mathrm{Ab}(\cat{D})$ is an abelian category.
        \item The morphism $\varphi$ is injective if and only if it is injective as a morphism of presheaves.
        \item The morphism $\varphi$ is surjective if and only if it is surjective as a morphism of sheaves of sets.
        \item A complex of abelian sheaves
        \[
        \cat{F} \rightarrow \cat{G} \rightarrow \cat{H}
        \]
        is exact at $\cat{G}$ if and only if for all $U \in \cat{C}$ and all $s \in \cat{G}(U)$ mapping to zero in $\cat{H}(U)$ there exists a covering $\left\{ U_i \rightarrow U\right\}_{i \in I}$ in $\cat{C}$ such that each $s_i = s|_{U_i}$ lies in the image of the map $\cat{F}(U_i) \rightarrow \cat{G}(U_i)$.
    \end{enumerate}
\end{lemma}

To begin with, let us have a look at the first ext group $\mathrm{Ext}^1_{\cat{A}}(B,A)$ in a general abelian category $\cat{A}$.

\begin{lemma}[\cite{stacks}, Lemma 13.27.6]
    Let $\cat{A}$ be an abelian category and $A,B$ objects in $\cat{A}$. Then $\mathrm{Ext}^1_{\cat{A}}(B,A)$ is the group given by isomorphism classes of extensions of $B$ by $A$. That is, the objects are given by classes $[E]$ represented by short exact sequences
    \[
    0 \rightarrow A \rightarrow E \rightarrow B \rightarrow 0
    \]
    up to isomorphism. More precisely we have $[E] = [E']$ if there is an isomorphism $E \xrightarrow{\cong} E'$ making the following diagram commute.
    \begin{center}
    \begin{tikzcd}
0 \arrow[r] & A \arrow[r] \arrow[d, "\mathrm{id}"] & E \arrow[r] \arrow[d, "\cong"] & B \arrow[r] \arrow[d, "\mathrm{id}"] & 0 \\
0 \arrow[r] & A \arrow[r]                          & E' \arrow[r]                   & B \arrow[r]                          & 0
    \end{tikzcd}
    \end{center}
\end{lemma}

As we are interested in the case $\mathrm{Ab}(\mathbf{Cart})$, recall that the category of diffeological spaces is equivalent to the category of concrete sheaves on the Cartesian site by Proposition \ref{Proposition: Diffeological Spaces are Concrete Sheaves on Cart}
\[
\mathbf{Diff} \xrightarrow{\simeq} \mathrm{Conc}(\mathbf{Cart}).
\]
Therefore, given diffeological abelian groups $N$ and $H$ we can on one side consider the abelian group $\mathrm{Ext}^1_{\mathrm{Ab}(\mathbf{Cart})}(H,N)$ of extensions of smooth abelian groups. On the other hand, there also exists then the notion of diffeological group extensions.

\begin{definition}[\cite{Iglesias-Zemmour}, Art. 7.3] \label{Definition: Diffeological Group Extension}
    Let $G$ be a diffeological group and $N$ a normal subgroup such that the quotient group $G/N$ is canonically a diffeological group with respect to the quotient diffeology. Then we get a short exact sequence
    \[
    1 \rightarrow N \xrightarrow{j} G \xrightarrow{\pi} H = G/N \rightarrow 1
    \]
    where $j$ is an induction and $\pi$ a subduction. Now given two diffeological groups $N$ and $H$, a \textbf{diffeological extension of} $H$ \textbf{by} $N$ is given by a diffeological group $G$ together with morphisms $j$ and $\pi$ satisfying the above diagram and such that $j$ is an induction and $\pi$ is a subduction.
\end{definition}

This raises the question if there is a difference between the abelian group of extensions $\mathrm{Ext}^1_{\mathrm{Ab}(\mathbf{Cart})}(H,N)$ and the abelian group of diffeological extensions $\mathrm{DiffExt}(H,N)$, given two diffeological abelian groups $N$ and $H$. \\

First, we note that a diffeological extension is a special case of a diffeological principal bundle.

\begin{lemma}[\cite{Iglesias-Zemmour}, Art. 7.3]
    Diffeological group extensions are a special case of diffeological principal fiber bundles. More precisely, for any extension $(G,j,\pi)$  of $H$ by $N$ we have that the projection
    \[
    \pi : G \rightarrow G/N = H
    \]
    exhibits $G$ as a diffeological principal fiber bundle of fiber type $N$ over $H$.
\end{lemma}

\begin{remark}
    Notice that for a diffeological group $G$, the associated sheaf of sets on $\mathbf{Cart}$ is of course concrete, that is the natural map
    \[
    G(U) \rightarrow \mathrm{Hom}_{\mathbf{Set}} \left( \mathbf{Cart}(\mathrm{pt},U),G(\mathrm{pt}) \right)
    \]
    is injective. However notice that this map naturally becomes an injective morphism of groups, when we equip the target set with the pointwise group structure in $G(\mathrm{pt})$. Then notice that given $\varphi,\varphi' \in G(U)$ two plots, we have that for all $d \in \mathbf{Cart}(\mathrm{pt},U)$:
    \[
    \underline{(\varphi \cdot \varphi')}(d) = G(d)\left( \varphi \cdot \varphi' \right) = G(d)(\varphi) \cdot G(d)(\varphi')
    \]
    Hence we indeed have that $\underline{(\varphi \cdot \varphi')} = \underline{\varphi} \cdot \underline{\varphi'}$.
\end{remark}

To see if smooth extensions and diffeological extensions agree, we first have to check, given concrete sheaves of abelian groups $N$ and $H$ and a short exact sequence of smooth abelian groups
\[
1 \rightarrow N \rightarrow G \rightarrow H \rightarrow 1,
\]
if the concreteness of $N$ and $H$ also implies that $G$ is concrete.

\begin{lemma}
    Let $\cat{D}$ be a site and consider the category of abelian sheaves on $\cat{D}$. Then the sections functor
    \[
     \Gamma(U,-) : \mathrm{Ab}(\cat{D}) \rightarrow \mathbf{Ab}
    \]
    is left exact, for all $U \in \cat{D}$.
\end{lemma}
\begin{proof}
    Let $U \in \cat{D}$. Recall that $\mathrm{Ab}(\cat{D})$ is an abelian category. Notice that the right adjoint to the sheafification functor preserves limits, i.e. is a left exact functor between abelian categories\footnote{For this fact we refer to \cite[Lemma 12.7.2]{stacks}.}
    \[
    \mathrm{Ab}(\cat{D}) \rightarrow \mathrm{PAb}(\cat{D}).
    \]
    Now using the fact that the sections functor
    \[
    \Gamma(U,-) : \mathrm{PAb}(\cat{D}) \rightarrow \mathbf{Ab}
    \]
    is exact for all $U \in \cat{D}$ implies that the sections functor
    \[
     \Gamma(U,-) : \mathrm{Ab}(\cat{D}) \rightarrow \mathbf{Ab},
    \]
    being the composition of a left exact and an exact functor is again left exact.
\end{proof}

The next proposition gives an affirmative answer to the question posed earlier, identifying smooth abelian extensions with diffeological extensions, provided $N$ and $H$ are both diffeological abelian groups.

\begin{proposition} \label{Proposition: Diffeological extensions and classical extensions}
    Consider a short exact sequence of abelian sheaves on $\mathbf{Cart}$
    \[
    0 \rightarrow N \xrightarrow{j} G \xrightarrow{\pi} H \rightarrow 0,
    \]
    where $N$ and $H$ are both concrete sheaves, i.e. diffeological abelian groups. Then the abelian sheaf $G$ is also concrete and the short exact sequence exhibits $G$ as a diffeological extension of $H$ by $N$. In particular, there is an isomorphism
    \[
    \mathrm{Ext}^1_{\mathrm{Ab}(\mathbf{Cart})}(H,N) \cong \mathrm{DiffExt}(H,N).
    \]
\end{proposition}
\begin{proof}
Let us first show that given a short exact sequence of abelian sheaves on $\mathbf{Cart}$ where $N$ and $H$ are concrete it automatically follows that $G$ is also concrete. Indeed, by evaluating at the terminal object we get a short exact sequence of abelian groups
    \[
    0 \rightarrow N(\mathrm{pt}) \rightarrow G(\mathrm{pt}) \rightarrow H(\mathrm{pt}) \rightarrow 0.
    \]
    This follows from the fact that $\mathbf{Cart}$ is a local site. Indeed, we notice first that by definition the morphism  $N(\mathrm{pt}) \rightarrow G(\mathrm{pt})$ is injective. Using the definition of surjectivity and exactness of abelian sheaves on site, together with the fact that the terminal object in a local site only allows the trivial cover, we conclude that the above sequence is exact. Now notice that the functor
    \[
    \mathrm{Hom}_{\mathbf{Set}}\left(\mathbf{Cart}(\mathrm{pt},U),- \right) : \mathbf{Ab} \rightarrow \mathbf{Ab}
    \]
    is left exact as it preserves all limits. This now implies that we get a left exact sequence
    \begin{equation*}
    \resizebox{.95\hsize}{!}{$0 \rightarrow     \mathrm{Hom}_{\mathbf{Set}}\left(\mathbf{Cart}(\mathrm{pt},U),N(\mathrm{pt}) \right) \rightarrow     \mathrm{Hom}_{\mathbf{Set}}\left(\mathbf{Cart}(\mathrm{pt},U), G(\mathrm{pt}) \right) \rightarrow     \mathrm{Hom}_{\mathbf{Set}}\left(\mathbf{Cart}(\mathrm{pt},U), H(\mathrm{pt}) \right),$}
    \end{equation*}
    where the abelian group structure is pointwise. Since $N$ and $H$ are by assumption concrete and by naturality we get a commutative diagram
    \begin{center}
    \adjustbox{scale=0.8,center}{%
        \begin{tikzcd}
0 \arrow[r] \arrow[d, "\mathrm{id}"] & N(U) \arrow[r] \arrow[d, hook]                                                   & G(U) \arrow[r] \arrow[d, "i"]                                                    & H(U) \arrow[d, hook]                                                   \\
0 \arrow[r]                          & {    \mathrm{Hom}_{\mathbf{Set}}\left(\mathbf{Cart}(\mathrm{pt},U),N(\mathrm{pt})\right) } \arrow[r] & {    \mathrm{Hom}_{\mathbf{Set}}\left(\mathbf{Cart}(\mathrm{pt},U),G(\mathrm{pt})\right) } \arrow[r] & {    \mathrm{Hom}_{\mathbf{Set}}\left(\mathbf{Cart}(\mathrm{pt},U),H(\mathrm{pt})\right) }
\end{tikzcd}
    }
\end{center}
where both rows are exact. The exactness of the top row follows from the fact that the sections functor is left exact. Now it follows by the 4-Lemma that the map $i$ is injective, i.e. $G$ is a concrete sheaf. \\

To conclude that the above short exact sequence is a diffeological extension of diffeological abelian groups, we need to show that the map $\pi$ is a subduction and that the map $j$ is an induction. \\

Indeed, first notice that since $\pi$ is a surjective map of sheaves on Cartesian spaces, we have that for any $U \in \mathbf{Cart}$ and every plot $p \in H(U)$ there is a good open covering $\left\{U_i \rightarrow U \right\}$ of $U$ such that the restriction $p|_{U_i}$ lies in the image of
\[
\pi_{U_i} : G(U_i) \rightarrow H(U_i)
\]
Hence there exists a plot $q : U_i \rightarrow G$ in $G$ such that $\pi \circ q = p|_{U_i}$, which shows that $\pi$ is a subduction.

We are left to show that $j$ is an induction. From the exactness of the sequence of abelian sheaves, we have that $\mathrm{Im}(j) \cong \mathrm{Ker}(\pi)$. By the concreteness of $G$ and $H$, it follows that $\mathrm{Ker}(\pi)$ is a diffeological subspace of $G$ and therefore the inclusion $\mathrm{Im}(j) \hookrightarrow G$ is an induction\footnote{This follows from the fact that the kernel of a morphism of diffeological spaces agrees with the kernel taken in the category of sheaves on Cartesian spaces. This is a consequence of how to compute limits in the category of diffeological spaces, see \cite[Proposition 39]{BaezHoffnung}.}. This now implies that the map $N \xrightarrow{j} G$ is an induction. \\

In the last step, we have to argue that any diffeological extension is also an extension of smooth abelian groups. Since $j$ is an induction the according morphism of sheaves is injective. Similarly for $\pi$ being a subduction, it follows that the underlying morphism of sheaves is surjective. This leaves us to check exactness. Let $U \in \mathbf{Cart}$ and let $p \colon U \rightarrow G$ be a plot such that $\pi \circ p = 0$. By definition of the subset diffeology, it follows that $p \colon U \rightarrow \mathrm{Ker}(\pi)$. Since $\mathrm{Ker}(\pi) \cong N$ it then follows that there is a plot $p' \colon U \rightarrow N$ such that $j \circ p' = p$. This shows exactness on the level of sheaves.
\end{proof}

However, Proposition \ref{Proposition: Diffeological extensions and classical extensions} allows us only to interpret a certain class of extensions in $\mathrm{Ab}(\mathbf{Cart})$ as diffeological extensions, which a priori does not make it easier to classify them. Also, the result only holds for the first extension group whereas we wish to understand also the higher ones. Nevertheless, this exercise helps to understand the main principle applied later to compute the higher extension groups. That is, for $A \in \mathbf{Ab}$ an abelian group and $H$ an abelian Lie group consider the two sheaves of abelian groups on $\mathbf{Cart}$
\begin{align*}
U &\mapsto \mathsf{c}(A)(U) = A, \text{ and} \\
U &\mapsto H(U) = C^{\infty}(U,H).
\end{align*}
In particular, both sheaves of abelian groups are concrete, corresponding to the discrete diffeological abelian group $\mathsf{c}(A)$ and the diffeological abelian group $H$. According to Proposition \ref{Proposition: Diffeological extensions and classical extensions} it then follows that
\[
\mathrm{Ext}^1_{\mathrm{Ab}(\mathbf{Cart})}(\mathsf{c}(A),H) \cong \mathrm{DiffExt}(\mathsf{c}(A),H)
\]
and hence any extension of $\mathsf{c}(A)$ by $H$ is equivalently a diffeological extension. Since a diffeological extension of $\mathsf{c}(A)$ by $H$ is represented by a diffeological principal $H$-bundle over the discrete diffeological space $\mathsf{c}(A)$, the bundle is trivial. As such, the ``smooth data" is not necessary to classify extensions of $\mathsf{c}(A)$ and one simply has
\begin{equation} \label{Equation: The Discrete First Extension}
\mathrm{Ext}^1_{\mathrm{Ab}(\mathbf{Cart})}(\mathsf{c}(A),H)  \cong \mathrm{Ext}^1_{\mathbf{Ab}}(A,H(\mathrm{pt})).
\end{equation}

The goal of the remaining part of this chapter is to prove a generalized version of equation \eqref{Equation: The Discrete First Extension}. More precisely, we show that for $A \in \mathbf{Ab}$ any abelian group and any $Y \in \cat{A} := \mathrm{Ab}(\mathbf{Cart})$ there is a natural isomorphism
\[
\mathrm{Ext}^k_{\cat{A}} (\mathsf{c}(A),Y) \cong \mathrm{Ext}^k_{\mathbf{Ab}}(A,Y(\mathrm{pt}))
\]
for all $k \geq 0$. This is done by identifying $\mathrm{Ext}^i_{\cat{A}}(A,B)$ as the path components of a derived mapping space for some suitable simplicial model category and using the fact that the constant simplicial presheaf functor $\mathsf{c}$ has a right adjoint provided by the evaluation functor $\mathrm{ev}_{\mathrm{pt}}$. \\

As a warm-up, let us show how this is indeed the case in the classical case where $\cat{A} = R\mathbf{Mod}$ is the abelian category of $R$-modules for $R$ an associative ring with unit. In this case, the required model category is given by the \textit{projective model structure on bounded below chain complexes} $\mathbf{Ch}_+(R)$, which is defined as follows.

\begin{theorem}[\cite{Goerss-Schemmerhorn}, Theorem 1.5]
    Define a map $ f : M \rightarrow N$ in $\mathbf{Ch}_+(R)$ to be
\begin{enumerate}
    \item a weak equivalence if the map $f$ induces isomorphisms $H_k(N) \rightarrow H_k(M)$ for all $k \geq 0$.
    \item a cofibration if for each $k \geq 0$ the map $f_k : M_k \rightarrow N_k$ is a monomorphism with a projective $R$-module as its cokernel, and
    \item a fibration if for each $k > 0$ the map $f_k : M_k \rightarrow N_k$ is an epimorphism.
\end{enumerate}
Then with these choices, $\mathbf{Ch}_+(R)$ is a model category called the \textbf{projective model structure on bounded chain complexes}.
\end{theorem}

The importance of the projective model structure is, that it can be compared with the standard model structure of simplicial $R$-modules via the Dold--Kan correspondence. Since we are eventually only interested in the case of sheaves of abelian groups, we simply restrict to the case $R = \mathbb{Z}$ for simplicity. However, all the statements do indeed hold also for $R\mathbf{Mod}$. \\

Recall the transferred model structure on simplicial abelian groups via the free-forgetful adjunction of Theorem \ref{Theorem: model structure simplicial groups}.

\begin{theorem} \label{simplicial abelian group model structure}
    The category of simplicial abelian groups $\mathbf{sAb}$ carries a model structure given by the transferred model structure along the forgetful functor over the Kan--Quillen model structure on simplicial sets
    \begin{center}
        \begin{tikzcd}
\mathbf{sAb} \arrow[r, "U"', shift right=3] \arrow[r, "\perp", phantom] & \mathbf{sSet}. \arrow[l, "\mathbb{Z}"', shift right=3]
\end{tikzcd}
    \end{center}
\end{theorem}

The Dold--Kan correspondence now serves as a bridge between the homotopy theory of simplicial abelian groups and the projective model structure on bounded below chain complexes.

\begin{theorem} \label{Theorem: Dold-Kan Quillen equivalence}
    The Dold--Kan correspondence provides Quillen equivalences between the projective model structure on bounded below chain complexes and the transferred model structure on simplicial abelian groups. More precisely, both adjunctions
    \begin{center}
        \begin{tikzcd}
\mathbf{Ch}_+(\mathbb{Z}) \arrow[r, "\Gamma"', shift right=2] \arrow[r, "\perp", phantom] & \mathbf{sAb} \arrow[l, "N"', shift right=2] & \mathbf{sAb} \arrow[r, "N"', shift right=2] \arrow[r, "\perp", phantom] & \mathbf{Ch}_+(\mathbb{Z}) \arrow[l, "\Gamma"', shift right=2]
\end{tikzcd}
    \end{center}
    provide Quillen equivalences with respect to these model structures.
\end{theorem}

This observation, together with the fact that the model structure on simplicial abelian groups is simplicial, allows us to compute the ext groups via path components of the derived mapping space.

\begin{proposition}[\cite{Dwyer-Spalinski}, Proposition 7.3] \label{Proposition: ext for abelian groups}
    Let $n,m \in \mathbb{N}$ and $A, B$ be abelian groups. Then there is a natural isomorphism
    \[
    \mathrm{Ext}_{\mathbf{Ab}}^{n-m} (B,A) \cong \pi_0 \mathbb{R} \mathrm{Map}(B[-m],A[-n]) \cong \pi_0\mathbb{R}  \mathrm{Map}(\overline{W}^m B,\overline{W}^n A)
    \]
    between the ext groups and the path components of the derived mapping spaces of the projective model structure on chain complexes and the model structure on simplicial abelian groups.
\end{proposition}
\begin{remark}
    Notice that if $k<0$ we set $\mathrm{Ext}_{\mathbf{Ab}}^k(B,A)$ to be zero.
\end{remark}

We now wish to generalize this statement to the case where the abelian category is given by the category $\mathrm{Ab}(\mathbf{Cart})$ of sheaves of abelian groups on Cartesian spaces. As a first guess one might try to construct a projective model structure on the category of bounded below chain complexes $\mathbf{Ch}_+(\cat{A})$ in some abelian category $\cat{A}$. Unfortunately, this model structure only exists if $\cat{A}$ has enough projectives\footnote{See Theorem 4 in Chapter II.4 of \cite{Quillen67}}, such as we have seen for $R\mathbf{Mod}$. The key is to work in the presheaf category $\mathrm{PSh}(\mathbf{Cart}, \mathbf{Ab})$, which indeed has enough projectives and is again abelian, instead of working with $\mathrm{Ab}(\mathbf{Cart})$. In particular, the category of chain complexes of presheaves of abelian groups $\mathbf{Ch}_+\left( \mathrm{PSh}(\mathbf{Cart}, \mathbf{Ab}) \right)$ can be endowed with an analog of the projective model structure, which we call here the \textit{standard model structure of presheaves of bounded-below chain complexes}. \\

This allows us to prove the following.

\begin{proposition} \label{Proposition: ext for smooth abelian groups}
    Let $X,Y \in \cat{A} :=\mathrm{Ab}(\mathbf{Cart})$ and $k \in \mathbb{N}$. Then there is a canonical isomorphism:
    \[
    \mathrm{Ext}_{\cat{A}}^k(X,Y) \cong \pi_0 \mathbb{R}\mathrm{Map}(X,\overline{W}^k(Y))
    \]
    where the derived mapping space is taken in the local injective model structure on \newline $\mathrm{PSh}(\mathbf{Cart}, \mathbf{sAb})$.
\end{proposition}
\begin{remark}
    Notice that as always, we consider $X$ and $Y$ as objects in $\mathrm{PSh}(\mathbf{Cart}, \mathbf{sAb})$ which are constant in the simplicial direction.
\end{remark}

To prove this proposition let us recall the definition of the ext groups for an arbitrary abelian category.
\newpage

\begin{definition} \label{Definition: definition of ext}
    Let $\cat{A}$ be an abelian category, $i \in \mathbb{Z}$ and $X,Y \in \cat{A}$ two objects. Then we define the $i$\textbf{th extension group of} $X$ \textbf{by} $Y$ to be the group
    \[
    \mathrm{Ext}^i_{\cat{A}}(X,Y) := \mathrm{Hom}_{D(\cat{A})}( X[0],Y[i] ),
    \]
    where $X[-i]$ is the chain complex concentrated in the $i$th degree and $D(\cat{A})$ the derived category of $\cat{A}$.
\end{definition}

Now, we construct a model structure on the category of \textbf{unbounded} chain complexes of presheaves with values in abelian groups denoted by $\mathbf{Ch} \left( \mathrm{PSh}(\mathbf{Cart}, \mathbf{Ab}) \right)$, whose homotopy category computes the derived category for $\mathrm{Ab}(\mathbf{Cart})$. To do so, observe that for $\cat{A}$ an abelian category the fully faithful functor
\[
i_0 : \mathbf{Ch}_{+}(\cat{A}) \rightarrow \mathbf{Ch}(\cat{A})
\]
sending a bounded below chain complex to the unbounded complex by inserting zeros in negative degrees, admits a right adjoint
\[
T_0 : \mathbf{Ch}(\cat{A}) \rightarrow  \mathbf{Ch}_{+}(\cat{A})
\]
called the \textbf{good truncation functor} given by
\[
(T_0E)_p = \begin{cases}
    E_p \text{ if } p > 0 \\
    \mathrm{Ker} \left( d : E_0 \rightarrow E_{-1} \right) \text{ if } p = 0.
\end{cases}
\]

Also, we now introduce the standard model structure on bounded below chain complexes of presheaves and then use the above inclusion-truncation adjunction to construct the desired model structure in the unbounded case. \\

First, endow the category of presheaves on $\mathbf{Cart}$ with values in simplicial abelian groups with the local injective model structure. This is achieved by considering the induced free-forgetful adjunction on the presheaf categories.
\begin{equation}\label{Equation: Free-Forgetful presheaf adjucntion}
\begin{tikzcd}
\mathrm{PSh}(\mathbf{Cart},\mathbf{sAb}) \arrow[r, "U"', shift right=3] \arrow[r, "\perp", phantom] & \mathrm{sPSh}(\mathbf{Cart}) \arrow[l, "\mathbb{Z}"', shift right=3]
\end{tikzcd}
\end{equation}
The local injective model structure on $\mathrm{PSh}(\mathbf{Cart},\mathbf{sAb})$  can now be constructed as the right transferred model structure of the local injective structure on $\mathrm{sPSh}(\mathbf{Cart})$.

\begin{theorem}[\cite{JardineBook}, Theorem 8.6]
 Define a map $f : X \rightarrow Y$ in $\mathrm{PSh}(\mathbf{Cart},\mathbf{sAb})$ to be a
    \begin{enumerate}
        \item weak equivalence if the induced map $U(f) : U(X) \rightarrow U(Y)$ is a weak equivalence in the local injective model structure.
        \item fibration if the induced map $U(f) : U(X) \rightarrow U(Y)$ is a fibration in the local injective model structure.
        \item cofibration if it has the left lifting property with respect to all trivial fibrations.
    \end{enumerate}
    With these definitions the category  $\mathrm{PSh}(\mathbf{Cart},\mathbf{sAb})$ has the structure of a proper closed simplicial model category. Denote this model category by $\mathrm{PSh}(\mathbf{Cart},\mathbf{sAb})_{\mathrm{inj, loc}}$.
\end{theorem}

In light of Theorem \ref{Theorem: Dold-Kan Quillen equivalence} the correct model structure to consider on the category of presheaves of bounded-below chain complexes is the one induced from $\mathrm{PSh}(\mathbf{Cart},\mathbf{sAb})_{\mathrm{inj, loc}}$ via the Dold--Kan equivalence for abelian presheaves. That is, the equivalence of categories
\begin{equation}\label{Equation: presheaf Dold-Kan}
\begin{tikzcd}
{N: \mathrm{PSh}(\mathbf{Cart},\mathbf{sAb})} \arrow[r, shift right=3] \arrow[r, "\simeq", phantom] & {\mathrm{Ch}_{+}(\mathrm{PSh}(\mathbf{Cart}, \mathbf{Ab})) : \Gamma} \arrow[l, shift right=3]
\end{tikzcd}
\end{equation}
allows us to introduce the \textit{standard model structure of presheaves of bounded-below chain complexes} as follows.
\begin{theorem}[\cite{JardineBook}, Corollary 8.7] \label{Theorem: standard model structure bounded below chain complexes}
    Define a map $f : C \rightarrow D$ in \newline $\mathbf{Ch}_{+}\left(\mathrm{PSh}(\mathbf{Cart}, \mathbf{Ab}) \right)$ to be a
    \begin{enumerate}
        \item weak equivalence if the induced map $\Gamma(f) : \Gamma(C) \rightarrow \Gamma(D)$ is a weak equivalence in the local injective model structure.
        \item fibration if the induced map $\Gamma(f) : \Gamma(C) \rightarrow \Gamma(D)$ is a fibration in the local injective model structure.
        \item cofibration if the induced map $\Gamma(f) : \Gamma(C) \rightarrow \Gamma(D)$ is a cofibration in the local injective model structure.
    \end{enumerate}
    With these definitions the category $\mathbf{Ch}_{+}\left(\mathrm{PSh}(\mathbf{Cart}, \mathbf{Ab}) \right)$ has the structure of a proper closed simplicial model category. We call this the \textbf{standard model structure of} \newline \textbf{presheaves of bounded-below chain complexes}, denoted $\mathbf{Ch}_{+}\left(\mathrm{PSh}(\mathbf{Cart}, \mathbf{Ab}) \right)_{\mathrm{std}}$.
\end{theorem}

Yet this model category still doesn't represent the derived category $D(\mathrm{Ab}(\mathbf{Cart}))$ as we are only considering \textbf{bounded below chain complexes}. Here is where the good truncation adjunction $ i_0 \dashv T_0$ comes into play.
\begin{equation}
\begin{tikzcd}
{\mathbf{Ch}_{+}\left(\mathrm{PSh}(\mathbf{Cart}, \mathbf{Ab}) \right)} \arrow[r, "i_0", shift left=3] \arrow[r, "\perp", phantom] & {\mathbf{Ch}\left(\mathrm{PSh}(\mathbf{Cart}, \mathbf{Ab}) \right)} \arrow[l, "T_0", shift left=3]
\end{tikzcd}
\end{equation}

\begin{theorem}[\cite{Jardine2003}, Theorem 2.6] \label{Theorem: stable model structure of chain complexes}
    Define a map $ f  : C \rightarrow D$ in \newline $\mathbf{Ch} \left( \mathrm{PSh}(\mathbf{Cart}, \mathbf{Ab}) \right)$ to be a
    \begin{enumerate}
        \item weak equivalence if $f$ induces isomorphisms on all the homology sheaves, i.e.
        \[
        \tilde{H}_k(C) \xrightarrow{\cong} \tilde{H}_k(D)
        \]
        for all $k \in \mathbb{Z}$.
        \item a fibration if all the induced maps
        \[
        T_0(C[-n]) \rightarrow T_0(D[-n])
        \]
        are fibrations in the standard model structure of presheaves of bounded-below chain complexes for all $n \in \mathbb{N}$.
        \item cofibration if it has the left lifting property with respect to trivial fibrations.
    \end{enumerate}
    With these definitions, the category $\mathbf{Ch} \left( \mathrm{PSh}(\mathbf{Cart}, \mathbf{Ab}) \right)$ has the structure of a proper, closed model category. We call this the \textbf{stable model structure of presheaves of chain complexes}, denoted $\mathbf{Ch} \left( \mathrm{PSh}(\mathbf{Cart}, \mathbf{Ab}) \right)_{\mathrm{stbl}}$.
\end{theorem}

As a corollary to the above theorem, we have the following result.

\begin{corollary}[\cite{JardineBook}]  \label{Corollary: inclusion-truncation adjunction}
    The inclusion-truncation adjunction
    \begin{center}
        \begin{tikzcd}
{\mathbf{Ch} \left( \mathrm{PSh}(\mathbf{Cart}, \mathbf{Ab}) \right)_{\mathrm{stbl}}} \arrow[r, "T_0"', shift right=2] \arrow[r, "\perp", phantom] & {\mathbf{Ch}_+ \left( \mathrm{PSh}(\mathbf{Cart}, \mathbf{Ab}) \right)_{\mathrm{std}}} \arrow[l, "i_0"', shift right=2]
\end{tikzcd}
    \end{center}
    is a Quillen adjunction with respect to the standard and the stable model structures of presheaves of chain complexes. Moreover, both functors $i_0$ and $T_0$ preserve weak equivalences.
\end{corollary}
\begin{proof}
    Notice that the right adjoint truncation functor $T_0$ preserves weak equivalences and fibrations. Indeed, for any fibration $f : C \rightarrow D$ in the stable model structure, we have by definition that
    \[
    T_0C \rightarrow T_0D
    \]
    is a fibration in the standard model structure on bounded-below complexes. Moreover, the good truncation functor $T_0$ preserves quasi-isomorphisms, which together with \cite[Lemma 8.4]{JardineBook} show that for $f \colon C \rightarrow D$ a stable weak equivalence the morphism $T_0(f)$ is indeed a standard weak equivalence. By Lemma \ref{Lemma: Quillen Adj Lemma} this shows that $i_0 \dashv T_0$ forms a Quillen pair. What is left to show is that $i_0$ preserves weak equivalences. Again by \cite[Lemma 8.4]{JardineBook} we have that a standard weak equivalence $f \colon C \rightarrow D$ of presheaves of bounded below chain complexes induces isomorphisms in all homology sheaves, i.e. is a quasi-isomorphism. By definition, $i_0$ preserves quasi-isomorphisms and hence preserves weak equivalences.
\end{proof}

\begin{theorem}\label{Theorem: Stable Model structure computes Derived category}
    Let $\cat{A} := \mathrm{Ab}(\mathbf{Cart})$ be the abelian category of sheaves of abelian groups on the Cartesian site $\mathbf{Cart}$. The homotopy category of the stable model category is equivalent to the derived category of $\cat{A}$:
    \[
    \mathrm{Ho}\, \mathbf{Ch} \left( \mathrm{PSh}(\mathbf{Cart}, \mathbf{Ab}) \right)_{\mathrm{stbl}} \simeq D(\cat{A}).
    \]
\end{theorem}
\begin{proof}
    This follows from \cite[Theorem 2.6]{Jardine2003} together with Proposition 8.16 and the subsequent discussion in \cite{JardineBook}.
\end{proof}

Now we are ready to prove Proposition \ref{Proposition: ext for smooth abelian groups}.

\begin{proof}[Proof of Proposition \ref{Proposition: ext for smooth abelian groups}]

Let $X,Y \in \cat{A} :=\mathrm{Sh}(\mathbf{Cart}, \mathbf{Ab})$ and $k \in \mathbb{N}$. To simplify notation introduce the abbreviations:
\begin{align*}
    &\cat{M} := \mathbf{Ch} \left( \mathrm{PSh}(\mathbf{Cart}, \mathbf{Ab}) \right)_{\mathrm{stable}} \\
    &\cat{M}_+ := \mathbf{Ch}_+ \left( \mathrm{PSh}(\mathbf{Cart}, \mathbf{Ab}) \right)_{\mathrm{std}}.
\end{align*}

By Definition \ref{Definition: definition of ext} we have that:
\[
\mathrm{Ext}_{\cat{A}}^k(X,Y) = \mathrm{Hom}_{D(\cat{A})}(X[0],Y[-k]).
\]
By Theorem \ref{Theorem: Stable Model structure computes Derived category} it follows that the derived category $D(\cat{A})$ is equivalent to the homotopy category of the stable model structure on $\mathbf{Ch} \left( \mathrm{PSh}(\mathbf{Cart}, \mathbf{Ab}) \right)$ and therefore
\[
\mathrm{Hom}_{D(\cat{A})}(X[0],Y[-k]) \cong \mathrm{Hom}_{\mathrm{Ho}\, \cat{M}} ( X[0],Y[-k] ) = \mathrm{Hom}_{\mathrm{Ho}\, \cat{M}} ( i_0(X[0]),Y[-k] ),
\]
where we have used in the last equality that $i_0(X[0]) = X[0]$. Recall from Definition \ref{Definition: Left/Right derived functors} that the left derived functor is of the form $\mathbb{L}i_0(X[0]) = i_0 \left( Q_{\cat{M}_+}(X[0]) \right)$. Corollary \ref{Corollary: inclusion-truncation adjunction} implies that the standard weak equivalence $Q_{\cat{M}_+}(X[0]) \xrightarrow{\simeq} X[0]$ is preserved by $i_0$, i.e. there is a stable weak equivalence $\mathbb{L}i_0(X[0]) \xrightarrow{\simeq} i_0(X[0])$. This induces a chain of isomorphisms, where the latter is given by Proposition \ref{Proposition: Derived Quillen adjunction is adjunction in Ho-Cat}:
\begin{align*}
\mathrm{Hom}_{\mathrm{Ho}\, \cat{M}} ( i_0(X[0]),Y[-k] ) &\cong \mathrm{Hom}_{\mathrm{Ho}\, \cat{M}} (\mathbb{L}i_0(X[0]),Y[-k] ) \\
&\cong \mathrm{Hom}_{\mathrm{Ho}\, \cat{M}_+} \left( X[0],\mathbb{R}T_0(Y[-k]) \right).
\end{align*}
Since $k \in \mathbb{N}$ notice that $T_0(Y[-k]) = Y[-k]$ and we also have that $T_0(Y[-k]) \xrightarrow{\simeq} \mathbb{R}T_0(Y[-k])$ is a standard weak equivalence by exactly the dual argument used before and the fact that $T_0$ preserves weak equivalences by Corollary \ref{Corollary: inclusion-truncation adjunction}. This shows that:
\[
\mathrm{Hom}_{\mathrm{Ho}\, \cat{M}_+} \left( X[0],\mathbb{R}T_0(Y[-k]) \right) \cong \mathrm{Hom}_{\mathrm{Ho}\, \cat{M}_+} \left( X[0],T_0(Y[-k]) \right) = \mathrm{Hom}_{\mathrm{Ho}\, \cat{M}_+} \left( X[0], Y[-k] \right).
\]

Recall that the standard model structure on bounded below chain complexes is simplicial by Theorem \ref{Theorem: standard model structure bounded below chain complexes} such that by Proposition \ref{Proposition: simplicial homotopy classes and morphism in Ho-Cat}
\begin{align*}
\mathrm{Hom}_{\mathrm{Ho}\, \cat{M}_+} ( X[0],Y[-k]) &\cong \pi_0 \mathbb{R}\mathrm{Map}(X[0],Y[-k]) \\
&= \pi_0 \mathbb{R}\mathrm{Map}(N(X),Y[-k] ) \\
&\cong \pi_0 \mathbb{R}\mathrm{Map}(X,\Gamma(Y[-k])) \\
&\cong \pi_0 \mathbb{R}\mathrm{Map}(X,\overline{W}^k(Y)),
\end{align*}
where for the second isomorphism we have used the fact that the Dold--Kan correspondence induces an equivalence of categories between homotopy categories associated with the standard model structure on bounded below complexes and the local injective model structure on $\mathrm{PSh}(\mathbf{Cart}, \mathbf{sAb})$ (\cite[Lemma 1.5]{Jardine2003}). For the other two isomorphisms note that we have $N(X) = X[0]$ since we consider $X$ as presheaf of simplicial abelian groups which is constant in the simplicial direction and $\Gamma(Y[-k]) \cong \overline{W}^k(Y)$ by Lemma \ref{Lemma: simplicial classifying space is Eilenberg-Mac Lane object}.
\end{proof}

\begin{corollary} \label{Corollary: ext of constant sheaf}
    Let $A \in \mathbf{Ab}$ be an abelian group and let $\mathsf{c}(A) \in \cat{A} := \mathrm{Sh}(\mathbf{Cart}, \mathbf{Ab})$ be the constant sheaf. Then for any $Y \in \cat{A}$ and any $k \geq 0$ there is a natural isomorphism
    \[
    \mathrm{Ext}^k_{\cat{A}} (\mathsf{c}(A),Y) \cong \mathrm{Ext}^k_{\mathbf{Ab}}(A,Y(\mathrm{pt})),
    \]
    where $Y(\mathrm{pt})$ denotes the abelian group obtained by evaluating $Y$ at the terminal object in $\mathbf{Cart}$.
\end{corollary}

To prove this corollary we need an adapted version of the constant-evaluation Quillen adjunction of \cite[Equation 2.16]{Bunk}.

\begin{lemma} \label{Lemma: constant -  evaluation adjunction + weak equivalences}
    The constant simplicial presheaf functor $\mathsf{c}$ and the evaluation at the point functor $\mathrm{ev}_{\mathrm{pt}}$ form a Quillen adjunction
    \begin{center}
    \begin{tikzcd}
\mathbf{sAb} \arrow[r, "\mathsf{c}", shift left=2] \arrow[r, "\perp", phantom] & {\mathrm{PSh}(\mathbf{Cart}, \mathbf{sAb})_{\mathrm{proj,\check{C}}}} \arrow[l, "\mathrm{ev}_{\mathrm{pt}}", shift left=2]
\end{tikzcd}
    \end{center}
    for the \v{C}ech-local projective model structure on presheaves of simplicial abelian groups. Moreover, the functor $\mathsf{c}$ preserves weak equivalences.
\end{lemma}

\begin{proof}
    We first recall by \cite[Equation 2.16]{Bunk} that for the \v{C}ech-local projective model structure, there is a Quillen adjunction
\begin{equation} \label{Equation: Cech-local projective Quillen Adj.}
\begin{tikzcd}
\mathbf{sSet} \arrow[r, "\mathsf{c}", shift left=2] \arrow[r, "\perp", phantom] & {\mathrm{sPSh}(\mathbf{Cart})_{\mathrm{proj,\check{C}}}.} \arrow[l, "\mathrm{ev}_{\mathrm{pt}}", shift left=2]
\end{tikzcd}
 \end{equation}
Now we recall that the model structure on simplicial abelian groups has been transferred from the Kan--Quillen model structure on simplicial sets. Likewise, the \v{C}ech-local projective structure $\mathrm{PSh}(\mathbf{Cart}, \mathbf{sAb})_{\mathrm{proj,\check{C}}}$ is right transferred from the corresponding structure on simplicial presheaves $\mathrm{sPSh}(\mathbf{Cart})_{\mathrm{proj,\check{C}}}$. In particular, notice that there is a commutative diagram:
\begin{equation}\label{Diagram: sAb proj}
\begin{tikzcd}
\mathbf{sSet}               & \mathrm{sPSh}(\mathbf{Cart}) \arrow[l, "\mathrm{ev}_{\mathrm{pt}}"']                              \\
\mathbf{sAb} \arrow[u, "U"] & {\mathrm{PSh}(\mathbf{Cart},\mathbf{sAb})} \arrow[l, "\mathrm{ev}_{\mathrm{pt}}"] \arrow[u, "U"']
\end{tikzcd}
\end{equation}
We wish to show that the functor $\mathrm{ev}_{\mathrm{pt}}$ defined for presheaves of simplicial abelian groups preserves fibrations and trivial fibrations. Let $f \colon X \rightarrow Y$ be a \v{C}ech-local projective fibration of presheaves of simplicial abelian groups, i.e. $U(f) \colon U(X) \rightarrow U(Y)$ is \v{C}ech-local projective fibration. By (\ref{Equation: Cech-local projective Quillen Adj.}) it follows that $U(f)(\mathrm{pt}) \colon U(X)(\mathrm{pt}) \rightarrow U(Y)(\mathrm{pt})$ is a Kan fibration. By the commutativity of (\ref{Diagram: sAb proj}) this morphisms equals $U(f(\mathrm{pt})) \colon U(X(\mathrm{pt})) \rightarrow U(Y(\mathrm{pt}))$, which now implies that the corresponding map $f(\mathrm{pt}) \colon X(\mathrm{pt}) \rightarrow Y(\mathrm{pt})$ of simplicial abelian groups is indeed a fibration. The same also holds for trivial fibrations, which then implies that $\mathrm{ev}_{\mathrm{pt}}$ is right Quillen. What is left to show is that $\mathsf{c}$ preserves weak equivalences. This follows from the fact that each object-wise weak equivalence is also a \v{C}ech-local weak equivalence. This finishes the proof.
\end{proof}

\begin{proof}[Proof of Corollary \ref{Corollary: ext of constant sheaf}:]
First, we use Proposition \ref{Proposition: ext for smooth abelian groups} and get
    \[
     \mathrm{Ext}^k_{\cat{A}} (\mathsf{c}(A),Y) \cong \pi_0 \mathbb{R}\mathrm{Map}(\mathsf{c}(A),\overline{W}^k(Y)),
    \]
    where the right-hand side denotes the derived mapping space taken in the local injective model structure on $\mathrm{PSh}(\mathbf{Cart}, \mathbf{sAb})$. Notice that there is a Quillen equivalence given by the identity
\begin{center}
\begin{tikzcd}
{\mathrm{PSh}(\mathbf{Cart},\mathbf{sAb})_{\mathrm{proj,\check{C}}}} \arrow[r, "\mathrm{id}", shift left=3] \arrow[r, "\perp", phantom] & {\mathrm{PSh}(\mathbf{Cart},\mathbf{sAb})_{\mathrm{inj,\check{C}}} = \mathrm{PSh}(\mathbf{Cart},\mathbf{sAb})_{\mathrm{inj,loc}}} \arrow[l, "\mathrm{id}", shift left=3]
\end{tikzcd}
\end{center}
and in particular, both model categories have the same weak equivalences\footnote{This follows from the fact that both model structures are right transferred from the corresponding local model structures of simplicial presheaves. Those are Quillen equivalent and have the same weak equivalences by \cite[Proposition~2.9]{Bunk}.}. It follows from this, that we have a natural isomorphism
\[
\mathrm{Hom}_{\mathrm{Ho} \, \cat{M}}(X,Y) \cong \mathrm{Hom}_{\mathrm{Ho} \, \cat{N}}(X,Y)
\]
for all $X,Y \in \mathrm{PSh}(\mathbf{Cart},\mathbf{sAb})$ where we abbreviate
\begin{align*}
    \cat{M} &:= \mathrm{PSh}(\mathbf{Cart},\mathbf{sAb})_{\mathrm{proj,\check{C}}} \\
    \cat{N} &:= \mathrm{PSh}(\mathbf{Cart},\mathbf{sAb})_{\mathrm{inj,loc}}
\end{align*}
to simplify notation. It then follows that
\[
     \mathrm{Ext}^k_{\cat{A}} (\mathsf{c}(A),Y) \cong \mathrm{Hom}_{\mathrm{Ho} \, \cat{N}} \left( \mathsf{c}(A),\overline{W}^kY \right) \cong \mathrm{Hom}_{\mathrm{Ho} \, \cat{M}} \left( \mathsf{c}(A),\overline{W}^kY \right).
\]
Hence assume from this moment on that we work in the \v{C}ech-local projective model structure. Using Lemma \ref{Lemma: constant -  evaluation adjunction + weak equivalences} one can show using the derived adjunction $\mathbb{L}\mathsf{c} \dashv \mathbb{R}\mathrm{ev}_{\mathrm{pt}}$ that
    \begin{align}
    \pi_0 \mathbb{R}\mathrm{Map}\left( \mathsf{c}(A),\overline{W}^kY \right) &\cong \pi_0 \mathbb{R}\mathrm{Map}\left( \mathbb{L}\mathsf{c}(A),\overline{W}^kY\right)  \\
    &\cong \pi_0 \mathbb{R}\mathrm{Map}\left(A,\mathbb{R}\mathrm{ev}_{\mathrm{pt}}\left( \overline{W}^kY \right) \right) \\
    &\cong \pi_0 \mathbb{R}\mathrm{Map}\left(A,(\overline{W}^kY)(\mathrm{pt}) \right), \label{Isomorphism: Projective Needed}
    \end{align}
    where the latter derived mapping space is taken in $\mathbf{sAb}$. Notice that it is because of the last isomorphism (\ref{Isomorphism: Projective Needed}) that we make use of the \v{C}ech-local projective structure. Indeed, since $Y$ is a sheaf of abelian groups on $\mathbf{Cart}$ one can show that $\overline{W}^kY$ is fibrant in the \v{C}ech-local projective structure by \cite[Proposition 4.13]{Pavlov}. It follows from this fact that $(\overline{W}^kY)(\mathrm{pt}) \xrightarrow{\simeq} \mathbb{R}\mathrm{ev}_{\mathrm{pt}}\left( \overline{W}^kY \right)$ is a weak equivalence, inducing this last isomorphism. By definition, we have:
    \[
    (\overline{W}^kY)(\mathrm{pt}) = \overline{W}^k\left( Y(\mathrm{pt}) \right) \cong \Gamma(Y(\mathrm{pt})[-k] ).
    \]
    Then the Dold--Kan correspondence gives together with Proposition \ref{Proposition: ext for abelian groups}:
    \[
    \pi_0 \mathbb{R}\mathrm{Map}(A,\Gamma(Y(\mathrm{pt})[-k] )) \cong \mathrm{Hom}_{\mathrm{Ho}\, \mathbf{Ch}_+(\mathbf{Ab}) } (A[0], Y(\mathrm{pt})[-k]) \cong \mathrm{Ext}_{\mathbf{Ab}}^k(A,Y(\mathrm{pt})).
    \]
\end{proof}

\chapter{Skeletal Diffeologies and Differential Characters} \label{Chapter: Skeletal Diffeologies and Differential Characters}

The purpose of this chapter is to prove the main result of this thesis.

\begin{theorem} \label{Theorem: The Main Theorem}
    Let $M$ be a smooth manifold without boundary. Then for all $k \geq 0$, there is an isomorphism of abelian groups
    \[
    H^k_{\infty}(\mathbb{M}_k,U(1)) \cong \hat{H}^{k+1}(M; \mathbb{Z} )
    \]
    identifying the Cheeger--Simons differential characters with the $k$th cohomology of the simplicial presheaf $\mathbb{M}_k$ with coefficients in $U(1)$.
\end{theorem}

The simplicial presheaf $\mathbb{M}_k$ is constructed as a refinement of the enriched simplicial presheaf $\overline{M}$ using $k$-skeletal diffeologies, to be introduced in the next section. First, we show that the universal coefficient spectral sequence can be used in this specific case to compute the cohomology $H^k_{\infty}(\mathbb{M}_k,U(1))$. Then, by switching from affine simplices $\mathbb{A}^{\bullet}$ to compact simplices $|\Delta^{\bullet}|$ we show how to recover the abelian group of differential characters.

\section{Skeletal Diffeologies}

\begin{definition}[\cite{Iglesias-Zemmour} Art. 1.66]
    Let $X$ be a set and $\cat{F}$ a family of maps \newline $\left\{ p \colon U \rightarrow X \text{ for } U \in \mathbf{Cart} \right\}$. The finest diffeology on $X$ containing the family $\cat{F}$ is called the \textbf{diffeology generated by} $\cat{F}$ and we denote it by $D_{\cat{F}}$. The plots for the diffeology generated by the family $\cat{F}$ can be characterized by the following property: \\

    A map $p \colon U \rightarrow X$ is a plot in $D_{\cat{F}}$ if and only if for all $z \in U$ there exists an open neighborhood $V$ of $z$ such that either $p|_V$ is constant, or there exists a map $f \colon W \rightarrow X$ belonging to the family $\cat{F}$ together with a smooth map of Cartesian spaces $q \colon V \rightarrow W$ such that $p|_V = f \circ q$.
\end{definition}

\begin{definition}[\cite{KiharaSk}]
    Let $X$ be a diffeological space. For an integer $ 0 \leq k < \infty$ consider the family of plots in $X$ of dimension $\leq k$, i.e.
    \[
    \cat{F}_k := \left\{ p \in C^{\infty}(U,X) \, | \, \mathrm{dim} \, U \leq k \right\}
    \]
    Then the diffeology on $X$ generated by the family $\cat{F}_k$ is called the $k$\textbf{-skeletal diffeology on} $X$. We denote the set $X$ endowed with the $k$-skeletal diffeology by $X_k$.
\end{definition}

Define the $k$-skeleton functor to assign to a diffeological space $X$ its $k$-skeletal diffeological space $X_k$.
\[
\begin{array}{rcl}
     (-)_k \colon \mathbf{Diff} &  \rightarrow & \mathbf{Diff}  \\
     X & \mapsto & X_k
\end{array}
\]

It is important to note that the $d$-skeleton functor preserves the underlying topological properties of a diffeological space. This stands in contrast to the smooth properties that are substantially altered. More precisely we have the following:

\begin{proposition}[\cite{KiharaSk,Christensen-Sinnamon-Wu}, Proposition 1.1 and Theorem 3.7] \label{Proposition: Skeletal D-topology}
    Let $X$ be a diffeological space and $k$ an integer $0 < k < \infty$. The identity map $\mathrm{id} \colon X_k \rightarrow X$ is a smooth map such that its induced map on the underlying topological spaces
    \[
    \mathrm{id} \colon \tau(X_k) \rightarrow \tau(X)
    \]
    is a homeomorphism.
\end{proposition}
\begin{remark}
Recall that $\tau \colon \mathbf{Diff} \rightarrow \mathbf{Top}$ denotes the functor sending a diffeological space to the underlying topological space equipped with the $D$-topology, see Definition \ref{Definition: D-topology}.
\end{remark}

Proposition \ref{Proposition: Skeletal D-topology} is based on the fact that the underlying topological space functor $\tau$ factors through the category of \textit{arc-generated topological spaces}. More specifically, call a topological space $X$ \textbf{arc-generated} if its topology is the final topology induced by all continuous curves $\mathbb{R} \rightarrow X$ in $X$. The underlying topological space $\tau(X)$ for any diffeological space $X$ is arc-generated \footnote{See \cite[Theorem 3.7]{Christensen-Sinnamon-Wu} or equivalently section 2.3 in \cite{KiharaMod}.}. Thus for $k > 0$, the set of smooth curves in $X_k$ agrees with the set of smooth curves in $X$ and the map $\mathrm{id} \colon \tau(X_k) \rightarrow \tau(X)$ is indeed a homeomorphism. \\

While skeletal diffeologies are well-behaved regarding their underlying topological spaces, they are pathological as diffeological spaces. What we mean by this is made precise by the next theorem.

\begin{theorem}[\cite{KiharaSk} Theorem 1.6] \label{Theorem: Smooth Homotopy of M_d}
    Let $(M,x_0)$ be a pointed smooth manifold and $k$ an integer with $ 0 < k < \mathrm{dim}\, M$. Then the natural homomorphism
    \[
    \pi_i^D(M_k,x_0) \rightarrow \pi_i(M,x_0)
    \]
    between the smooth homotopy group of $M_k$ and the standard topological homotopy group $\pi_i(M,x_0)$ is an isomorphism for $i < k$ and surjective for $i = k$. In addition, for $i = k$ the kernel is uncountable.
\end{theorem}

Recall that for $X$ a diffeological space its \textbf{smooth singular homology} $H_*(X)$ is given by the simplicial homology of the extended smooth singular complex:
\[
H_*(X) := H_*(S_e(X)) = H_* \left( N(\mathbb{Z}[S_e(X)]) \right)
\]

Similarly to Theorem \ref{Theorem: Smooth Homotopy of M_d} the following provides an analog regarding the smooth homology of skeletal diffeological spaces.

\begin{lemma} \label{Lemma: Smooth Singular Homology X_d}
    Let $X$ be a diffeological space and $d$ an integer with $0 < d < \infty$. Then the natural homomorphism induced by the smooth map $\mathrm{id} \colon X_d \rightarrow X$
    \[
    H_j(X_d) \rightarrow H_j(X)
    \]
    is the identity for all $j < d$.
\end{lemma}
\begin{proof}
    This follows easily from the fact that for $j < d$ the notions of smooth $j$-cycles and $j$-boundaries in $X_d$ and $X$ agree. That is, for $j < d$:
    \[
    Z_j(S_e(X_d)) = Z_j(S_e(X)) \text{ and } B_j(S_e(X_d)) = B_j(S_e(X))
    \]
\end{proof}

\begin{corollary} \label{Corollary: Constant Homology sheaves}
    Let $X$ be a diffeological space. For every $U \in \mathbf{Cart}$ the unique map $U \rightarrow \mathrm{pt}$ induces an objectwise weak equivalence in $\mathrm{sPSh}(\mathbf{Cart})$ between
    \begin{equation*}
               \mathsf{c}(S_e(X)) \rightarrow \overline{X}.
    \end{equation*}
    As a consequence, the homotopy and homology sheaves of the simplicial presheaf $\overline{X}$ are constant.
\end{corollary}
\begin{remark}
    Recall that given a diffeological space $X$ the enriched simplicial presheaf $\overline{X}$ is given by
    \[
    U \mapsto S_e(D(U,X)).
    \]
\end{remark}
\begin{proof}
Let $X$ be a diffeological space and let $U \in \mathbf{Cart}$ be a Cartesian space. Since any $U$ is smoothly contractible there exists $r_U \colon U \times \mathbb{R} \rightarrow U$ a smooth homotopy from the identity to a constant map, say $\mathrm{const}_{u_0}$. Note that for every $U \in \mathbf{Cart}$, the map $U \rightarrow \mathrm{pt}$ induces inductions
\[
i \colon X = D(\mathrm{pt},X) \hookrightarrow D(U,X)
\]
sending $x_0 \in X$ to the associated constant map $\mathrm{const}_{x_0} \colon U \rightarrow X$. We now wish to show that $X$ is a smooth deformation retract of $D(U,X)$, identifying $X$ with its image in $D(U,X)$. Indeed, for every $U$ define a smooth map by
\begin{equation*}
    \begin{array}{rcl}
        H \colon D(U,X) \times \mathbb{R} & \rightarrow & D(U,X)  \\
         (f,t) & \mapsto & U \xrightarrow{(r_U)_t} U \xrightarrow{f} X
    \end{array}
\end{equation*}
which is readily checked to satisfy:
\begin{align*}
    &H(f,0) = (r_U)_0 \circ f = f \\
    &H(f,1) = (\mathrm{const}_{u_0}) \circ f = \mathrm{const}_{f(u_0)} \in \mathrm{Im}(i) \\
    &H(\mathrm{const}_{x_0},t) = (r_U)_t \circ \mathrm{const}_{x_0} = \mathrm{const}_{x_0} \text{ for all } t \in \mathbb{R}
\end{align*}
It then follows that the induced maps
\[
S_e(X) \rightarrow S_e \left(D(U,X)\right) = \overline{X}(U)
\]
are indeed weak equivalences for all $U \in \mathbf{Cart}$. That is, we have an objectwise weak equivalence
\[
\mathsf{c}(S_e(X)) \rightarrow \overline{X}
\]
implying that the associated homotopy and homology sheaves are isomorphic and therefore constant.
\end{proof}

Corollary \ref{Corollary: Constant Homology sheaves} is a simplified variant of a much more general recent result \cite[Theorem 1.1]{BBP}. More precisely, given a simplicial presheaf $F$ on $\mathbf{Cart}$ denote by $\mathbf{B}F$ the simplicial presheaf
\[
U \mapsto \mathbf{B}F(U):= \delta \left( F( U \times \mathbb{A}^{\bullet})_{\bullet} \right)
\]
where $\delta$ denotes the diagonal simplicial set functor taking a bisimplicial set $K_{\bullet,\bullet}$ to its diagonal $\delta(K)_n = K_{n,n}$. If $F$ is simplicial presheaf satisfying \v{C}ech-descent, the evaluation map
\[
\mathbf{B}F(U) \rightarrow \mathrm{Map} \left( S_e(U),\mathrm{B}F \right)
\]
is a natural weak equivalence of simplicial sets where $\mathrm{B}F$ is a fibrant replacement of the simplicial set $\mathbf{B}F(\mathrm{pt})$. In the case the simplicial presheaf $F$ is given by some diffeological space $X$, the presheaves $\mathbf{B}X = \overline{X}$ agree. \\

The skeletal diffeology can now be used to \textit{refine} the simplicial presheaf $\overline{X}$ of diffeological space $X$.

\begin{definition} \label{Definition: The Skeletal Simplicial presheaf M_k}
    Given a diffeological space $X$ and an integer $0 \leq k < \infty$ the simplicial presheaf denoted by $\mathbb{X}_k$ is given by
    \[
    \begin{array}{rcl}
         \mathbf{Cart}^{\mathrm{op}} & \rightarrow & \mathbf{sSet}  \\
         U & \mapsto & S_e \left( D(U,X)_k \right).
    \end{array}
    \]
    In the case the diffeological space is given by a smooth manifold $M$ then the simplicial presheaf is denoted $\mathbb{M}_k$.
\end{definition}

\begin{proposition} \label{Proposition: Homology Sheaf M_k constant}
    Let $M$ be a smooth manifold and $ 0 < k \leq \mathrm{dim} \, M$. Then for $j < k$ the $j$th homology sheaf $\tilde{H}_j(\mathbb{M}_k)$ of the simplicial presheaf $\mathbb{M}_k$ is constant. Specifically, it is isomorphic to the sheafification of the constant presheaf
    \[
    U \mapsto H_j(M)
    \]
    where $H_j(M)$ is the usual singular homology of the manifold $M$.
\end{proposition}
\begin{remark}
    As introduced in Definition \ref{Definition: The Skeletal Simplicial presheaf M_k} above, the simplicial presheaf $\mathbb{M}_k$ takes the values
    \[
        \mathbb{M}_k(U) := S_e(D(U,M)_k)
        \]
        for $U \in \mathbf{Cart}$.
\end{remark}
\begin{proof}
    Denote by $H_j(\mathbb{M}_k)$ the homology presheaf associated with the simplicial presheaf $\mathbb{M}_k$. That is the presheaf of abelian groups given by
    \[
    U \mapsto H_j \left( S_e \left( D(U,M)_k \right) \right)
    \]
    By Lemma \ref{Lemma: Smooth Singular Homology X_d} it now follows that since $j < k$ for all $U$
    \[
    H_j \left( S_e \left( D(U,M)_k \right) \right) =  H_j \left( S_e \left( D(U,M) \right) \right)
    \]
    In other words, the $j$th homology presheaf of $\mathbb{M}_k$ agrees with the $j$th homology presheaf of $\overline{M}$ whenever $j < k$. It follows now from Corollary \ref{Corollary: Constant Homology sheaves} that the homology presheaf associated with $\overline{M}$ is constant for all $j < k$.
    \end{proof}

\subsection{The Derived Mapping Space}

The fact that for $j < k$ the homology sheaves $\tilde{H}_j(\mathbb{M}_k)$ are constant and agree with the usual homology of $M$ allows us to compute the cohomology $H^k_{\infty}(\mathbb{M}_k,A)$ via the universal coefficient spectral sequence of Proposition \ref{Proposition: Universal coefficient in stack cohomology}.

\begin{proposition} \label{Proposition: Universal Coefficients Theorem for M_k}
    Let $M$ be a manifold and consider the associated simplicial presheaf $\mathbb{M}_k$ on $\mathbf{Cart}$ for $k \geq 0$. Then for any sheaf of abelian groups $A$ on $\mathbf{Cart}$ there is a short exact sequence
    \[
    0 \rightarrow \mathrm{Ext}^1_{\mathbf{Ab}} \left( H_{k-1}(M), \underline{A} \right) \hookrightarrow H^k_{\infty}\left(\mathbb{M}_k , A \right) \rightarrow \mathrm{Hom}_{\mathrm{Sh}(\mathbf{Cart}, \mathbf{Ab})} \left( \tilde{H}_k(\mathbb{M}_k) , A \right) \rightarrow 0
    \]
    where here $\underline{A}$ denotes the underlying abelian group $A(\mathrm{pt})$.
\end{proposition}
\begin{proof}
Recall from Proposition \ref{Proposition: Universal coefficient in stack cohomology} that the universal coefficients spectral sequence converges to $H^{*}_{\infty}(\mathbb{M}_k, A)$. The $E_2$ page of this spectral sequence reads
\[
E_2^{p,q} =  \mathrm{Ext}^q \left( \tilde{H}_p(\mathbb{M}_k) , A \right).
\]
In the case $k=0$ the statement now follows immediately, with the convention $H_{-1}(M) = 0$. For $k > 0$ we claim that the entries for $p < k, q \geq 0$ are given by
\[
E_2^{p,q} =
\begin{cases}
     \mathrm{Hom}_{\mathbf{Ab}} \left( H_{p}(M) , \underline{A} \right) &\text{ for } p < k,q=0 \\
      \mathrm{Ext}_{\mathbf{Ab}}^1 \left( H_{p}(M) , \underline{A} \right) &\text{ for } p< k,q=1 \\
      0 &\text{ for } p < k,q >1.
\end{cases}
\]
Indeed, for $p < k$ by Proposition \ref{Proposition: Homology Sheaf M_k constant} we have that $\tilde{H}_p(\mathbb{M}_k)$ is the constant sheaf given by $H_p(M)$. It then follows from Corollary \ref{Corollary: ext of constant sheaf} that
\[
\mathrm{Ext}_{\mathrm{Sh}(\mathbf{Ab})}^q \left( \tilde{H}_p(\mathbb{M}_k) , A \right) \cong \mathrm{Ext}_{\mathbf{Ab}}^q \left( H_{p}(M) , \underline{A} \right)
\]
and as such vanish for all $q \geq 2$. That is, the $E_2$ page reads \\

\begin{center}
\begin{tikzpicture}[x=0.75pt,y=0.75pt,yscale=-1,xscale=1, scale=0.95]

\draw (348,341.9) node [anchor=north west][inner sep=0.75pt]    {$\mathrm{Hom}\left(\tilde{H}_{k}(\mathbb{M}_{k}) ,A\right)$};
\draw (500.5,342.9) node [anchor=north west][inner sep=0.75pt]    {$\mathrm{Hom}\left(\tilde{H}_{k+1}(\mathbb{M}_{k}) ,A\right)$};
\draw (509,294.9) node [anchor=north west][inner sep=0.75pt]    {$\mathrm{Ext}^{1}\left(\tilde{H}_{k+1}(\mathbb{M}_{k}) ,A\right)$};
\draw (510,248.9) node [anchor=north west][inner sep=0.75pt]    {$\mathrm{Ext}^{2}\left(\tilde{H}_{k+1}(\mathbb{M}_{k}) ,A\right)$};
\draw (577,202.4) node [anchor=north west][inner sep=0.75pt]    {$\vdots $};
\draw (202,349.9) node [anchor=north west][inner sep=0.75pt]    {$\mathrm{Hom}( H_{k-1}( M) ,\underline{A})$};
\draw (201.5,305.9) node [anchor=north west][inner sep=0.75pt]    {$\mathrm{Ext}^{1}( H_{k-1}( M) ,\underline{A})$};
\draw (262,257.4) node [anchor=north west][inner sep=0.75pt]    {$0$};
\draw (258,202.4) node [anchor=north west][inner sep=0.75pt]    {$\vdots $};
\draw (56,349.4) node [anchor=north west][inner sep=0.75pt]    {$\mathrm{Hom}( H_{k-2}( M) ,\underline{A})$};
\draw (62.5,306.9) node [anchor=north west][inner sep=0.75pt]    {$\mathrm{Ext}^{1}( H_{k-2}( M) ,\underline{A})$};
\draw (130,257.4) node [anchor=north west][inner sep=0.75pt]    {$0$};
\draw (126,202.4) node [anchor=north west][inner sep=0.75pt]    {$\vdots $};
\draw (351,295.9) node [anchor=north west][inner sep=0.75pt]    {$\mathrm{Ext}^{1}\left(\tilde{H}_{k}(\mathbb{M}_{k}) ,A\right)$};
\draw (348,246.9) node [anchor=north west][inner sep=0.75pt]    {$\mathrm{Ext}^{2}\left(\tilde{H}_{k}(\mathbb{M}_{k}) ,A\right)$};
\draw (16.5,349.4) node [anchor=north west][inner sep=0.75pt]    {$\cdots $};
\draw (17.5,308.9) node [anchor=north west][inner sep=0.75pt]    {$\cdots $};
\draw (401,202.4) node [anchor=north west][inner sep=0.75pt]    {$\vdots $};

\end{tikzpicture}
\end{center}
Since the differential $d_r$ on the $E_r$ page is of degree $(1-r,r)$ it follows that the diagonal $p+q = k$ stabilizes after the $E_2$ page. That is,
\[
E_{\infty}^{p,q} =
\begin{cases}
     \mathrm{Hom} \left( \tilde{H}_k(\mathbb{M}_k) , A \right) &\text{ for } p=k,q=0 \\
      \mathrm{Ext}_{\mathbf{Ab}}^1 \left( H_{k-1}(M) , \underline{A} \right) &\text{ for } p=k-1,q=1 \\
      0 &\text{ for } p = k -q, q >1
\end{cases}
\]
from which we conclude that the filtration of Remark \ref{Remark: The aboutment Remark}
\[
H^k_{\infty}(\mathbb{M}_k,A) = J_{k,0} \supset J_{k-1,1} \supset \cdots \supset J_{0,k} \supset J_{-1,k+1} = 0,
\]
where $J_{p,q}/J_{p-1,q+1} \cong E_{\infty}^{p,q}$ simplifies to
\[
H^k_{\infty}(\mathbb{M}_k,A) = J_{k,0} \supset J_{k-1,1} \supset  J_{k-2,2} = \cdots = J_{0,k} = J_{-1,k+1} = 0.
\]
We conclude that $J_{k-1,1} =\mathrm{Ext}_{\mathbf{Ab}}^1 \left( H_{k-1}(M) , \underline{A} \right)$ and that there is a short exact sequence
\[
    0 \rightarrow \mathrm{Ext}^1_{\mathbf{Ab}} \left( H_{k-1}(M), \underline{A} \right) \hookrightarrow H^k_{\infty}\left(\mathbb{M}_k , A \right) \rightarrow \mathrm{Hom}_{\mathrm{Sh}(\mathbf{Cart}, \mathbf{Ab})} \left( \tilde{H}_k(\mathbb{M}_k) , A \right) \rightarrow 0
\]
which finishes the proof.
\end{proof}

A direct consequence in the case where $A = U(1)$ is given by the circle group is the following.

\begin{corollary} \label{Corollary: Derived Hom space in the case U(1)}
Let $M$ be a manifold and consider the associated simplicial presheaf $\mathbb{M}_k$ on $\mathbf{Cart}$ for $k \geq 0$. Then there is an isomorphism of abelian groups
\[
H^k_{\infty} \left( \mathbb{M}_k , U(1) \right) \cong \mathrm{Hom}_{\mathrm{Sh}(\mathbf{Cart}, \mathbf{Ab})} \left( \tilde{H}_k(\mathbb{M}_k) , U(1) \right).
\]
\end{corollary}
\begin{proof}
    This follows from Proposition \ref{Proposition: Universal Coefficients Theorem for M_k} together with the fact that $U(1)$ as an abelian group is divisible.
\end{proof}

As a sanity check for Theorem \ref{Theorem: The Main Theorem} let us look at the easiest case, that is $k = 0$. From the previous corollary, it follows that
\[
H^0_{\infty} \left( \mathbb{M}_0 , U(1) \right) \cong \mathrm{Hom}_{\mathrm{Sh}(\mathbf{Cart}, \mathbf{Ab})} \left( \tilde{H}_0(\mathbb{M}_0) , U(1) \right)
\]
By definition the simplicial presheaf $\mathbb{M}_0$ assigns to a Cartesian space $U$ the simplicial set
\[
\mathbb{M}_0(U) = S_e(D(U,M)_0),
\]
where here $D(U,M)_0$ is simply the set $C^{\infty}(U,M)$ endowed with the discrete diffeology. As such, it follows that
\[
H_0\left(\mathbb{M}_0 \right)(U) = \bigoplus_{p \in C^{\infty}(U,M)} \mathbb{Z}.
\]
This then implies that
\begin{align*}
    \mathrm{Hom}_{\mathrm{Sh}(\mathbf{Cart}, \mathbf{Ab})} \left( \tilde{H}_0(\mathbb{M}_0) , U(1) \right) &\cong \mathrm{Hom}_{\mathrm{PSh}(\mathbf{Cart}, \mathbf{Ab})} \left( H_0(\mathbb{M}_0) , U(1) \right) \\
    &\cong C^{\infty}\left( M,U(1) \right),
\end{align*}
where the group structure in the set of smooth maps into $U(1)$ is pointwise. This is precisely a degree $1$-differential character. That is, we have shown the following.

\begin{proposition} \label{Proposition: The case k = 0}
Given $M$ a smooth manifold. Then there is an isomorphism of abelian groups
    \[
    H^0_{\infty}(\mathbb{M}_0,U(1)) \cong \hat{H}^{1}(M; \mathbb{Z} ) = C^{\infty}(M,U(1)).
    \]
\end{proposition}

\section{Extracting the curvature} \label{Section: The Direct Approach}

Having reduced the cohomology $H_{\infty}^k \left( \mathbb{M}_k,U(1) \right)$ via the universal coefficients spectral sequence to
\[
\mathrm{Hom}_{\mathrm{Sh}(\mathbf{Cart},\mathbf{Ab})} \left( \tilde{H}_k(\mathbb{M}_k) ,U(1) \right)
\]
we now wish to show that the morphisms $h \colon \tilde{H}_k(\mathbb{M}_k) \rightarrow U(1)$ represent the holonomy of circle $k$-bundles with connection. Accordingly, we will call a morphism $h$ a \textit{smooth holonomy morphism} throughout this chapter.  Using the Cheeger--Simons differential characters as our preferred model for circle $k$-bundles with connection amounts to constructing an isomorphism of abelian groups
\[
\Psi \colon \hat{H}^{k+1}(M;\mathbb{Z}) \xrightarrow{\cong} \mathrm{Hom}_{\mathrm{Sh}(\mathbf{Cart},\mathbf{Ab})} \left( \tilde{H}_k(\mathbb{M}_k) ,U(1) \right)
\]
for any smooth manifold $M$ without boundary, which is deemed fixed for the remainder of this chapter. The inverse $\Phi$ of this isomorphism is then constructed by extracting a curvature $(k+1)$-form from any holonomy morphism $h \colon \tilde{H}_k(\mathbb{M}_k) \rightarrow U(1)$. \\

Let us fix a differential character of degree $k+1$ given by the pair $(\chi,\omega_{\chi})$ where $\chi \colon Z_k(M) \rightarrow U(1)$ is a morphism of abelian groups and $\omega_{\chi} \in \Omega^{k+1}(M)$ the associated curvature form.
\begin{remark} \label{Remark: chains with the subset diffeology}
Here $Z_k(M)$ denotes the abelian group of $k$-cycles in $M$ considered as a subgroup of $C_k(M)$, the free abelian group of $k$-chains in $M$. That is, $C_k(M)$ is the free abelian group
\[
\mathbb{Z} \left(S(M)_k\right)
\]
where $S(M)$ denotes the smooth singular complex as opposed to the extended smooth singular complex $S_e(M)$ used in the previous section. That is, a chain is a formal sum of smooth maps $\sigma \colon |\Delta^k| \rightarrow M$ where the compact $k$-simplex $|\Delta^k| \subset \mathbb{R}^k$ is considered as a diffeological space endowed with the subset diffeology. In Definition \ref{Definition: Differential Character}, a differential character is however based on smooth cycles considering $|\Delta^k| \subset \mathbb{R}^k$ as a closed subset and not as a diffeological space. This technicality is resolved by Lemma \ref{Lemma: Subset diffeology of compact simplex}, that is for $M$ a smooth manifold without boundary any smooth map $\sigma \colon |\Delta^k| \rightarrow M$ is smooth when considered as a map from a closed subset of $\mathbb{R}^k$ into $M$. Indeed, choose an embedding $M \hookrightarrow \mathbb{R}^m$ and consider a tubular neighborhood $(W,q)$ of $M$ in $\mathbb{R}^m$. The map $\sigma$ now extends to give a smooth map
\[
\sigma \colon |\Delta^k| \rightarrow \mathbb{R}^m.
\]
By Lemma \ref{Lemma: Subset diffeology of compact simplex} there exists an open neighborhood $U'$ of $|\Delta^k|$ in $\mathbb{R}^k$ together with a smooth extension
\[
\sigma' \colon U \rightarrow \mathbb{R}^m.
\]
Let $U := \sigma'^{-1}(W)$ and define the extension $\tilde{\sigma} := q \circ \sigma'|_U$ where $q \colon W \rightarrow M$ is the smooth contraction of the tubular neighborhood $W$ of $M$. This shows that $\sigma$ is smooth when regarding $|\Delta^k|$ as a closed subset of $\mathbb{R}^k$. \\

To distinguish the smooth singular homology of a diffeological space $X$ using the compact simplices $|\Delta^{\bullet}|$ from the singular homology introduced in the previous section using the extended simplices $\mathbb{A}^{\bullet}$, write
\[
H^{\Delta}_{*}(X) := H_* \left( S(X) \right) = H_* \left( N(\mathbb{Z}[S(X)]) \right).
\]
\end{remark}

\begin{lemma} \label{Lemma: Compact k-skeletal k-chains}
    Let $M$ be a manifold without boundary and $ 0 < k$. Then for any Cartesian space $U \in \mathbf{Cart}$ the $k$-chains taking values in $D(U,M)_k$ agree with the usual $k$-chains in $D(U,M)$,
    \[
    C_k(D(U,M)_k) = C_k(D(U,M)).
    \]
\end{lemma}
\begin{proof}
    First it is clear that $Z_k(D(U,M)_k) \subset Z_k(D(U,M))$. Let therefore $\sigma \colon |\Delta^k| \rightarrow D(U,M)$ be a smooth $k$-chain. That is, $\sigma \colon |\Delta^k| \times U \rightarrow M$ is a smooth map of diffeological spaces. Note that Remark \ref{Remark: chains with the subset diffeology} implies that $\sigma$ is also smooth when considering $|\Delta^k| \times U \subset \mathbb{R}^k \times U$ as a closed subset. We now wish to show that $\sigma$ extends to a smooth map $\tilde{\sigma} \colon \mathbb{R}^k \rightarrow D(U,M)$. Using the extension Lemma for smooth maps \cite[Corollary 6.27]{Lee} it suffices to find a continuous extension to $\mathbb{R}^k \times U$. Since $|\Delta^k| \times U \subset \mathbb{R}^k \times U$ is a closed convex subset of some Cartesian space it is a strong deformation retract. That is, there is a continuous map $r \colon \mathbb{R}^k \times U \rightarrow |\Delta^k| \times U$ and we define the continuous extension by $\tilde{\sigma} := \sigma \circ r$. \\

    Assume therefore that $\tilde{\sigma} \colon \mathbb{R}^k \times U \rightarrow M$ is such a smooth extension of $\sigma$. Then it follows that $\sigma \colon |\Delta^k| \hookrightarrow \mathbb{R}^k \xrightarrow{\tilde{\sigma}} D(U,M)$ factors through a plot of dimension $\leq k$ and therefore indeed defines a smooth map $\sigma \colon |\Delta^k| \rightarrow D(U,M)_k$.
\end{proof}
\begin{remark}
    Note that this lemma can not be generalized to arbitrary diffeological spaces $X$. As an example consider $|\Delta^1| = [0,1] \subset \mathbb{R}$ endowed with the subset diffeology. Like the half-line $[0,\infty)$ also the interval with the subset diffeology has diffeological dimension $\mathrm{dim}([0,1]) = \infty$\footnote{The argument is essentially the same as for the half-line. For the latter see \cite{Iglesias-Zemmour-Dim}.}. Therefore, assuming $C^{\infty}(|\Delta^1|,|\Delta^1|_1) = C^{\infty}(|\Delta^1|,|\Delta^1|)$ would imply that $|\Delta^1|_1 = |\Delta^1|$. This however is a contradiction since $X = X_k \Leftrightarrow \mathrm{dim}\, X \leq k$.
\end{remark}

Given any Cartesian space $U$ for any point $u \in U$ consider the inclusion map $\mathrm{pt} \xrightarrow{u} U$ sending the one-point space to $u$. This map induces a morphism of chain complexes
\[
(-)|_{ \{ u \} } \colon C_i(D(U,M)_k) \rightarrow C_i(M_k)
\]
which then induces morphisms of homology groups
\[
\begin{array}{rcl}
(-)|_{ \{ u \} } \colon H^{\Delta}_k(D(U,M)_k) & \rightarrow & H^{\Delta}_k(M_k) \\
   {[} \sigma \colon |\Delta^{k}| \times U \rightarrow M  {]_k} & \mapsto & {[} \sigma|_{\{ u \}} \colon |\Delta^k| \rightarrow M {]_k}
\end{array}
\]
Note that to indicate that homology classes are taken with respect to the $k$-skeletal diffeology we write $[\sigma]_k$. By the nature of differential characters, i.e. the fact that they are thin-invariant,  we have that $\chi$ factors through $H_k(M_k)$ as a morphism of abelian groups. That is, there is a map $\chi'$ such that the diagram is commutative.
\begin{equation}\label{Equation: Differential Character is thin invariant}
\begin{tikzcd}
Z_k(M_k) \overset{\ref{Lemma: Compact k-skeletal k-chains}}{=} Z_k(M) \arrow[rd, "\chi"] \arrow[d, two heads] &      \\
H^{\Delta}_k(M_k) \arrow[r, "\chi'"']                            & U(1)
\end{tikzcd}
\end{equation}
Indeed, recall that a differential character $\chi$ evaluated on a boundary of a $(k+1)$-cycle $\partial c$
\[
\chi(\partial c) = \mathrm{exp} \left(2\pi i \int_c \omega \right)
\]
for some $(k+1)$-form $\omega \in \Omega^{k+1}(M)$. If, however, we consider a smooth chain $c \in C_{k+1}(M_k)$ such as $c \colon |\Delta^{k+1}| \rightarrow M_k$ then the pullback $c^*\omega$ is a differential form on $|\Delta^{k+1}|$ considered as a diffeological space with the subset diffeology. By \cite{Gurer-Iglesias-Zemmour} it follows that this is equivalent to the standard notion of differential forms on $|\Delta^{k+1}|$ considered as closed subsets. That is, they extend to a differential form on a neighborhood in $\mathbb{R}^{k+1}$. Notice, that for every plot $p \colon U \rightarrow |\Delta^{k+1}|$ the pullback $p^*(c^*\omega) = (c \circ p)^* \omega = 0$ vanishes since $c \circ p \colon U \rightarrow M_k$ locally factors through $\mathbb{R}^k$. By the definition of $c^*\omega$ as a differential form on $|\Delta^{k+1}|$ considered as a diffeological space this shows $c^*\omega = 0$. Therefore, it follows that
\[
\chi(\partial c) = \mathrm{exp} \left(2\pi i \int_{|\Delta^{k+1}|} c^*\omega \right) = 1.
\]
Accordingly, for every $[\sigma]_k \in H^{\Delta}_k(D(U,M)_k)$ there is an assignment:
\begin{equation}\label{Equation: Smooth Assignemt DiffChar }
\begin{array}{rcl}
   (h_{\chi})_U({[} \sigma {]_k}) \colon U  & \rightarrow & U(1)  \\
     u & \mapsto & \chi'({[} \sigma|_{\{u\}} {]_k})
\end{array}
\end{equation}

\begin{proposition} \label{Proposition: Holonomy morphism associated to Differential Character}
    Let $(\chi,\omega_{\chi})$ be a differential character of degree $(k+1)$ on $M$. Then the assignment $(h_{\chi})_U({[} \sigma {]})$ defined via equation (\ref{Equation: Smooth Assignemt DiffChar }) is smooth and defines a morphism $h_{\chi}$ of abelian presheaves on $\mathbf{Cart}$. That is, for every $U \in \mathbf{Cart}$
    \[
    \begin{array}{rcl}
    (h_{\chi})_U \colon H_k(\mathbb{M}_k)(U) \xrightarrow{\cong} H^{\Delta}_k(D(U,M)_k) & \rightarrow & C^{\infty}(U,U(1)) \\
    {[} \sigma {]_k} & \mapsto &  (h_{\chi})_U({[} \sigma {]_k}).
    \end{array}
    \]
\end{proposition}
\begin{proof}
The natural isomorphisms $H_k(\mathbb{M}_k)(U) \cong H^{\Delta}_k(D(U,M)_k)$ are induced by the weak equivalences of simplicial sets $S_e(D(U,M)_k) \rightarrow S(D(U,M)_k)$ of Theorem \ref{Theorem: The Fundamental Theorem of Diffeological Homotopy}.
    It is left to show that given a differential character $(\chi, \omega_{\chi})$ and a fixed thin homology class $[\sigma] \in H^{\Delta}_k(D(U,M)_k)$ the assignment
    \[
    u \in U \longmapsto \chi'\left( {[} \sigma|_{\{u\}} {]} \right)
    \]
    smooth. Choose a representative of the thin class $\sigma \colon U \times |\Delta^k| \rightarrow M$ and notice that it follows from equation (\ref{Equation: Differential Character is thin invariant}) that for all $u \in U$
    \[
   \chi'\left( {[} \sigma|_{\{u\}} {]} \right) = \chi \left(  \sigma|_{\{u\}} \right)
    \]
    Let now $u_0 \in U$ be any point, then the compact image of the smooth $k$-cycle $\sigma|_{\{u_0\}} \colon |\Delta^k| \rightarrow M$ in $M$ admits a neighborhood $W_0$ whose integral cohomology vanishes above $k$. Such a neighborhood is called a $k$\textbf{-good neighborhood}. The existence of such a $k$-good neighborhood of the compact image of  $\sigma|_{\{u_0\}}$ follows from Fact 2.1 of Simons--Sullivan \cite{Simons-Sullivan}. In addition the preimage $\sigma^{-1}(W_0)$ is an open neighborhood of $\{ u_0 \} \times |\Delta^{k}|$ inside $U \times |\Delta^{k}|$ hence we can assume by the compactness of $|\Delta^{k}|$, that it contains an open of the form $V \times |\Delta^{k}|$ where $V$ is a suitable neighborhood of $u_0$ in $U$. Denote by $i \colon W_0 \hookrightarrow M$ the inclusion of the $k$-good neighborhood in $M$. Pulling back the differential character $\chi$ along the inclusion $i$ we obtain $i^*\chi \in \hat{H}^{k+1}(W_0, \mathbb{Z})$. Since the neighborhood $W_0$ is $k$-good, i.e. $H^{k+1}(W_0, \mathbb{Z}) = 0$, it follows from the differential cohomology hexagon that $i^*\chi = j([\theta])$ for some class $[\theta] \in \Omega^k(W_0)/\Omega^k_{\mathrm{cl}}(W_0)_{\mathbb{Z}}$. That is, for all $\tau \in Z^{\Delta}_k(W_0)$
    \[
    i^*\chi(\tau) = \mathrm{exp} \left( 2\pi i \cdot \int_\tau  \theta \right)
    \]
    where $\theta \in \Omega^{k}(W_0)$ some $k$-form determined up to an element of $\Omega^k_{\mathrm{cl}}(W_0)_{\mathbb{Z}}$. \\

    This now shows that for any $u \in V$ we have that
    \[
    \chi\left( \sigma|_{\{ u \}}\right) = i^*\chi\left( \sigma|_{\{ u \}} \right) = \mathrm{exp} \left( 2\pi i \cdot \int_{\sigma|_{\{ u \}}} \theta \right)
    \]
    which shows that locally around the point $u_0$ the assignment
    \[
    u \in V \longmapsto \chi\left( \sigma|_{\{ u \}} \right) = \chi'\left( {[} \sigma|_{\{u\}} {]} \right) \in U(1)
    \]
    is indeed smooth. Since this holds for every point $u_0 \in U$, the statement follows.
\end{proof}

An immediate outcome of the preceding discourse is the following corollary.

\begin{corollary} \label{Corollary: The morphism Phi}
    For any smooth manifold $M$ without boundary and $k > 0$ there is a homomorphism of abelian groups
    \[
    \begin{array}{rcl}
    \Phi \colon \hat{H}^{k+1}(M;\mathbb{Z}) & \rightarrow & \mathrm{Hom}_{\mathrm{PSh}(\mathbf{Cart},\mathbf{Ab})} \left( \tilde{H}_k(\mathbb{M}_k) ,U(1) \right) \\
        (\chi,\omega) & \mapsto & h_{\chi}.
    \end{array}
    \]
\end{corollary}

\subsection{The Van Est homomorphism for Lie Groupoids} \label{Section: Van Est Map}
Before proceeding further with the extraction of the curvature form $A_h$ from a smooth holonomy morphism $h$, we take a short detour and introduce the Van Est map for Lie groupoids\footnote{For an introduction to Lie groupoids and Lie algebroids see \cite{Moerdijk-Mrcun}.}. The cause is that the Van Est homomorphism turns out to be an important tool to prove a key result in the next section, specifically Lemma \ref{Lemma: Curvature form gives differential character}.

\begin{definition}
    Let $X$ be a manifold. Denote by $\mathrm{Pair}(X)$ the \textbf{pair Lie groupoid of} $X$ given by $ X \times X \rightrightarrows X$ with source map $s(x',x) = x$ and target map $t(x',x) = x'$. The units are given by the inclusion of the diagonal $X \hookrightarrow X \times X$ and the multiplication is simply given by $(x_1',x_1)(x_2',x_2) = (x_1',x_2)$ whenever $x_1 = x_2'$.
\end{definition}

\begin{definition}[\cite{Weinstein-Xu} Definition 1.2] \label{Definition: Smooth normalized Cochains}
    Let $(\Gamma \rightrightarrows X,s,t)$ be a Lie groupoid. Denote by $\Gamma^{(n)}$ the iterated fiber product
    \[
    \Gamma^{(n)} = \Gamma \tensor[_{s}]{\times}{_{t}} \cdots \tensor[_{s}]{\times}{_{t}} \Gamma
    \]
    of $n$-tuples of composable morphisms in $\Gamma$. Let $C^n(\Gamma,\mathbb{R})_X$ be the space of all maps $\sigma \colon \Gamma^{(n)} \rightarrow \mathbb{R}$ such that $\sigma$ is smooth in a neighborhood of the diagonal \newline $X^{(n)} = \left\{ ( i(x), ..., i(x) ) \, | x \in X \right\} \subset \Gamma^{(n)}$ and such that $\sigma(g_1, ..., g_n) = 0$ if some $g_j = i(x)$ is a unit for some $x \in X$. The elements of the space $C^n(\Gamma, \mathbb{R})_X$ are called \textbf{normalized} $n$\textbf{-cochains on the Lie groupoid} $\Gamma \rightrightarrows X$. There exists a differential
    \[
    \delta \colon C^n(\Gamma, \mathbb{R})_X \rightarrow C^{n+1}(\Gamma, \mathbb{R})_X
    \]
    turning the pair $(C^{\bullet}(\Gamma, \mathbb{R})_X, \delta)$ into a cochain complex. It is given by
    \begin{align*}
        \sigma \in C^n(\Gamma, \mathbb{R})_X  \mapsto \delta(\sigma)(g_0, ..., g_n) := & \sigma(g_1, ..., g_n) + \sum_{i = 1}^{n} (-1)^i \sigma(g_0, ..., g_{i-1}g_i, ..., g_n) \\
        &+ (-1)^{n+1} \sigma(g_0, ..., g_{n-1}).
    \end{align*}
\end{definition}
\begin{example}
    In the case of the pair groupoid $\mathrm{Pair}(X)$ for some manifold $X$, the associated complex of normalized $n$-cochains $C^n(\mathrm{Pair}(X), \mathbb{R})_X$ is given by maps of the form $\sigma \colon X^{n+1} \rightarrow \mathbb{R}$ which are smooth in a neighborhood around the diagonal $X \hookrightarrow X^{n+1}$ and are such that $\sigma(x_0, ..., x_{n}) = 0$ whenever $x_i = x_{i-1}$ for some $1 \leq i \leq n$. This complex is also known as the Alexander--Spanier complex\footnote{To be precise, this complex is the \textit{normalized subcomplex} of the Alexander--Spanier complex.} of $X$, see \cite[Example 1]{Li-Bland-Meinrenken}.
\end{example}

Given a Lie groupoid $(\Gamma \rightrightarrows X,s,t)$ we wish to define an action of the symmetric group $S_n$ on the manifold $\Gamma^{(n)}$. However, we can not simply act by permutation of the components, since permutations do not respect the source-target conditions imposed on the tuples. To resolve this issue, we introduce a slight variation of the iterated fiber product. \\

To any Lie groupoid we associate the smooth manifold $\Gamma_s^{(n)}$ given by the iterated fiber product
\[
    \Gamma_s^{(n)} := \underbrace{\Gamma \tensor[_{s}]{\times}{_{s}} \cdots \tensor[_{s}]{\times}{_{s}} \Gamma}_{n\text{-times}}
\]
only with respect to the source map. Notice that $\Gamma^{(n)}$ and $\Gamma_s^{(n)}$ are diffeomorphic via
\[
\begin{array}{rcl}
    \varphi \colon \Gamma^{(n)} & \rightarrow & \Gamma_s^{(n)} \\
    (g_1,g_2,...,g_n) & \mapsto & (g_1,g_1g_2, ..., g_1g_2 \cdots g_n).
\end{array}
\]
The manifold $\Gamma_s^{(n)}$ now carries an $S_n$-action simply given by the permutation of components. That is, for $(g_1, ..., g_n) \in \Gamma_s^{(n)}$ and $s \in S_n$ we define
\[
    s \cdot (g_1, ..., g_n) := (g_{s(1)}, ..., g_{s(n)}).
    \]
This action can now be used together with the diffeomorphism $\varphi$ to define an action on the manifold $\Gamma^{(n)}$ simply by
\[
s \cdot (g_1, ..., g_n) := \varphi^{-1}(s \cdot \varphi(g_1, ..., g_n)).
\]
This in turn defines an $S_n$-action on $n$-cochains such that for any $\sigma \in C^n(\Gamma,\mathbb{R})_X$ and $s \in S_n$
\begin{equation} \label{Equation: S_n action on n-cochains}
    (s \cdot \sigma)(g_1, ...,g_n) := \sigma\left( s \cdot (g_{1}, ..., g_{n}) \right).
\end{equation}

\begin{remark}
    In the case we are given the pair Lie groupoid $\Gamma:= \mathrm{Pair}(X)$ the $S_n$-action on
    \[
        \Gamma^{(n)} \cong X^{n+1}
    \]
     is simply given by permutations of the elements in $X^{n+1}$ fixing the first component.
\end{remark}

\begin{definition}\label{Definition: S_n antisymmetric cochain}
 An $n$-cochain $\sigma \in C^n(\Gamma, \mathbb{R})_X$ is said to be $S_n$\textbf{-antisymmetric} if for all $\sigma \in S_n$ the identity $s \cdot \sigma = \mathrm{sgn}(s) \sigma$ holds.
\end{definition}

A similar complex can be associated with a Lie algebroid called the Chevalley--Eilenberg complex introduced in Definition \ref{Definition: Chevalley-Eilenberg complex}. First, however, we briefly recall the notion of a Lie algebroid.

\begin{definition}
    A \textbf{Lie algebroid} is given by a triple $(\mathfrak{g}, [-,-],\alpha)$ consisting of the following data.
    \begin{enumerate}[label={(\arabic*)}]
        \item A vector bundle  $\mathfrak{g} \rightarrow X$ over some manifold $X$,
        \item a vector bundle map $\alpha \colon \mathfrak{g} \rightarrow TX$ called the \textbf{anchor},
        \item and a Lie bracket on the space of sections $\Gamma(\mathfrak{g})$,
    \end{enumerate}
   such that the anchor $\alpha$ induces a Lie algebra homomorphism from the Lie algebra of section $\Gamma(\mathfrak{g})$ to the Lie algebra of vector fields of $X$. Moreover for all $f \in C^{\infty}(X)$ we have that
   \[
   [Y,fZ] = f[Y,Z] + \alpha(Y)(f)Z
   \]
   for all sections $Y,Z \in \Gamma(\mathfrak{g})$.
\end{definition}

\begin{example} \label{Example: Tangent Lie algebroid}
    The most important example of a Lie algebroid for our purpose is given by the tangent bundle $TX \rightarrow X$ with anchor map the identity and Lie bracket given by the usual bracket of vector fields on $X$.
\end{example}

Like to any Lie group $G$ we associate the Lie algebra $\mathrm{Lie}(G)$ given by all smooth left-invariant vector fields on $G$ we can associate to any Lie groupoid $(\Gamma \rightrightarrows X)$ a Lie algebroid $(\mathfrak{g}, [-,-],\alpha)$ known as the \textbf{associated Lie algebroid}. A detailed construction of this associated Lie algebroid is presented in \cite[Section 6.1]{Moerdijk-Mrcun}.

\begin{definition} \label{Definition: Chevalley-Eilenberg complex}
    Given a Lie algebroid $(\mathfrak{g}, [-,-],\alpha)$ the \textbf{Chevalley--Eilenberg complex} of $\mathfrak{g}$ is given by the graded differential algebra
    \[
    C^{\bullet}(\mathfrak{g}, \mathbb{R}) := \Gamma \left( \Lambda^{\bullet}\mathfrak{g}^* \right)
    \]
    where $\mathfrak{g}^*$ denotes the associated dual vector bundle over the base manifold $X$, with differential
    \[
    d_{\mathrm{CE}} \colon  C^{\bullet}(\mathfrak{g}, \mathbb{R}) \rightarrow  C^{\bullet +1}(\mathfrak{g}, \mathbb{R}).
    \]
    It is given by
        \begin{align*}
        \phi \in C^{n}(\mathfrak{g}, \mathbb{R}) \mapsto d_{\mathrm{CE}}(\phi)(X_0,...,X_n)  := & \sum_{i = 0}^{n} (-1)^i \alpha(X_i) \phi(X_0, ..., \hat{X}_i, ...,X_n) \\
        &+ \sum_{i < j} (-1)^{i+j} \phi ([X_i,X_j],X_0, ..., \hat{X}_i, ..., \hat{X}_j, ..., X_n).
    \end{align*}
\end{definition}

The Van Est homomorphism then provides a morphism of cochain complexes from the Lie groupoid cochain complex to the Chevalley--Eilenberg complex of its associated Lie algebroid $\mathfrak{g} = \mathrm{Lie}(\Gamma \rightrightarrows X)$.
\[
VE \colon C^{\bullet}(\Gamma, \mathbb{R})_X \rightarrow C^{\bullet}(\mathfrak{g}, \mathbb{R})
\]
Specifically, for the pair Lie groupoid $\mathrm{Pair}(X)$ of some smooth manifold $X$, the Van Est homomorphism defines a map
\[
VE \colon C^{\bullet}(\mathrm{Pair}(X), \mathbb{R})_X \rightarrow \Omega^{\bullet}(X)
\]
to the cochain complex of differential forms on $X$. This follows from the fact that the associated Lie algebroid $\mathrm{Lie} \left( \mathrm{Pair}(X) \rightrightarrows X \right)$ to the pair Lie groupoid is the Lie algebroid of Example \ref{Example: Tangent Lie algebroid}, see \cite[Example 6.3]{Moerdijk-Mrcun}. The associated Chevalley--Eilenberg complex is then simply given by the de Rham complex $(\Omega^{\bullet}(X), d)$ on $X$.

\begin{definition}[\cite{Weinstein-Xu}]
    Let $(\Gamma \rightrightarrows X)$ be a Lie groupoid and $(\mathfrak{g}, [-,-],\alpha)$ its associated Lie algebroid. Define the \textbf{Van Est homomorphism} to be the map
    \[
    \begin{array}{rcl}
       VE \colon C^n(\Gamma, \mathbb{R})_X & \rightarrow & C^n(\mathfrak{g}, \mathbb{R}) \\
    \sigma & \mapsto & VE(\sigma)
    \end{array}
    \]
    where for any $X_1, ..., X_n \in \Gamma(\mathfrak{g})$ and all $x \in X$:
    \[
    VE(\sigma)(X_1, ..., X_n)(x) = \sum_{s \in S_n} (-1)^{\mathrm{sgn}(s)} (X_{s(1)} ...,X_{s(n)} \sigma)(x)
    \]
    \end{definition}
    \begin{remark} \label{Remark: Formula of the Van Est morphism}
        The notation $(X_1 ... X_n \sigma)(x)$ denotes the function on $X$ given as follows. Recall that $\sigma \colon \Gamma^{(n)} \rightarrow \mathbb{R}$ is a map that is smooth in a neighborhood around the diagonal. By fixing the first $(n-1)$-variables $(g_1, ..., g_{n-1}) \in \Gamma^{(n)}$ we can consider $\sigma$ as a function of the variable $g_n$ alone defined on the fiber over $(g_1, ..., g_{n-1})$. Now apply $X_n$ to this function and evaluate it at $g_n = s(g_{n-1})$. This defines a function on $\Gamma^{(n-1)}$. By repeating this process $n$-times one obtains a function on $X$ which is denoted by $(X_1 ... X_n \sigma)(x)$.
    \end{remark}

Let us have a closer look at the Van Est homomorphism in the case of the pair Lie groupoid $\mathrm{Pair}(X)$. As already mentioned before, the cochain complex is given by maps of the form
\[
\sigma \colon X^{n+1} \rightarrow \mathbb{R}
\]
that are smooth in a neighborhood of the diagonal. The associated Lie algebroid is then the tangent bundle $TX \rightarrow X$ with the standard Lie bracket of vector fields. Identify $\Omega^n(X)$ with the vector space of alternating $C^{\infty}(X)$-multilinear functions
\[
\underbrace{\mathfrak{X}(X) \times \cdots \times \mathfrak{X}(X)}_{n\text{-times}} \rightarrow C^{\infty}(X).
\]
Given an $n$-cochain $\sigma$ and vector fields $X_1, ..., X_n \in \mathfrak{X}(X)$  denote by $(X_1 ... X_n \sigma)$ the real-valued function on $X$ defined as follows. \\

For $n = 1$, we have $\sigma \colon X \times X \rightarrow \mathbb{R}$ smooth in a neighborhood of the diagonal. Given $X_1$ a vector field define
\begin{equation*}
    (X_1\sigma)(x) := X_1(\sigma(x,-))(x).
\end{equation*}
That is, for any fixed $x \in X$ consider $\sigma(x,-)$ as a function of one variable $X \rightarrow \mathbb{R}$ that is smooth in a neighborhood of $x$. Now apply the vector field $X_1$ to this function, which yields a smooth\footnote{By smooth we mean here smooth in a neighborhood of $x \in X$.} function $X_1(\sigma(x,-)) \colon X \rightarrow  \mathbb{R}$ which now we evaluate at $x$. \\

For $n =2$ consider now two vector fields $X_1,X_2$ on $X$ and let $(x_0,x_1) \in X \times X$ be fixed. Then
\[
X_2(\sigma(x_0,x_1, -)) \colon X \rightarrow \mathbb{R}
\]
which now by evaluating at $x_1$ gives the map
\[
\begin{array}{rcl}
  (X_2\sigma) \colon   X \times X & \rightarrow & \mathbb{R}  \\
     (x_0,x_1) & \mapsto & X_2(\sigma(x_0,x_1, -))(x_1)
\end{array}
\]
which for the moment we call $(X_2\sigma)$. Now we iterate the process to define the smooth map
\[
\begin{array}{rcl}
     (X_1 X_2 \sigma) \colon X & \rightarrow   & \mathbb{R} \\
    x & \mapsto & X_1\left( (X_2\sigma)(x,-) \right)(x) = \textcolor{blue}{X_1} \left( \textcolor{red}{X_2}(\sigma(x, \textcolor{blue}{-}, \textcolor{red}{-}) ) \right)(x)
\end{array}
\]
where the vector field $\color{red}X_2$ acts on the last variable of $\sigma(x, \textcolor{blue}{-}, \textcolor{red}{-})$ and $\color{blue}X_1$ on the first. Similarly one has
\[
(X_1...X_n \sigma)(x) = X_1 \left( \cdots X_n(\sigma(x,-,...,-) ) \cdots \right)(x).
\]

\vspace{3mm}

\begin{theorem}[\cite{Lackman}] \label{Theorem: Riemann Integral}
    Let $X$ be an oriented $n$-dimensional compact manifold possibly with boundary and corners. Let $\omega$ be an $n$-form on $X$ and consider $\Omega$ a normalized, $S_n$-antisymmetric $n$-cochain such that $VE(\Omega) = \omega$. For any smooth oriented triangulation $(K,f)$ of $X$ denote by $S_{\leq(K)}$ the simplicial set associated with $K$. Define the Riemann sum subordinate to $(K,f)$
    \[
    S_{(K,f)}(\omega) := \sum_{\sigma \in S_{\leq}(K)_n } (i_f)^*\Omega(\sigma).
    \]
    Then the Riemann integral agrees with the usual integral, that is
    \[
    \int_X \omega = \underset{(K,f)}{\mathrm{lim}} \, S_{(K,f)}(\omega)
    \]
    where the limit is taken over the directed set of oriented triangulations $(K,f)$ of $X$ ordered by subdivision.
\end{theorem}

\begin{remark} $ $
\begin{enumerate}
    \item Recall that a \textbf{net} in a topological space $X$ is a function $f \colon I \rightarrow X$ where $(I,\leq)$ is a directed set. Given some $x \in X$, we say that the net $f$ \textbf{converges to} $x$ provided for each neighborhood $U$ of $x$, there is some $i_0 \in I$ such that for all $i_0 \leq i$ we have $f(i) \in U$. Since oriented triangulations form a directed set with respect to subdivision, the Riemann sum provides a net valued in $\mathbb{R}$
    \[
    S_{(-)}(\omega) \colon (K,f) \mapsto S_{(K,f)}(\omega) \in \mathbb{R}.
    \]
    Then the statement
    \[
      \int_X \omega = \underset{(K,f)}{\mathrm{lim}} \, S_{(K,f)}(\omega)
    \]
    is interpreted as, the net $S_{(-)}(\Omega)$ converges to the value $\int_X \omega$ in $\mathbb{R}$.

    \item A smooth triangulation of a manifold $X$ is a pair $(K,f)$ where $K$ is a simplicial complex and $f \colon |K| \rightarrow X$ a piecewise differentiable map satisfying further properties, see Definition \ref{Definition: Smooth Triangulation}. In the case that $X$ is an oriented manifold, the orientation determines a partial order $\leq$ on the set of vertices $V(K)$ which is then used to define the simplicial set $S_{\leq}(K)$. This is what we call an oriented triangulation, see Definition \ref{Definition: Oriented smooth triangulation}. Note, that the Riemann sum
    \[
    S_{(K,f)}(\omega) := \sum_{\sigma \in S_{\leq}(K)_n } (i_f)^*\Omega(\sigma) = \sum_{\sigma \in NS_{\leq}(K)_n } (i_f)^*\Omega(\sigma)
    \]
    can be reduced to be taken over the non-degenerate simplices $NS_{\leq}(K)$ using the fact that $\Omega$ is normalized.
\end{enumerate}
\end{remark}

Let us have a closer look at the construction of the Riemann sum in Theorem \ref{Theorem: Riemann Integral}. Let $X$ be an oriented $n$-dimensional manifold and $\omega \in \Omega^n(X)$ an $n$-form together with its corresponding anti-derivative $\Omega \in C^n(\mathrm{Pair}(X),\mathbb{R})$, i.e. a normalized, $S_n$-antisymmetric $n$-cochain such that $VE(\Omega) = \omega$. Any oriented triangulation $(K,f)$ of $X$ now induces a map of simplicial sets
\[
i_f \colon S_{\leq}(K) \rightarrow N\left( \mathrm{Pair}(X) \right),
\]
where on the right-hand side we are given the nerve of the pair groupoid. That is, for $n \geq 0$
\[
N(\mathrm{Pair}(X))_n = X^{n+1}.
\]
The map $i_f$ is now defined as follows. The vertices $v \in S_{\leq}(K)_0$ are given by the vertices of the triangulation $V(K)$ so that $i_f$ sends $v \in V(K)$ to its image under the homeomorphism $f$
\[
\begin{array}{rcl}
     (i_f)_0 \colon K_0 & \rightarrow & X  \\
     v & \mapsto & f(v).
\end{array}
\]
For general $n$ the map is given by
\[
\begin{array}{rcl}
     (i_f)_n \colon S_{\leq}(K)_n & \rightarrow & X^{n+1}  \\
     (v_0,...,v_n) & \mapsto &  \left( f(v_n), ..., f(v_0) \right),
\end{array}
\]
where we recall that an $n$-simplex in $S_{\leq}(K)$ is given by an ordered $(n+1)$-tuple of vertices $(v_0, ..., v_n)$. \\

This now allows us to properly define the Riemann sum with respect to a triangulation $(K,f)$ of $X$. Given some $n$-simplex $\sigma = (v_0, ..., v_n)$ in $S_{\leq}(K)$, denote by
\[
(i_f)^*\Omega(\sigma) := \Omega \left( i_f(\sigma) \right) = \Omega(v_n, ..., v_0) \in \mathbb{R}.
\]
Taking the sum over all $n$-simplices defines the Riemann sum
\[
    S_{(K,f)}(\omega) := \sum_{\sigma \in S_{\leq}(K)_n} (i_f)^*\Omega(\sigma).
\]

\subsection{Extracting the Curvature}

Previously we constructed a homomorphism of abelian groups
\[
    \begin{array}{rcl}
    \Phi \colon \hat{H}^{k+1}(M;\mathbb{Z}) & \rightarrow & \mathrm{Hom}_{\mathrm{PSh}(\mathbf{Cart},\mathbf{Ab})} \left( H_k(\mathbb{M}_k) ,U(1) \right) \\
        (\chi,\omega) & \mapsto & h_{\chi}.
    \end{array}
\]
associating to every differential character $\chi$ a smooth holonomy morphism $h_{\chi}$. The aim of this section is now two-fold. First, we assign a differential $(k+1)$-form $A_h \in \Omega^{k+1}(M)$ to each morphism $h$. In the second step, we show, using the Van Est homomorphism, that for all smooth $(k+1)$-chains $\eta \in C_{k+1}(M)$ in $M$
\begin{equation} \label{Equation: differential character eq.}
h_{\mathrm{pt}}\left( [\partial \eta]_k \right)  = \mathrm{exp} \left( 2 \pi i \int_{\eta} A_h \right).
\end{equation}

\begin{remark}
A smooth holonomy morphism $h$ is by construction a morphism of presheaves, i.e. a natural transformation. For every Cartesian space $U$ denote by $h_U$ the component of $h$ at $U$. Accordingly, $h_{\mathrm{pt}}$ denotes the component of $h$ at the one-point space $\mathrm{pt}$
\[
h_{\mathrm{pt}} \colon H_k(M_k) \rightarrow U(1).
\]
However, to simplify notation we will usually omit the subscript.
\end{remark}

From Lemma \ref{Lemma: Compact k-skeletal k-chains} it follows that $Z_k(D(U,M)_k) = Z_k(D(U,M))$ and so for any holonomy morphism $h$ there is an associated morphism of abelian groups $\chi_U$ for all $U \in \mathbf{Cart}$.
\begin{equation} \label{Diagram: Associated Character map to h}
\begin{tikzcd}
{Z_k(D(U,M))} \arrow[d, two heads] \arrow[rd, "\chi_U"] &                      \\
{H_k^{\Delta}\left( D(U,M)_k \right)} \arrow[r, "h_U"]           & {C^{\infty}(U,U(1))}
\end{tikzcd}
\end{equation}

The methods now used to extract the curvature form $A_h$ from the map $h$ are based on the ideas presented in \cite[Section 2.2]{Schreiber-Waldorf-Smooth}. \\

First, introduce the \textbf{standard family of smooth} $(k+1)$\textbf{-simplices} in $\mathbb{R}^{k+1}$ given by
\[
\begin{array}{rcl}
    \Delta_{\mathbb{R}} \colon \mathbb{R}^{k+1} & \rightarrow & D(|\Delta^{k+1}|,\mathbb{R}^{k+1}) \\
     x = (x_0,...,x_k) & \mapsto  & \Delta_{\mathbb{R}}(x)
\end{array}
\]
where $\Delta_{\mathbb{R}}(x)$ is the standard smooth $k+1$-simplex in $\mathbb{R}^{k+1}$ spanned by $(x_0,...,x_k)$. More precisely, consider the $(k+1)$-simplex $|\Delta^{k+1}|$ as the subset of $\mathbb{R}^{k+1}$ given by
\[
|\Delta^{k+1}| := \left\{ (z_0, ..., z_k) \in \mathbb{R}^{k+1} \, | \, \sum_{i = 0}^k z_i \leq 1 \text{ and } z_i \geq 0 \right\}
\]
For some fixed $x = (x_0, ..., x_k)$ denote by $m_x$ the coordinate-wise multiplication in $\mathbb{R}^{k+1}$ by $x$, i.e. $m_x(z) = (x_0z_0, ..., x_kz_k)$ and define
\[
\Delta_{\mathbb{R}}(x) \colon |\Delta^{k+1}| \hookrightarrow \mathbb{R}^{k+1} \xrightarrow{m_x} \mathbb{R}^{k+1}
\]
For example in the case $k=1$ the standard smooth 2-simplex in $\mathbb{R}^2$ by $(x_0,x_1)$ is given by the following smooth map:
\begin{center}
\begin{tikzpicture}[x=0.75pt,y=0.75pt,yscale=-1,xscale=1]

\draw    (48,192) -- (48,26) ;
\draw [shift={(48,24)}, rotate = 90] [color={rgb, 255:red, 0; green, 0; blue, 0 }  ][line width=0.75]    (10.93,-3.29) .. controls (6.95,-1.4) and (3.31,-0.3) .. (0,0) .. controls (3.31,0.3) and (6.95,1.4) .. (10.93,3.29)   ;
\draw    (24,168) -- (214,168) ;
\draw [shift={(216,168)}, rotate = 180] [color={rgb, 255:red, 0; green, 0; blue, 0 }  ][line width=0.75]    (10.93,-3.29) .. controls (6.95,-1.4) and (3.31,-0.3) .. (0,0) .. controls (3.31,0.3) and (6.95,1.4) .. (10.93,3.29)   ;
\draw    (42,72) -- (54,72) ;
\draw    (144,174) -- (144,162) ;
\draw  [draw opacity=0][fill={rgb, 255:red, 74; green, 144; blue, 226 }  ,fill opacity=0.5 ] (48,72) -- (144,168) -- (48,168) -- cycle ;
\draw    (186,96) -- (256,96) ;
\draw [shift={(258,96)}, rotate = 180] [color={rgb, 255:red, 0; green, 0; blue, 0 }  ][line width=0.75]    (10.93,-3.29) .. controls (6.95,-1.4) and (3.31,-0.3) .. (0,0) .. controls (3.31,0.3) and (6.95,1.4) .. (10.93,3.29)   ;
\draw    (336,192) -- (336,26) ;
\draw [shift={(336,24)}, rotate = 90] [color={rgb, 255:red, 0; green, 0; blue, 0 }  ][line width=0.75]    (10.93,-3.29) .. controls (6.95,-1.4) and (3.31,-0.3) .. (0,0) .. controls (3.31,0.3) and (6.95,1.4) .. (10.93,3.29)   ;
\draw    (312,168) -- (502,168) ;
\draw [shift={(504,168)}, rotate = 180] [color={rgb, 255:red, 0; green, 0; blue, 0 }  ][line width=0.75]    (10.93,-3.29) .. controls (6.95,-1.4) and (3.31,-0.3) .. (0,0) .. controls (3.31,0.3) and (6.95,1.4) .. (10.93,3.29)   ;
\draw    (330,96) -- (342,96) ;
\draw    (456,174) -- (456,162) ;
\draw  [draw opacity=0][fill={rgb, 255:red, 208; green, 2; blue, 27 }  ,fill opacity=0.5 ] (336,96) -- (456,168) -- (336,168) -- cycle ;

\draw (121,182.4) node [anchor=north west][inner sep=0.75pt]    {$( 1,0)$};
\draw (1,50.4) node [anchor=north west][inner sep=0.75pt]    {$( 0,1)$};
\draw (99,86.4) node [anchor=north west][inner sep=0.75pt]    {$|\Delta ^{2} |$};
\draw (433,182.4) node [anchor=north west][inner sep=0.75pt]    {$( x_{0} ,0)$};
\draw (289,70.4) node [anchor=north west][inner sep=0.75pt]    {$( 0,x_{1})$};
\draw (182,56.4) node [anchor=north west][inner sep=0.75pt]    {$\Delta _{\mathbb{R}}( x_{0} ,x_{1})$};

\end{tikzpicture}
\end{center}
For $x = 0$ the associated simplex is just given by the constant map at the origin. \\

The differential $(k+1)$-form $A_h$ is now defined pointwise. Therefore, fix some point $p \in M$ and let $v_0,...,v_k$ be tangent vectors $v_i \in T_pM$. Choose a smooth map $\Gamma \colon \mathbb{R}^{k+1} \rightarrow M$ satisfying
\begin{equation} \label{Equation: Gamma Conditions}
    \Gamma(0) = p \text{ and }  \left.\frac{\partial }{\partial x_i} \right\vert_{x_i = 0} \Gamma(0,..,x_i,...,0) = v_i \text{ for all } 0 \leq i \leq k
\end{equation}

The push-forward of the standard smooth families of simplices along the map $\Gamma$ defines a smooth map of diffeological spaces
\[
\mathbb{R}^{k+1} \xrightarrow{\Delta_{\mathbb{R}}} D(|\Delta^{k+1}|,\mathbb{R}^{k+1})  \xrightarrow{\Gamma_*}   D(|\Delta^{k+1}|,M)
\]
which can be considered likewise as a smooth map
\[
\Gamma_* \circ \Delta_{\mathbb{R}} \colon |\Delta^{k+1}| \rightarrow D(\mathbb{R}^{k+1},M)
\]
and defines a smooth $(k+1)$-chain in $D(\mathbb{R}^{k+1},M)$
\[
\Gamma_* \circ \Delta_{\mathbb{R}} \in C_{k+1}(D(\mathbb{R}^{k+1},M)).
\]
Its boundary gives a smooth $k$-cycle $\partial(\Gamma_* \circ \Delta_{\mathbb{R}}) \in Z_k \left( D(\mathbb{R}^{k+1},M)_k \right)$ and taking its corresponding $k$-skeletal homology class $[\partial(\Gamma_* \circ \Delta_{\mathbb{R}})]_k \in H^{\Delta}_k \left( C^{\infty}(\mathbb{R}^{k+1},M)_k \right)$ defines in general a \textit{non-trivial class} since the quotient is taken with respect to "$k$-thin" or $k$-skeletal boundaries. The evaluation of $h$ at this homology class gives a smooth map into $U(1)$
\[
h \left( \left[ \partial (\Gamma_* \circ \Delta_{\mathbb{R}} \right]_k) \right) \in C^{\infty}(\mathbb{R}^{k+1},U(1)).
\]
To simplify notation, denote this map by $h_{\Gamma}$, keeping in mind that the choice of $\Gamma$ depends on the fixed tangent vectors $(v_0, ..., v_k)$. Accordingly, we define
\begin{equation}\label{Equation: Definition of Curvature Form}
\begin{array}{rcl}
   A_h \colon  TM \times_M \cdots \times_M TM & \rightarrow & \mathbb{R}  \\
     (p,v_0, ..., v_k) & \mapsto & \left.\frac{\partial^{k+1} h_{\Gamma}}{\partial x_0 \cdots \partial x_k} \right\vert_{(0,...,0)},
\end{array}
\end{equation}
which is a priori neither well-defined, smooth, or antisymmetric.

\begin{lemma} \label{Lemma: The pointwise form is well defined}
Let $h$ be a fixed smooth holonomy morphism
\[
h \colon H_k \left( \mathbb{M}_k \right) \rightarrow U(1).
\]
Then the associated map $A_h$ is well defined, i.e. is independent of the choice of the function $\Gamma$. More precisely, for any two smooth maps $\Gamma_0,\Gamma_1 \colon \mathbb{R}^{k+1} \rightarrow M$ both satisfying (\ref{Equation: Gamma Conditions}), then
\[
    \left.\frac{\partial^{k+1} h_{\Gamma_0}}{\partial x_0 \cdots \partial x_k} \right\vert_{(0,...,0)} = \left.\frac{\partial^{k+1} h_{\Gamma_1}}{\partial x_0 \cdots \partial x_k} \right\vert_{(0,...,0)}.
\]
\end{lemma}

The proof of Lemma \ref{Lemma: The pointwise form is well defined} is divided into two steps. The first step is to construct a smooth homotopy $H \colon V \times [0,1] \rightarrow D \left( |\Delta^{k+1}|,M\right)$ where $V$ is some open neighborhood of the origin in $\mathbb{R}^{k+1}$ such that $H_j = (\Gamma_j)_* \circ \Delta_{\mathbb{R}}$ for $j = 0,1$. Moreover, this homotopy can be expressed as a composition of two smooth maps $f \colon V \times [0,1] \rightarrow Z$ and $ B \colon Z \rightarrow D \left( |\Delta^{k+1}|, M \right)$. This is the content of the next Lemma. The second step then uses the chain rule on the smooth homotopy $H = B \circ f$ to conclude that the pointwise defined differential form is indeed independent of the choice of $\Gamma$. This proof is a generalization of \cite[Lemma 2.6]{Schreiber-Waldorf-Smooth}.

\begin{lemma} \label{Lemma: Auxiliary Lemma curvature form}
    Let $M$ be a smooth manifold and let $\Gamma_0,\Gamma_1 \colon \mathbb{R}^{k+1} \rightarrow M$ be two smooth maps both satisfying (\ref{Equation: Gamma Conditions}). Then there is a smooth homotopy $H \colon V \times [0,1] \rightarrow D \left( |\Delta^{k+1}|,M\right)$ where $V$ is some open neighborhood of the origin in $\mathbb{R}^{k+1}$ such that $H_j = (\Gamma_j)_* \circ \Delta_{\mathbb{R}}$ for $j = 0,1$. Moreover, the homotopy $H$ can be expressed as a composition $H = B \circ f$ where $f \colon V \times [0,1] \rightarrow Z \subset V \times [0,1]$ is a smooth map given by $f(x,\alpha) = (x, \| x \|^2 \alpha )$ and $B \colon Z \rightarrow D \left( |\Delta^{k+1}|,M\right)$ some smooth map where $Z = \left\{ (x,\alpha) \in V \times [0,1] \, | \, 0 \leq \alpha \leq \| x \|^2\right\}$.
\end{lemma}
\begin{proof}
To simplify notation denote the maps $(\Gamma_j)_* \circ \Delta_{\mathbb{R}}$ by $\Delta_j$. Consider $U$ a chart of $M$ centered at $p$ and let $V$ be a small enough neighborhood of the origin in $\mathbb{R}^{k+1}$ such that we can define the homotopy via linear interpolation
    \begin{equation*}
        \begin{array}{rcl}
             \theta \colon V \times [0,1] \times |\Delta^{k+1}| & \rightarrow & U \subset M  \\
            (x,\alpha,t) & \mapsto & \Delta_0(x)(t) + \alpha \cdot d(x,t)
        \end{array}
    \end{equation*}
    where the distance function is simply given by
    \[
    d(x,t) := \Delta_1(x)(t) - \Delta_0(x)(t).
    \]
    Note that to make sense of the above sum the subset $U$ has implicitly been identified with $\mathbb{R}^n$ via the coordinate chart. Also to guarantee the existence of such a small enough neighborhood $V$ it is essential to consider the compact smooth simplices $|\Delta^{k+1}|$. Define the homotopy $H$ by
    \[
    \begin{array}{rcl}
         H \colon V \times [0,1] & \rightarrow  &  D(|\Delta^{k+1}|,M) \\
        (x,\alpha) & \mapsto & \left[   t \mapsto \theta(x,\alpha,t) \right].
    \end{array}
    \]
    The key observation to show that $H$ factors through $f$ and some other function $B$ is the following. Since both $\Gamma_0$ and $\Gamma_1$ satisfy the condition (\ref{Equation: Gamma Conditions}) we apply Hadamard's Lemma twice such that there are smooth maps
    \[
    g_{ij} \colon V \times |\Delta^{k+1}| \rightarrow U
    \]
    satisfying
    \begin{equation} \label{Equation: Hadamard}
    d(x_0, ..., x_k,t) = \sum_{0 \leq i,j \leq k} x_ix_j g_{ij}(x,t).
    \end{equation}
    Using the change of coordinates
    \[
    \begin{array}{rcl}
         \tau \colon [0,\infty) \times S^k & \rightarrow & \mathbb{R}^{k+1}  \\
        (r, u) &  \mapsto & r u
    \end{array}
    \]
 where $S^{k} \subset \mathbb{R}^{k+1}$ let $W := \tau^{-1}(V)$ and define
 \[
 \begin{array}{rcl}
      d_S \colon W \times |\Delta^{k+1}| & \rightarrow & U \\
     (r,u,t) & \mapsto & d(\tau(r,u),t).
 \end{array}
 \]
From equation \eqref{Equation: Hadamard} it then follows that $d_S(r,u,t) = r^2 \left( \sum_{0 \leq i,j \leq k} u_i u_j g_{ij}(r u,t) \right)$. Then define
\[
\tilde{d}_S(r,u,t) := \sum_{0 \leq i,j \leq k} u_i u_j g_{ij}(r u,t)
\]
such that $d_S(r,u,t) = r^2 \tilde{d}_S(r,u,t)$ for all $(r,u) \in W$ and $t \in |\Delta^{k+1}|$. \\

Define the map
\[
\begin{array}{rcl}
     B_S \colon W \times |\Delta^{k+1}| \times [0,1] & \rightarrow & U  \\
     (r,u,t,\alpha) & \mapsto & \Delta_0(\tau(r,u),t) + \alpha \tilde{d}_S(r,u,t)
\end{array}
\]
and let $Z_S := \left\{ (r,u,\alpha) \in W \times [0,1] \, | \, 0 \leq \alpha \leq r^2 \right\}$. First notice that under $\tau$ the subset $Z_S$ is indeed mapped into $Z$. The claim is, that there exists a map $B \colon Z \times |\Delta^{k+1}| \rightarrow U$ such that the following diagram commutes.
\begin{center}
    \begin{tikzcd}
Z_S \times |\Delta^{k+1}| \arrow[d, "\tau \times \mathrm{id}"'] \arrow[r, "B_S"] & U \\
Z \times |\Delta^{k+1}| \arrow[ru, "B"']                                         &
\end{tikzcd}
\end{center}
First notice that to find an extension of $B_S$ it suffices to find an extension for the term $\alpha \tilde{d}_S(r,u,t)$. This term however, together with all of its derivatives in the $r$-direction vanishes at $r= 0$ (by the definition of $Z_S$) and therefore extends to $Z \times |\Delta^{k+1}|$. \\

What is still left to show is that $H = B \circ f$. Since $\tau$ is surjective, it suffices to check this identity for $(r,u,\alpha) \in W \times [0,1]$. It then follows
\begin{align*}
    (B \circ f)(\tau(r,u),\alpha)(t) &= B_S(r,u, r^2 \alpha ,t) \\
    &=\Delta_0(\tau(r,u),t) + \alpha r^2 \tilde{d}_S(r,u,t) \\
    &= \Delta_0(\tau(r,u),t) + \alpha d_S(r,u,t) \\
    &= \Delta_0(\tau(r,u),t) + \alpha d(\tau(r,u),t) \\
    &= H(\tau(r,u),\alpha)(t).
\end{align*}
This finishes the proof.
\end{proof}

\begin{proof}[Proof of Lemma \ref{Lemma: The pointwise form is well defined}]
    Let $\Gamma_0$ and $\Gamma_1$ be two smooth maps $\Gamma_j \colon \mathbb{R}^{k+1} \rightarrow M$ both satisfying the condition (\ref{Equation: Gamma Conditions}). Using the standard family of simplices in $\mathbb{R}^{k+1}$ this gives two maps
    \[
    (\Gamma_j)_* \circ \Delta_{\mathbb{R}} \colon \mathbb{R}^{k+1} \rightarrow D(|\Delta|^{k+1}, M).
    \]
    Again to simplify notation denote the maps $(\Gamma_j)_* \circ \Delta_{\mathbb{R}}$ by $\Delta_j$. From the auxiliary Lemma \ref{Lemma: Auxiliary Lemma curvature form} we know that there exists a smooth homotopy $H \colon V \times [0,1] \rightarrow D \left( |\Delta^{k+1}|,M \right)$ where $V$ is some open neighborhood of the origin in $\mathbb{R}^{k+1}$ such that $H_j = (\Gamma_j)_* \circ \Delta_{\mathbb{R}}$ for $j = 0,1$. Also we have that the homotopy $H$ can be expressed as a composition $H = B \circ f$ where $f \colon V \times [0,1] \rightarrow Z \subset V \times [0,1]$ is a smooth map given by $f(x,\alpha) = (x, \| x \|^2 \alpha )$ and $B \colon Z \rightarrow D \left( |\Delta^{k+1}|,M\right)$ some smooth map where $Z = \left\{ (x,\alpha) \in V \times [0,1] \, | \, 0 \leq \alpha \leq \| x \|^2\right\}$. \\

    Now given a smooth holonomy map $h$ recall that we denote by $h_{\Gamma_j} = h([\partial(\Delta_j)]_k)$ for $j = 0,1$.  Then since $H_j = \Delta_j$ it follows
    \[
     \left.\frac{\partial^{k+1} h_{\Gamma_j}}{\partial x_0 \cdots \partial x_k} \right\vert_{(0,...,0)} = \left.\frac{\partial^{k+1} \left( h([\partial(-)]_k) \circ B \circ f  \right)}{\partial x_0 \cdots \partial x_k} \right\vert_{(0,...,0,j)}
    \]
    Now we apply the chain rule to compute the partial derivative where we consider $( h([\partial(-)]_k) \circ B) \circ f$ as a composition of smooth maps. Write $\psi :=  h([\partial(-)]_k) \circ B)$. By the generalized version of Fa\`a di Bruno's Formula, see Proposition \ref{Proposition: Faa di Bruno}, the partial derivative is given by a sum of products
    \begin{equation} \label{Equation: The Chain rule equation}
         \left.\frac{\partial^{k+1} h_{\Gamma_j}}{\partial x_0 \cdots \partial x_k} \right\vert_{(0,...,0)} = \sum_{ |\sigma| = 1}^{k+1} \left.\frac{\partial^{|\sigma|} \psi }{\partial z^{\sigma} } \right\vert_{ f(0,...,0,j)} \sum_{E_{\sigma}} \prod_{i = 1}^{k+1}  \prod_{A_{\beta}} \frac{1}{e_{i\beta^i}!} \left( \left.\frac{\partial^{|\beta^i|} f_i }{\partial x^{\beta^i}} \right\vert_{(0,...,0,j)} \right)^{e_{i\beta^i}}.
    \end{equation}
   For the exact meaning of the indices, we refer to Proposition \ref{Proposition: Faa di Bruno}. First, note that $f(0,...,0,1) = f(0,...,0,0) = 0$, and hence the terms
   \[
   \left.\frac{\partial^{|\sigma|} \psi }{\partial z^{\sigma} } \right\vert_{ f(0,...,0,j)}
   \]
   are all independent of $j$. Further, the left-hand side is a mixed partial derivative, i.e. the multi-index reads $\beta = (1,\cdots , 1)$. Since the multi-indices $\beta^i$ need to satisfy the equation
       \[
    \sum_{i = 1}^{k+1} \sum_{|\beta^i|}^{k+1} e_{i\beta^i} \beta^i = \beta =  (1, ....,1),
    \]
   where $e_{i\beta^i}$ are positive integers, it follows that the multi-indices $\beta^i$ all have components $\leq 1$. This fact together with the specific way we defined the map $f$, the values for all the partial derivatives of the form
    \[
    \left.\frac{\partial^{|\beta^i|} f_i }{\partial x^{\beta^i} } \right\vert_{(0,...,0,j)}
    \]
    are independent of $j$. We conclude that the right-hand side is independent of $j$ and thus also the left-hand side of equation (\ref{Equation: The Chain rule equation}). This finishes the proof.
\end{proof}

The next step is to show that the assignment $A_h$ defines a differential $(k+1)$-form on $M$. This is the content of the next Lemma.

\begin{lemma} \label{Lemma: A_h is antisymmetric, smooth and multilinear}
For fixed $p \in M$ the map $A_h \colon T_pM \times \cdots \times T_pM \rightarrow \mathbb{R}$ is antisymmetric and multilinear. Moreover, the map $A_h \colon TM \times_M \cdots \times_M TM \rightarrow \mathbb{R}$ is smooth.
\end{lemma}

\begin{proof}
Let us first prove the multilinearity of the map $A_h \colon T_pM \times \cdots \times T_pM \rightarrow \mathbb{R}$ for every fixed $p \in M$. This follows from an analog approach as given in \cite[Lemma 2.7]{Schreiber-Waldorf-Smooth}. Consider therefore tangent vectors $v_i \in T_pM$ for $0 \leq i \leq k$ and $v'_j \in T_pM$ for some fixed $0 \leq j \leq k$. We now wish to show that
\begin{equation} \label{Equation: Multilinearity final}
    A_h(p,v_0, ..., v_j + \lambda v'_j,..,v_k) =  A_h(p,v_0, ..., v_j,..,v_k) + \lambda A_h(p,v_1, ...,v'_j, ...,v_k).
\end{equation}
Let $\Gamma,\Gamma' \colon \mathbb{R}^{k+1} \rightarrow M$ be two smooth functions for the tangent vectors $(v_0, ...,v_j,...,v_k)$ and $(v_0,...,v'_j,...,v_k)$, that is they satisfy equation (\ref{Equation: Gamma Conditions}). Choose now $\varphi \colon U \rightarrow M$ a chart centered at $p$ and construct the smooth map $\widetilde{\Gamma} \colon W \rightarrow M$ defined on a suitable neighborhood $W \subset \mathbb{R}^{k+1}$ of the origin by
\begin{equation*}
    \widetilde{\Gamma} := \varphi \left(  \varphi^{-1} \circ \Gamma + \lambda \varphi^{-1} \circ \Gamma' \right)
\end{equation*}
First, notice that $\widetilde{\Gamma}(0) = p$. Moreover, we have that
\[
 \left.\frac{\partial }{\partial x_i} \right\vert_{x_i = 0} \widetilde{\Gamma}(0,..,x_i,...,0) = \begin{cases}
      v_j + \lambda v'_j &\text{ for } i = j \\
      v_i &\text{ else }

 \end{cases}
\]
which then implies by Lemma \ref{Lemma: The pointwise form is well defined} that
\begin{equation} \label{Equation: Multilinearity}
A_h(p,v_0, ..., v_j + \lambda v'_j, ..., v_k) =  \left.\frac{\partial^{k+1} h_{\widetilde{\Gamma}}}{\partial x_0 \cdots \partial x_k} \right\vert_{(0,...,0)}.
\end{equation}
Recall that $h_{\widetilde{\Gamma}}$ denotes the smooth map
\[
h([\partial(\widetilde{\Gamma}_* \circ \Delta_{\mathbb{R}})]_k) \colon \mathbb{R}^{k+1} \rightarrow U(1).
\]
Now notice that
\[
\widetilde{\Gamma}_* \circ \Delta_{\mathbb{R}} = \varphi_* \left(   \varphi^{-1}_* \left( \Gamma_* \circ \Delta_{\mathbb{R}} \right) + \lambda \varphi_*^{-1} \left( \Gamma'_* \circ \Delta_{\mathbb{R}} \right) \right)
\]
which gives
\[
\left.\frac{\partial^{k+1} h_{\widetilde{\Gamma}}}{\partial x_0 \cdots \partial x_k} \right\vert_{(0,...,0)} = \left.\frac{\partial^{k+1} h_{{\Gamma}}}{\partial x_0 \cdots \partial x_k} \right\vert_{(0,...,0)} + \lambda \left.\frac{\partial^{k+1} h_{{\Gamma'}}}{\partial x_0 \cdots \partial x_k} \right\vert_{(0,...,0)}.
\]
Together with equation (\ref{Equation: Multilinearity}), this shows that identity (\ref{Equation: Multilinearity final}) holds. \\

The antisymmetry property of $A_h$ essentially follows from Corollary \ref{Corollary: Change of Orientation}. Indeed, consider again for some $p \in M$ fixed tangent vectors $v_i \in T_pM$ for $0 \leq i \leq k$. Let $\Gamma \colon \mathbb{R}^{k+1} \rightarrow M$ be a smooth map such that $\Gamma(0) = p$ and satisfying equation (\ref{Equation: Gamma Conditions}) such that
\[
A_h(p,v_0, ...,v_k) = \left.\frac{\partial^{k+1} h_{{\Gamma}}}{\partial x_0 \cdots \partial x_k} \right\vert_{(0,...,0)}.
\]
Given any permutation $\sigma \in S_{k+1}$ it then follows that the smooth map $\sigma_* \Gamma$ given by
\[
\mathbb{R}^{k+1} \xrightarrow{\sigma_*} \mathbb{R}^{k+1} \xrightarrow{\Gamma} M
\]
is precisely such that
\[
A_h(p,v_{\sigma(0)}, ..., v_{\sigma(k)}) = \left.\frac{\partial^{k+1} h_{\sigma_*{\Gamma}}}{\partial x_0 \cdots \partial x_k} \right\vert_{(0,...,0)}.
\]
The smooth map $h_{\sigma_*{\Gamma}}$ is by definition
\[
h([\partial((\sigma_*\Gamma)_* \circ \Delta_{\mathbb{R}})]_k) \colon \mathbb{R}^{k+1} \rightarrow U(1).
\]
where
\[
(\sigma_*\Gamma)_* \circ \Delta_{\mathbb{R}} \colon | \Delta^{k+1} | \times \mathbb{R}^{k+1} \xrightarrow{\Delta_{\mathbb{R}}} \mathbb{R}^{k+1} \xrightarrow{\sigma_*} \mathbb{R}^{k+1} \xrightarrow{\Gamma} M
\]
This shows that as smooth $(k+1)$-simplices in $D(\mathbb{R}^{k+1},M)$
\[
(\sigma_*\Gamma)_* \circ \Delta_{\mathbb{R}} = \sigma_* \left( \Gamma_* \circ \Delta_{\mathbb{R}} \right),
\]
where the right-hand notation is the one introduced in preparation for Corollary \ref{Corollary: Change of Orientation}. Notice that on the right-hand side to act with $\sigma$ on a smooth $(k+1)$-simplex consider $\sigma$ as an element in $S_{k+2}$ preserving $0 \in \{0,...,k+1\}$. It follows now from Corollary \ref{Corollary: Change of Orientation} that
\[
h([\partial((\sigma_*\Gamma)_* \circ \Delta_{\mathbb{R}})]_k) = h([\partial(\sigma_*(\Gamma_* \circ \Delta_{\mathbb{R}})]_k) = h([\partial((\Gamma \circ \Delta_{\mathbb{R}})]_k)^{\mathrm{sgn}(\sigma)}.
\]
By taking the derivatives it then follows
\[
\left.\frac{\partial^{k+1} h_{\sigma_*{\Gamma}}}{\partial x_0 \cdots \partial x_k} \right\vert_{(0,...,0)} = \mathrm{sgn}(\sigma) \left.\frac{\partial^{k+1} h_{{\Gamma}}}{\partial x_0 \cdots \partial x_k} \right\vert_{(0,...,0)}
\]
which shows precisely that $A_h$ is antisymmetric. \\

The smoothness of $A_h$ follows essentially from the fact that $h$ as a morphism of presheaves on $\mathbf{Cart}$ is smooth.
\end{proof}

To recap briefly, Lemma \ref{Lemma: The pointwise form is well defined} together with Lemma \ref{Lemma: A_h is antisymmetric, smooth and multilinear} show that to every smooth holonomy map $h \colon H_k(\mathbb{M}_k) \rightarrow U(1)$ we can associate a differential $(k+1)$-form $A_h \in \Omega^{k+1}(M)$. An alternative construction of the form $A_h$ relying only on antisymmetric, normalized cochains together with the Van Est morphism is presented briefly in Remark \ref{Remark: The form A_h via the Van Est map}.    \\

\begin{lemma} \label{Lemma: Curvature form gives differential character}
    Let $h$ be as above and denote by $A_h$ its associated curvature $(k+1)$-form in $\Omega^{k+1}(M)$. The component $h_{\mathrm{pt}}$ induces a morphism of abelian groups $\chi \colon Z_k(M) \rightarrow U(1)$. Then for any smooth $(k+1)$-chain $\eta \in C_{k+1}(M)$ in $M$
    \[
    \chi(\partial \eta) = \mathrm{exp} \left( 2 \pi i \int_{\eta} A_h \right).
    \]
    In particular, the pair $(\chi,A_h)$ defines a differential character of degree $k+1$ on $M$.
\end{lemma}

The key ingredient to prove this lemma is Theorem \ref{Theorem: Riemann Integral} expressing the integral as a Riemann sum where the partition is taken over triangulations instead of cubes. To apply this theorem, we first have to find for any smooth $(k+1)$-chain $\eta \in C_{k+1}(M)$ a suitable anti-derivative $\Omega_{(\eta,h)}$ of $\eta^*A_h$, i.e.
\[
VE(\Omega_{(\eta,h)}) = \eta^* A_h.
\]
To simplify notation we omit the dependence of $\Omega_{(\eta,h)}$ on $h$ and write $\Omega_{\eta}$.

\begin{definition}
    Given any smooth $(k+1)$-chain $\eta \in C_{k+1}(M)$ define the normalized $(k+1)$-cochain $\Omega_{\eta} \in C^{k+1}(\mathrm{Pair}(|\Delta^{k+1}|), \mathbb{R})$ on a suitable neighborhood $D$ of the diagonal
    \[
    \begin{array}{rcl}
      \Omega_{\eta} \colon  D \subset |\Delta^{k+1}| \times \cdots \times |\Delta^{k+1}| & \rightarrow & \mathbb{R}  \\
          (x_0, ..., x_{k+1}) & \mapsto & \mathrm{log} \left( h\left( [\partial(\eta|_{[x_0, ..., x_{k+1}]})]_k\right) \right) \cdot \frac{1}{2 \pi i},
    \end{array}
    \]
    where here $[x_0, ..., x_{k+1}]$ denotes the convex hull or linear $(k+1)$-simplex as a subset of $|\Delta^{k+1}|$ such that the restriction of $\eta$ to $[x_0, ..., x_{k+1}]$ gives a smooth $(k+1)$-chain denoted by $\eta|_{[x_0, ..., x_{k+1}]}$.
\end{definition}

\begin{lemma} \label{Lemma: Van Est Antiderivative}
    The $(k+1)$-cochain $\Omega_{\eta}$ is a normalized, $S_{k+1}$-antisymmetric $(k+1)$-cochain such that
    \[
    VE(\Omega_{\eta}) = \eta^*A_h.
    \]
\end{lemma}

\begin{proof}
Notice that it suffices to consider the case $\eta \colon |\Delta^k| \rightarrow M$. Also, fix a smooth extension $\tilde{\eta} \colon U \rightarrow M$ where $|\Delta^{k+1}| \subset U \subset \mathbb{R}^{k+1}$ is a convex open neighborhood\footnote{See Remark \ref{Remark: chains with the subset diffeology}.}. Accordingly, we define
\[
\begin{array}{rcl}
        \tilde{\Omega}_{\eta} \colon D' \subset U \times \cdots \times U  & \rightarrow & \mathbb{R} \\
        (x_0, ..., x_{k+1}) & \mapsto & \mathrm{log} \left( h_*\left( [\partial(\tilde{\eta}|_{[x_0, ..., x_{k+1}]})] \right) \right) \cdot \frac{1}{2 \pi i}
\end{array}
\]
as an extension of $\Omega_{\eta}$, where again $D'$ denotes a suitable neighborhood of the diagonal. This extension is well defined since for all $(x_0, ..., x_{k+1}) \in D$ the convex hull $[x_0, ..., x_{k+1}] \subset U$ by the convexity of $U$. Hence $\tilde{\eta}|_{[x_0, ..., x_{k+1}]}$ indeed defines a smooth $(k+1)$-simplex in $M$. \\

Let us first show that $\Omega_\eta$ is $S_{k+1}$-antisymmetric and normalized. The $S_{k+1}$-action on
\[
\mathrm{Pair}(|\Delta^{k+1}|)^{(k+1)} = \underbrace{|\Delta^{k+1}| \times \cdots \times |\Delta^{k+1}|}_{(k+2)\text{-times}}
\]
is given by permutations that fix the first entry. That is, for $s \in S_{k+1}$
\[
s \cdot (x_0, ..., x_{k+1}) = (x_0, x_{s(1)}, ..., x_{s(k+1)}).
\]
The $S_{k+1}$-antisymmetry property of $\Omega_{\eta}$ now follows from Corollary \ref{Corollary: Change of Orientation}:
\[
\left[ \partial(\eta|_{[x_0,x_1,..., x_{k+1}]})\right]_k = \mathrm{sgn}(s) \left[ \partial(\eta|_{[x_0,x_{s(1)}..., x_{s(k+1)}]})\right]_k.
\]

Tho show that $\Omega_{\eta}$ is normalized, consider a tuple $(x_0, ..., x_{k+1}) \in D$ such that there is some $1 \leq i \leq k+1$ with $x_i = x_{i-1}$. The associated $(k+1)$-simplex $[x_0,...,x_{k+1}]$ is under this hypothesis a $k$-thin smooth simplex, i.e. $\eta|_{[x_0,...,x_{k+1}]} \colon |\Delta^{k+1}| \rightarrow M_k$ and hence the homology class vanishes $[\partial [x_0, ..., x_{k+1}] ]_k = 0$. \\

We are left to show that $\tilde{\Omega}_{\eta}$ is an anti-derivative of $\tilde{\eta}^*A_h$. The pullback $\tilde{\eta}^*A_h$ is a $(k+1)$-form on the open subset $U$ of $\mathbb{R}^{k+1}$ and is therefore fully specified by its values on $\left.\frac{\partial}{\partial x_i} \right\vert_{x_0}$ for $ 1 \leq i \leq k +1$ and $x_0 \in U$. This fact, together with by the $S_{k+1}$-antisymmetry of $\tilde{\Omega}_{\eta}$ imply that it suffices to show that for all $x_0 \in U$
    \[
    VE(\tilde{\Omega}_{\eta})\left( \frac{\partial}{\partial x_1}, ..., \frac{\partial}{\partial x_{k+1}} \right)(x_0) = \tilde{\eta}^*A_h\left( \left.\frac{\partial}{\partial x_1} \right\vert_{x_0}, ...,  \left.\frac{\partial}{\partial x_{k+1}} \right\vert_{x_0} \right).
    \]
By Remark \ref{Remark: Formula of the Van Est morphism} it follows that:
\begin{equation} \label{Equation: Partial Derivative in each direction}
    VE(\tilde{\Omega}_{\eta})\left( \frac{\partial}{\partial x_1} ... \frac{\partial}{\partial x_{k+1}} \right)(x_0) = \left.\frac{\partial^{k+1} \tilde{\Omega}_{\eta}(x_0,-) }{\partial x_1 ... \partial x_{k+1}} \right\vert_{(x_0,...,x_0)}.
    \end{equation}
    Notice that for $x_0 \in U$ fixed we have
    \[
    \Omega_{\eta}(x_0,-) \colon D'_{x_0} \subset \underbrace{U \times \cdots \times U}_{(k+1)\text{-times}} \rightarrow \mathbb{R}
    \]
    where $D'_{x_0} = (i_{x_0})^{-1}(D')$ for $i_{x_0} \colon \left\{x_0 \right\} \times U \times \cdots \times U \hookrightarrow U \times \cdots \times U$ and equation (\ref{Equation: Partial Derivative in each direction}) means that we take the $i$-th partial derivative with respect to the $i$th copy of $U$. \\

    Keeping $x_0 \in U$ fixed, define the map
    \[
    U_0 \subset \mathbb{R}^{k+1} \rightarrow U \times \cdots \times U
    \]
    on a neighborhood of the origin sending $(z_1, ..., z_{k+1})$ to the tuple $(x_0 + z_1 e_1, ..., x_0 + z_i e_i, ..., x_0 + z_{k+1} e_{k+1})$. Here $e_i$ denotes the $i$-th standard basis element in $\mathbb{R}^{k+1}$. By composition with $\tilde{\Omega}_h(x_0,-)$ we get a map $\gamma_{x_0}$, which has the property that
    \begin{equation} \label{Equation: Little gamma}
    \left.\frac{\partial \gamma_{x_0} }{\partial z_1 \cdots \partial z_{k+1}} \right\vert_{(0,..,0)} =  VE(\tilde{\Omega}_{\eta})\left( \frac{\partial}{\partial x_1},...,\frac{\partial}{\partial x_{k+1}} \right)(x_0).
    \end{equation}
    We claim the following:
    \begin{equation}\label{Equation: Identity}
    \left.\frac{\partial \gamma_{x_0} }{\partial z_1 \cdots \partial z_{k+1}} \right\vert_{(0,..,0)} = \tilde{\eta}^*A_h\left( \left.\frac{\partial}{\partial x_1} \right\vert_{x_0}, ...,  \left.\frac{\partial}{\partial x_{k+1}} \right\vert_{x_0} \right).
     \end{equation}
    Indeed, denote by $v_i:= \left.\frac{\partial \tilde{\eta}}{\partial x_i} \right\vert_{x_0}$. Then by definition of the curvature form $A_h$ we have that
    \begin{equation} \label{Equation: Big Gamma}
    A_h(v_1, ..., v_{k+1}) = \left.\frac{\partial^{k+1} h_{\Gamma}}{\partial x_1 \cdots \partial x_{k+1}} \right\vert_{(0,...,0)},
    \end{equation}
    for $\Gamma \colon \mathbb{R}^{k+1} \rightarrow M$ some smooth map with $\Gamma(0) = \tilde{\eta}(x_0)$ and  $\left.\frac{\partial }{\partial x_i} \right\vert_{x_i = 0} \Gamma(0,..,x_i,...,0) = v_i$  for all $1 \leq i \leq k+1$. We have a preferred candidate for the map $\Gamma$
    \[
    \Gamma(z_1, ..., z_{k+1}) =  \tilde{\eta} \left(  x_0 + \sum_{i = 1}^{k+1} z_i e_i \right).
    \]
    Notice that for any fixed $x_0 \in U$ this is well defined on a suitable neighborhood $V \subset \mathbb{R}^{k+1}$ around the origin. This map
    \[
    \Gamma \colon V \rightarrow M
    \]
    satisfies $\Gamma(0) = \tilde{\eta}(x_0)$ and the partial derivatives at the origin precisely agree with the $v_i$. Now we recall the definition of $h_{\Gamma}$. Using the standard family of simplices in $\mathbb{R}^{k+1}$ we define in a suitable neighborhood of the origin $W \subset  \mathbb{R}^{k+1}$
    \[
    \begin{array}{rcl}
       W \times |\Delta^{k+1}| & \rightarrow &  M \\
         (z , t) & \mapsto & \Gamma( zt )
    \end{array}
    \]
    and in addition, we notice that for every $z \in W$ the smooth $(k+1)$-simplex in $M$ defined by
    \[
    t \mapsto \Gamma(z \cdot t) = \tilde{\eta} \left( x_0 + \sum_{i = 1}^{k+1} (z_it_i) e_i\right) ,
    \]
    is precisely the restriction of $\tilde{\eta}$  to the simplex spanned by $\left\{x_0, x_0 + z_ie_i, ..., x_0 + z_{k+1}e_{k+1} \right\} $ i.e. $\eta|_{[x_0, x_0 + z_ie_i, ..., x_0 + z_{k+1}e_{k+1} ]}$. By definition of $h_{\Gamma}$
    \[
    h_{\Gamma} := h \left( \left[ \partial (\Gamma_* \circ \Delta_{\mathbb{R}})\right]_k \right) \colon \mathbb{R}^{k+1} \rightarrow U(1),
    \]
    it follows that
    \[
    \mathrm{log} \left( h_{\Gamma} \right) = 2 \pi i \cdot \gamma_{x_0},
    \]
    when restricted to the open neighborhood of the origin $U_0 \cap W$. Applying the chain rule together with the fact that $h_{\Gamma}(0) = 1$ it follows
    \[
     \left.\frac{\partial \gamma_{x_0} }{\partial z_1\cdots \partial z_{k+1}} \right\vert_{(0,..,0)} = \left.\frac{\partial^{k+1} h_{\Gamma}}{\partial z_1 \cdots \partial z_{k+1}} \right\vert_{(0,...,0)},
    \]
    which proves the claim of (\ref{Equation: Identity}). Now equation (\ref{Equation: Little gamma}) together with equation (\ref{Equation: Big Gamma}) imply that indeed for any $x_0 \in U$ one has
    \[
    VE(\tilde{\Omega}_{\eta})\left( \frac{\partial}{\partial x_1}, ...,\frac{\partial}{\partial x_{k+1}} \right)(x_0) = \tilde{\eta}^*A_h\left( \left.\frac{\partial}{\partial x_1} \right\vert_x, ...,  \left.\frac{\partial}{\partial x_{k+1}} \right\vert_{x_0} \right).
    \]
    which implies that for all $x_0 \in |\Delta^{k+1}|$
    \[
    VE(\Omega_{\eta})\left( \frac{\partial}{\partial x_1}, ...,\frac{\partial}{\partial x_{k+1}} \right)(x_0) = \eta^*A_h\left( \left.\frac{\partial}{\partial x_1} \right\vert_x, ...,  \left.\frac{\partial}{\partial x_{k+1}} \right\vert_{x_0} \right).
    \]
    This concludes the proof.
\end{proof}

Now we are ready to give a proof of Lemma \ref{Lemma: Curvature form gives differential character}.

\begin{proof}[Proof of Lemma \ref{Lemma: Curvature form gives differential character}]

Let $h$ be a fixed smooth holonomy morphism and let $A_h$ be its associated curvature $(k+1)$-form. Let $\eta$ be any $(k+1)$-chain in $M$. By Theorem \ref{Theorem: Riemann Integral} we have that the integral can be computed as a limit of Riemann sums
\[
\int_{\eta} A_h = \int_{|\Delta^{k+1}|} \eta^*A_h = \underset{(K,f)}{\mathrm{lim}} S_{(K,f)}(\eta^*A_h)
\]
which then gives
\[
\mathrm{exp} \left( 2 \pi i \int_{\eta} A_h \right) = \underset{(K,f)}{\mathrm{lim}} \mathrm{exp} \left( 2\pi i   S_{(K,f)}(\eta^*A_h)  \right)
\]

Let therefore $(K,f)$ be an oriented triangulation of $|\Delta^{k+1}|$. For a fine enough\footnote{This means that the vertices of all $(k+1)$-dimensional simplices in $(K,f)$ lie in the neighborhood $D$ of the diagonal.} triangulation $(K,f)$ it follows that the Riemann sum $ S_{(K,f)}(\eta^*A_h)$ is given by
\begin{equation*}
\sum_{\sigma \in NS_{\leq} (K)_{k+1}} \Omega_h( (i_f)(\sigma) ) = \sum_{\sigma \in NS_{\leq}(K)_{k+1}} \mathrm{log} \left( h \left( [\partial(\eta|_{\sigma})]_k \right) \right) \cdot \frac{1}{2 \pi i}
\end{equation*}
such that by taking the exponential we get
\begin{align*}
\mathrm{exp} \left( 2\pi i \cdot S_{(K,f)}(\eta^*A_h)  \right)  &= \prod_{\sigma \in NS_{\leq}(K)_{k+1}} h([\partial(\eta|_{\sigma})]_k) \\
&= \left( h \left( \left[ \sum_{\sigma \in NS_{\leq}(K)_{k+1}} \partial(\eta|_{\sigma}) \right]_k \right) \right).
\end{align*}

We can always produce such a fine enough triangulation by an iterated application of the barycentric subdivision since this creates arbitrary small simplices \footnote{This follows from the fact that the diameter of each simplex in the barycentric subdivision of $|\Delta^{n}|$ is at most $\frac{n}{n+1}$ times the diameter of $|\Delta^n|$. An iterated application of barycentric subdivision therefore creates simplices of diameter at most $\left( \frac{n}{n+1} \right)^r$ which approaches $0$ for $r \rightarrow \infty$. For a discussion of this fact see \cite[Proposition 2.21]{Hatcher}.}. Therefore, consider $(K,f)$ to be a triangulation induced by a repeated barycentric subdivision of the smooth simplex $|\Delta^{k+1}|$ endowed with the standard orientation. We wish to show that the homology classes
\begin{equation} \label{Equation: The Subdivision equation}
\left[ \sum_{\sigma \in NS_{\leq}(K)_{k+1}} \partial(\eta|_{\sigma}) \right]_k = \left[ \partial(\eta) \right]_k
\end{equation}
agree in $H_{k}(M_k)$, such that in turn, the right-hand side is independent of the triangulation $(K,f)$. To see this, note first by Corollary \ref{Corollary: Subdivision Operator} that
\begin{equation} \label{Equation: The Subdivision operator equation}
\left[ S^r(\partial \eta ) \right]_k = [\partial \eta]_k.
\end{equation}

where $S^r$ is the iterated subdivision operator. From Example \ref{Example: Subdivision} we conclude that
\begin{equation} \label{Equation: The Baycenter Equation}
\left[ \sum_{\sigma \in NS_{\leq}(K)_{k+1}} \partial(\eta|_{\sigma}) \right]_k  =  \left[ \partial\left( S^r(\eta) \right) \right]_k = \left[ S^r( \partial \eta ) \right]_k
\end{equation}
and identity (\ref{Equation: The Subdivision equation}) now follows from (\ref{Equation: The Subdivision operator equation}) and (\ref{Equation: The Baycenter Equation}).
We conclude that
\[
\mathrm{exp} \left( 2\pi i  S_{(K,f)}(\eta^*A_h)  \right)=  h([\partial \eta]_k)
\]
and as such the right-hand side is independent of the triangulation. That is, the net \newline $\mathrm{exp}\left( 2\pi i \cdot S_{(-)}(\eta^*A_h) \right)$ indeed converges to $h([\partial \eta])$, resulting in the desired identity
\[
\mathrm{exp} \left( 2 \pi i \int_{\eta} A_h  \right) =  h([\partial \eta]).
\]
\end{proof}

\begin{theorem} \label{Corollary: The morphism Phi is an iso}
For $M$ a smooth manifold and $k \geq 1$, the morphism of abelian groups
\[
\begin{array}{rcl}
     \Phi \colon \hat{H}^{k+1}(M;\mathbb{Z}) & \rightarrow & \mathrm{Hom}_{\mathrm{PSh}(\mathbf{Cart},\mathbf{Ab})} \left(H_k(\mathbb{M}_k) ,U(1) \right) \\
     \chi & \mapsto & h_{\chi}
\end{array}
\]
is an isomorphism whose inverse is given by the assignment
\[
\begin{array}{rcl}
   \Psi \colon \mathrm{Hom}_{\mathrm{PSh}(\mathbf{Cart},\mathbf{Ab})} \left( H_k(\mathbb{M}_k) ,U(1) \right)  & \rightarrow &  \hat{H}^{k+1}(M;\mathbb{Z}) \\
    h & \mapsto &  \chi = h_{\mathrm{pt}} \circ q
\end{array}
\]
where $q$ denotes the quotient map $Z_k(M) \rightarrow H_k(M_k)$.
\end{theorem}
\begin{proof}
    From Lemma \ref{Lemma: Curvature form gives differential character} it follows that for every smooth holonomy morphism $h$ the associated morphism of abelian groups
    \[
    \chi \colon Z_k(M) \xrightarrow{q} H^{\Delta}_k(M_k) \xrightarrow{h_{\mathrm{pt}}} U(1)
    \]
    is indeed a differential character. This shows that the inverse map is well-defined. To see that it defines an inverse, let $(\chi, A_{\chi})$ denote a degree $(k+1)$-differential character. Then we have that
    \[
   \Psi\left( \Phi(\chi, A_{\chi}) \right) = (\Phi(\chi, A_{\chi}))_{\mathrm{pt}} \circ q = \chi
    \]
    by construction of the morphism $\Phi$. Also since the curvature $A_{\chi}$ is uniquely determined by the character $\chi$ it follows that
    \[
    \Psi\left( \Phi(\chi, A_{\chi}) \right) = (\chi, A_{\chi}).
    \]
    For the other direction, let $h$ be a smooth holonomy map. Since $U(1)$ is a diffeological group, i.e. a concrete sheaf of abelian groups on $\mathbf{Cart}$ it follows that there is a chain of isomorphisms
    \begin{align*}
    \mathrm{Hom}_{\mathrm{PSh}(\mathbf{Cart},\mathbf{Ab})} \left( H_k(\mathbb{M}_k) ,U(1) \right) &\cong \mathrm{Hom}_{\mathrm{CPSh}(\mathbf{Cart},\mathbf{Ab})} \left( C \left( H_k(\mathbb{M}_k) \right) ,U(1) \right) \\
    &\cong \mathrm{Hom}_{\mathbf{DiffAb}} \left( \left(C\left( H_k(\mathbb{M}_k)\right)\right)^{\#} ,U(1) \right)
    \end{align*}
    Therefore, checking the equality
    \[
    \Phi(\Psi(h)) = h
    \]
    as morphisms of presheaves, is equivalent to checking the equality as morphisms of diffeological abelian groups. By concreteness, this amounts to check their equality only on the points. Indeed, for all $\sigma \in Z_{k}(M)$
    \begin{align*}
        \Phi(\Psi(h))([\sigma]_k)_{\mathrm{pt}} &= \Phi \left( h_{\mathrm{pt}} \circ q \right)_{\mathrm{pt}}([\sigma]_k) \\
        &= h_{\mathrm{pt}}([\sigma])
    \end{align*}
    which concludes the proof.
\end{proof}

The proof of the main theorem, that is Theorem \ref{Theorem: The Main Theorem}, follows now directly.

\begin{proof}[Proof of Theorem \ref{Theorem: The Main Theorem}]
    First recall that the case $k = 0$ has been already proved separately in Proposition \ref{Proposition: The case k = 0}. Hence assume $k>0$. Then Theorem \ref{Corollary: The morphism Phi is an iso} together with Corollary \ref{Corollary: Derived Hom space in the case U(1)} show that there is a sequence of isomorphisms of abelian groups
    \begin{align*}
        \hat{H}^{k+1}(M;\mathbb{Z}) &\cong \mathrm{Hom}_{\mathrm{PSh}(\mathbf{Cart},\mathbf{Ab})} \left( H_k(\mathbb{M}_k) ,U(1) \right) \\
        &\cong \mathrm{Hom}_{\mathrm{Sh}(\mathbf{Cart},\mathbf{Ab})} \left( \tilde{H}_k(\mathbb{M}_k) ,U(1) \right) \\
        &\cong H^k_{\infty} \left( \mathbb{M}_k,U(1) \right).
    \end{align*}
\end{proof}

\begin{remark} \label{Remark: The form A_h via the Van Est map}
    As mentioned earlier, the curvature form $A_h$ extracted from a holonomy map $h$ can also be constructed by directly specifying an $S_{k+1}$-antisymmetric, normalized smooth cochain $\Omega_h \in C^{k+1}(\mathrm{Pair}(M),\mathbb{R})$ such that the curvature form $A_h$ is then defined as its image under the Van Est map:
    \[
    A_h := VE(\Omega_h).
    \]
    Lemma \ref{Lemma: A_h is antisymmetric, smooth and multilinear} is then the analog of proving that the cochain $\Omega_h$ is indeed $S_{k+1}$-antisymmetric and normalized. An interesting observation is that in contrast to the approach presented earlier where the proof of antisymmetry and linearity of $A_h$ relied on the very technical Lemma \ref{Lemma: The pointwise form is well defined}, the approach via the Van Est morphism provides a much cleaner way of defining the curvature form. However, this has the downside of a broader theoretical effort introducing the Van Est morphism together with the additional requirement for $M$ to be a compact smooth manifold, as we will observe next. \\

    Recall from Definition \ref{Definition: Smooth normalized Cochains} that a normalized smooth cochain $\Omega_h \in C^{k+1}(\mathrm{Pair}(M),\mathbb{R})$ is a smooth function
    \[
    \Omega_h \colon U_M \subset \overbrace{M \times \cdots \times M}^{(k+2)\text{-times}} \rightarrow \mathbb{R}
    \]
    defined on a neighborhood $U_M$ of the diagonal $\{ (p,...,p) \in M^{k+2} \, | \, p \in M \}$. Given now a smooth holonomy morphism $h \colon H_k(\mathbb{M}_k) \rightarrow U(1)$ the most natural choice for $\Omega_h$ is to associate to a $(k+2)$-tuple $(p_0,...,p_{k+1}) \in U_M$ some smooth $(k+1)$-simplex $[p_0,...,p_{k+1}] \colon |\Delta^{k+1}| \rightarrow M$ in $M$ spanned by the points $\{p_i\}_{0 \leq i \leq k+1}$ and then set
    \[
    \Omega_h (p_0,...,p_{k+1})  = \frac{1}{2 \pi i} \cdot \mathrm{log} \left( h \left( [ \partial [p_0, ..., p_{k+1}] ]_k \right) \right).
    \]
    The main difficulty is how to construct a smooth $(k+2)$-simplex in $M$ given a $(k+2)$-tuple of points $(p_0, ..., p_{k+1})$. To do so, assume that $M$ is compact and endow $M$ with a Riemannian metric $(M,g)$. At every point $p \in M$ there now exist \textit{normal coordinates centered at} $p$ provided by the exponential map $ U \xrightarrow{\mathrm{exp}_p^{-1}} T_pM \xrightarrow{B} \mathbb{R}^n$. The compactness of $M$ guarantees the existence of a small enough neighborhood $U_M$ of the diagonal, such that for every $(p_0,...,p_{k+1}) \in U_M$ there exists a normal coordinate chart $\varphi \colon U \rightarrow \mathbb{R}^n$ centered at $p_0$ with $p_i \in U$ for all $0 \leq i \leq k+1$. Then, define the smooth $(k+1)$-simplex $[p_0,...,p_{k+1}]$ by
    \begin{equation}
    [p_0,...,p_{k+1}] := \varphi^{-1} \circ [\varphi(p_0), ..., \varphi(p_{k+1})]
    \end{equation}
    where $[\varphi(p_0), ..., \varphi(p_{k+1})]$ denotes the linear map $|\Delta^{k+1}| \rightarrow \mathbb{R}^n$ specified by sending the $i$th vertex of $|\Delta^{k+1}|$ to $\varphi(p_i)$. Then define
    \[
    \Omega_h(p_0,...,p_{k+1}) := \frac{1}{2 \pi i} \cdot \mathrm{log} \left( h \left( [ \partial [p_0, ..., p_{k+1}] ]_k \right) \right).
    \]
    To show that this gives indeed a well-defined smooth map, assume we are given two normal coordinate charts $\varphi$ and $\psi$ defined on $U$. The advantage of using normal coordinates lies in the fact that the transition functions $\psi \circ \varphi^{-1}$ are given by some orthogonal linear transformation $A \in O(n)$\footnote{See for example \cite[Proposition 5.23]{Lee-Riemannian}.}. We now wish to show that
    \[
    \varphi^{-1} \circ [\varphi(p_0), ...,\varphi(p_{k+1})] =   \psi^{-1} \circ [\psi(p_0), ...,\psi(p_{k+1})].
    \]
    First, it follows from the linearity of the transition function $\psi \circ \varphi^{-1} = A$ that
    \[
    A \circ [\varphi(p_0), ...,\varphi(p_{k+1})] = [\psi(p_0), ...,\psi(p_{k+1})]
    \]
    which then in turn indeed shows that
    \begin{align*}
    \psi^{-1} \circ [\psi(p_0), ...,\psi(p_{k+1})] &= \psi^{-1} \circ A \circ [\varphi(p_0), ...,\varphi(p_{k+1})] \\
    &= \psi^{-1} \circ \psi \circ \varphi^{-1} \circ [\varphi(p_0), ...,\varphi(p_{k+1})] \\
    &= \varphi^{-1} \circ [\varphi(p_0), ...,\varphi(p_{k+1})]
    \end{align*}
    and as such, it shows that the partial function $\Omega_h$ is well-defined. \\

    It is left to show that $\Omega_h$ defines a normalized and $S_{k+1}$-antisymmetric cochain. The fact that $\Omega_h$ is normalized follows from the observation that if we are given a tuple $(p_0, ..., p_{k+1}) \in U_M$ such that there is some $1 \leq i \leq k+1$ with $p_i = p_0$, then the smooth $(k+1)$-simplex $[p_0,...,p_{k+1}]$ is taking values in $M_k$, i.e. is $k$-thin. Therefore, the homology class vanishes $[\partial [p_0, ..., p_{k+1}] ]_k = 0$. The $S_{k+1}$-antisymmetry follows again from Corollary \ref{Corollary: Change of Orientation}.
\end{remark}

\chapter{Geometric Loop Groups} \label{Chapter: Geometric Loop Groups}

\section{The Geometric Loop Group} \label{Section: Geometric and Simplicial Loop Groups}

The main concept upon which the definition of the \textit{geometric loop group} arises is the classical path fibration. That is, given a pointed topological manifold $(M,x_0)$ we denote by $\Omega(M,x_0)$ the topological space of continuous loops in $M$ based at $x_0 \in M$. Similarly, we denote by $P(M,x_0)$ the topological space of continuous paths in $M$ starting at $x_0 \in M$. This then gives the well-known path-space fibration
\[
    \Omega(M,x_0) \hookrightarrow P(M,x_0) \xrightarrow{\mathrm{ev}_1} M
\]
with $P(M,x_0)$ contractible and the evaluation map $\mathrm{ev}_1$ a fibration. However, as discussed in Example 4.44 of Christensen--Wu \cite{Christensen-Wu} the path-space fibration fails to produce a fibration of diffeological spaces as defined in Definition \ref{Definition: fibration of diffeological spaces} for arbitrary diffeological spaces $X$. To see how this fails, assume the evaluation map $\mathrm{ev}_1 \colon P(X,x_0) \rightarrow X$ were a fibration. In particular, the loop space $\Omega(X,x_0)$ which arises as the fiber over the point $x_0 \in X$ is necessarily a fibrant diffeological space and so any diffeological space $X$ whose loop space is non-fibrant provides a counter-example.\\

One way of circumventing the problem of having non-fibrant fibers is to force more structure upon them, such that they become fibrant. More precisely, introduce an equivalence relation $\sim$ on the space of paths $P(X,x_0)$ such that the evaluation map $\mathrm{ev}_1$ passes through this quotient yielding a map
\begin{equation} \label{Equation: Geometric Loop Fibration}
\widetilde{\mathrm{ev}}_1 \colon P(X,x_0) / \! \sim \rightarrow X
\end{equation}
whose fiber over the point $x_0$, denoted $G(X,x_0)$, has the structure of a diffeological group. In particular, the map (\ref{Equation: Geometric Loop Fibration}) turns out to be a diffeological fibration in general, and in the case $(X,x_0)$ is connected (\ref{Equation: Geometric Loop Fibration}) defines a diffeological principal $G(X,x_0)$-bundle over $X$. \\

An important property of diffeological spaces that will be used various times is the following.

\begin{lemma}[\cite{Iglesias-Zemmour}, Art. 1.51]\label{Lemma: Smoothness and Subductions}
    Let $X,X'$ and $X''$ be diffeological spaces, and let $\pi \colon X \rightarrow X'$ be a subduction. A map $f \colon X' \rightarrow X''$ is smooth if and only if $f \circ \pi$ is smooth.
    \begin{center}
    \begin{tikzcd}
X \arrow[d, "\pi"'] \arrow[rd, "f \circ  \pi"] &     \\
X' \arrow[r, "f"']                             & X''
\end{tikzcd}
\end{center}

    In particular, $f$ is a subduction if and only if $f \circ \pi$ is a subduction.
\end{lemma}

\begin{definition}[\cite{GajerHigher}]\label{Definition: geometric path space}
    Let $(X,x_0)$ be a pointed diffeological space. Introduce the following three relations $\sim_1, \sim_2, \sim_3$ on the diffeological based path space
    \[
    P(X,x_0) := D \left( (I,\{0\}) ,(X,\{x_0\}) \right)
    \]
    where $I = [0,1]$ is endowed with the subset diffeology of $\mathbb{R}$.
    \begin{enumerate}
        \item \textsc{Reparametrization:} For $\gamma \in P(X,x_0)$ and every orientation preserving diffeomorphism $\varphi$ of $I$ we have
        \[
        \gamma \sim_1 \gamma \circ \varphi.
        \]
        Notice that since $\varphi$ is orientation preserving we have $\varphi(0) = 0$ and $\varphi(1) = 1$ and as such $\gamma \circ \varphi \in P(X,x_0)$.

        \item \textsc{Sitting instances:} For $\gamma \in P(X,x_0)$ and every smooth map $\psi \colon I \rightarrow I$ such that there exists a proper subinterval $I_0 \subset I$ where $\psi|_{I_0}$ is a constant map with $\psi(I_0) = t_0 \in I$ and $\psi|_{I - I_0}$ is an orientation preserving diffeomorphism onto $I - \{t_0 \}$, then
        \[
        \gamma \sim_2 \gamma \circ \psi.
        \]

        \item \textsc{Spikes:} For $\gamma \in P(X,x_0)$ of the form $\alpha * \beta * \beta^{-1} * \eta$ where $\alpha$ and $\beta$ are allowed to be empty paths the relation is given by
        \[
        \gamma = \alpha * \beta * \beta^{-1} * \eta  \sim_3 \alpha * \eta.
        \]
    \end{enumerate}
    Denote by $\sim$ the equivalence relation generated by the three relations $\sim_1, \sim_2$ and $\sim_3$. The \textbf{geometric path space} $E(X,x_0)$ is defined as the quotient diffeological space of $P(X,x_0)$ under the equivalence relation $\sim$.
\end{definition}
\newpage
\begin{remark} $ $
\begin{itemize}
    \item The second relation $\sim_2$ is non-empty by the existence of smooth cut-off functions.
    \item Recall that given a set $X$ a relation denotes a subset $R \subset X \times X$. The \textit{equivalence relation generated by} $R$ is such that $x \sim y$ if and only if one of the following statements hold
    \begin{enumerate}
        \item $x = y$
        \item $x R' y$
        \item There exists a finite sequence of elements $z_1, ...,z_n$ in $X$ such that
        \[
        x R' z_1 R' \cdots R' z_n R' y
        \]
    \end{enumerate}
    where here the expression $x R' y$ means $(x,y) \in R$ or $(y,x) \in R$.
    \item Given a pointed diffeological space $(X,x_0)$ the evaluation at $1$ defines a smooth map
        \[
        \begin{array}{rcl}
             \mathrm{ev}_1 \colon P(X,x_0) & \rightarrow & X \\
             \gamma & \mapsto & \gamma(1)
        \end{array}
        \]
        which factors through the quotient and defines a smooth map.
        \[
        \begin{array}{rcl}
             \widetilde{\mathrm{ev}}_1 \colon E(X,x_0) & \rightarrow & X \\
             {[}\gamma{]} & \mapsto & {[}\gamma(1){]}
        \end{array}
        \]
        To show that $\widetilde{\mathrm{ev}}_1$ is well defined it suffices to note that the three relations $\sim_1$,$\sim_2$ and $\sim_3$ are all endpoint preserving. Therefore, via the commutative diagram
        \begin{center}
        \begin{tikzcd}
        {P(X,x_0)} \arrow[rd, "\mathrm{ev}_1"] \arrow[d, "q"']      &   \\
        {E(X,x_0)} \arrow[r, "\mathrm{\widetilde{\mathrm{ev}}_1}"'] & X
        \end{tikzcd}
        \end{center}
        the smoothness of $\widetilde{\mathrm{ev}}_1$ is now a consequence of Lemma \ref{Lemma: Smoothness and Subductions} using the fact that $q$ is a subduction by construction.
        \end{itemize}
\end{remark}

\begin{definition} \label{Definition: Geometric Loop Group}
    Let $(X,x_0)$ be a pointed diffeological space. The fiber of the smooth map $\widetilde{\mathrm{ev}}_1 \colon E(X,x_0) \rightarrow X$ over the basepoint $x_0$ endowed with the subset diffeology is called the \textbf{geometric loop group} and denoted by $G(X,x_0)$. Its elements $[\gamma] \in G(X,x_0)$ are given by equivalence classes of smooth loops $\gamma \in \Omega(X,x_0)$ based at $x_0$. The equivalence class of the constant loop at $x_0$ is simply denoted by $[x_0]$.
\end{definition}

As the name suggests, the geometric loop group can be equipped with the structure of a diffeological group. The multiplication operation is given by concatenation of loops. From this perspective, the three relations $\sim_1$,$\sim_2$ and $\sim_3$ can now be identified each with one of the group axioms. That is, $\sim_1$ makes sure that the operation of concatenation is associative, $\sim_2$ makes sure that the class of the constant path $[x_0]$ is the identity in $G(X,x_0)$ and $\sim_3$ establishes that the class of the reverse loop $\gamma^{-1} \colon t \mapsto \gamma(1-t)$ is the inverse of $[\gamma]$. This is the content of the next Lemma.

\begin{lemma}\label{Lemma: Geometric Loop Group is smooth Group}
    Given a pointed diffeological space $(X,x_0)$ the diffeological space $G(X,x_0)$ admits the structure of a diffeological group where the multiplication is given by the concatenation of loops in $X$. The identity element is given by the class $[x_0]$ of the constant loop at $x_0$ and the inverse of some element $[\gamma]$ by the class $[\gamma^{-1}]$.
\end{lemma}
To show that the operations of concatenation and taking inverses are indeed smooth, proceed by first showing that the corresponding morphisms are smooth before passing to the quotient. Then after showing that these maps all pass through the quotient, smoothness is then implied by Lemma \ref{Lemma: Smoothness and Subductions}. The concatenation of paths however is not well defined for smooth loops $\Omega(X,x_0) = D\left( (I,\{0,1\}),(X,\{x_0\}) \right)$. Therefore, consider the diffeological path space $P_{\mathrm{st}}(X) \subset D(I,X)$ of paths with sitting instants, also called stationary paths, in $X$. Recall from section \ref{Subsection: Diffeological Spaces} that a path $\gamma \colon I \rightarrow X$ has sitting instants if there exists $0 < \varepsilon < 1/2$ such that $\gamma$ is constant on $[0,\varepsilon)$ and $(1-\varepsilon,1]$.

\begin{lemma}[\cite{Collier-Lerman-Wolbert}, Lemma A.22.] \label{Lemma: Smooth Concatenation}
 Let $X$ be a diffeological space. The concatenation of paths with sitting instants in $X$ defined by
 \[
 \left( \gamma_0 * \gamma_1\right)(t) := \begin{cases}
     \gamma_0(2t) &\text{ for } t \in [0,1/2] \\
     \gamma_1(2t - 1) &\text{ for } t \in [1/2,1]
 \end{cases}
 \]
 is well defined and yields a smooth map of diffeological spaces
 \[
 - * - \colon P_{\mathrm{st}}(X)  \tensor[_{\mathrm{ev}_1}]{\times}{_{\mathrm{ev}_0}} P_{\mathrm{st}}(X) \rightarrow  P_{\mathrm{st}}(X)
 \]
\end{lemma}
\begin{proof}
  The proof \cite[Lemma A.22.]{Collier-Lerman-Wolbert} considering only the special case where $X$ is a smooth manifold generalizes to arbitrary diffeological spaces without further adjustments.
\end{proof}

\begin{proof}[Proof of Lemma \ref{Lemma: Geometric Loop Group is smooth Group}]
Consider first the case of concatenation. The smooth concatenation map of Lemma \ref{Lemma: Smooth Concatenation} restricts to a smooth map
\[
* \colon \Omega_{\mathrm{st}}(X,x_0) \times \Omega_{\mathrm{st}}(X,x_0) \rightarrow \Omega_{\mathrm{st}}(X,x_0)
\]
of diffeological loop spaces with sitting instants. By construction the quotient map $q \colon \Omega(X,x_0) \rightarrow G(X,x_0)$ is a subduction. Relation $\sim_2$ however makes sure that every class $[\gamma] \in G(X,x_0)$ contains a representative $\gamma' \in [\gamma]$ having sitting instants. Therefore, also the quotient map $\Omega_{\mathrm{st}}(X,x_0) \rightarrow G(X,x_0)$ is a subduction \footnote{See \cite[Art. 1.53.]{Iglesias-Zemmour} for more details.}. Considering Lemma \ref{Lemma: Smoothness and Subductions} it thus suffices to show that the concatenation map descends to the quotient, i.e. there exists a map making the square commutative.
\begin{center}
    \begin{tikzcd}
{\Omega_{\mathrm{st}}(X,x_0) \times \Omega_{\mathrm{st}}(X,x_0)} \arrow[d, "q \times q"'] \arrow[r, "*"] & {\Omega_{\mathrm{st}}(X,x_0)} \arrow[d, "q"] \\
{G(X,x_0) \times G(X,x_0)} \arrow[r, dashed]                                                             & {G(X,x_0)}
\end{tikzcd}
\end{center}
Then $q \times q$ being a subduction then implies the smoothness of said map. For $[\gamma],[\gamma'] \in G(X,x_0)$ let
\[
[\gamma] * [\gamma'] := [\gamma * \gamma']
\]
To see that this operation is well defined it suffices to check this for the three relations $\sim_1$,$\sim_2$, and $\sim_3$.
\begin{itemize}
    \item Relation $\sim_1$: Consider $\gamma,\gamma' \in \Omega_{\mathrm{st}}(X,x_0)$ and let $\varphi,\varphi' \in \mathrm{Diff}^+(I)$ be two orientation preserving diffeomorphisms. Then $\gamma \sim_1 \gamma \circ \varphi$ and $\gamma' \sim_1 \gamma' \circ \varphi'$. It is to show that there exists an orientation preserving diffeomorphism $\varphi'' \in \mathrm{Diff}^+(I)$ such that
    \[
    (\gamma * \gamma') \circ \varphi'' = (\gamma \circ \varphi) * (\gamma' \circ \varphi')
    \]
    First, denote by $\varphi * \varphi'$ the piecewise smooth diffeomorphism given by
    \[
    (\varphi * \varphi')(t) :=
    \begin{cases}
    \frac{\varphi(2t)}{2} &\text{ for } t \in [0,1/2] \\
   \frac{1 +  \varphi'(2t-1)}{2} &\text{ for } t \in [1/2,1].
    \end{cases}
    \]
    Since both loops $\gamma,\gamma' \in \Omega_{\mathrm{st}}(X,x_0)$ have sitting instants, we can assume that there is an $0 < \varepsilon < \frac{1}{2}$ such that both loops are constant on $[0,\varepsilon)$ and $(1- \varepsilon, 1]$. The concatenation $\gamma * \gamma'$ is then constant on the neighborhood $I_{\varepsilon}:= (\frac{1}{2}- \frac{\varepsilon}{2}, \frac{1}{2} + \frac{\varepsilon}{2})$. Now let $0 < \delta$ be such that $(\frac{1}{2} - \delta , \frac{1}{2} + \delta) \subset (\varphi * \varphi')^{-1}(I_{\varepsilon})$. Using mollifiers and cut-off functions\footnote{See for example \cite[Theorem 2.5]{Hirsch}.} one can show that there exists a smooth orientation preserving diffeomorphism $\varphi''$ of $I$ such that
    \[
    \varphi''(t) = \begin{cases}
        \frac{\varphi(2t)}{2} &\text{ for } t \in [0,1/2 - \delta] \\
         \frac{1 +  \varphi'(2t-1)}{2} &\text{ for } t \in [1/2 + \delta,1],
    \end{cases}
    \]
    i.e. $\varphi''$ is smooth and agrees with $\varphi * \varphi'$ outside of a neighborhood of $1/2$, see Figure \ref{fig: Concatenation well defined}.
    \begin{figure}[h]
        \centering
        \begin{tikzpicture}[x=0.75pt,y=0.75pt,yscale=-1,xscale=1]

\draw    (150,170) -- (150,32) ;
\draw [shift={(150,30)}, rotate = 90] [color={rgb, 255:red, 0; green, 0; blue, 0 }  ][line width=0.75]    (10.93,-3.29) .. controls (6.95,-1.4) and (3.31,-0.3) .. (0,0) .. controls (3.31,0.3) and (6.95,1.4) .. (10.93,3.29)   ;
\draw    (150,170) -- (288,170) ;
\draw [shift={(290,170)}, rotate = 180] [color={rgb, 255:red, 0; green, 0; blue, 0 }  ][line width=0.75]    (10.93,-3.29) .. controls (6.95,-1.4) and (3.31,-0.3) .. (0,0) .. controls (3.31,0.3) and (6.95,1.4) .. (10.93,3.29)   ;
\draw    (210,170) -- (210,175) ;
\draw    (270,170) -- (270,175) ;
\draw    (145,50) -- (150,50) ;
\draw    (145,110) -- (150,110) ;
\draw [color={rgb, 255:red, 208; green, 2; blue, 27 }  ,draw opacity=1 ]   (150,170) .. controls (177.25,167) and (199.17,143.33) .. (210,110) ;
\draw [color={rgb, 255:red, 208; green, 2; blue, 27 }  ,draw opacity=1 ]   (210,110) .. controls (261.25,104) and (218.25,60.5) .. (270,50) ;
\draw  [dash pattern={on 0.84pt off 2.51pt}]  (150,100) -- (235,100) ;
\draw  [dash pattern={on 0.84pt off 2.51pt}]  (150,120) -- (205,120) ;
\draw  [dash pattern={on 0.84pt off 2.51pt}]  (205,120) -- (205,170) ;
\draw  [dash pattern={on 0.84pt off 2.51pt}]  (235,100) -- (235,170) ;
\draw    (206.63,119) .. controls (209.13,112.25) and (221.13,110.25) .. (229.38,103.75) ;
\draw    (273.25,125) .. controls (254.92,133.2) and (237.12,124.02) .. (220.55,116.78) ;
\draw [shift={(218.75,116)}, rotate = 23.23] [color={rgb, 255:red, 0; green, 0; blue, 0 }  ][line width=0.75]    (10.93,-3.29) .. controls (6.95,-1.4) and (3.31,-0.3) .. (0,0) .. controls (3.31,0.3) and (6.95,1.4) .. (10.93,3.29)   ;

\draw (132,42.4) node [anchor=north west][inner sep=0.75pt]    {$1$};
\draw (130,89.4) node [anchor=north west][inner sep=0.75pt]    {$\frac{1}{2}$};
\draw (262,177.4) node [anchor=north west][inner sep=0.75pt]    {$1$};
\draw (201,177.4) node [anchor=north west][inner sep=0.75pt]    {$\frac{1}{2}$};
\draw (186,57.9) node [anchor=north west][inner sep=0.75pt]    {$\varphi \ *\ \varphi '$};
\draw (277.5,113.9) node [anchor=north west][inner sep=0.75pt]    {$\varphi ''$};

\end{tikzpicture}
        \caption{}
        \label{fig: Concatenation well defined}
    \end{figure}

    In particular one computes
    \[
    (\gamma * \gamma') \circ \varphi''(t) =
    \begin{cases}
        \gamma(\varphi(2t)) &\text{ for } t \in [0,1/2 - \delta] \\
        \gamma(1) = \gamma'(0) &\text{ for } t \in [\frac{1}{2} - \delta , \frac{1}{2} + \delta] \\
        \gamma'( \varphi'(2t-1)) &\text{ for } t \in [1/2 + \delta,1],
    \end{cases}
    \]
    which shows that
    \[
    (\gamma * \gamma') \circ \varphi'' = (\gamma \circ \varphi) * (\gamma' \circ \varphi').
    \]
    This then implies by definition
    \[
    (\gamma \circ \varphi) * (\gamma' \circ \varphi') \sim_1 (\gamma * \gamma').
    \]

    \item Relation $\sim_2$: This can be shown using the same strategy as for relation $\sim_1$.
    \item Relation $\sim_3$: This follows directly from the definition of $\sim_3$.
\end{itemize}

Next, notice that the assignment
\[
\begin{array}{rcl}
     (-)^{-1} \colon \Omega(X,x_0) & \rightarrow & \Omega(X,x_0)  \\
     \gamma & \mapsto & \gamma^{-1}
\end{array}
\]
where $\gamma^{-1}(t) := \gamma(1-t)$ denotes the reversed path, is smooth. By the same argument as for the concatenation it suffices to show that this map factors through the quotient since the smoothness is then automatically implied. That is, define
\[
[\gamma]^{-1} := [\gamma^{-1}]
\]
for any class $[\gamma] \in G(X,x_0)$. To check that its well defined consider $\gamma \sim_1 \gamma \circ \varphi$ for $\varphi \in \mathrm{Diff}^+(I)$. Then define $\varphi'(t):= 1 - \varphi(1-t)$ which is again an orientation preserving diffeomorphism. Then a simple computation shows:
\begin{align*}
(\gamma \circ \varphi)^{-1}(t) &= \gamma(\varphi(1-t)) \\
&= \gamma(1 - \varphi'(t)) \\
&= (\gamma^{-1} \circ \varphi')(t)
\end{align*}
i.e. $(\gamma \circ \varphi)^{-1} \sim_1 \gamma^{-1}$. Again the same strategy can be applied for relation $\sim_2$ and for the relation $\sim_3$ it follows directly. \\

To conclude, there are well-defined smooth maps
\begin{align*}
    * \colon G(X,x_0) \times G(X,x_0) \rightarrow G(X,x_0) \\
    (-)^{-1} \colon G(X,x_0) \rightarrow G(X,x_0)
\end{align*}
given by concatenation and reversing loops. What is left to show is that together with the class of the constant loop $[x_0]$ this data satisfies the axioms of a diffeological group. Indeed, the associativity of the concatenation follows from relation $\sim_1$ using an orientation preserving diffeomorphism $\alpha$ obtained from the piecewise-smooth function $\tilde{\alpha}$ of Figure \ref{fig: Associativity and left/right inverse} by smooth approximation. That is, there exists an orientation preserving diffeomorphism $\alpha$ which agrees with $\tilde{\alpha}$ outside a neighborhood of the two points $1/2$ and $3/4$, as we have argued earlier.

The fact that for any class $[\gamma]$ one has $[\gamma] * [\gamma^{-1}] = [\gamma * \gamma^{-1}] = [\gamma]$ follows directly from relation $\sim_3$. And finally, the smooth maps $\psi_0$ and $\psi_1$, again constructed by smooth approximation from the piecewise-smooth maps $\widetilde{\psi_0}$ and $\widetilde{\psi_1}$ of Figure \ref{fig: Associativity and left/right inverse}, show that indeed by the second relation
\begin{align*}
    [x_0] * [\gamma] = [\gamma] \text{ and } [\gamma] * [x_0] = [\gamma].
\end{align*}

\begin{figure}[h]
    \centering
    \begin{tikzpicture}[x=0.75pt,y=0.75pt,yscale=-1,xscale=1]

\draw    (40,160) -- (40,22) ;
\draw [shift={(40,20)}, rotate = 90] [color={rgb, 255:red, 0; green, 0; blue, 0 }  ][line width=0.75]    (10.93,-3.29) .. controls (6.95,-1.4) and (3.31,-0.3) .. (0,0) .. controls (3.31,0.3) and (6.95,1.4) .. (10.93,3.29)   ;
\draw    (40,160) -- (178,160) ;
\draw [shift={(180,160)}, rotate = 180] [color={rgb, 255:red, 0; green, 0; blue, 0 }  ][line width=0.75]    (10.93,-3.29) .. controls (6.95,-1.4) and (3.31,-0.3) .. (0,0) .. controls (3.31,0.3) and (6.95,1.4) .. (10.93,3.29)   ;
\draw [color={rgb, 255:red, 208; green, 2; blue, 27 }  ,draw opacity=1 ]   (40,160) -- (100,160) ;
\draw [color={rgb, 255:red, 208; green, 2; blue, 27 }  ,draw opacity=1 ]   (100,160) -- (160,40) ;
\draw    (100,160) -- (100,165) ;
\draw    (160,160) -- (160,165) ;
\draw    (35,40) -- (40,40) ;
\draw    (35,100) -- (40,100) ;
\draw    (260,160) -- (260,22) ;
\draw [shift={(260,20)}, rotate = 90] [color={rgb, 255:red, 0; green, 0; blue, 0 }  ][line width=0.75]    (10.93,-3.29) .. controls (6.95,-1.4) and (3.31,-0.3) .. (0,0) .. controls (3.31,0.3) and (6.95,1.4) .. (10.93,3.29)   ;
\draw    (260,160) -- (398,160) ;
\draw [shift={(400,160)}, rotate = 180] [color={rgb, 255:red, 0; green, 0; blue, 0 }  ][line width=0.75]    (10.93,-3.29) .. controls (6.95,-1.4) and (3.31,-0.3) .. (0,0) .. controls (3.31,0.3) and (6.95,1.4) .. (10.93,3.29)   ;
\draw    (320,160) -- (320,165) ;
\draw    (380,160) -- (380,165) ;
\draw    (255,40) -- (260,40) ;
\draw    (255,100) -- (260,100) ;
\draw    (480,160) -- (480,22) ;
\draw [shift={(480,20)}, rotate = 90] [color={rgb, 255:red, 0; green, 0; blue, 0 }  ][line width=0.75]    (10.93,-3.29) .. controls (6.95,-1.4) and (3.31,-0.3) .. (0,0) .. controls (3.31,0.3) and (6.95,1.4) .. (10.93,3.29)   ;
\draw    (480,160) -- (618,160) ;
\draw [shift={(620,160)}, rotate = 180] [color={rgb, 255:red, 0; green, 0; blue, 0 }  ][line width=0.75]    (10.93,-3.29) .. controls (6.95,-1.4) and (3.31,-0.3) .. (0,0) .. controls (3.31,0.3) and (6.95,1.4) .. (10.93,3.29)   ;
\draw    (540,160) -- (540,165) ;
\draw    (600,160) -- (600,165) ;
\draw    (475,40) -- (480,40) ;
\draw    (475,100) -- (480,100) ;
\draw [color={rgb, 255:red, 208; green, 2; blue, 27 }  ,draw opacity=1 ]   (260,160) -- (320,40) ;
\draw [color={rgb, 255:red, 208; green, 2; blue, 27 }  ,draw opacity=1 ]   (320,40) -- (380,40) ;
\draw    (475,130) -- (480,130) ;
\draw    (570,160) -- (570,165) ;
\draw [color={rgb, 255:red, 208; green, 2; blue, 27 }  ,draw opacity=1 ]   (480,160) -- (540,130) ;
\draw [color={rgb, 255:red, 208; green, 2; blue, 27 }  ,draw opacity=1 ]   (540,130) -- (570,100) ;
\draw [color={rgb, 255:red, 208; green, 2; blue, 27 }  ,draw opacity=1 ]   (570,100) -- (600,40) ;

\draw (22,32.4) node [anchor=north west][inner sep=0.75pt]    {$1$};
\draw (20,79.4) node [anchor=north west][inner sep=0.75pt]    {$\frac{1}{2}$};
\draw (152,167.4) node [anchor=north west][inner sep=0.75pt]    {$1$};
\draw (91,167.4) node [anchor=north west][inner sep=0.75pt]    {$\frac{1}{2}$};
\draw (242,32.4) node [anchor=north west][inner sep=0.75pt]    {$1$};
\draw (240,79.4) node [anchor=north west][inner sep=0.75pt]    {$\frac{1}{2}$};
\draw (372,167.4) node [anchor=north west][inner sep=0.75pt]    {$1$};
\draw (311,167.4) node [anchor=north west][inner sep=0.75pt]    {$\frac{1}{2}$};
\draw (462,32.4) node [anchor=north west][inner sep=0.75pt]    {$1$};
\draw (460,74.4) node [anchor=north west][inner sep=0.75pt]    {$\frac{1}{2}$};
\draw (592,167.4) node [anchor=north west][inner sep=0.75pt]    {$1$};
\draw (531,167.4) node [anchor=north west][inner sep=0.75pt]    {$\frac{1}{2}$};
\draw (527,82.4) node [anchor=north west][inner sep=0.75pt]    {$\tilde{\alpha }$};
\draw (78,83.4) node [anchor=north west][inner sep=0.75pt]    {$\widetilde{\psi _{0}}$};
\draw (337,83.4) node [anchor=north west][inner sep=0.75pt]    {$\widetilde{\psi _{1}}$};
\end{tikzpicture}
\caption{}
\label{fig: Associativity and left/right inverse}
\end{figure}

\end{proof}
\begin{remark}
Using the same arguments as above it follows that the geometric loop group $G(X,x_0)$ acts smoothly on the geometric path space $E(X,x_0)$ by concatenation. That is, there is a smooth map
\[
\begin{array}{rcl}
     G(X,x_0) \times E(X,x_0) & \rightarrow & E(X,x_0)  \\
     ([\gamma] , [\delta]) & \mapsto & [\gamma] * [\delta] := [\gamma * \delta]
\end{array}
\]
\end{remark}

The next step is to show that $E(X,x_0) \rightarrow X$ is a diffeological principal $G(X,x_0)$-bundle whenever $(X,x_0)$ is connected. Recall from Definition \ref{Definition: Diffeological bundle} that for $\widetilde{\mathrm{ev}}_1$ to be a bundle it is required for this map to be a subduction.

\begin{lemma} \label{Lemma: Evaluation function is subduction for connected case}
Let $(X,x_0)$ be a connected and pointed diffeological space. Then the evaluation map $\mathrm{ev}_1 \colon P(X,x_0) \rightarrow X$ is a subduction.
\end{lemma}
\begin{proof}
    The proof \cite[Art. 5.6.]{Iglesias-Zemmour} using the path space $D \left( (\mathbb{R},\{0\}) ,(X,\{x_0\})\right)$ instead of $P(X,x_0)$ goes through without any necessary modification.
\end{proof}

\begin{proposition} \label{Proposition: principal bundle associated to geometric loop group}
Let $(X,x_0)$ be a connected, pointed diffeological space. Then the evaluation map
\[
\widetilde{\mathrm{ev}}_1 : E(X,x_0) \rightarrow X, \hspace{3mm} [\gamma] \mapsto \gamma(1)
\]
is a diffeological principal $G(X,x_0)$-bundle. For a general pointed diffeological space $(X,x_0)$, i.e. not necessarily connected, the evaluation map is a fibration of diffeological spaces.
\end{proposition}
\begin{proof}
    Consider first the case where $(X,x_0)$ is connected. By Definition \ref{Definition: Diffeological bundle} it suffices to show that
    \begin{enumerate}[label={(\arabic*)}]
        \item the evaluation map $\widetilde{\mathrm{ev}}_1 : E(X,x_0) \rightarrow X$ is a subduction,
        \item the action of $G(X,x_0)$ is fiber preserving, and
        \item the shear map $\tau \colon G(X,x_0) \times E(X,x_0) \rightarrow E(X,x_0) \times_X E(X,x_0)$ is a diffeomorphism.
    \end{enumerate}

    Property $(1)$ follows from the fact that the quotient map $q \colon P(X,x_0) \rightarrow E(X,x_0)$ is a subduction by construction. Indeed, since then Lemma \ref{Lemma: Evaluation function is subduction for connected case} together with Lemma \ref{Lemma: Smoothness and Subductions} show that $\widetilde{\mathrm{ev}}_1$ is a subduction. \\

    Property $(2)$ follows from construction of the $G(X,x_0)$-action on $E(X,x_0)$ via concatenation. \\

    To prove property $(3)$ recall that the shear map is given by
    \[
    \tau : G(X) \times E^o(X) \rightarrow E^o(X) \times_X E^o(X), \hspace{3mm} ([l],[\gamma]) \mapsto ([\gamma], [l * \gamma]).
    \]
    Since the action of $G(X,x_0)$ on $E(X,_0)$ is smooth also the shear map is smooth. Define a smooth inverse to the shear map by
    \[
    \tau^{-1} : E^o(X) \times_X E^o(X) \rightarrow G(X) \times E^o(X), \hspace{3mm} ([\gamma],[\gamma']) \mapsto ([\gamma' * \gamma^{-1}], [\gamma]).
    \]
   This is also a smooth map and it is readily checked using relation $\sim_3$ that it provides an inverse to $\tau$:
    \begin{align*}
    &\tau(\tau^{-1}([\gamma],[\gamma']) = \tau([\gamma' * \gamma^{-1}],[\gamma]) = ([\gamma], [\gamma' * \gamma^{-1} * \gamma]) = ([\gamma],[\gamma']) \\
    &\tau^{-1}(\tau([l],[\gamma])) = \tau^{-1} ( [\gamma], [l * \gamma]) = ([l * \gamma * \gamma^{-1}], [\gamma]) = ([l],[\gamma])
    \end{align*}  \\

    In the case $(X,x_0)$ is not connected, property $(1)$ fails since any subduction is necessarily surjective. However, any commutative diagram of the form
    \begin{center}
    \begin{tikzcd}
\Lambda^n \arrow[d, "a"'] \arrow[r, "b"] & {E(X,x_0)} \arrow[d, "\widetilde{\mathrm{ev}}_1"] \\
\mathbb{R}^n \arrow[r, "p"']             & X
\end{tikzcd}
    \end{center}
    takes values in the connected component $X_0 \subset X$ of $x_0$. Notice that here $\Lambda^n = \{ (x_1, ..., x_n) \in \mathbb{R}^n \, | \, x_i = 0 \text{ for some } i \}$ is the diffeological space constructed as a coequalizer which has the property that $|\Lambda^n_k|_D \cong \Lambda^n$, i.e. it is diffeomorphic to the diffeological realization of all the simplicial horns $\Lambda^n_k$. Thus the above diagram refines to a diagram of the form
    \begin{center}
    \begin{tikzcd}
\Lambda^n \arrow[d, "a"'] \arrow[r, "b"]        & {E(X,x_0)} \arrow[d, "\widetilde{\mathrm{ev}}_1"'] \arrow[rd, "\widetilde{\mathrm{ev}}_1"] &   \\
\mathbb{R}^n \arrow[r, "p"'] \arrow[ru, dashed] & X_0 \arrow[r, hook]                                                                        & X
\end{tikzcd}
    \end{center}
    whose lift is then constructed using the fact that $E(X,x_0) \rightarrow X_0$ is now a principal $G(X,x_0)$-bundle. More precisely, since $G(X,x_0)$ is a diffeological group it is a fibrant diffeological space\footnote{This is \cite[Proposition 4.30.]{Christensen-Wu}} such that we apply \cite[Proposition 4.28]{Christensen-Wu} which shows that the dashed lift in the refined commutative square indeed exists. This concludes the proof.
\end{proof}

\begin{remark}\label{Remark: Differences with Gajer's Paper} $ $
\begin{enumerate}[label={(\arabic*)}]
\item The geometric loop group $G(X,x_0)$ as given by Definition \ref{Definition: Geometric Loop Group} slightly differs from the version introduced by Gajer in \cite{GajerHigher} and \cite{Gajer_Geometry}. Gajer constructs $E_{\mathrm{ps}}(X,x_0)$ as the quotient space of the space of piecewise smooth paths in $X$. That is, let $P_{\mathrm{ps}}(X,x_0)$ denote the set of piecewise smooth paths in $X$ based at $x_0$. This set can be equipped with a diffeology as follows. A map $p \colon U \rightarrow P_{\mathrm{ps}}(X,x_0)$ for $U \in \mathbf{Cart}$ is a plot if there exists a good open covering $\cat{U}$ of $U$ such that for every $V \in \cat{U}$ there exists a partition of the interval $0 = t_{0} < t_{1} < \cdots < t_{m} = 1$ such that the induced map
\[
p \colon U \times I \rightarrow X
\]
restricted to each $[t_{j-1},t_j] \times V$ is a smooth map into $X$. This has the advantage that one only needs to consider the equivalence relation $\sim$ generated by two relations $\sim_1$ and $\sim_3$, where here $\gamma \sim_1 \gamma \circ \varphi$ for every orientation preserving piecewise diffeomorphism $\varphi$ of the interval $I$ and $\sim_3$ denotes the same relation as introduced earlier. Working with piecewise smooth paths allows reducing relation $\sim_2$ to relation $\sim_3$ by simply considering $\beta$ to be the constant path. Then $E_{ps}(X,x_0) := P_{\mathrm{ps}}(X,x_0) / \sim$ is defined as the diffeological quotient space. \\

The existence of cut-off functions, however, guarantees the existence of a smooth map
\[
\begin{array}{rcl}
    S_{\phi} \colon P_{\mathrm{ps}}(X,x_0) & \rightarrow & P_{\mathrm{st}}(X,x_0)  \\
     \gamma & \mapsto & \gamma_{\phi}
\end{array}
\]
for any choice of cut-off function $\phi \in C^{\infty}([0,1])$. That is, given a piecewise smooth path $\gamma \in P_{\mathrm{ps}}(X,x_0)$ there is a partition of the interval $0 = t_0 < t_1 < \cdots < t_m = 1$ such that the restriction $\gamma_j := \gamma|_{[t_{j-1},t_j]}$ is smooth for all $1 \leq j \leq m$. Likewise set $\phi_j(t):= t_{j-1}(1 - \phi(t)) + \phi(t)t_j$. Then define
\[
\gamma_{\phi} := (\gamma_1 \circ \phi_1 ) * (\gamma_2 \circ \phi_2 ) * \cdots * (\gamma_m \circ \phi_m).
\]
In other words, using cut-off functions, every piecewise smooth path gives rise to a smooth path with sitting instances. Therefore, for the rest of this chapter the smooth versions $E(X,x_0)$ and $G(X,x_0)$ are preferred over their piecewise smooth counterparts for simplicity.

\item Gajer shows that for any connected, pointed manifold $(M,x_0)$ the evaluation map \newline $E(M,x_0) \xrightarrow{\widetilde{\mathrm{ev}}_1} M$ is a principal $G(X,x_0)$-bundle. First, it is important to point out that Proposition \ref{Proposition: principal bundle associated to geometric loop group} is a slightly different statement compared to Gajer's result \cite[Theorem 1.2]{GajerHigher}. Whereas we choose to regard $E(M,x_0)$ and $G(M,x_0)$ as diffeological spaces only, Gajer endows the space of paths $P(M,x_0)$ with both a topology and a diffeological structure. The topology in question is the compact-open topology\footnote{Note that it was observed by \cite[pp. 1181-1182]{Barrett} that ordinary holonomy maps are in general not continuous with respect to the compact-open topology on the loop space. One should therefore avoid working in the compact-open topology and rather use the functional topology induced by the diffeological structure.} and is as such always strictly coarser than the $D$-topology of the diffeological space $P(M,x_0)$, see \cite[Example 4.5]{Christensen-Sinnamon-Wu}. The quotient $E(M,x_0)$ is then endowed with the corresponding quotient topology and shown to be a principal $G(M,x_0)$-bundle in the classical sense, where here $G(M,x_0)$ is equipped with the subset topology of $E(M,x_0)$. Further, to argue that this bundle is universal, it is claimed that the total space $E(M,x_0)$ is smoothly contractible via the standard smooth retract
\[
\begin{array}{rcl}
     E(M,x_0) \times [0,1] & \rightarrow & E(M,x_0)   \\
     ([\gamma] ,t) & \mapsto & [\gamma_t]
\end{array}
\]
where here $\gamma_t(s) := \gamma(ts)$. It has already been pointed out in \cite{Ghazel-Kallel} that the proof of the contractibility of $E(M,x_0)$ is incomplete. Indeed, the suggested smooth deformation retract is not well defined. To see where this fails, consider a non-constant smooth path $\gamma \colon I \rightarrow M$ with sitting instants starting at $x_0$ such that $[\gamma] \neq [x_0]$. Then the concatenation $\gamma * \gamma^{-1}$ again defines a smooth path starting at $x_0$. By the third relation $\sim_3$ it is clear that $[\gamma * \gamma^{-1}] = [x_0]$. However, the one-parameter family of smooth paths given by $(\gamma * \gamma^{-1})_t(s) := (\gamma * \gamma^{-1})(ts)$ is such that for $t = \frac{1}{2}$
\[
(\gamma * \gamma^{-1})_{\frac{1}{2}}(s) = \gamma(s)
\]
Thus, by assumption it follows $[(\gamma * \gamma^{-1})_{\frac{1}{2}}] = [\gamma] \neq [x_0]$. This shows that
\[
([\gamma * \gamma^{-1}],t) \mapsto [(\gamma * \gamma^{-1})_{t}]
\]
is not well defined. Therefore, we consider the contractibility of $E(X,x_0)$ to be an open problem.
\end{enumerate}
\end{remark}

\begin{conjecture} \label{Hypothesis: Weak Contractibility of the total space}
    Let $(X,x_0)$ be a pointed diffeological space. Then the geometric path space $E(X,x_0)$ is weakly contractible. More precisely, the simplicial set $S_e \left( E(X,x_0) \right)$ is weakly contractible.
\end{conjecture}

In \cite{GajerHigher} the geometric loop group $G(M,x_0)$ serves as a suitable refinement of the fundamental group $\pi_1(M,x_0)$ such that the classical isomorphism between
\begin{equation*}
    \begin{array}{rcl}
        \mathrm{Bun}(M,G,\nabla^{\mathrm{flat}}) & \xrightarrow{\cong} & \mathrm{Hom}(\pi_1(M,x_0) , G)  \\
         {[P,p_0,A]} & \mapsto & h^{A}
    \end{array}
\end{equation*}
the pointed set of isomorphism classes $[P,p_0,A]$ to the pointed set of group morphisms can be refined to the case of bundles with arbitrary connection. Here the triple $[P,p_0,A]$ denotes an isomorphism class of a smooth principal $G$-bundle $P$ over $M$ endowed with a flat connection $A$ together with a designated point $p_0 \in P_{x_0}$ in the fiber over the basepoint $x_0 \in M$.

\begin{theorem}[\cite{GajerHigher}, Theorem 1.4.] \label{Theorem: Principal Bundle Version of Theorem}
    Let $(M,x_0)$ be a connected and pointed manifold and $G$ a Lie group. Then there is an isomorphism of pointed sets
    \begin{equation*}
    \begin{array}{rcl}
        \mathrm{Bun}(M, G, \nabla) & \xrightarrow{\cong} & \mathrm{Hom}_{\mathbf{Diff}}(G(M,x_0) , G).  \\
         {[P,p_0,A]} & \mapsto & h^A
    \end{array}
\end{equation*}
The left-hand side denotes the pointed set of isomorphism classes of principal $G$-bundles $P$ with connection $A$ over $M$ together with a point $p_0 \in P_{x_0}$, and the right-hand side the pointed set of morphisms of diffeological groups.
\end{theorem}
\begin{remark}
    The proof Theorem 1.4. in \cite{GajerHigher} is based on the existence of a universal parallel transport structure on the principal $G(M,x_0)$-bundle $E(M,x_0) \rightarrow M$. As such this proof relies on Conjecture \ref{Hypothesis: Weak Contractibility of the total space}. More precisely, the universal parallel transport structure\footnote{Or equivalently the universal connection.} $\mathbb{P}$ on $E(M,x_0) \rightarrow M$ is such that for any smooth holonomy map $h \colon G(M,x_0) \rightarrow G$ the associated bundle $E(M,x_0) \times^h G \rightarrow M$ carries the induced parallel transport structure $\mathbb{P}^{h}$. However, for the assignment
    \[
        \begin{array}{rcl}
         \mathrm{Hom}_{\mathbf{Diff}}(G(M,x_0) , G) & \rightarrow & \mathrm{Bun}(M,G,\nabla) \\
         h & \mapsto & \left( E(M,x_0) \times^h G \rightarrow M,[\mathrm{const}_{x_0},e], \mathbb{P}^h \right)
    \end{array}
    \]
    to be surjective, it is essential for the bundle $E(M,x_0) \rightarrow M$ to be universal. \\

    We believe that an alternative proof of Theorem \ref{Theorem: Principal Bundle Version of Theorem} can be given following the proof of \cite[Theorem 6.]{Caetano-Picken} where the thin homotopy group $\pi_1^1(M,x_0)$ is replaced with the geometric loop group $G(M,x_0)$. This avoids the need for $E(M,x_0)$ to be weakly contractible.
\end{remark}

The next lemma shows how the geometric loop group and the $1$-skeletal diffeological fundamental group are related.

\begin{lemma} \label{Geeometric loop group factors through pi}
    Let $(X,x_0)$ be a pointed diffeological space. Then there exists a surjective group morphism $\theta_1$ completing the diagram:
    \begin{center}
    \begin{tikzcd}
\Omega_{\mathrm{st}}(X,x_0) \arrow[d, "q" description] \arrow[rd, "q_1" description, shift left] \\
G(X,x_0) \arrow[r, "\theta_1"']    & \pi^{D}_1(X_1,x_0)
\end{tikzcd}
    \end{center}
\end{lemma}

\begin{proof}
The equivalence relation defining the quotient $G(X,x_0)$ is generated by three types of relations. It suffices therefore to show the following:
    \begin{enumerate}[label={(\arabic*)}]
        \item For any $\gamma \in \Omega_{\mathrm{st}}(X,x_0)$ and $\varphi$ an orientation preserving diffeomorphism of $I$ we have:
        \[
        [\gamma \circ \varphi]_1 = [\gamma]_1
        \]
        \item For any $\gamma \in \Omega_{\mathrm{st}}(X,x_0)$ and any $\psi : I \rightarrow I$ such as given in the above definition we have:
        \[
        [\gamma \circ \psi]_1 = [\gamma]_1
        \]
        \item
        For any two $\gamma, \eta \in \Omega_{\mathrm{st}}(M)$ such that $\gamma \sim_3 \eta$ we have
        \[
        [\gamma]_1 = [\eta]_1
        \]
    \end{enumerate}
        To show $(1)$ let $\gamma \in \Omega_{\mathrm{st}}(X,x_0)$ and let $\varphi \in \mathrm{Diff}^+(I)$ be an orientation preserving diffeomorphism. Since $\gamma$ is stationary extend it trivially to a smooth map $\gamma \colon \mathbb{R} \rightarrow X$. The standard smooth homotopy given by
        \[
        \begin{array}{rcccl}
              H \colon I \times I & \rightarrow & \mathbb{R} & \xrightarrow{\gamma} & X \\
              (s,t) & \mapsto & (1-t)\varphi(s) + ts & \mapsto & \gamma\left( (1-t)\varphi(s) + ts \right)
        \end{array}
        \]
        by definition factors trough $\mathbb{R}$ hence indeed $[\gamma \circ \varphi]_1 = [\gamma]_1$.

        To show $(2)$ let $\gamma \in \Omega_{\mathrm{st}}(X,x_0)$ and consider $I_0$ a proper subinterval of $I$ together with a smooth map $\psi : I \rightarrow I$ such that $\psi$ restricted to $I_0$ is a constant map with $\psi(I_0) = t_0$ and $\psi$ is an orientation preserving diffeomorphism of $I - I_0$ onto $I - \left\{ t_0 \right\}$. Again extend $\gamma$ trivially to a smooth map $\gamma \colon \mathbb{R} \rightarrow X$. The same homotopy as in the previous case
        \[
        \begin{array}{rcccl}
              H \colon I \times I & \rightarrow & \mathbb{R} & \xrightarrow{\gamma} & X \\
              (s,t) & \mapsto & (1-t)\psi(s) + ts & \mapsto & \gamma\left( (1-t)\psi(s) + ts \right)
        \end{array}
        \]
        provides a homotopy which factors trough $\mathbb{R}$. Therefore $[\gamma \circ \psi]_1 = [\gamma]_1$.

        To show $(3)$ it suffices to show that given a loop of the form $\alpha * \beta * \beta^{-1} * \eta$ there is a smooth homotopy which locally factors trough $\mathbb{R}$ such that $[\alpha * \beta * \beta^{-1} * \eta]_1 = [\alpha * \eta]_1$. Consider therefore $\alpha : x_0 \rightarrow y$, $\beta : y \rightarrow y'$ and $\eta : y \rightarrow x_0$ smooth paths with sitting instants in $X$. Then $\alpha * \beta * \beta^{-1} * \eta$ defines a loop in $X$ based at $x_0$ and likewise $\alpha * \eta$. Recall the standard contraction homotopy $h_t = (\beta)_t * (\beta_t)^{-1}$ where we define
        \[
        \beta_t(s) =
        \begin{cases}
        \beta(0) \text{ for } s \in [0,t] \\
        \beta(s-t) \text{ for } s \in [t,1]
        \end{cases}
        \]
        Notice that the smoothness at $\beta(0)$ is guaranteed by the fact that $\beta$ is chosen to have sitting instances. However, their concatenation $(\beta)_t * (\beta_t)^{-1}$ is not necessarily smooth as $\beta_t$ does not need to have a sitting instant at $s = 1$. Therefore via a smooth function $\psi \in C^{\infty}([0,1])$ with $\psi([\delta,1]) = 1$ for some $\frac{1}{2} < \delta < 1$ and such that $\psi$ is an orientation preserving diffeomorphism on $[0,\delta)$ construct the family $\beta'_t := \beta_t \circ \psi$. Now by construction it follows that $\beta'_0 = \beta_0 \circ \psi = \beta \circ \psi \sim_2 \beta$. Hence it suffices to show that the smooth homotopy $H'_t :=\beta'_t * (\beta'_t)^{-1}$ locally factors trough $\mathbb{R}$. By construction, it follows that
        \[
        H'(s,t) = \begin{cases}
            \beta'_t(2s) &\text{ for } (s,t) \in [0,1/2] \times [0,1] \\
            (\beta'_t)^{-1}(2s -1) &\text{ for } (s,t) \in [1/2,1] \times [0,1]
        \end{cases}
        \]
        Therefore, on the open subset $[0,1/2) \times [0,1]$ the homotopy factors by definition
        \[
        \begin{array}{rcccl}
         H' \colon [0,1/2) \times [0,1] & \rightarrow & \mathbb{R} & \xrightarrow{\beta'} & X \\
             (s,t) & \mapsto & 2s-t & \mapsto & \beta'(2s-t)
        \end{array}
        \]
        and likewise for $(1/2,1] \times [0,1]$. What is left to show is that any point of the form $(1/2,t)$ for $t \in [0,1]$ admits a neighborhood where $H'$ locally factors trough $\mathbb{R}$. Since for $t$ fixed by construction $\beta'_t([\delta,1]) = \beta'_t(1) = \beta'(1-t)$ it follows that
        \[
        \begin{array}{rcccl}
             H' \colon (1/2 - \delta,1/2 + \delta) \times I & \rightarrow & \mathbb{R} & \xrightarrow{\beta'} & X  \\
             (s,t) & \mapsto & t & \mapsto & \beta'(1-t).
        \end{array}
        \]
        By concatenation with the identity homotopies on the left and the right, which trivially factor trough $\mathbb{R}$, we get the desired smooth homotopy
        \[
        H : \alpha * \beta * \beta^{-1} * \eta \Rightarrow \alpha * \eta.
        \]
\end{proof}

\begin{remark}
    In the case where are given a smooth manifold $(M,x_0)$ then there is a tower of surjective maps that assemble into a commutative diagram.
    \begin{center}
    \begin{tikzcd}
\Omega(M,x_0) \arrow[d, "q" description] \arrow[rd, "q_1" description, shift left] \arrow[rrrd, "q" description, shift left=3] \arrow[rrd, "q_{\leq 1}" description, shift left=2] &                                             &                                 &          \\
G(M,x_0) \arrow[r, "\theta_1"'] & \pi^{D}_1(M_1,x_0) \arrow[r, "\theta_{\leq1}"'] & \pi_1^1(M,x_0) \arrow[r, "\theta"'] & \pi_1(M,x_0)
\end{tikzcd}
    \end{center}
    Here $\pi_1^1(M,x_0)$ denotes the \textit{thin homotopy group} introduced by Caetano--Picken \cite{Caetano-Picken}. Elements in $\pi_1^1(M,x_0)$ are homotopy classes of smooth paths in $M$ where the homotopy $H$ is such that its rank is strictly smaller than $2$.
\end{remark}

\section{Comparing Higher Holonomies}

My first strategy to show that the morphism
\[
\Phi \colon \hat{H}^{k+1}(M;\mathbb{Z}) \rightarrow \mathrm{Hom}_{\mathrm{Sh}(\mathbf{Cart},\mathbf{Ab})} \left( \tilde{H}_k(\mathbb{M}_k) ,U(1) \right)
\]
from the abelian group of differential characters to the abelian group of smooth holonomy morphisms is an isomorphism, was to compare the group of smooth holonomy maps to the higher holonomy morphisms of Gajer. Recall that for a smooth manifold $M$ the simplicial presheaf $\mathbb{M}_k$ of Definition \ref{Definition: The Skeletal Simplicial presheaf M_k} is such that
\[
U \in \mathbf{Cart} \mapsto \mathbb{M}_k(U) := S_e\left( D(U,M)_k \right)
\]
and its associated homology sheaf $\tilde{H}_k(\mathbb{M}_k)$ is the sheafification of the presheaf of abelian groups
\[
U \in \mathbf{Cart} \mapsto H_k(D(U,M)_k).
\]
The idea was to construct a map of presheaves on $\mathbf{Cart}$
\[
\pi \colon G^{k-1}(A(M)) \rightarrow H_{k}(\mathbb{M}_k)
\]
such that accordingly the induced pullback morphism $\pi^*$  would fit into the commutative diagram:
\begin{equation} \label{Equation: Diagram of Differential Characters}
\begin{tikzcd}[row sep=large, column sep=small]
 & \hat{H}^{k+1}(M; \mathbb{Z}) \arrow[dr,"\hat{\kappa}"] \arrow[dl, "\Phi"'] & \\
{\mathrm{Hom} \left( H_{k}(\mathbb{M}_k), U(1)   \right)} \arrow[rr,"\pi^*"] & & {\mathrm{Hom}\left(  G^{k-1}(A(M)) , U(1)   \right)}
\end{tikzcd}
\end{equation}
Here $A(M,x_0)$ denotes the diffeological abelianization of the geometric loop group $G(M,x_0)$. The morphism $\hat{\kappa}$ was constructed by Gajer\footnote{see Theorem \ref{Theorem: Gajer's holonomy Theorem}} and he shows that under the assumption that Conjecture \ref{Hypothesis: Weak Contractibility of the total space} is true, given any $(M,x_0)$ connected, pointed smooth manifold
\[
\hat{\kappa} \colon \hat{H}^{k+1}(M; \mathbb{Z}) \xrightarrow{\cong} \mathrm{Hom}_{\mathrm{PSh}(\mathbf{Cart}, \mathbf{Ab})} \left(  G^{k-1}(A(M)) , U(1)   \right)
\]
is an isomorphism of abelian groups. This approach, however, turned out to be rather challenging due to the following two reasons.
\begin{enumerate}[label={(R\arabic*)}]
    \item \label{Point: First Obstacle} It is unknown to us if the quotient map $\Omega(X,x_0) \rightarrow G(X,x_0)$ is a weak equivalence of diffeological spaces, even for special cases,  such as $(X,x_0)$ a smooth manifold. This issue has already been pointed out through Conjecture \ref{Hypothesis: Weak Contractibility of the total space}.
    \item \label{Point: Second Obstacle} A key ingredient to establish the isomorphism $\hat{\kappa}$ are the natural isomorphisms
    \[
    \pi_i(A(M,x_0), [x_0]) \cong H_{i+1}(M).
    \]
    given for every connected smooth manifold $(M,x_0)$ which are due to \cite[Theorem 4.1.]{GajerHigher}. It is important to point out that here $\pi_i(A(M,x_0),[x_0])$ does not represent the diffeological homotopy group, but the homotopy group of the topological space $A(M,x_0)$ as discussed in Remark \ref{Remark: Differences with Gajer's Paper}. The proof based on the smooth approximation theorem for manifolds can not be generalized to arbitrary diffeological spaces and, in particular, does not apply to the skeletal diffeologies of the form $M_d$.
\end{enumerate}

We show that by assuming $X$ to be $(k-1)$-connected, obstacle \ref{Point: Second Obstacle} can simply be reduced to obstacle \ref{Point: First Obstacle}. That is, for $X$ an $(k-1)$-connected diffeological space, the Hurewicz map $\pi^D_{k}(X,x_0) \xrightarrow{\cong} H_{k}(X)$ is an isomorphism. Note that the Hurewicz map can be related to the diffeological quotient map $G(X,x_0) \rightarrow A(X,x_0)$ as we will show in Theorem \ref{Theorem: Kan's Theorem but diffeological}. The purely simplicial analog of that statement is due to Kan \cite{Kan}. Given a connected and pointed simplicial set $(K,x_0)$ denote by $\mathbb{G}(K,x_0)$ Kan's simplicial loop group, see Definition \ref{Definition: Dwyer-Kan simplicial loop group}. Its abelianization is denoted by $\mathbb{A}(K,x_0)$.

\begin{theorem}[\cite{Kan}, Theorem 15.1 and Theorem 16.1] \label{Theorem: Kan's Theorem}
    Let $(K,x_0)$ be a pointed and connected simplicial set. Then for $i > 0$ there are natural isomorphisms
    \[
    \pi_{i-1}\left( \mathbb{A}(K,x_0) \right) \cong H_{i}(K).
    \]
    Moreover, commutativity holds in the diagram for all $i>0$
\begin{center}
\begin{tikzcd}
{\pi_i(K,x_0)} \arrow[r, "h"] \arrow[d, "\cong"'] & H_i(K) \arrow[d, "\cong"]      \\
{\pi_{i-1}(\mathbb{G}(K,x_0))} \arrow[r]          & {\pi_{i-1}(\mathbb{A}(K,x_0))}
\end{tikzcd}
\end{center}
where the morphism $h$ on the top denotes the simplicial Hurewicz morphism.
\end{theorem}

Using the homotopy theory of diffeological spaces, Theorem \ref{Theorem: Kan's Theorem} has its partial\footnote{Partial in the sense that not all vertical arrows in the diagram of Theorem \ref{Theorem: Kan's Theorem but diffeological} are known to be isomorphisms.} diffeological counterpart.

\begin{theorem} \label{Theorem: Kan's Theorem but diffeological}
    Let $(X,x_0)$ be a pointed and connected diffeological space. Denote by $G(X,x_0)$ the geometric loop group and consider $A(X,x_0)$ its abelianization. The quotient map $q \colon G(X,x_0) \rightarrow A(X,x_0)$ fits into the commutative diagram
    \begin{center}
    \begin{tikzcd}
{\pi^D_i(X,x_0)} \arrow[r, "h"] \arrow[d, "\cong"]                            & H_i(X) \arrow[d, "\cong"]                                           \\
{\pi_{i-1}\left( \mathbb{G}(S_e(X),x_0) \right)} \arrow[r] \arrow[d] & {\pi_{i-1}\left( \mathbb{A}(S_e(X),x_0) \right)} \arrow[d] \\
{\pi^D_{i-1}\left( G(X,x_0) \right)} \arrow[r, "q_*"]                & {\pi^D_{i-1}(A(X,x_0))}
\end{tikzcd}
\end{center}
where the morphism $h$ on top denotes the Hurewicz morphism and $q_*$ the morphism induced by the quotient map $q$.
\end{theorem}
\begin{remark}
    It is important to note that for the quotient map $q \colon G(X,x_0) \rightarrow A(X,x_0)$ to completely recover the Hurewicz map all vertical morphisms in the diagram need to be isomorphisms, which on the left-hand side is equivalent to Conjecture \ref{Hypothesis: Weak Contractibility of the total space} and on the right-hand side equivalent to observation \ref{Point: Second Obstacle}. The commutativity of the diagram, however, holds in any case.
\end{remark}

\begin{proof}
First recall from Proposition \ref{Proposition: principal bundle associated to geometric loop group} that $E(X,x_0) \xrightarrow{\widetilde{\mathrm{ev}}_1} X$ is a diffeological principal $G(X,x_0)$-bundle. By Theorem \ref{Theorem: Diffeological bundles are principal fibrations} it follows that applying the extended smooth singular complex functor produces $S_e \left( E(X,x_0) \right) \rightarrow S_e(X)$ a principal $S_e(G(X,x_0))$-fibration. Let
\[
f \colon S_e(X) \rightarrow \overline{W} S_e(G(X,x_0))
\]
be any choice of classifying map for $S_e \left( E(X,x_0) \right) \rightarrow S_e(X)$. Using the adjoint pair $\cat{G} \dashv \overline{W}$ of Theorem \ref{Theorem: Classifying space adjunction} this gives a morphism of simplicial groupoids
\[
f \colon \cat{G}(S_e(X)) \rightarrow S_e(G(X,x_0)).
\]
Note that by definition $\mathbb{G}(X,x_0) = \cat{G}(S_e(X))(x_0,x_0)$ and as such by restriction of $f$ to the simplicial group of automorphisms at $x_0$  we obtain a morphism of simplicial groups
\[
f \colon \mathbb{G}(X,x_0) \rightarrow S_e(G(X,x_0)).
\]
By the universal property of the abelianization, the map $f$ descends
\begin{center}
\begin{tikzcd}
{\mathbb{G}(S_e(X),x_0)} \arrow[d] \arrow[r, "f"]             & {S_e(G(X,x_0))} \arrow[d, "q"] \\
{\mathbb{A}(S_e(X),x_0)} \arrow[r, "f_{\mathrm{ab}}", dashed] & {S_e(A(X,x_0))}
\end{tikzcd}
\end{center}
to give a morphism of simplicial abelian groups $f_{\mathrm{ab}}$. The induced diagram of homotopy groups
\begin{center}
\begin{tikzcd}
{ \pi_{i-1} \left( \mathbb{G}(S_e(X),x_0) \right)} \arrow[d] \arrow[r, "f_*"]             & {\pi^D_{i-1}( G(X,x_0) )} \arrow[d, "q_*"] \\
{\pi_{i-1} \left( \mathbb{A}(S_e(X),x_0) \right)} \arrow[r, "(f_{\mathrm{ab}})_*", dashed] & {\pi^D_{i-1}(A(X,x_0))}
\end{tikzcd}
\end{center}
is independent of the choice of classifying map $f$. Together with the commutative square of Theorem \ref{Theorem: Kan's Theorem} this concludes the proof.
\end{proof}
\newpage

\textbf{Assume for the rest of this chapter that we are always considering a smooth, pointed and} $(k-1)$\textbf{-connected diffeological space} $(X,x_0)$ \textbf{and further assume that Conjecture} \ref{Hypothesis: Weak Contractibility of the total space} \textbf{is true}. \\

An immediate consequence is the following.
\begin{corollary} \label{Corollary: Homotopy type of the geometric loop space}
 For every pointed diffeological space $(X,x_0)$ there are natural isomorphisms
    \[
    \pi^D_{k+1}(X,x_0) \xrightarrow{\cong} \pi^D_k(G(X,x_0),[x_0])
    \]
    for all $k \geq 0$.
\end{corollary}
\begin{proof}
    Recall from Proposition \ref{Proposition: principal bundle associated to geometric loop group} that the evaluation map $E(X,x_0) \rightarrow X$ is a diffeological fibration. That is, the associated morphism of simplicial sets $S_e(E(X,x_0)) \rightarrow S_e(X)$ is a Kan fibration. First notice that since $S_e$ preserves limits, the fiber
    \begin{center}
    \begin{tikzcd}
{S_e(G(X,x_0))} \arrow[d] \arrow[r] & {S_e(E(X,x_0))} \arrow[d] \\
\Delta^0 \arrow[r, "x_0"]           & S_e(X)
\end{tikzcd}
    \end{center}
    can be identified with $S_e(G(X,x_0))$. The long exact sequence of a Kan fibration now gives
        \begin{center}
        \begin{tikzcd}
  \cdots \rar & \pi_{k+1}(S_e(E(X,x_0)),[x_0]) \rar
             \ar[draw=none]{d}[name=X, anchor=center]{}
    & \pi_{k+1}(S_e(X),x_0) \ar[rounded corners,
            to path={ -- ([xshift=2ex]\tikztostart.east)
                      |- (X.center) \tikztonodes
                      -| ([xshift=-2ex]\tikztotarget.west)
                      -- (\tikztotarget)}]{dll}[at end]{\delta} \\
  \pi_{k}(S_e(G(X,x_0)),[x_0]) \rar & \pi_{k}(S_e(E(X,x_0)),[x_0]) \rar & \cdots
  \end{tikzcd}
  \end{center}
    from which we conclude by Conjecture \ref{Hypothesis: Weak Contractibility of the total space} that the natural maps
    \[
    \pi_{k+1}(S_e(X),x_0) \rightarrow \pi_{k}(S_e(G(X,x_0)),[x_0])
    \]
    are indeed isomorphisms for all $k \geq 0$.
\end{proof}

\begin{remark} \label{Remark: Diffeological and ordinary homotopy groups for G(M,x_0)}
    Considering the case where we are given a smooth manifold $(M,x_0)$, then the natural morphism
    \[
    \pi_i^D(M,x_0) \xrightarrow{\cong} \pi_i(M,x_0)
    \]
    from the diffeological homotopy group into the ordinary homotopy group is an isomorphism. Further, recall that Gajer endows the geometric loop groups $G(M,x_0)$ with the quotient topology of the compact-open topology on $P(M,x_0)$. First, we wish to show that there is a continuous map
    \[
    \tau(G(M,x_0)) \xrightarrow{\mathrm{id}} G(M,x_0),
    \]
    where on the left-hand side the geometric loop group is endowed with the $D$-topology and on the right-hand side with the quotient topology of the compact-open topology. From Remark \ref{Remark: Functional Topologies} it first follows that there is a continuous map
    \[
    \tau(D(I,M)) \xrightarrow{\mathrm{id}} C^{\infty}(I,M)
    \]
    since the functional topology is finer than the compact-open topology. Further, the discussion in \cite[Section 3.3]{Christensen-Sinnamon-Wu} then implies that there are indeed continuous maps
    \begin{align*}
        &\tau(P(M,x_0)) \xrightarrow{\mathrm{id}} \Omega(M,x_0) \\
        &\tau(G(M,x_0)) \xrightarrow{\mathrm{id}} G(M,x_0).
    \end{align*}
    Therefore, there is a natural morphism
    \[
    \pi_i^D(G(M,x_0),[x_0]) \rightarrow \pi_i(G(M,x_0),[x_0])
    \]
    such that the following diagram commutes.
\begin{center}
\begin{tikzcd}
{\pi_{i+1}^D(M,x_0)} \arrow[d, "\cong" description] \arrow[r, "\cong"] & {\pi_{i+1}(M,x_0)} \arrow[d, "\cong" description] \\
{\pi_{i}^D(G(M,x_0),[x_0])} \arrow[r]                                  & {\pi_{i}(G(M,x_0),[x_0])}
\end{tikzcd}
\end{center}
The left-hand vertical morphism is an isomorphism from Corollary \ref{Corollary: Homotopy type of the geometric loop space} and the right-hand side vertical morphism follows from \cite{GajerHigher}. In particular, this implies that
    \[
    \pi_i^D(G(M,x_0),[x_0]) \xrightarrow{\cong} \pi_i(G(M,x_0),[x_0])
    \]
    is also an isomorphism.
\end{remark}

\subsection{Three diffeologies on $G^k(X,x_0)$}

The core concept developed in \cite{GajerHigher, Gajer_Geometry} is based on the iterative application of the geometric loop group. That is, the geometric loop group defines a functor
\[
\begin{array}{rcl}
     G \colon \mathbf{Diff}_* & \rightarrow & \mathbf{Diff}_*  \\
    (X,x_0) & \mapsto & \left( G(X,x_0) , [x_0] \right)
\end{array}
\]
such that the iterated geometric loop group $G^k(X,x_0)$ is understood to be equipped with the \textit{standard diffeology} or \textit{canonical diffeology} induced from the functor $G$. \\

Alternatively, consider the smooth map
    \begin{equation} \label{Equation: Iterated quotient map}
    q_M^k \colon \Omega^k(M,x_0) \rightarrow G^k(M,x_0)
    \end{equation}
    induced by an iterative application of the smooth quotient map $q_X \colon \Omega(X,x_0) \rightarrow G(X,x_0)$. For example in the case $k = 2$ the smooth map (\ref{Equation: Iterated quotient map}) is given by:
    \begin{center}
    \begin{tikzcd}
    {\Omega^2(M,x_0)} \arrow[r, "\Omega(q_M)"] \arrow[rd] & {\Omega(G(M,x_0),[x_0])} \arrow[d, "{q_{G(M,x_0)}}"] \\
                                                      & {G(G(M,x_0),[x_0])}
    \end{tikzcd}
    \end{center}
    Now the set $G(G(M,x_0),[x_0])$ can be endowed with the pushforward diffeology of the standard diffeology on the loop space $\Omega^2(M,x_0)$ along the map $q^2$. This is the diffeology that has been considered by Gajer \cite{Gajer_Geometry}.

\begin{definition}[\cite{Gajer_Geometry}]
    Let $(X,x_0)$ be a pointed diffeological space and $k >0$ an integer. Endow the set $G^k(X,x_0)$ with the diffeology where a function $p \colon U \rightarrow G^k(X,x_0)$ is a plot if and only if it locally lifts to a plot $\tilde{p} \colon V \rightarrow \Omega^k(X,x_0)$. Denote this diffeological space by $\widetilde{G}^k(X,x_0)$. In particular the diffeology of $\widetilde{G}^k(X,x_0)$ represents the pushforward diffeology on the set $G^k(X,x_0)$ along the map
    \[
    q_X^k \colon \Omega^k(X,x_0) \rightarrow G^k(X,x_0).
    \]
\end{definition}

\begin{remark}
Note that in the case $k = 1$ we simply recover the canonical diffeology
\[
\tilde{G}(X,x_0) = G(X,x_0).
\]
\end{remark}

For the canonical diffeology $G^k(X,x_0)$ the map
\[
q^k_X \colon \Omega^k(X,x_0) \rightarrow G^k(X,x_0)
\]
is smooth and therefore, using the fact that the pushforward diffeology is the finest diffeology such that the map $q^k_X$ smooth\footnote{See for example Art. 1.43 in \cite{Iglesias-Zemmour}.}, it follows that the set-theoretical identity yields a smooth map
\[
\mathrm{id} \colon \widetilde{G}^k(X,x_0) \rightarrow G^k(X,x_0).
\]

The third diffeology to consider is the one induced by the presheaf of groups on $\mathbf{Cart}$
\[
U \mapsto G^k(D(U,X),[x_0])
\]
where here $x_0$ denotes the constant smooth map $x_0 \colon U \rightarrow X$ at $x_0 \in X$. This means that $G^k(D(U,X),[x_0])$ is considered simply as a set-theoretic group by forgetting the diffeology. Via concretization and sheafification this defines a diffeological group denoted $\hat{G}^k(X,x_0)$. Also for this diffeology the natural map
\[
q^k_X \colon \Omega^k(X,x_0) \rightarrow \hat{G}^k(X,x_0)
\]
is smooth, i.e. again by the universal property of the pushforward the set-theoretic identity map is smooth
\begin{equation} \label{Equation: The Morphism Id}
\mathrm{id} \colon \tilde{G}^k(X,x_0) \rightarrow \hat{G}^k(X,x_0).
\end{equation}

\begin{proposition} \label{Proposition: The (k-1)-connected Gajer theorem}
    Let $(M,x_0)$ be a pointed $(k-1)$-connected smooth manifold for $k > 0$. Then there is an isomorphism of abelian groups
    \[
    \hat{\kappa} \colon \hat{H}^{k+1}(M,\mathbb{Z}) \xrightarrow{\cong} \mathrm{Hom}_{\mathbf{DiffGrp}} \left(  \widetilde{G}^{k}(M) , U(1)   \right).
    \]
\end{proposition}

\begin{proof}[Proof of Proposition \ref{Proposition: The (k-1)-connected Gajer theorem}]
Let us first consider the case $k=1$ where $\tilde{G}(M,x_0) = G(M,x_0)$ and the statement simply follows from Theorem \ref{Theorem: Principal Bundle Version of Theorem}. Assume now $k> 1$ and notice that there is a short exact sequence\footnote{For reference see the proof of Theorem 3.6 in \cite{Gajer_Geometry}.}
\[
0 \rightarrow \mathrm{Hom} \left( \pi_{0} \left( G^{k-1}(A(M)) \right), U(1) \right) \xrightarrow{i} \mathrm{Hom}_{}\left(  \tilde{G}^{k-1}(A(M)), U(1) \right) \xrightarrow{\mathrm{curv}} \Omega_{\mathrm{cl}}^{k+1}(M)_{\mathbb{Z}} \rightarrow 0
\]
which under the identification with $ \pi_{0} \left( G^{k-1}(A(M)) \right) \cong H_{k}(M)$ together with the universal coefficients theorem $\mathrm{Hom} \left( \pi_{0} \left( G^{k-1}(A(M)) \right), U(1) \right) \cong H^k(M,U(1))$ corresponds to one of the diagonals in the differential cohomology hexagon. \\

The smooth abelianization map $q_{\mathrm{ab}} \colon G(M) \rightarrow A(M)$ induces smooth morphisms $q \colon \tilde{G}^k(M) \rightarrow \tilde{G}^{k-1}(A(M))$. We now claim that there is a commutative diagram of short exact sequences.
\begin{center}
\adjustbox{scale=0.85,center}{%
\begin{tikzcd}
0 \arrow[r] & {\mathrm{Hom} \left( \pi_{0} \left( G^{k-1}(A(M)) \right), U(1) \right)} \arrow[r, "i", hook] \arrow[d, "\pi_0(q)_*"'] & {\mathrm{Hom} \left( \tilde{G}^{k-1}(A(M)), U(1) \right)} \arrow[r, "\mathrm{curv}"] \arrow[d, "q_*"] & \Omega_{\mathrm{cl}}^{k+1}(M)_{\mathbb{Z}} \arrow[r] \arrow[d, "\mathrm{id}"] & 0 \\
0 \arrow[r] & {\mathrm{Hom} \left( \pi_{0} \left( G^{k}(M) \right), U(1) \right)} \arrow[r, "j", hook]                               & {\mathrm{Hom} \left( \tilde{G}^{k}(M), U(1) \right)} \arrow[r, "\mathrm{curv}"]                       & \Omega_{\mathrm{cl}}^{k+1}(M)_{\mathbb{Z}} \arrow[r]                          & 0
\end{tikzcd}
}
\end{center}
Let us first address the commutativity of the diagram. The fact that the left-hand square commutes follows from the fact that it is induced by precomposition with the map $q$ and the associated map of path-components $\pi_0(q)$ respectively. The fact that the right-hand square commutes follows from the definition of the curvature form. Indeed, to a smooth morphism $h \colon  \tilde{G}^{k-1}(A(M)) \rightarrow U(1)$ the associated curvature form locally at some point $x \in M$ is constructed as follows. Let $\varphi \colon U \rightarrow \mathbb{R}^n$ be a coordinate chart of $M$ centered at $x$. Given a collection of vectors $w := (w_0, ..., w_k)$ in $\mathbb{R}^n$ consider the smooth map\footnote{See p. 231 in \cite{GajerHigher} for more details.} $\gamma_{w} \colon \mathbb{R}^k \rightarrow \mathbb{R}^n$ parametrizing the parallelepiped spanned by the vectors $w_i$ in $\mathbb{R}^n$. The coordinate chart now gives a smooth map $\gamma_{w} \colon \mathbb{R}^k \rightarrow M$ which by construction defines an element $\gamma_{w} \in \Omega^k(M)$. Then locally at $x$ the curvature form is defined by
\[
\sum_{i_0 < \cdots < i_k} K_{i_0,...,i_k}(x) dx_{i_0} \cdots dx_{i_k}
\]
where
\[
K_{i_0,...,i_k}(x) = \underset{t \rightarrow 0}{\mathrm{lim}} \, \frac{\mathrm{log} \left( h \left( \left[\gamma_{t \cdot (e_{i_0}, ..., e_{i_k})} \right]  \right) \right)}{t^{k+1}}.
\]
Here $\left[\gamma_{t \cdot (e_{i_0}, ..., e_{i_k})} \right]$ denotes the image of $\gamma_{t \cdot (e_{i_0}, ..., e_{i_k})} \in \Omega^k(M)$ under the quotient map $\Omega^k(M) \rightarrow \tilde{G}^{k-1}(A(M))$. In the case where we are given a smooth map $h' \colon \tilde{G}^{k}(M) \rightarrow U(1)$, the same construction defines a curvature $(k+1)$-form locally at $x$
\[
K'_{i_0,...,i_k}(x) = \underset{t \rightarrow 0}{\mathrm{lim}} \, \frac{\mathrm{log} \left( h' \left( \left[\gamma_{t \cdot (e_{i_0}, ..., e_{i_k})} \right]  \right) \right)}{t^{k+1}}
\]
but here $\left[\gamma_{t \cdot (e_{i_0}, ..., e_{i_k})} \right]$ denotes the image of $\gamma_{t \cdot (e_{i_0}, ..., e_{i_k})} \in \Omega^k(M)$ under the quotient map $\Omega^k(M) \rightarrow \tilde{G}^{k}(M)$. It follows now from the commutativity of the diagram
\begin{center}
    \begin{tikzcd}
\Omega^k(M) \arrow[d] \arrow[rd] &               \\
\tilde{G}^{k}(M) \arrow[r, "q"']         & \tilde{G}^{k-1}(A(M))
\end{tikzcd}
\end{center}
that these two curvature forms agree in the case $h' = q \circ h$. That is, $\mathrm{curv}( q_*(h)) = \mathrm{curv}( h \circ q ) = \mathrm{curv}(h)$ which shows that the right-hand square indeed commutes. \\

Let us now show the exactness of the lower sequence. The surjectivity of the curvature morphism now follows from the surjectivity of the curvature morphism above. The injectivity of $j$ follows from the fact that the morphism is induced by precomposition with the quotient map to the path components. By construction of the curvature, we see that $\mathrm{curv}(h') = 0$ if and only if $h'$ is locally constant, i.e. $h \in \mathrm{Im}(j)$. \\

Notice that until this point, we have not used the fact that $M$ is $(k-1)$-connected. This assumption however implies that $\pi_0(q)$ is an isomorphism, as we can deduce from the following diagram:
\begin{center}
    \begin{tikzcd}
\pi_0\left( G^{k}(M) \right) \arrow[d, "\pi_0(q)"] & {\pi_k(M,x_0)} \arrow[d, "\cong"] \arrow[l, "\cong"'] \\
\pi_0\left( G^{k-1}(A(M)) \right)      & H_k(M) \arrow[l, "\cong"]
\end{tikzcd}
\end{center}
Here the vertical morphism on the right is the Hurewicz morphism and commutativity is implied by Theorem \ref{Theorem: Kan's Theorem but diffeological} together with Remark \ref{Remark: Diffeological and ordinary homotopy groups for G(M,x_0)}. It follows now from the 5-Lemma, that the morphism in middle $q_*$ is an isomorphism.
\end{proof}

\subsection{Constructing the map $\pi$}

Start with the presheaf of groups used to define the diffeological group $\hat{G}^k(M,x_0)$:
\begin{equation} \label{Equation: The presheaf G^k}
\begin{array}{rcl}
   \cat{G}^k(M,x_0) \colon \mathbf{Cart}^{\mathrm{op}}   & \rightarrow & \mathbf{Grp}  \\
     U & \mapsto & G^{k}\left( D(U,M) ,[x_0]  \right)
\end{array}
\end{equation}
which is denoted for the moment by $\cat{G}^k(M,x_0)$ \footnote{We point out the slight abuse of notation since $\cat{G}$ has been used before to denote the Dwyer--Kan simplicial loop groupoid. However, in this section, it always denotes the presheaf of groups as defined in \eqref{Equation: The presheaf G^k}.}. Since the diffeological group $\hat{G}^k(M,x_0)$ is defined by
\[
\hat{G}^k(M,x_0) := \left( C \left( \cat{G}^k(M,x_0) \right) \right)^{\#}
\]
it follows that there are natural isomorphisms
\[
\mathrm{Hom}_{\mathbf{DiffGrp}} \left( \hat{G}^k(M,x_0) , U(1) \right) \cong  \mathrm{Hom}_{\mathrm{PSh}(\mathbf{Cart}, \mathbf{Grp})} \left( \cat{G}^k(M,x_0) , U(1) \right)
\]
using the fact that sheafification and concretization are both left adjoints, see section \ref{Subsection: Diffeological Spaces}. The idea is to construct a morphism of presheaves
\begin{equation} \label{Equation: presheaf map into homology presheaf}
\cat{G}^{k}(M,x_0) \rightarrow \pi_{k}(\mathbb{M}_k,x_0)
\end{equation}
and show how this morphism then gives rise to the desired morphism we have denoted by $\pi$. Recall that by $\pi_{k}(\mathbb{M}_k,x_0)$ we denote the homotopy presheaf associated with the simplicial presheaf $\mathbb{M}_k$, see Definition \ref{Definition: Homotopy Sheaf}.

\begin{lemma} \label{Lemma: Surjection path space thinness}
    For $(X,x_0)$ a pointed diffeological space and any integer $k > 0$, the set theoretical identity map
    \[
    \mathrm{id} \colon G(X,x_0)_k \rightarrow G(X_{k+1},x_0)
    \]
    is smooth and induces a surjection on the diffeological homotopy groups:
    \[
    \pi_i^D(G(X,x_0)_k) \rightarrow \pi_i^D(G(X_{k+1},x_0))
    \]
\end{lemma}
\begin{proof}
    To check the smoothness let's have a look at the diffeologies of the corresponding spaces. Let $U \in \mathbf{Cart}$ be some cartesian space, then a parametrization $p : U \rightarrow G(X)_k$ is a plot of $G(X)_k$ if there exists a good open covering $\left\{ U_i \right\}_{i \in I}$ of $U$ and plots $p_i : \mathbb{R}^k \rightarrow G(X)$ together with smooth maps $\pi_i \colon U_i \rightarrow \mathbb{R}^k$ fitting into a commutative diagram
    \begin{center}
        \begin{tikzcd}
U \arrow[rr, "p"]              &                             & G(X,x_0)_k \arrow[d] \\
U_i \arrow[u, hook'] \arrow[r, "\pi_i"] & \mathbb{R}^k \arrow[r, "p_i"] & G(X,x_0)
\end{tikzcd}
    \end{center}
    That is the restriction of $p$ to $U_i$ is given by a plot of $G(X)$ which factors trough $\mathbb{R}^k$. On the other hand, for $i \in I$ fixed we can find an open covering $\{V_j\}_{j \in J_i}$ of $\mathbb{R}^k$ and plots $(p_i)_j : V_j \rightarrow \Omega(X,x_0)$ such that
    \begin{equation} \label{Diagram: The plot diagram of G(X)_k}
    \begin{tikzcd}
\mathbb{R}^k \arrow[r, "p_i"]               & G(X,x_0)                     \\
V_j \arrow[u, hook'] \arrow[r, "(p_i)_j"] & \Omega(X,x_0) \arrow[u, "q"']
\end{tikzcd}
    \end{equation}
By definition of the functional diffeology, the plots $(p_i)_j$ are given by smooth maps.
\[
p_{i,j}: V_j \times I \rightarrow X.
\]
Let $\psi \in C^{\infty}([0,1])$ be a smooth cut-off function. Define
\[
\tilde{p}_{i,j}(u,t) := p_{i,j}(u,\psi(t))
\]
which gives a smooth map $\tilde{p}_{i,j} \colon V_j \rightarrow \Omega(X,x_0)$ also making the diagram (\ref{Diagram: The plot diagram of G(X)_k}) commute. This follows from the fact that for any $u \in V_j$ the loops $p_{i,j}(u)(t) \sim_2 \tilde{p}_{i,j}(u)(t)$. The family $\tilde{p}_{i,j}$ has the advantage that it trivially extends to a family of smooth maps
\[
\tilde{p}_{i,j}: V_j \times \mathbb{R} \rightarrow X.
\]
By pulling back the covering $\left\{ V_j \right\}_{j \in J_i}$ via $\pi_i$ we get a covering $\left\{ U_{i,j}\right\}_{i \in I,j \in J_i}$ of $U$ such that locally $p : U \rightarrow G(X)_k$ admits a lift of the form
\[
\tilde{p}_{i,j} : U_{i,j} \times \mathbb{R} \xrightarrow{\pi_i \times \mathrm{id}} V_j \times \mathbb{R} \rightarrow X
\]
\\

Since $V_{j} \subset \mathbb{R}^k$ it follows that
\begin{equation} \label{Equation: Smooth plot X_k+1}
    \tilde{p}_{i,j} : U_{i,j} \times \mathbb{R} \xrightarrow{\pi_i \times \mathrm{id}} V_j \times \mathbb{R} \rightarrow X_{k+1}
\end{equation}
is smooth, i.e. $\tilde{p}_{i,j} \colon U_{i,j} \rightarrow \Omega(X_{k+1},x_0)$ define plots. \\

On the other hand, a parametrization $p : W \rightarrow G(X_{k+1})$ is a plot if and only if there is a good open covering $\left\{ W_i \right\}_i$ of $W$ and plots $p_i : W_i \rightarrow \Omega (X_{k+1},x_0)$ such that:
\begin{center}
    \begin{tikzcd}
W \arrow[r, "p"]                      & G(X_{k+1},x_0)                     \\
W_i \arrow[u, hook'] \arrow[r, "p_i"] & \Omega(X_{k+1},x_0) \arrow[u, "q"']
\end{tikzcd}
\end{center}

From observation (\ref{Equation: Smooth plot X_k+1}) it is now clear that any smooth plot $p \colon U \rightarrow G(X,x_0)_k$ is in particular also a smooth plot for $G(X_{k+1},x_0)$. This shows that
\[
\cat{D}_{G(X,x_0)_k} \subset \cat{D}_{G(X_{k+1},x_0)}
\]
and therefore, there are more smooth homotopies between paths in $G(X_{k+1})$ than in $G(X)_k$ since the target diffeology is finer. This shows that the induced map of fundamental groups is surjective.
\end{proof}

Recall now from Lemma \ref{Geeometric loop group factors through pi} that there is a natural group morphism
\begin{equation} \label{Eqzation: G(X) to thin homotopy group}
G(X,x_0) \xrightarrow{\theta_1} \pi^D_1(X_1,x_0),
\end{equation}
which now gives rise to the following iterative construction.
\begin{center}
    \begin{tikzcd}
{G^{k}\left(D(U,M), [x_0] \right)} \arrow[r, "\theta_1"] & {\pi^D_1\left( G^{k-1}\left(D(U,M) \right)_1, [x_0] \right)} \arrow[d] &                                                     \\
                                           & \vdots \arrow[d]                                                  &                                                     \\
                                           & {\pi_1^D \left( G^{k-1}\left(D(U,M)_k \right) ,[x_0] \right)} \arrow[r, "\cong"] & {\pi^D_{k} \left(D(U,M)_k ,[x_0] \right)}
\end{tikzcd}
\end{center}
where the chain of vertical arrows is given by the morphisms from Lemma \ref{Lemma: Surjection path space thinness}.
In particular, the chain of arrows gives a surjective group morphism for every $k > 0$ of the form
\begin{equation} \label{Equation: The Morphism xi}
G^{k}\left(D(U,M),[x_0] \right) \twoheadrightarrow \pi^D_{k}\left( D(U,M)_k,[x_0]\right) = \pi_k \left( \mathbb{M}_k,[x_0] \right)(U).
\end{equation}

This defines a morphism of presheaves of groups $\xi$. Also since $M$ is $(k-1)$-connected, so is $M_k$. This follows simply from the fact that for $i < k$ the homotopy groups $\pi_i^D(M) = \pi_i(M_k)$ agree. Likewise, using the contractibility of the cartesian spaces $U$ it follows that $D(U,M)$ and $D(U,M)_k$ are both $(k-1)$-connected. It therefore follows that
\[
\pi_k \left( \mathbb{M}_k \right)(U) \cong H_k \left( \mathbb{M}_k \right)(U)
\]
via the Hurewicz map for $k > 1$. In the case $k = 1$ we simply compose with the surjective map
\[
\pi_1(\mathbb{M}_1)(U) \rightarrow H_1 \left( \mathbb{M}_1 \right)(U).
\]
This defines the morphism
\[
\xi_U \colon G^{k}\left(D(U,M),[x_0] \right) \rightarrow H_k \left( \mathbb{M}_k \right)(U).
\]
for all $k > 0$, which is natural in $U \in \mathbf{Cart}$.

\begin{proposition} \label{Proposition: The commutative triangle result}
   Let $(M,x_0)$ be a pointed $(k-1)$-connected smooth manifold for $k > 0$. The morphism $\xi$ specified in (\ref{Equation: The Morphism xi}) together with the smooth identity map (\ref{Equation: The Morphism Id}) induce a morphism of abelian groups
\[
\pi^* \colon \mathrm{Hom}_{\mathrm{PSh}(\mathbf{Cart}, \mathbf{Ab})} \left( H_{k}(\mathbb{M}_k), U(1)   \right) \rightarrow \mathrm{Hom}_{\mathbf{DiffGrp}} \left(  \tilde{G}^{k}(M) , U(1) \right)
\]
In particular, this morphism fits into a commutative diagram
\begin{center}
\begin{tikzcd}[row sep=large, column sep=small]
 & \hat{H}^{k+1}(M; \mathbb{Z}) \arrow[dr,"\hat{\kappa}"] \arrow[dl, "\Phi"'] & \\
{\mathrm{Hom} \left( H_{k}(\mathbb{M}_k), U(1)   \right)} \arrow[rr,"\pi^*"] & & {\mathrm{Hom}\left(  \tilde{G}^k(M) , U(1)   \right)}
\end{tikzcd}
\end{center}
and is therefore an isomorphism.
\end{proposition}
\begin{remark}
The proof of this proposition relies on the construction of the map $\hat{\kappa}$ which we will sketch out in the following and the main reference is \cite{GajerHigher}. \\

Recall that the group of Cheeger--Simons differential characters $\hat{H}^{k+1}(M;\mathbb{Z})$ has as elements morphism of abelian groups
    \[
    \chi : Z_{k}(M;\mathbb{Z}) \rightarrow U(1)
    \]
    such that there exists a $(k+1)$-form $\omega_{\chi}\in \Omega^k(M)$ with the property that for every smooth chain $c \in C_{k+1}(M;\mathbb{Z})$ we have
    \[
    \chi(\partial c) = \mathrm{exp} \left( 2\pi i \int_c \omega_{\chi} \right).
    \]
    Here $Z_{\bullet}(M;\mathbb{Z})$ and $C_{\bullet}(M;\mathbb{Z})$ denote the free abelian groups of smooth singular chains and cycles in $M$. Under certain circumstances however, it is favorable to have a representation of differential characters based on \textit{geometric chains} which one should think of as representatives of singular homology classes of $M$ by manifolds. This however is not possible in general, so one needs to consider more general geometric objects that allow for singularities, such as \textit{stratifolds} introduced by Kreck \cite{Kreck}. Our main reference to how to use geometric chains as an alternative presentation of differential characters is \cite[Section 4]{Bär-Becker}. \\

    For $k \geq 0$ denote by $\cat{C}_k(M)$ the set of diffeomorphism classes of smooth maps $f \colon X \rightarrow M$ where $X$ denotes an oriented compact $k$-dimensional regular stratifold. An equivalence class is then denoted by $[X \xrightarrow{f} M]$. Disjoint union endows $\cat{C}_k(M)$ with the structure of an abelian semigroup \footnote{Recall that a semigroup is a set equipped with an associative binary operation.}. The boundary operator $\partial \colon \cat{C}_k(M) \rightarrow \cat{C}_{k-1}(M)$ is then given by restriction the geometric boundary, i.e $[X \rightarrow M] \mapsto [\partial X \hookrightarrow X \rightarrow M]$. Like for singular homology denote by $\cat{Z}_k(M) := \left\{ \zeta \in \cat{C}_{k}(M) \, | \, \partial \zeta = 0 \right\}$ the \textit{geometric cycles} and $\cat{B}_k(M) := \left\{ \zeta \in \cat{C}_{k}(M) \, | \, \exists \beta \in \cat{C}_{k+1}(M) : \partial \beta = \zeta \right\}$ the \textit{geometric boundaries}. Then define \textit{geometric homology} via
    \[
    \cat{H}_k(M) := \cat{Z}_k(M) / \cat{B}_k(M)
    \]
    which has the structure of an abelian group with inverse elements given by reversing the orientation of the stratifolds. In particular, there is a natural isomorphism to the ordinary singular homology of $M$
    \[
    \cat{H}_k(M) \xrightarrow{\cong} H_k(M).
    \]
\end{remark}

\begin{proof}[Proof of Proposition \ref{Proposition: The commutative triangle result}]\textit{(Sketch)}
Assume first $k > 1$. 
The content of the previous discussion was to construct a morphism of presheaves of groups on $\mathbf{Cart}$
\[
\xi \colon \cat{G}^k(M,x_0) \rightarrow H_k(\mathbb{M}_k)
\]
which then induces the corresponding pullback morphism
\begin{equation} \label{Equation: xi}
\xi^* \colon \mathrm{Hom}_{\mathrm{Psh}(\mathbf{Cart},\mathbf{Ab})}(H_k(\mathbb{M}_k),U(1)) \rightarrow \mathrm{Hom}_{\mathrm{Psh}(\mathbf{Cart},\mathbf{Grp})}(\cat{G}^k(M,x_0),U(1))
\end{equation}
Since $U(1)$ is a diffeological group together with the fact that $\hat{G}^k(M,x_0) := \left( C \left( \cat{G}^k(M,x_0) \right) \right)^{\#}$ gives
\begin{equation} \label{Equation: from cat G to G hat}
\mathrm{Hom}_{\mathrm{Psh}(\mathbf{Cart},\mathbf{Ab})}(\cat{G}^k(M,x_0),U(1)) \cong \mathrm{Hom}_{\mathbf{DiffGrp}}(\hat{G}^k(M,x_0),U(1))
\end{equation}
From \eqref{Equation: The Morphism Id} it follows that there is an induced morphism
\begin{equation} \label{Equation: The pullback via identity}
\mathrm{Hom}_{\mathbf{DiffGrp}}(\hat{G}^k(M,x_0),U(1)) \rightarrow \mathrm{Hom}_{\mathbf{DiffGrp}}(\tilde{G}^k(M,x_0),U(1))
\end{equation}
such that the morphism $\pi^*$ is then simply defined as the composition of \eqref{Equation: xi}, \eqref{Equation: from cat G to G hat} and \eqref{Equation: The pullback via identity}. \\

To check that for the morphism $\pi^*$ commutativity holds in the diagram it is essential to recall the construction of $\hat{\kappa}$ of \cite{GajerHigher}. Denote by $\cat{Z}_k(M)$ the geometric cycles in $M$. Gajer introduces a further equivalence relation $\sim$ on the geometric cycles such that the map
\[
\begin{array}{rcl}
     \Omega_{\mathrm{st}}^k(M,x_0) & \rightarrow & \cat{Z}_{k}(M)  \\
     f \colon S^k \rightarrow M & \mapsto & [f \colon S^k \rightarrow M]
\end{array}
\]
descends to a map
\begin{center}
\begin{tikzcd}
{\Omega_{\mathrm{st}}^k(M,x_0)} \arrow[d] \arrow[r] & \cat{Z}_k(M) \arrow[d] \\
{G^k(M,x_0)} \arrow[r, "\kappa"]                & \cat{Z}_k(M)/\!\sim
\end{tikzcd}
\end{center}
Notice that we have used here the fact that every element in the loop space $\Omega_{\mathrm{st}}^k(M,x_0)$ can be extended to give a smooth map $S^k \rightarrow M$. Moreover, Gajer shows that a Cheeger--Simons differential character can be represented by a pair $(\chi,\omega_{\chi})$ where $\chi \colon \cat{Z}_k(M)/\!\sim \rightarrow U(1)$ a morphism and $\omega_{\chi} \in \Omega^{k+1}(M)$ a differential form such that for all $[X \rightarrow M] \in \cat{C}_{k+1}(M)$
\[
\chi(\partial [X \xrightarrow{f} M] ) = \mathrm{exp} \left( 2\pi i \int_X f^*\omega_{\chi}  \right).
\]
To any such pair $(\chi,\omega_{\chi})$ via precomposition with $\kappa$ we get a group morphism
\[
G^k(M,x_0) \xrightarrow{\kappa} \cat{Z}_k(M)/\!\sim \xrightarrow{h} U(1).
\]
This assignment defines the morphism $\hat{\kappa}$. The commutativity of the diagram now follows from the fact that the isomorphism of homology groups
    \[
    \cat{H}_k(M) \xrightarrow{\cong} H_k(M)
    \]
is induced via the map\footnote{See \cite[p.9]{Bär-Becker}}
\[
\begin{array}{rcl}
     \psi_k \colon \cat{Z}_{k}(M) & \rightarrow & Z_k(M)/\partial C_{k+1}(M_k) = H_k(M_k) \\
     {[}X \xrightarrow{f} M{]} & \mapsto & {[}f_*(c){]}
\end{array}
\]
where here $c$ represents a smooth singular $k$-cycle representing the fundamental class of $X$ in $H_k(X)$. Note that for dimensional reasons it follows that $H_k(X) = H_k(X_k)$. Then denote by $[f_*(c)]$ the image of the class $[c]$ in $H_k(M_k)$ under the map $f$. In the case $X = S^k$ it therefore follows that this assignment agrees with the Hurewicz morphism. That is, commutativity holds in the diagram
\begin{center}
\begin{tikzcd}
{\Omega_{\mathrm{st}}^k(M,x_0)} \arrow[d] \arrow[r] & \cat{Z}_k(M) \arrow[d,"\psi_k"] \\
{G^k(M,x_0)} \arrow[r, "\xi"]                & H_k(M_k)
\end{tikzcd}
\end{center}
where the bottom map $\xi$ is to be considered as the component of the morphism of presheaves $\xi$ constructed above at the point $U = \mathrm{pt}$. This finishes the proof.
\end{proof}

\clearpage
\pagestyle{fancy}
\fancyhead[RO, LE]{\thepage}
\fancyhead[LO, RE]{\slshape\nouppercase{\leftmark}}
\fancyfoot{} 
\chapter{Outlook} \label{Chapter: Outlook}

So far we have only considered the case of higher circle bundles with connection. In the non-abelian setting, connections on non-abelian $H$-gerbes were introduced by Breen--Messing \cite{Breen-Messing} and Aschieri--Cantini--Jur\v{c}o \cite{Aschieri-Cantini-Jurco}, and by Schreiber--Waldorf for $\mathfrak{G}$-bundle gerbes, where $\mathfrak{G}$ is a Lie 2-group. Notice, that a non-abelian $H$-gerbe is equivalently an $\mathrm{Aut}(H)$-bundle gerbe where $\mathrm{Aut}(H)$ is the \textit{automorphism Lie 2-group} of $H$. The main difference between the two approaches to non-abelian differential cohomology, on one side given by Breen-Messing and Aschieri--Cantini--Jur\v{c}o and on the other side by Schreiber--Waldorf is, that the latter is based on the existence of higher dimensional parallel transport. More precisely, in a series of papers\footnote{ \cite{Schreiber-Waldorf-Smooth}, \cite{Schreiber-Waldorf-Parallel} and \cite{Schreiber-Waldorf-NonAb}} Schreiber--Waldorf show that $\mathfrak{G}$-bundle gerbes with \textit{fake-flat connections} are modeled by \textit{transport 2-functors} being a particular class of smooth 2-functors with domain the \textit{smooth path 2-groupoid} and target the delooping of $\mathfrak{G}$. This extra condition imposed on the connection data, called fake-flatness, is necessary for a consistent parallel transport to exist. Accordingly, they define the differential non-abelian cohomology of $M$ with values in $\mathfrak{G}$ as
\[
\hat{H}^2(M,\mathfrak{G}) := h_0 \mathrm{Trans}^2 \left(M, \mathfrak{G} \right)
\]
isomorphism classes of transport 2-functors on $M$ with $\mathbf{B}\mathfrak{G}$-structure. Here a transport $2$-functor $T$ on $M$ with $\mathbf{B}\mathfrak{G}$-structure is a smooth $2$-functor
\[
T \colon P_2(M) \rightarrow \mathbf{B} \mathfrak{G}
\]
from the path $2$-groupoid of the manifold $M$ into the delooping of $\mathfrak{G}$, such that there exists an open covering $\left\{ U_i\right\}_{i \in I}$ of $M$ trivializing\footnote{See \cite[Definition 2.1.1]{Schreiber-Waldorf-local} for a precise definition.} the functor $T$. Of particular interest is here the path 2-groupoid $P_2(M)$ whose construction we briefly recall now. \\

Given a smooth manifold $M$, denote by $P_{\mathrm{st}}M$ the set of stationary smooth paths in $M$. Recall that this is the set of smooth paths $\gamma : [0,1] \rightarrow M$ such that there exists $0 < \varepsilon < \frac{1}{2}$ with $\gamma(t) = \gamma(0)$ for all $t \in [0,\varepsilon)$ and $\gamma(t) = \gamma(1)$ for all $t \in (1-\varepsilon,1]$. Two such stationary paths $\gamma_1,\gamma_2 \colon x \rightarrow y$ are called \textit{thin homotopy equivalent} if there exists a stationary smooth homotopy $h \colon [0,1]^2 \rightarrow M$ between them, such that the differential of $h$ has at most rank 1. Thin homotopy equivalence defines an equivalence relation $\sim_1$ on the set of smooth stationary paths $P_{\mathrm{st}}M$ and its quotient is denoted $P^1M := P_{\mathrm{st}}M / \! \sim_1$. Recall also that the set of smooth paths $P_{\mathrm{st}}M$ carries a natural diffeology being a subspace of the function space $D([0,1],M)$ and we endow the quotient $P^1M$ with the quotient diffeology. The space $P^1M$ serves as the space of 1-morphisms of the smooth path 2-groupoid $P_2(M)$.   \\

Given two stationary paths $\gamma_1,\gamma_2 \colon x \rightarrow y$ in $M$ a \textit{bigon} $\Sigma \colon \gamma_1 \Rightarrow \gamma_2$ is given by a smooth map $\Sigma \colon [0,1]^2 \rightarrow$ which is stationary, i.e. there is $0 < \varepsilon < \frac{1}{2}$ such that
\begin{enumerate}
    \item $\Sigma(s,t) = x$ for $t \in [0,\varepsilon)$ and $\Sigma(s,t) = y$ for $t \in (1-\varepsilon,1]$.
    \item $\Sigma(s,t) = \gamma_1(t)$ for $s \in [0,\varepsilon)$ and $\Sigma(s,t) = \gamma_2(t)$ for $s \in (1-\varepsilon,1]$.
\end{enumerate}
The set of all such bigons is denoted $BX$ and is analogously endowed with the diffeology induced from the function space $D([0,1]^2,M)$. Two bigons $\Sigma \colon \gamma_1 \Rightarrow \gamma_2$ and $\Sigma' \colon \gamma'_1 \Rightarrow \gamma'_2$ are said to be \textit{thin homotopy equivalent} if there exists a smooth stationary homotopy $h \colon [0,1]^3 \rightarrow M$ between them, such that
\begin{enumerate}
    \item $\mathrm{rk} \left( dh|_{(r,s,t)} \right) \leq 2$ for all $(r,s,t) \in [0,1]^3$, and
    \item $\mathrm{rk} \left( dh|_{(r,i,t)} \right) \leq 1$ for $i = 0,1$ fixed.
\end{enumerate}
Thin homotopy equivalence of bigons defines an equivalence relation $\sim_2$ on the space of bigons and its quotient space, denoted $B^2M := BM / \! \sim_2$, is endowed with the quotient diffeology. The space $B^2M$ serves as the space of 2-morphisms of the smooth path 2-groupoid $P_2(M)$. \\

\begin{definition}[\cite{Schreiber-Waldorf-Smooth}]
    The \textbf{path 2-groupoid} $P_2(M)$ of a smooth manifold is the (diffeological) 2-category whose diffeological space of objects is given by $M$, the diffeological space of 1-morphisms by $P^1M$ and the diffeological space of 2-morphisms by $B^2M$.
\end{definition}

The natural question that now arises is if there exists a way to connect the approach to non-abelian cohomology given by transport functors to the approach presented in this thesis, i.e. to find a refinement $X$ of the enriched simplicial presheaf $\overline{M}$ such that given a Lie $2$-group $\mathfrak{G}$ we have
\[
\hat{H}^2(M,\mathfrak{G}) \overset{?}{\cong} H^1_{\infty}(X,\mathfrak{G}).
\]
To try to answer this question, we first have to be clear about what we mean by the right-hand expression. \\

For us, a 2-group means a crossed module $(G,H,\alpha,t)$, which always gives rise to a group object in the category of categories $\mathfrak{G}$. The objects of $\mathfrak{G}$ are given by $\mathrm{Ob}(\mathfrak{G}) = G$ and the morphisms by $\mathrm{Mor}(\mathfrak{G}) = G \ltimes H$. To see why $\mathfrak{G} \in \mathrm{Grp}(\mathbf{Cat})$ we refer to \cite[Appendix A.2]{Schreiber-Waldorf-Smooth}. Taking now the nerve of this category $\mathfrak{G}$ gives a simplicial group $N(\mathfrak{G})$.
\[
\begin{array}{rcccl}
\mathbf{CMod} & \rightarrow & \mathrm{Grp}\left( \mathbf{Cat} \right) & \xrightarrow{N} & \mathbf{sGrp} \\
(G,H,\alpha,t) & \mapsto & \mathfrak{G} & \mapsto & N(\mathfrak{G})
\end{array}
\]

In the case we are given a Lie crossed module $(G,H,\alpha,t)$ the above procedure gives accordingly a functor of presheaves on $\mathbf{Cart}$
\[
[\mathbf{Cart}^{\mathrm{op}},\mathbf{CMod}] \rightarrow [\mathbf{Cart}^{\mathrm{op}}, \mathbf{sGrp}].
\]
Hence, we identify a Lie crossed module $(G,H,\alpha,t)$ with a presheaf of simplicial groups on $\mathbf{Cart}$, which we denote by $N(\mathfrak{G})$. The non-abelian cohomology of a simplicial presheaf $X$ with values in the Lie crossed module $N(\mathfrak{G})$ is given by
\[
H^1_{\infty}(X,\mathfrak{G}) := \pi_0 \mathbb{R}\mathrm{Map} \left( X, \overline{W}N(\mathfrak{G}) \right),
\]
where the right-hand side denotes the derived mapping space for the local injective model structure $\mathrm{sPSh}(\mathbf{Cart})_{\mathrm{inj,loc}}$. \\

Now that the right-hand side has been defined, we wish to relate this non-abelian cohomology to the transport 2-functors introduced by Schreiber--Waldorf, upon choosing a suitable candidate for the simplicial presheaf $X$. To do so, notice that the adjunction $\cat{G} \dashv \overline{W}$ of Theorem \ref{Theorem: Classifying space adjunction} induces via pointwise application an adjunction of presheaves
\begin{center}
\begin{tikzcd}
\cat{G} \colon \mathrm{sPSh}(\mathbf{Cart}) \arrow[r, shift left=3] \arrow[r, "\perp", phantom] & {\mathrm{PSh}(\mathbf{Cart},\mathbf{sGpd}) \colon \overline{W}} \arrow[l, shift left=3]
\end{tikzcd}
\end{center}
By the right transfer of the local injective model structure along $\overline{W}$, the category of presheaves of simplicial groupoids $\mathrm{PSh}(\mathbf{Cart},\mathbf{sGpd})$ admits a model structure such that the adjunction $\cat{G} \dashv \overline{W}$ is a Quillen equivalence, see \cite[Theorem 9.50]{JardineBook}. Denote this model structure by $\mathrm{PSh}(\mathbf{Cart},\mathbf{sGpd})_{\mathrm{inj,loc}}$. Moreover, both functors $\overline{W}$ and $\cat{G}$ preserve weak equivalences. This is immediate in the case of $\overline{W}$ and follows from \cite[Lemma 3.4]{Luo-Bubenik-Kim} in the case of $\cat{G}$. As a direct consequence, there is a natural isomorphism
\begin{equation}\label{Equation: W-G Presheaf Adjunction}
\pi_0 \mathbb{R}\mathrm{Map} \left( X, \overline{W}N(\mathfrak{G}) \right) \cong \pi_0 \mathbb{R}\mathrm{Map} \left( \cat{G}(X), N(\mathfrak{G}) \right)
\end{equation}

for all simplicial presheaves $X$ and Lie 2-groups $\mathfrak{G}$. \\

To relate the derived mapping space on the right-hand side of (\ref{Equation: W-G Presheaf Adjunction}) to certain smooth 2-functors, consider $\mathbf{2Gpd}$ the category of 2-groupoids\footnote{By a 2-groupoid we always mean a strict 2-groupoid.} and notice that the classical fundamental groupoid and nerve adjunction $\pi \dashv N$ induces an adjoint pair of functors
\begin{center}
\begin{tikzcd}
\mathbf{sGpd} \arrow[r, "\Pi", shift right=-3] \arrow[r, "\perp", phantom] &  \mathbf{2Gpd}.  \arrow[l, "B", shift right=-3]
\end{tikzcd}
\end{center}
More precisely, given a 2-groupoid $H$ the simplicially enriched groupoid $B(H)$ has as set of objects $\mathrm{Ob}(B(H)) = \mathrm{Ob}(H)$ and as simplicial set of morphisms $\mathrm{Mor}(B(H)) = N(\mathrm{Mor}(H))$ given by the nerve of the groupoid of morphisms $\mathrm{Mor}(H)$ of $H$. \\

Conversely, given a simplicially enriched groupoid $\cat{H}$, the 2-groupoid $\Pi(\cat{H})$ is given by
\[
\mathrm{Ob} \left( \Pi(\cat{H}) \right) := \mathrm{Ob}(\cat{H}) \text{ and } \mathrm{Mor} \left( \Pi(\cat{H}) \right) :=  \pi \left(  \mathrm{Mor}(\cat{H})\right)
\]
where here $\pi \colon \mathbf{sSet} \rightarrow \mathbf{Gpd}$ denotes the fundamental groupoid functor. Unwinding the definition, the 2-groupoid $\Pi(\cat{H})$ has as objects the objects of $\cat{H}$, as 1-morphisms the morphisms of $\cat{H}$ and as 2-morphisms homotopy classes of paths in the space $\mathrm{Mor}(\cat{H})$. \\

Assume we are given a 2-group $\mathfrak{G}$ and let $\mathbf{B}\mathfrak{G}$ denote the delooping of $\mathfrak{G}$ as a strict 2-groupoid with a single object. We claim that $N(\mathfrak{G}) = B\left( \mathbf{B}\mathfrak{G} \right)$ where we consider $N(\mathfrak{G})$ as a simplicially enriched groupoid with a single object under the standard inclusion $\mathbf{sGrp} \hookrightarrow \mathbf{sGpd}$. Indeed, by definition $\mathrm{Ob}(\mathbf{B}\mathfrak{G}) = *$ and $\mathrm{Mor}(\mathbf{B}\mathfrak{G}) = \mathfrak{G}$ hence it follows that $\mathrm{Mor}(B(\mathbf{B}\mathfrak{G})) = N(\mathfrak{G})$.\\

Again a pointwise application of $\Pi \dashv B$ induces an adjunction of presheaves accordingly
\begin{center}
\begin{tikzcd}
{\Pi \colon \mathrm{PSh}(\mathbf{Cart},\mathbf{sGpd})} \arrow[r, shift left=3] \arrow[r, "\perp", phantom] & {\mathrm{PSh}(\mathbf{Cart},\mathbf{2Gpd}) \colon B} \arrow[l, shift left=3]
\end{tikzcd}
\end{center}
It follows from \cite[Theorem 9.57]{JardineBook} that the right transfer of the local injective model structure $\mathrm{PSh}(\mathbf{Cart},\mathbf{sGpd})_{\mathrm{inj,loc}}$ along the functor $B$ exists and that the adjoint pair defines a Quillen adjunction. In particular, this shows that $B$ preserves weak equivalences which also holds for $\Pi$ by \cite[Proposition 9.52]{JardineBook}. Therefore we can deduce that there are natural isomorphisms
\[
\pi_0 \mathbb{R}\mathrm{Map} \left( \cat{G}(X), N(\mathfrak{G}) \right) = \pi_0 \mathbb{R}\mathrm{Map} \left( \cat{G}(X), B\left( \mathbf{B}\mathfrak{G} \right) \right) \cong \left[ \Pi \cat{G}(X), \mathbf{B}\mathfrak{G} \right]
\]
where the right-hand side denotes the hom-set in the homotopy category of \newline $\mathrm{PSh}(\mathbf{Cart}, \mathbf{2Gpd})_{\mathrm{inj,loc}}$. \\

In conclusion, for an arbitrary simplicial presheaf $X$ and $\mathfrak{G}$ a Lie 2-group there is a chain of isomorphisms identifying the non-abelian cohomology of $X$ with values in $\mathfrak{G}$ with
\[
H^1_{\infty}(X,\mathfrak{G}) \cong \left[ \Pi \cat{G}(X), \mathbf{B}\mathfrak{G} \right]
\]
homotopy classes of smooth 2-functors taking values in $\mathbf{B}\mathfrak{G}$. This looks already much closer to the definition of non-abelian differential cohomology via transport 2-functors as before. \\

By construction, $\Pi \cat{G}(X)$ is a presheaf of 2-groupoids on the site $\mathbf{Cart}$ that can be understood as the fundamental 2-groupoid of the simplicial presheaf $X$. Thus we are left to find a suitable simplicial presheaf $X$ such that in turn the homotopy classes $\left[ \Pi \cat{G}(X), \mathbf{B}\mathfrak{G} \right]$ can be identified with isomorphism classes of transport 2-functors. We propose the conjecture that for $X = \mathbb{M}_{1,2}$ the simplicial presheaf given by
\begin{equation}
    \begin{array}{rcl}
      \mathbb{M}_{1,2} \colon \mathbf{Cart}^{\mathrm{op}} \times \Delta^{\mathrm{op}} & \rightarrow & \mathbf{Set}  \\
         (U, [n]) & \mapsto & C^{\infty}\left( \mathbb{A}^n, D\left( \mathbb{A}^n, D(U,M)_2 \right)_1 \right),
    \end{array}
\end{equation}
is such that its fundamental 2-groupoid $\Pi \cat{G}(\mathbb{M}_{1,2})$ incorporates an analog of the path 2-groupoid $P_2(M)$. This would then lead to the desired identification
\[
\hat{H}^2(M,\mathfrak{G}) \overset{?}{\cong} H^1_{\infty}\left( \mathbb{M}_{1,2} ,\mathfrak{G} \right).
\]

\clearpage
\pagestyle{fancy}
\fancyhead[RO, LE]{\thepage}
\fancyhead[LO, RE]{\slshape\nouppercase{\leftmark}}
\fancyfoot{}

\appendix

\chapter{Smooth Triangulations and the Subdivision Operator}

The goal of this appendix is twofold. First, to introduce the classical subdivision operator $S \colon C_{\bullet}(X) \rightarrow C_{\bullet}(X)$ defined on the chain complex of smooth singular chains taking values in a diffeological space $X$. The main reference here is \cite[Section 2.1]{Hatcher} for the classical chain complex of continuous singular chains. In the second part, we will then review smooth triangulations of manifolds and see how to associate a simplicial set with a simplicial complex upon choosing an ordering of the vertices. In particular, we show how the subdivision operator $S$ and smooth triangulations via barycentric subdivision are related to each other.

\section{The Subdivision Operator for Smooth Singular Homology}

The idea behind the construction of the subdivision operator $S$ for smooth chains taking values in diffeological spaces $X$ is analogous to the one for continuous chains. That is, one first defines a subdivision operator $ S \colon LC_{\bullet}(Y) \rightarrow \colon LC_{\bullet}(Y)$ where $Y$ is a convex subset in some cartesian space and where $LC_n(Y)$ denotes the free abelian group of linear chains in $Y$. Endowing the standard topological simplices and $Y \subset \mathbb{R}^k$ with the corresponding subset diffeologies, any linear chain $\sigma \colon |\Delta^n| \rightarrow Y$ is in particular also a smooth linear chain in $Y$. This then allows us to define the smooth subdivision operator for general smooth chains in some diffeological space $X$. \\

First notice that a linear map $\lambda \colon |\Delta^k| \rightarrow Y$ is completely determined by $[y_0,..., y_k]$ where $y_i$ is the image under $\lambda$ of the $i$th vertex in $|\Delta^k|$. As we will construct the subdivision using an iteration, it is convenient to introduce $LC_{-1}(Y) = \mathbb{Z}$ to be the free abelian group with generator given by the empty simplex. Then declare $\partial[y_0] = [\emptyset]$ for all $[y_0] \in LC_0(Y)$. \\

Define first the \textit{cone operator} for any choice $b \in Y$ as follows:
\begin{equation*}
    \begin{array}{rcl}
         b \colon LC_n(Y) & \rightarrow & LC_{n+1}(Y)  \\
        {[}y_0, ..., y_n {]}& \mapsto & {[}b,y_0, ..., y_n{]}.
    \end{array}
\end{equation*}
The cone operator in particular satisfies $\partial b + b \partial = \mathrm{id}$. Now define the subdivision operator $S$ inductively for $\lambda \colon |\Delta^k| \rightarrow Y$
\[
S(\lambda) := b_{\lambda}\left( S\left( \partial \lambda \right) \right),
\]
where $b_{\lambda}$ denotes the image under $\lambda$ of the barycenter in $|\Delta^k|$. Set $S([\emptyset]) = [\emptyset]$ so that we have $S([y_0]) = [y_0]$, i.e. $S$ is the identity on $0$-simplices. Given a linear $1$-simplex $\lambda = [y_0,y_1]$ we then have $S([y_0,y_1]) = [b_{\lambda},y_1] - [b_{\lambda},y_0]$. \\

Given a linear $2$-simplex $\lambda = [y_0,y_1,y_2]$ we then have
\begin{align*}
S([y_0,y_1,y_2]) &= b_{\lambda} \left( S\left( [y_0,y_1] + [y_1,y_2] - [y_0,y_2] \right)    \right)  \\
&= b_{\lambda} \left(   [b_2,y_1] - [b_2,y_0] + [b_0, y_2] - [b_0, y_1] - [b_1,y_2] + [b_1,y_0]     \right) \\
&= [b,b_2,y_1] - [b,b_2,y_0] + [b,b_0, y_2] - [b,b_0, y_1] - [b,b_1,y_2] + [b,b_1,y_0],
\end{align*}
which we think of as the following subdivision of the $2$-simplex.
\begin{center}
\begin{tikzpicture}[x=0.75pt,y=0.75pt,yscale=-1,xscale=1]

\draw    (152,136) -- (196,56) ;
\draw [shift={(196,56)}, rotate = 298.81] [color={rgb, 255:red, 0; green, 0; blue, 0 }  ][fill={rgb, 255:red, 0; green, 0; blue, 0 }  ][line width=0.75]      (0, 0) circle [x radius= 2.01, y radius= 2.01]   ;
\draw [shift={(175.73,92.85)}, rotate = 118.81] [color={rgb, 255:red, 0; green, 0; blue, 0 }  ][line width=0.75]    (6.56,-2.94) .. controls (4.17,-1.38) and (1.99,-0.4) .. (0,0) .. controls (1.99,0.4) and (4.17,1.38) .. (6.56,2.94)   ;
\draw [shift={(152,136)}, rotate = 298.81] [color={rgb, 255:red, 0; green, 0; blue, 0 }  ][fill={rgb, 255:red, 0; green, 0; blue, 0 }  ][line width=0.75]      (0, 0) circle [x radius= 2.01, y radius= 2.01]   ;
\draw    (196,160) -- (108,212) ;
\draw [shift={(108,212)}, rotate = 149.42] [color={rgb, 255:red, 0; green, 0; blue, 0 }  ][fill={rgb, 255:red, 0; green, 0; blue, 0 }  ][line width=0.75]      (0, 0) circle [x radius= 2.01, y radius= 2.01]   ;
\draw [shift={(148.9,187.83)}, rotate = 329.42] [color={rgb, 255:red, 0; green, 0; blue, 0 }  ][line width=0.75]    (6.56,-2.94) .. controls (4.17,-1.38) and (1.99,-0.4) .. (0,0) .. controls (1.99,0.4) and (4.17,1.38) .. (6.56,2.94)   ;
\draw [shift={(196,160)}, rotate = 149.42] [color={rgb, 255:red, 0; green, 0; blue, 0 }  ][fill={rgb, 255:red, 0; green, 0; blue, 0 }  ][line width=0.75]      (0, 0) circle [x radius= 2.01, y radius= 2.01]   ;
\draw    (152,136) -- (108,212) ;
\draw [shift={(108,212)}, rotate = 120.07] [color={rgb, 255:red, 0; green, 0; blue, 0 }  ][fill={rgb, 255:red, 0; green, 0; blue, 0 }  ][line width=0.75]      (0, 0) circle [x radius= 2.01, y radius= 2.01]   ;
\draw [shift={(128.2,177.12)}, rotate = 300.07] [color={rgb, 255:red, 0; green, 0; blue, 0 }  ][line width=0.75]    (6.56,-2.94) .. controls (4.17,-1.38) and (1.99,-0.4) .. (0,0) .. controls (1.99,0.4) and (4.17,1.38) .. (6.56,2.94)   ;
\draw [shift={(152,136)}, rotate = 120.07] [color={rgb, 255:red, 0; green, 0; blue, 0 }  ][fill={rgb, 255:red, 0; green, 0; blue, 0 }  ][line width=0.75]      (0, 0) circle [x radius= 2.01, y radius= 2.01]   ;
\draw    (196,212) -- (108,212) ;
\draw [shift={(108,212)}, rotate = 180] [color={rgb, 255:red, 0; green, 0; blue, 0 }  ][fill={rgb, 255:red, 0; green, 0; blue, 0 }  ][line width=0.75]      (0, 0) circle [x radius= 2.01, y radius= 2.01]   ;
\draw [shift={(148.4,212)}, rotate = 360] [color={rgb, 255:red, 0; green, 0; blue, 0 }  ][line width=0.75]    (6.56,-2.94) .. controls (4.17,-1.38) and (1.99,-0.4) .. (0,0) .. controls (1.99,0.4) and (4.17,1.38) .. (6.56,2.94)   ;
\draw [shift={(196,212)}, rotate = 180] [color={rgb, 255:red, 0; green, 0; blue, 0 }  ][fill={rgb, 255:red, 0; green, 0; blue, 0 }  ][line width=0.75]      (0, 0) circle [x radius= 2.01, y radius= 2.01]   ;
\draw    (196,212) -- (284,212) ;
\draw [shift={(284,212)}, rotate = 0] [color={rgb, 255:red, 0; green, 0; blue, 0 }  ][fill={rgb, 255:red, 0; green, 0; blue, 0 }  ][line width=0.75]      (0, 0) circle [x radius= 2.01, y radius= 2.01]   ;
\draw [shift={(243.6,212)}, rotate = 180] [color={rgb, 255:red, 0; green, 0; blue, 0 }  ][line width=0.75]    (6.56,-2.94) .. controls (4.17,-1.38) and (1.99,-0.4) .. (0,0) .. controls (1.99,0.4) and (4.17,1.38) .. (6.56,2.94)   ;
\draw [shift={(196,212)}, rotate = 0] [color={rgb, 255:red, 0; green, 0; blue, 0 }  ][fill={rgb, 255:red, 0; green, 0; blue, 0 }  ][line width=0.75]      (0, 0) circle [x radius= 2.01, y radius= 2.01]   ;
\draw    (240,136) -- (196,56) ;
\draw [shift={(196,56)}, rotate = 241.19] [color={rgb, 255:red, 0; green, 0; blue, 0 }  ][fill={rgb, 255:red, 0; green, 0; blue, 0 }  ][line width=0.75]      (0, 0) circle [x radius= 2.01, y radius= 2.01]   ;
\draw [shift={(216.27,92.85)}, rotate = 61.19] [color={rgb, 255:red, 0; green, 0; blue, 0 }  ][line width=0.75]    (6.56,-2.94) .. controls (4.17,-1.38) and (1.99,-0.4) .. (0,0) .. controls (1.99,0.4) and (4.17,1.38) .. (6.56,2.94)   ;
\draw [shift={(240,136)}, rotate = 241.19] [color={rgb, 255:red, 0; green, 0; blue, 0 }  ][fill={rgb, 255:red, 0; green, 0; blue, 0 }  ][line width=0.75]      (0, 0) circle [x radius= 2.01, y radius= 2.01]   ;
\draw    (240,136) -- (284,212) ;
\draw [shift={(284,212)}, rotate = 59.93] [color={rgb, 255:red, 0; green, 0; blue, 0 }  ][fill={rgb, 255:red, 0; green, 0; blue, 0 }  ][line width=0.75]      (0, 0) circle [x radius= 2.01, y radius= 2.01]   ;
\draw [shift={(263.8,177.12)}, rotate = 239.93] [color={rgb, 255:red, 0; green, 0; blue, 0 }  ][line width=0.75]    (6.56,-2.94) .. controls (4.17,-1.38) and (1.99,-0.4) .. (0,0) .. controls (1.99,0.4) and (4.17,1.38) .. (6.56,2.94)   ;
\draw [shift={(240,136)}, rotate = 59.93] [color={rgb, 255:red, 0; green, 0; blue, 0 }  ][fill={rgb, 255:red, 0; green, 0; blue, 0 }  ][line width=0.75]      (0, 0) circle [x radius= 2.01, y radius= 2.01]   ;
\draw    (196,160) -- (240,136) ;
\draw [shift={(240,136)}, rotate = 331.39] [color={rgb, 255:red, 0; green, 0; blue, 0 }  ][fill={rgb, 255:red, 0; green, 0; blue, 0 }  ][line width=0.75]      (0, 0) circle [x radius= 2.01, y radius= 2.01]   ;
\draw [shift={(221.16,146.28)}, rotate = 151.39] [color={rgb, 255:red, 0; green, 0; blue, 0 }  ][line width=0.75]    (6.56,-2.94) .. controls (4.17,-1.38) and (1.99,-0.4) .. (0,0) .. controls (1.99,0.4) and (4.17,1.38) .. (6.56,2.94)   ;
\draw [shift={(196,160)}, rotate = 331.39] [color={rgb, 255:red, 0; green, 0; blue, 0 }  ][fill={rgb, 255:red, 0; green, 0; blue, 0 }  ][line width=0.75]      (0, 0) circle [x radius= 2.01, y radius= 2.01]   ;
\draw    (196,160) -- (284,212) ;
\draw [shift={(284,212)}, rotate = 30.58] [color={rgb, 255:red, 0; green, 0; blue, 0 }  ][fill={rgb, 255:red, 0; green, 0; blue, 0 }  ][line width=0.75]      (0, 0) circle [x radius= 2.01, y radius= 2.01]   ;
\draw [shift={(243.1,187.83)}, rotate = 210.58] [color={rgb, 255:red, 0; green, 0; blue, 0 }  ][line width=0.75]    (6.56,-2.94) .. controls (4.17,-1.38) and (1.99,-0.4) .. (0,0) .. controls (1.99,0.4) and (4.17,1.38) .. (6.56,2.94)   ;
\draw [shift={(196,160)}, rotate = 30.58] [color={rgb, 255:red, 0; green, 0; blue, 0 }  ][fill={rgb, 255:red, 0; green, 0; blue, 0 }  ][line width=0.75]      (0, 0) circle [x radius= 2.01, y radius= 2.01]   ;
\draw    (196,160) -- (152,136) ;
\draw [shift={(152,136)}, rotate = 208.61] [color={rgb, 255:red, 0; green, 0; blue, 0 }  ][fill={rgb, 255:red, 0; green, 0; blue, 0 }  ][line width=0.75]      (0, 0) circle [x radius= 2.01, y radius= 2.01]   ;
\draw [shift={(170.84,146.28)}, rotate = 28.61] [color={rgb, 255:red, 0; green, 0; blue, 0 }  ][line width=0.75]    (6.56,-2.94) .. controls (4.17,-1.38) and (1.99,-0.4) .. (0,0) .. controls (1.99,0.4) and (4.17,1.38) .. (6.56,2.94)   ;
\draw [shift={(196,160)}, rotate = 208.61] [color={rgb, 255:red, 0; green, 0; blue, 0 }  ][fill={rgb, 255:red, 0; green, 0; blue, 0 }  ][line width=0.75]      (0, 0) circle [x radius= 2.01, y radius= 2.01]   ;
\draw    (196,160) -- (196,56) ;
\draw [shift={(196,56)}, rotate = 270] [color={rgb, 255:red, 0; green, 0; blue, 0 }  ][fill={rgb, 255:red, 0; green, 0; blue, 0 }  ][line width=0.75]      (0, 0) circle [x radius= 2.01, y radius= 2.01]   ;
\draw [shift={(196,104.4)}, rotate = 90] [color={rgb, 255:red, 0; green, 0; blue, 0 }  ][line width=0.75]    (6.56,-2.94) .. controls (4.17,-1.38) and (1.99,-0.4) .. (0,0) .. controls (1.99,0.4) and (4.17,1.38) .. (6.56,2.94)   ;
\draw [shift={(196,160)}, rotate = 270] [color={rgb, 255:red, 0; green, 0; blue, 0 }  ][fill={rgb, 255:red, 0; green, 0; blue, 0 }  ][line width=0.75]      (0, 0) circle [x radius= 2.01, y radius= 2.01]   ;
\draw    (196,160) -- (196,212) ;
\draw [shift={(196,212)}, rotate = 90] [color={rgb, 255:red, 0; green, 0; blue, 0 }  ][fill={rgb, 255:red, 0; green, 0; blue, 0 }  ][line width=0.75]      (0, 0) circle [x radius= 2.01, y radius= 2.01]   ;
\draw [shift={(196,189.6)}, rotate = 270] [color={rgb, 255:red, 0; green, 0; blue, 0 }  ][line width=0.75]    (6.56,-2.94) .. controls (4.17,-1.38) and (1.99,-0.4) .. (0,0) .. controls (1.99,0.4) and (4.17,1.38) .. (6.56,2.94)   ;
\draw [shift={(196,160)}, rotate = 90] [color={rgb, 255:red, 0; green, 0; blue, 0 }  ][fill={rgb, 255:red, 0; green, 0; blue, 0 }  ][line width=0.75]      (0, 0) circle [x radius= 2.01, y radius= 2.01]   ;

\draw (85,204.4) node [anchor=north west][inner sep=0.75pt]    {$v_{0}$};
\draw (293,202.4) node [anchor=north west][inner sep=0.75pt]    {$v_{1}$};
\draw (189,34.4) node [anchor=north west][inner sep=0.75pt]    {$v_{2}$};
\draw (242,118.4) node [anchor=north west][inner sep=0.75pt]    {$b_{0}$};
\draw (130,116.4) node [anchor=north west][inner sep=0.75pt]    {$b_{1}$};
\draw (189,216.4) node [anchor=north west][inner sep=0.75pt]    {$b_{2}$};
\draw (199,134.4) node [anchor=north west][inner sep=0.75pt]    {$b$};

\end{tikzpicture}
\end{center}
Notice that the signs in the sum correspond to the orientation of the specific simplex and are exactly such that taking the boundary $\partial S(\lambda) = S(\partial \lambda)$, in other words, $S$ defines a chain map. Now define the subdivision operator on smooth singular chains taking values in $X$ as follows. For some smooth map $f \colon Y \rightarrow X$ where $Y \subset \mathbb{R}^n$ is some convex subset endowed with the subset diffeology, denote by
\begin{equation*}
    \begin{array}{rcl}
         f_{\#} \colon LC_{k}(Y) & \rightarrow & C_k(X)  \\
         \lambda & \mapsto & f \circ \lambda
    \end{array}
\end{equation*}
the induced chain map from smooth linear chains in $Y$ to smooth singular chains in $X$. Given now some smooth simplex $\sigma \colon |\Delta^k| \rightarrow X$ define
\begin{equation*}
\begin{array}{rcl}
     S \colon C_{k}(X) & \rightarrow & C_{k}(X) \\
    \sigma  & \mapsto & \sigma_{\#}S(|\Delta^k|).
\end{array}
\end{equation*}
That is, $S(\sigma)$ is a signed sum of $\sigma$ restricted to the subsimplices in $|\Delta^{k}|$ specified by the barycentric subdivision. \\

We now wish to define a chain homotopy $T \colon C_k(X) \rightarrow C_{k+1}(X)$ between $S$ and the identity. As before, we first define $T$ inductively as an operator acting on linear chains. That is, $T = 0$ for $k = -1$ and  then further $T(\lambda) = b_{\lambda}(\lambda - T(\partial \lambda))$. That is, $T([y_0]) = y_0 \left( [y_0] \right) = [y_0,y_0]$ and $T([y_0,y_1]) = [b,y_0,y_1] -[b,y_1,y_1] + [b,y_0,y_0]$. After verifying that the chain homotopy formula $\partial T + T \partial = \mathrm{id} - S$ is satisfied for linear chains, extend the chain homotopy to smooth chains as before via
\begin{equation*}
    \begin{array}{rcl}
        T \colon C_k(X) & \rightarrow & C_{k+1}(X) \\
         \sigma \colon |\Delta^k| \rightarrow X & \mapsto & \sigma_{\#} T(|\Delta^k|).
    \end{array}
\end{equation*}

\begin{lemma} \label{Lemma: Subdivision Operator}
Given a diffeological space $X$ the subdivision operator \newline $S \colon C_{\bullet}(X) \rightarrow \colon C_{\bullet}(X)$ is chain homotopic to the identity with chain homotopy provided by $T \colon C_{\bullet}(X) \rightarrow \colon C_{\bullet + 1}(X)$. Moreover, for any $r \geq 1$ the iterated barycentric subdivision operator $S^r$ is chain homotopic to the identity with chain homotopy provided by $D_r = \sum_{0 \leq i < r} TS^i$.
\end{lemma}
\begin{proof}
    The proof is exactly like in the continuous case, see \cite[pp.~121-124]{Hatcher}.
 \end{proof}

\begin{corollary} \label{Corollary: Subdivision Operator}
    Let $M$ be a smooth manifold, $U \in \mathbf{Cart}$ some Cartesian space and $\eta \colon |\Delta^{k+1}| \rightarrow D(U,M)$ a smooth $(k+1)$-simplex. Its boundary $\partial(\eta) \in Z_{k}(D(U,M))$ is a smooth $k$-cycle in $D(U,M)$ and likewise $S(\partial(\eta)) = \partial S(\eta) \in Z_k(D(U,M))$ defines a smooth $k$-cycle in $D(U,M)$. In particular, the two homology classes
    \[
    [\partial \eta]_k = [\partial S(\eta)]_k
    \]
    agree as elements in $H^{\Delta}_k(D(U,M)_k)$. Moreover, the same equality also holds considering the iterated subdivision operator $S^r$, i.e.
     \[
    [\partial \eta]_k = [\partial S^r(\eta)]_k.
    \]
\end{corollary}

\begin{proof}
    This is a direct consequence of Lemma \ref{Lemma: Subdivision Operator}. Indeed, we wish to show that for any $\eta$ there exists a $k$-thin smooth $(k+1)$-chain $\gamma \in C_{k+1}(D(U,M)_k)$ such that
    \begin{equation} \label{Equation: The thin chain}
    \partial \gamma = \partial \eta - \partial S(\eta).
    \end{equation}
    Of course, the chain homotopy $T$ constructed in the previous Lemma allows for the construction of such a $k$-thin chain for any $\eta$. Indeed, $T$ being a chain homotopy between $S$ and the identity  implies that for any $\eta \in C_{k+1}(D(U,M))$ one has
    \[
    \partial T(\eta) + T(\partial \eta) = \eta - S(\eta).
    \]
    Applying the boundary operator then gives
    \[
    \partial T\left( \partial \eta \right) = \partial \eta - \partial S(\eta).
    \]
    Notice that here $T \colon C_{k}(D(U,M)) \rightarrow C_{k+1}(D(U,M))$ since $T$ acts on the boundary of $\eta$. By definition of the chain homotopy $T$ together with Lemma \ref{Lemma: Compact k-skeletal k-chains} one confirms that $T$ actually takes values in $C_{k+1}(D(U,M)_k)$, i.e.
    \[
    T \colon \colon C_{k}(D(U,M)) \rightarrow C_{k+1}(D(U,M)_k).
    \]
    This in particular implies that $\gamma:= T(\partial \eta)$ can be constructed for any $\eta$ and indeed satisfies (\ref{Equation: The thin chain}). The same argument applies also in the case of the iterated subdivision operator.
\end{proof}

Given a smooth singular $(k+1)$-simplex $\eta \colon |\Delta^{k+1}| \rightarrow X$ taking values in some diffeological space $X$ and $\sigma \in S_{k+2}$ a permutation of the vertices $\{v_0, ..., v_{k+1}\}$  of $|\Delta^{k+1}|$, define the homomorphism
\begin{equation*}
    \begin{array}{rcl}
        \sigma_* \colon C_{k+1}(X) & \rightarrow & C_{k+1}(X) \\
         \eta & \mapsto & \left( \eta \circ \sigma \right)
    \end{array}
\end{equation*}
where here the expression $\eta \circ \sigma$ denotes the composition of the linear map $\sigma \colon |\Delta^{k+1}| \rightarrow |\Delta^{k+1}|$ permuting the vertices, i.e. $\sigma(v_i) = v_{\sigma(i)}$, with the smooth simplex $\eta$. \\

\begin{corollary} \label{Corollary: Change of Orientation}
    Let $M$ be a smooth manifold, $U \in \mathbf{Cart}$ some Cartesian space and $\eta \colon |\Delta^{k+1}| \rightarrow M$ a smooth $(k+1)$-simplex in $M$ and $\sigma \in S_{k+2}$ a permutation, then the boundary $\partial(\eta) \in Z_{k}(D(U,M))$ is a smooth $k$-cycle in $D(U,M)$ and likewise $\partial(\sigma_*(\eta)) \in Z_{k}(D(U,M))$. In particular, the homology classes
    \[
    \left[ \partial(\eta) \right]_k =  \mathrm{sgn}(\sigma)  \left[ \partial\left( \sigma_*(\eta)  \right) \right]_k =   \mathrm{sgn}(\sigma) \left[ \partial\left( \eta \circ \sigma \right) \right]_k
    \]
    agree as elements in $H_k(D(U,M)_k)$.
\end{corollary}
\begin{proof}
The key idea behind this result is that the barycentric subdivision of a smooth $(k+1)$-simplex $\eta$ satisfies
\begin{equation} \label{Equation: Subdivision and permutations}
S(\eta) = \mathrm{sgn}(\sigma) S \left( \sigma_*(\eta) \right).
\end{equation}

The observation of this fact for simplices of low dimension goes back to the work of Barr \cite[p.~7]{Barr} and was shown to hold in general by Kock \cite[p.~261]{Kock}. Identity (\ref{Equation: Subdivision and permutations}) together with Corollary \ref{Corollary: Change of Orientation} yields
\begin{equation*}
\mathrm{sgn}(\sigma)[ \partial(\sigma_*(\eta))]_k = \mathrm{sgn}(\sigma)[ \partial S(\sigma_*(\eta))]_k = [\partial S(\eta)]_k = [\partial \eta]_k.
\end{equation*}

\end{proof}

\section{Smooth Triangulations} \label{Appendix: Smooth Triangulations}

\begin{definition}[\cite{Munkres}, Definition 7.1]
The convex hull of linearly independent points $\left\{v_0,...,v_m \right\}$ in $\mathbb{R}^n$ we call a \textbf{linear} $m$\textbf{-simplex}. A \textbf{simplicial complex} $K$ in $\mathbb{R}^n$ is a collection of linear simplices in $\mathbb{R}^n$ such that
\begin{enumerate}
    \item if $\sigma \in K$ and $\tau$ is a face of $\sigma$, then also $\tau \in K$,
    \item if $\sigma,\tau \in K$ and $\sigma \cap \tau \neq \emptyset$ then $\sigma \cap \tau$ is a face of both $\tau$ and $\sigma$,
    \item for each point $x \in |K|:= \bigcup_{\sigma \in K} \sigma$ there exists a neighbourhood in $\mathbb{R}^n$ of $x$ intersecting only finitely many simplices of $K$.
\end{enumerate}
The union of all the simplices in $K$ denoted by $|K|$ is called the \textbf{polyhedron of} $K$.
\end{definition}
\begin{remark}
   A simplicial complex $K$ together with a partial order $\leq$ on the set of vertices $V(K)$, such that when restricted to each simplex in $K$ it becomes a total ordering, we call an \textbf{ordered simplicial complex}. To each ordered simplicial complex $(K,\leq)$ we can associate a simplicial set
   \[
   S_{\leq}(K)_n := \left\{ (v_0, ..., v_n) \in V(K)^{n+1} \, | \, v_0 \leq \cdots \leq v_n,  \left\{v_0, ..., v_n \right\} \in K \right\},
   \]
   where here $\left\{v_0, ..., v_n \right\}$ denotes the convex hull spanned by the vertices $v_i$. Notice that we are allowed to repeat vertices, i.e. $(v_0,v_0)$ is indeed a $1$-simplex since $v_0 \leq v_0$ and  $\left\{v_0,v_0 \right\} = v_0 \in K$. Elements of this form provide precisely the degenerate simplices in $S_{\leq}(K)$. The face and degeneracy maps are then defined by:
   \begin{align*}
       &d_i \left( v_0, ..., v_n \right) = (v_0, ..., v_{i-1}, v_{i+1}, ..., v_n) \\
       &s_i \left( v_0 , ..., v_n \right) = (v_0, ..., v_i, v_i, ..., v_n).
   \end{align*}
\end{remark}

\begin{lemma}[\cite{Curtis}, 1.29] \label{Lemma: Geometric Realization Simplicial Complex}
    Let $(K,\leq)$ be an ordered simplicial complex. Then the geometric realization of its associated simplicial set $S_{\leq}(K)$ is homeomorphic to the geometric realization of K,
    \[
    |S_{\leq}(K)| \cong |K|.
    \]
\end{lemma}

\begin{definition}[\cite{Munkres}, Definition 8.3] \label{Definition: Smooth Triangulation}
    Let $K$ be a simplicial complex and $M$ a smooth manifold. A map $f \colon |K| \rightarrow M$ is said to be \textbf{piecewise differentiable} if the restriction of $f$ to each simplex $\sigma \in K$ is smooth. We call $f$ a \textbf{smooth triangulation of} $M$ if $f$ is a piecewise differentiable homeomorphism and the restriction of $f$ to each simplex is an immersion.
\end{definition}

In the case where $M$ is an oriented manifold and we are given a smooth triangulation $(K,f)$ of $M$ the orientation of $M$ determines a partial order $\leq$ on the set of vertices of $K$. This order turns $(K,\leq)$ into an ordered simplicial complex, and accordingly the polyhedron of $K$ is homeomorphic to the geometric realization of $S_{\leq}(K)$, i.e. $|K| \cong |S_{\leq}(K)|$. This is made precise in the next definition.

\begin{definition} \label{Definition: Oriented smooth triangulation}
Let $M$ be an $n$-dimensional, connected, compact, and oriented manifold with boundary $\partial M$. The orientation is specified by a choice of generator $[M,\partial M] \in H_n(M,\partial M; \mathbb{Z})$ also called a \textbf{fundamental class}. An \textbf{oriented smooth triangulation} of $M$ is a smooth triangulation $(K,f)$ of $M$ together with a partial order $\leq$ on $V(K)$ such that the the chain $\sum_{ \sigma \in NS_{\leq}(K)_n} \sigma$  of non-degenerate $n$-simplices of $S_{\leq}(K)$ is closed, and its class $[K,\partial K]$
\[
\begin{array}{rcl}
  f_* \colon H_n(K,\partial K;\mathbb{Z})  & \xrightarrow{\cong} & H_n(M,\partial M; \mathbb{Z}) \\
    {[}K,\partial K{]} & \mapsto & {[}M,\partial M{]}
\end{array}
\]
is mapped to the fundamental class under the pushforward by $f$.
\end{definition}

\begin{example} \label{Example: Subdivision}
    Consider $M = |\Delta^k|$ the smooth compact $k$-simplex as a manifold with boundary and corners embedded in $\mathbb{R}^k$. Performing barycentric subdivision of $|\Delta^k|$ defines a simplicial complex $K$ together with a homeomorphism $f \colon |K| \rightarrow |\Delta^k|$ which exhibits the pair $(K,f)$ as a smooth triangulation of $|\Delta^k|$. Choosing a partial order $\leq$ on the set of vertices in $K$ compatible with the standard orientation of $|\Delta^k|$ gives an oriented simplicial complex $(K,\leq)$ and in particular an oriented triangulation. The following diagram shows such a partial order in the case $k = 2$.
\begin{center}
   \begin{tikzpicture}[x=0.75pt,y=0.75pt,yscale=-1,xscale=1]

\draw    (196,56) -- (152,136) ;
\draw [shift={(152,136)}, rotate = 118.81] [color={rgb, 255:red, 0; green, 0; blue, 0 }  ][fill={rgb, 255:red, 0; green, 0; blue, 0 }  ][line width=0.75]      (0, 0) circle [x radius= 2.01, y radius= 2.01]   ;
\draw [shift={(172.27,99.15)}, rotate = 298.81] [color={rgb, 255:red, 0; green, 0; blue, 0 }  ][line width=0.75]    (6.56,-2.94) .. controls (4.17,-1.38) and (1.99,-0.4) .. (0,0) .. controls (1.99,0.4) and (4.17,1.38) .. (6.56,2.94)   ;
\draw [shift={(196,56)}, rotate = 118.81] [color={rgb, 255:red, 0; green, 0; blue, 0 }  ][fill={rgb, 255:red, 0; green, 0; blue, 0 }  ][line width=0.75]      (0, 0) circle [x radius= 2.01, y radius= 2.01]   ;
\draw    (196,160) -- (108,212) ;
\draw [shift={(108,212)}, rotate = 149.42] [color={rgb, 255:red, 0; green, 0; blue, 0 }  ][fill={rgb, 255:red, 0; green, 0; blue, 0 }  ][line width=0.75]      (0, 0) circle [x radius= 2.01, y radius= 2.01]   ;
\draw [shift={(148.9,187.83)}, rotate = 329.42] [color={rgb, 255:red, 0; green, 0; blue, 0 }  ][line width=0.75]    (6.56,-2.94) .. controls (4.17,-1.38) and (1.99,-0.4) .. (0,0) .. controls (1.99,0.4) and (4.17,1.38) .. (6.56,2.94)   ;
\draw [shift={(196,160)}, rotate = 149.42] [color={rgb, 255:red, 0; green, 0; blue, 0 }  ][fill={rgb, 255:red, 0; green, 0; blue, 0 }  ][line width=0.75]      (0, 0) circle [x radius= 2.01, y radius= 2.01]   ;
\draw    (152,136) -- (108,212) ;
\draw [shift={(108,212)}, rotate = 120.07] [color={rgb, 255:red, 0; green, 0; blue, 0 }  ][fill={rgb, 255:red, 0; green, 0; blue, 0 }  ][line width=0.75]      (0, 0) circle [x radius= 2.01, y radius= 2.01]   ;
\draw [shift={(128.2,177.12)}, rotate = 300.07] [color={rgb, 255:red, 0; green, 0; blue, 0 }  ][line width=0.75]    (6.56,-2.94) .. controls (4.17,-1.38) and (1.99,-0.4) .. (0,0) .. controls (1.99,0.4) and (4.17,1.38) .. (6.56,2.94)   ;
\draw [shift={(152,136)}, rotate = 120.07] [color={rgb, 255:red, 0; green, 0; blue, 0 }  ][fill={rgb, 255:red, 0; green, 0; blue, 0 }  ][line width=0.75]      (0, 0) circle [x radius= 2.01, y radius= 2.01]   ;
\draw    (108,212) -- (196,212) ;
\draw [shift={(196,212)}, rotate = 0] [color={rgb, 255:red, 0; green, 0; blue, 0 }  ][fill={rgb, 255:red, 0; green, 0; blue, 0 }  ][line width=0.75]      (0, 0) circle [x radius= 2.01, y radius= 2.01]   ;
\draw [shift={(155.6,212)}, rotate = 180] [color={rgb, 255:red, 0; green, 0; blue, 0 }  ][line width=0.75]    (6.56,-2.94) .. controls (4.17,-1.38) and (1.99,-0.4) .. (0,0) .. controls (1.99,0.4) and (4.17,1.38) .. (6.56,2.94)   ;
\draw [shift={(108,212)}, rotate = 0] [color={rgb, 255:red, 0; green, 0; blue, 0 }  ][fill={rgb, 255:red, 0; green, 0; blue, 0 }  ][line width=0.75]      (0, 0) circle [x radius= 2.01, y radius= 2.01]   ;
\draw    (196,212) -- (284,212) ;
\draw [shift={(284,212)}, rotate = 0] [color={rgb, 255:red, 0; green, 0; blue, 0 }  ][fill={rgb, 255:red, 0; green, 0; blue, 0 }  ][line width=0.75]      (0, 0) circle [x radius= 2.01, y radius= 2.01]   ;
\draw [shift={(243.6,212)}, rotate = 180] [color={rgb, 255:red, 0; green, 0; blue, 0 }  ][line width=0.75]    (6.56,-2.94) .. controls (4.17,-1.38) and (1.99,-0.4) .. (0,0) .. controls (1.99,0.4) and (4.17,1.38) .. (6.56,2.94)   ;
\draw [shift={(196,212)}, rotate = 0] [color={rgb, 255:red, 0; green, 0; blue, 0 }  ][fill={rgb, 255:red, 0; green, 0; blue, 0 }  ][line width=0.75]      (0, 0) circle [x radius= 2.01, y radius= 2.01]   ;
\draw    (240,136) -- (196,56) ;
\draw [shift={(196,56)}, rotate = 241.19] [color={rgb, 255:red, 0; green, 0; blue, 0 }  ][fill={rgb, 255:red, 0; green, 0; blue, 0 }  ][line width=0.75]      (0, 0) circle [x radius= 2.01, y radius= 2.01]   ;
\draw [shift={(216.27,92.85)}, rotate = 61.19] [color={rgb, 255:red, 0; green, 0; blue, 0 }  ][line width=0.75]    (6.56,-2.94) .. controls (4.17,-1.38) and (1.99,-0.4) .. (0,0) .. controls (1.99,0.4) and (4.17,1.38) .. (6.56,2.94)   ;
\draw [shift={(240,136)}, rotate = 241.19] [color={rgb, 255:red, 0; green, 0; blue, 0 }  ][fill={rgb, 255:red, 0; green, 0; blue, 0 }  ][line width=0.75]      (0, 0) circle [x radius= 2.01, y radius= 2.01]   ;
\draw    (284,212) -- (240,136) ;
\draw [shift={(240,136)}, rotate = 239.93] [color={rgb, 255:red, 0; green, 0; blue, 0 }  ][fill={rgb, 255:red, 0; green, 0; blue, 0 }  ][line width=0.75]      (0, 0) circle [x radius= 2.01, y radius= 2.01]   ;
\draw [shift={(260.2,170.88)}, rotate = 59.93] [color={rgb, 255:red, 0; green, 0; blue, 0 }  ][line width=0.75]    (6.56,-2.94) .. controls (4.17,-1.38) and (1.99,-0.4) .. (0,0) .. controls (1.99,0.4) and (4.17,1.38) .. (6.56,2.94)   ;
\draw [shift={(284,212)}, rotate = 239.93] [color={rgb, 255:red, 0; green, 0; blue, 0 }  ][fill={rgb, 255:red, 0; green, 0; blue, 0 }  ][line width=0.75]      (0, 0) circle [x radius= 2.01, y radius= 2.01]   ;
\draw    (196,160) -- (240,136) ;
\draw [shift={(240,136)}, rotate = 331.39] [color={rgb, 255:red, 0; green, 0; blue, 0 }  ][fill={rgb, 255:red, 0; green, 0; blue, 0 }  ][line width=0.75]      (0, 0) circle [x radius= 2.01, y radius= 2.01]   ;
\draw [shift={(221.16,146.28)}, rotate = 151.39] [color={rgb, 255:red, 0; green, 0; blue, 0 }  ][line width=0.75]    (6.56,-2.94) .. controls (4.17,-1.38) and (1.99,-0.4) .. (0,0) .. controls (1.99,0.4) and (4.17,1.38) .. (6.56,2.94)   ;
\draw [shift={(196,160)}, rotate = 331.39] [color={rgb, 255:red, 0; green, 0; blue, 0 }  ][fill={rgb, 255:red, 0; green, 0; blue, 0 }  ][line width=0.75]      (0, 0) circle [x radius= 2.01, y radius= 2.01]   ;
\draw    (196,160) -- (284,212) ;
\draw [shift={(284,212)}, rotate = 30.58] [color={rgb, 255:red, 0; green, 0; blue, 0 }  ][fill={rgb, 255:red, 0; green, 0; blue, 0 }  ][line width=0.75]      (0, 0) circle [x radius= 2.01, y radius= 2.01]   ;
\draw [shift={(243.1,187.83)}, rotate = 210.58] [color={rgb, 255:red, 0; green, 0; blue, 0 }  ][line width=0.75]    (6.56,-2.94) .. controls (4.17,-1.38) and (1.99,-0.4) .. (0,0) .. controls (1.99,0.4) and (4.17,1.38) .. (6.56,2.94)   ;
\draw [shift={(196,160)}, rotate = 30.58] [color={rgb, 255:red, 0; green, 0; blue, 0 }  ][fill={rgb, 255:red, 0; green, 0; blue, 0 }  ][line width=0.75]      (0, 0) circle [x radius= 2.01, y radius= 2.01]   ;
\draw    (196,160) -- (152,136) ;
\draw [shift={(152,136)}, rotate = 208.61] [color={rgb, 255:red, 0; green, 0; blue, 0 }  ][fill={rgb, 255:red, 0; green, 0; blue, 0 }  ][line width=0.75]      (0, 0) circle [x radius= 2.01, y radius= 2.01]   ;
\draw [shift={(170.84,146.28)}, rotate = 28.61] [color={rgb, 255:red, 0; green, 0; blue, 0 }  ][line width=0.75]    (6.56,-2.94) .. controls (4.17,-1.38) and (1.99,-0.4) .. (0,0) .. controls (1.99,0.4) and (4.17,1.38) .. (6.56,2.94)   ;
\draw [shift={(196,160)}, rotate = 208.61] [color={rgb, 255:red, 0; green, 0; blue, 0 }  ][fill={rgb, 255:red, 0; green, 0; blue, 0 }  ][line width=0.75]      (0, 0) circle [x radius= 2.01, y radius= 2.01]   ;
\draw    (196,160) -- (196,56) ;
\draw [shift={(196,56)}, rotate = 270] [color={rgb, 255:red, 0; green, 0; blue, 0 }  ][fill={rgb, 255:red, 0; green, 0; blue, 0 }  ][line width=0.75]      (0, 0) circle [x radius= 2.01, y radius= 2.01]   ;
\draw [shift={(196,104.4)}, rotate = 90] [color={rgb, 255:red, 0; green, 0; blue, 0 }  ][line width=0.75]    (6.56,-2.94) .. controls (4.17,-1.38) and (1.99,-0.4) .. (0,0) .. controls (1.99,0.4) and (4.17,1.38) .. (6.56,2.94)   ;
\draw [shift={(196,160)}, rotate = 270] [color={rgb, 255:red, 0; green, 0; blue, 0 }  ][fill={rgb, 255:red, 0; green, 0; blue, 0 }  ][line width=0.75]      (0, 0) circle [x radius= 2.01, y radius= 2.01]   ;
\draw    (196,160) -- (196,212) ;
\draw [shift={(196,212)}, rotate = 90] [color={rgb, 255:red, 0; green, 0; blue, 0 }  ][fill={rgb, 255:red, 0; green, 0; blue, 0 }  ][line width=0.75]      (0, 0) circle [x radius= 2.01, y radius= 2.01]   ;
\draw [shift={(196,189.6)}, rotate = 270] [color={rgb, 255:red, 0; green, 0; blue, 0 }  ][line width=0.75]    (6.56,-2.94) .. controls (4.17,-1.38) and (1.99,-0.4) .. (0,0) .. controls (1.99,0.4) and (4.17,1.38) .. (6.56,2.94)   ;
\draw [shift={(196,160)}, rotate = 90] [color={rgb, 255:red, 0; green, 0; blue, 0 }  ][fill={rgb, 255:red, 0; green, 0; blue, 0 }  ][line width=0.75]      (0, 0) circle [x radius= 2.01, y radius= 2.01]   ;

\draw (85,204.4) node [anchor=north west][inner sep=0.75pt]    {$v_{0}$};
\draw (293,202.4) node [anchor=north west][inner sep=0.75pt]    {$v_{1}$};
\draw (189,34.4) node [anchor=north west][inner sep=0.75pt]    {$v_{2}$};
\draw (242,118.4) node [anchor=north west][inner sep=0.75pt]    {$b_{0}$};
\draw (130,116.4) node [anchor=north west][inner sep=0.75pt]    {$b_{1}$};
\draw (189,216.4) node [anchor=north west][inner sep=0.75pt]    {$b_{2}$};
\draw (199,134.4) node [anchor=north west][inner sep=0.75pt]    {$b$};

\end{tikzpicture}
\end{center}
Observe that for general $k$ the ordering on the vertices can be taken to be the one induced from the barycentric subdivision operator $S(|\Delta^k|)$ considering the signs.
\end{example}

We conclude this appendix with the generalized Faà di Bruno's formula used in Chapter \ref{Chapter: Skeletal Diffeologies and Differential Characters}.

\begin{proposition}[\cite{Encinas-Masque}, Theorem 1] \label{Proposition: Faa di Bruno}
    Let $U \subset \mathbb{R}^n$, $V \subset \mathbb{R}^m$ be open subsets, and $U \xrightarrow{f} V \xrightarrow{g} \mathbb{R}$ be two smooth maps. Let $\beta = (\beta_1, ..., \beta_n) \in \mathbb{N}^n$ be a multi-index of order $|\beta| = \beta_1 + \cdots + \beta_n$. Let $x_0 \in U$ and write $y_0 = f(x_0)$. Then the following identity holds:
    \begin{equation}
                 \left.\frac{\partial^{|\beta|} \left( g \circ f \right)}{\partial x^{\beta}} \right\vert_{x_0} = \sum_{ |\sigma| = 1}^{|\beta|} \left.\frac{\partial^{|\sigma|} g }{\partial y^{\sigma} } \right\vert_{ y_0 } \sum_{E_{\sigma}} \prod_{i = 1}^{m}  \prod_{A_{\beta}} \frac{1}{e_{i\beta^i}!} \left( \left.\frac{\partial^{|\beta^i|} f_i }{\partial x^{\beta^i}} \right\vert_{x_0} \right)^{e_{i\beta^i}},
    \end{equation}
    where
    \begin{equation*}
        E_{\sigma} := \left\{  (e_{1\beta^1}, ..., e_{m\beta^m}) \, \bigg| \, e_{i \beta^i} \in \mathbb{N}, 1 \leq |\beta^i| \leq |\beta|, \sum_{|\beta^i| = 1}^{|\beta|} e_{i\beta^i} = \sigma_i, 1 \leq i \leq m  \right\}
    \end{equation*}
    and
    \begin{equation*}
        A_{\beta} = \left\{ (\beta^1, ..., \beta^m) \, \bigg| \, 1 \leq |\beta^i| \leq |\beta|, 1 \leq i \leq m, \sum_{i= 1}^m \sum_{|\beta^i| = 1}^{|\beta|} e_{i \beta^i} \beta^i  = \beta \right\}.
    \end{equation*}
\end{proposition}

\chapter{Simplicial Loop Groups}

This appendix aims to introduce the simplicial loop group $\mathbb{G}(K,x_0)$ provided some connected, pointed simplicial set $(K,x_0)$. Further, we show how the simplicial loop group relates to the simplicial classifying space $\overline{W}G$ and how we can use these simplicial techniques to study diffeological principal bundles via the extended smooth singular complex. The main references are Goerss--Jardine \cite{GoerssJardine} for the simplicial techniques and Kihara \cite{Kihara22} for the study of diffeological bundles in terms of their smooth singular complexes.

\section{Principal Fibrations}

The extended smooth singular complex $S_e$ allows us to translate the smooth homotopy theory of diffeological spaces to the homotopy theory of simplicial sets. It is therefore natural to ask how diffeological bundles are related to their simplicial counterparts which we call \textit{principal fibrations}. More precisely, given a smooth map $\pi \colon P \rightarrow X$ of diffeological spaces and considering its associated simplicial map $S_e(\pi) \colon S_e(P) \rightarrow S_e(X)$ are we able to relate the condition of $\pi$ being a diffeological principal bundle to the condition of $S_e(\pi)$ being a principal fibration? \\

Luckily we can give an affirmative answer from the following recent characterization theorem for diffeological principal bundles.

\begin{theorem}[\cite{Kihara22}] \label{Theorem: Diffeological bundles are principal fibrations}
    Given a diffeological group $G$ together with a smooth map $\pi \colon P \rightarrow X$ where $P$ is endowed with a fiber-preserving right $G$-action, then $\pi \colon P \to X$ is a diffeological principal $G$-bundle if and only if
    \[
    S_e(\pi) \colon S_e(P) \rightarrow S_e(X)
    \]
    is a principal $S_e(G)$-fibration.
\end{theorem}

Let us start first by introducing the term \textit{principal} $G$\textit{-fibration} given some simplicial group $G$.
\newpage

\begin{definition} \label{Definition: principal G-fibration}
    Let $G$ be a simplicial group. A \textbf{principal} $G$\textbf{-fibration} $f \colon  E \rightarrow B$ is a Kan fibration between simplicial $G$-spaces such that
    \begin{enumerate}[label={(\arabic*)}]
        \item the simplicial base $B$ carries the trivial $G$ action,
        \item $E$ is a level-wise free $G$-space, i.e. for all $n \in \mathbb{N}$ the set $E_n$ is a free $G_n$-set.
        \item the induced map $E/G \rightarrow B$ is an isomorphism.
    \end{enumerate}
\end{definition}

\begin{remark}
    For the notion of \textit{Kan fibration} see Definition \ref{Definition: Kan fibration}. Also, notice that the third condition implies that a principal $G$-fibration $f \colon E \rightarrow B$ is isomorphic to a quotient map $q \colon X \rightarrow X/G$ for some level-wise free $G$-space $X$.
\end{remark}

Given some fixed simplicial group $G$ and two principal $G$-fibrations $f_1 \colon E_1 \to B$ and $f_2 \colon E_2 \to B$ we say that $f_1$ is isomorphic to $f_2$ if there exists an isomorphism of simplicial $G$-spaces $g \colon E_1 \xrightarrow{\cong} E_2$ such that the following diagram commutes.
\begin{center}
    \begin{tikzcd}
E_1 \arrow[rd, "f_1"'] \arrow[rr, "g"] &   & E_2 \arrow[ld, "f_2"] \\
                                       & B &
\end{tikzcd}
\end{center}
The category with objects given by principal $G$-fibrations over $B$ and morphisms given by isomorphisms of principal $G$-fibrations is denoted by $\mathbf{PF}_G(B)$. The set of isomorphism classes of principal $G$-fibrations over $B$ is denoted by $\pi_0\mathbf{PF}_G(B)$.  \\

To introduce classifying spaces, recall the notion of a \textit{universal principal fibration} for a fixed simplicial group $G$.

\begin{definition} \label{Definition: Universal Principal Fibration}
    Let $G$ be a simplicial group and $\pi \colon E \rightarrow L$ a principal $G$-fibration. We call $\pi$ a \textbf{universal principal} $G$\textbf{-fibration} if $L$ is a Kan complex and if the natural map
    \[
    [K,L] \rightarrow \pi_0\mathbf{PF}_G(K), \hspace{3mm} [f : K \rightarrow L] \mapsto [f^*E \rightarrow K]
    \]
    is bijective.
\end{definition}

\begin{lemma}[\cite{Kihara22}, Lemma 6.1] \label{Lemma: Characterization of universal fibrations}
    Given a simplicial group $H$ and a principal $H$-fibration $\pi : E \rightarrow L$ the following are equivalent:
    \begin{enumerate}
        \item $\pi : E \rightarrow L$ is universal.
        \item $L$ is a Kan complex and $E$ is weakly contractible.
        \item $E$ is a contractible Kan complex.
    \end{enumerate}
\end{lemma}

Commonly universal principal $G$-fibrations are denoted by $EG \rightarrow BG$, however in the case of simplicial sets, there is a standard construction for a universal principal fibration called the $W$-construction. \\

\begin{construction} \label{Construction: W-construction}
    Fix a simplicial group $G$ and define the total space $WG$ to be a simplicial set with
    \[
    WG_n := G_n \times G_{n-1} \times \cdots \times G_0,
    \]
    where the face and degeneracy maps are given by
\begin{align*}
d_i(g_n, g_{n-1}, \cdots ,g_0) &=
\begin{cases}
\left(d_ig_n, d_{i-1}g_{n-1}, ... , (d_0g_{n-i})g_{n-i-1}, g_{n-i-2}, ... , g_0 \right) &\text{ for } i < n \\
\left( d_n g_n, d_{n-1}g_{n-1}, ..., d_1g_1\right) &\text{ for } i = n
\end{cases} \\
s_i(g_n, g_{n-1}, ..., g_0) &= (s_ig_n, s_{i-1}g_{n-1}, ..., s_0g_{n-i}, e, g_{n-i-1}, ..., g_0).
\end{align*}
The space $WG$ carries a natural $G$-action via multiplication in the first variable, i.e. level-wise given by
\begin{align*}
    G_n \times WG_n &\longrightarrow WG_n \\
    \left( h ,(g_n,g_{n-1}, ..., g_0) \right) &\mapsto (hg_n, g_{n-1}, ..., g_0).
\end{align*}
Denote by $\overline{W}G:= WG/G$ the quotient space of $WG$ by the $G$-action and write
\[
q \colon WG \to \overline{W}G
\]
for the quotient map. The space $\overline{W}G$ is the \textit{standard simplicial classifying complex} for the simplicial group $G$. Unwinding the effect of the quotient, it is presented as a simplicial set by:

\begin{align*}
        &(\overline{W}G)_0 = * \\
        &(\overline{W}G)_n = G_{n-1} \times G_{n-2} \times \cdots \times G_0
    \end{align*}
    with face and degeneracy maps given by:
    \begin{align*}
    d_i(g_{n-1}, ..., g_0) &=
    \begin{cases}
    (g_{n-2}, ..., g_0) \hspace{10mm} &\text{ if } i = 0 \\
    \left( d_{i-1}g_{n-1}, ..., d_1g_{n-i +1}, g_{n-i-1} (d_0g_{n-i}), g_{n-i-2}, ..., g_0 \right) \hspace{3mm} &\text{ if } 0 < i < n \\
    \left( d_{n-1}g_{n-1}, ..., d_1g_1 \right) &\text{ if } i = n
    \end{cases} \\
    s_i(g_{n-1}, ..., g_0) &=
    \begin{cases}
    (1,g_{n-1}, ..., g_0) \hspace{10mm} &\text{ if } i = 0 \\
    \left( s_{i-1}g_{n-1}, ..., s_0g_{n-i}, 1 ,g_{n-i-1},, g_{n-i-2}, ..., g_0 \right) \hspace{3mm} &\text{ if } 1 \leq i \leq n.
    \end{cases}
\end{align*}
\end{construction}

The $W$-construction determines a functor from the category of simplicial groups to simplicial sets

\[
\overline{W} \colon \mathbf{sGrp} \rightarrow \mathbf{sSet}.
\]
To introduce its left adjoint, the \textit{simplicial loop groupoid} functor $\cat{G}$, a generalized version of the $W$-construction is needed, which associates to every \textit{simplicial groupoid} $\cat{G}$ a classifying space $\overline{W}\cat{G}$. This extends the previously defined $W$-construction along the subcategory inclusion of simplicial groups into simplicial groupoids.

\begin{definition}
    A \textbf{simplicial groupoid} is a simplicial object in the category of groupoids whose simplicial set of objects is discrete. That is, a simplicial groupoid is given by the following data: \\

        For every $n \geq 0$ a small groupoid $\cat{G}_n$ together with functors $\theta^* : \cat{G}_m \rightarrow \cat{G}_n$ for every map of simplices $[n] \rightarrow [m]$ satisfying the simplicial identities. Moreover, we have that $\mathrm{ob}(\cat{G}_n) = \mathrm{ob}(\cat{G}_0)$, i.e. the set of objects is constant in the simplicial direction and the functors $\theta^*$ induce the identity on the set of objects. \\

        Given $x,y$ two objects in the simplicial groupoid $\cat{G}$ denote by $\cat{G}(x,y)$ the simplicial set of morphisms between $x$ and $y$ whose set of n-simplices is given by
        \[
        \cat{G}(x,y)_n = \cat{G}_n(x,y)
        \]
\end{definition}
\begin{remark}
    Notice that alternatively, we could define a simplicial groupoid as a groupoid $\cat{G}$ internal to simplicial sets with a discrete simplicial set of objects, i.e. $\cat{G}_1 \rightrightarrows \cat{G}_0$ a simplicial set of morphisms $\cat{G}_1$ and a constant simplicial set of objects $\cat{G}_0$ together with the simplicial source and target maps.
\end{remark}

Given an ordinary groupoid $H$ write $\pi_0H$ for the set of \textbf{path components} of $H$. More precisely, we define:
\[
\pi_0 H = \mathrm{Ob}(H)/ \sim
\]
where we declare two objects $x$ and $y$ to be equivalent iff there is an morphism $x \rightarrow y$ in $H$. \\

In the case of a simplicial groupoid, we have seen that $\cat{G}_n$ all have the same object set, hence we conclude that the simplicial structure functors $\theta^*$ induce isomorphisms
\[
\pi_0 \cat{G}_n \cong \pi_0 \cat{G}_m
\]
hence we define the set of path components of the simplicial groupoid $\cat{G}$ to be:
\[
\pi_0 \cat{G} := \pi_0 \cat{G}_0
\]

As already mentioned, the classifying space functor assigning to a simplicial groupoid $\cat{G}$ the simplicial set $\overline{W}\cat{G}$ admits a left adjoint provided by the Dwyer--Kan loop groupoid to be introduced shortly. However there is more, that is, the classifying space functor $\overline{W}$ will be part of a Quillen equivalence of model categories. For this reason, we briefly recall the Dwyer--Kan model structure of simplicial groupoids.

\begin{definition}
    A map of simplicial groupoids $f : \cat{G} \rightarrow \cat{H}$ is said to be a \textbf{weak equivalence of simplicial groupoids} if
    \begin{itemize}
        \item $f$ induces an isomorphism on the path components: $\pi_0 \cat{G} \cong \pi_0 \cat{H}$, and
        \item each induced map $f : \cat{G}(x,x) \rightarrow \cat{H}(f(x),f(x))$ is a weak equivalence of simplicial sets.
    \end{itemize}
\end{definition}

The following example introduces an important class of weak equivalences of simplicial groupoids.

\begin{example} \label{Example: Retraction of simplicial groupoid}
    Let $\cat{G}$ be a simplicial groupoid. Let $[x] \in \pi_0 \cat{G}$ be a path component and consider $x \in [x]$ a representative. Then the morphism
    \[
    i : \coprod_{[x] \in \pi_0 \cat{G}} \cat{G}(x,x) \rightarrow \cat{G}
    \]
    is by definition a weak equivalence of simplicial groupoids. In particular, the left-hand side is a deformation retract of $\cat{G}$. Indeed, for each $y \in [x]$ and each $[x] \in \pi_0 \cat{G}$ pick morphisms
    \[
    \omega_y : y \rightarrow x \in \cat{G}_0(y,x)
    \]
    such that $\omega_x = \mathrm{id}_x$ for all the fixed choices of representatives $x \in [x]$. Then we define the map
    \[
    r : \cat{G} \rightarrow \coprod_{[x] \in \pi_0 \cat{G}} \cat{G}(x,x),
    \]
    which is given by conjugation by the paths $\omega_y$. More precisely, given $y,z \in [x]$ the map $r$ restricts to
    \begin{align*}
    r : \cat{G}(y,z) \rightarrow \cat{G}(x,x) \\
    \alpha \mapsto \omega_z \alpha \omega^{-1}_y,
    \end{align*}
    which gives the desired retraction.
\end{example}

Further, a map $f \colon \cat{G} \rightarrow \cat{H}$ of simplicial groupoids is said to be a \textbf{fibration of simplicial groupoids} if:
\begin{enumerate}[label={(\arabic*)}]
    \item The morphism $f$ satisfies the \textit{path lifting property}. That is, for every object $x \in \cat{G}$ and morphism $\varphi \colon f(x) \rightarrow y$ of the simplicial groupoid $\cat{H}$, there is a morphism $\hat{\varphi} \colon x \rightarrow z$ in $\cat{G}$ such that $f(\hat{\varphi}) = \varphi$.
    \item For every object $x$ of $\cat{G}$ the induced map of simplicial sets $f \colon \cat{G}(x,x) \rightarrow \cat{H}(f(x),f(x))$ is a Kan fibration.
\end{enumerate}
Having specified weak equivalences and fibrations of simplicial groupoids, we simply define the \textbf{cofibrations of simplicial groupoids} to be morphisms $i \colon \cat{G} \rightarrow \cat{H}$ which satisfy the left lifting property with respect to the trivial fibrations.
\begin{theorem}[\cite{GoerssJardine}, Theorem 7.6] \label{Theorem: Model Structure on simplicial Groupoids}
    The category of simplicial groupoids $sGpd$ together with the definitions of weak equivalence, fibrations, and cofibrations satisfy the axioms for a closed model category.
\end{theorem}
\begin{remark}
    Notice that in this model structure, all simplicial groupoids are fibrant since for $\cat{G}$ a simplicial groupoid, the space of morphisms $\cat{G}(x,x)$ for any object $x \in \cat{G}$ has the structure of a simplicial group, which is indeed a Kan complex.
\end{remark}
\newpage

\begin{theorem}[\cite{GoerssJardine}, Theorem 7.8] \label{Theorem: Classifying space adjunction} $ $
    \begin{enumerate}[label={(\arabic*)}]
        \item The functor $\cat{G} : \mathbf{sSet} \rightarrow \mathbf{sGpd}$ preserves cofibrations and weak equivalences.
        \item The functor $\overline{W} : \mathbf{sGpd} \rightarrow \mathbf{sSet}$ preserves fibrations and weak equivalences.
        \item The functors $\cat{G}$ and $\overline{W}$ form an adjoint pair, that is for all $X \in \mathbf{sSet}$ and $\cat{H} \in \mathbf{sGpd}$ we have natural isomorphisms:
        \[
        \mathrm{Hom}_{\mathbf{sGpd}}(\cat{G}(X),\cat{H}) \cong \mathrm{Hom}_{\mathbf{sSet}}(X, \overline{W}\cat{H})
        \]
        \item A map $K \rightarrow \overline{W}\cat{H}$ is a weak equivalence if and only if its adjoint $\cat{G}(K) \rightarrow \cat{H}$ is a weak equivalence.
    \end{enumerate}
\end{theorem}

The construction of the simplicial loop groupoid $\cat{G}X$ associated to a simplicial set $X$ goes back to the work of Dwyer and Kan \cite{Dwyer-Kan}. It is constructed as a free simplicial groupoid on a set of generators. Therefore recall the notion of \textit{free groupoids generated by a graph} following \cite[Section 8.2]{Brown}.

\begin{definition}
    A graph $\Gamma$ consists of a set $\mathrm{Ob}(\Gamma)$ of objects or vertices and for each pair $(x,y)$ of objects a set $\Gamma(x,y)$ of morphisms or edges from $x$ to $y$. Since we are only concerned with directed graphs we suppose that for various objects $x,y$ the sets $\Gamma(x,y)$ are disjoint. In particular, the two sets $\Gamma(x,y)$ and $\Gamma(y,x)$ are disjoint for $x \neq y$. A subgraph $\Gamma \subset \Delta$ is called \textbf{wide} if we have that $\mathrm{Ob}(\Gamma) = \mathrm{Ob}(\Delta)$.
\end{definition}

\begin{remark}
    Note that any groupoid $\cat{G}$ defines a graph whose set of objects is $\mathrm{Ob}(\cat{G})$ and whose set of morphisms is $\cat{G}(x,y)$. Therefore we say that a \textbf{graph in a groupoid} $\cat{G}$ is a graph $\Gamma$ which is a subgraph of $\cat{G}$. The \textbf{subgroupoid generated by} $\Gamma$ is the smallest subgroupoid of $\cat{G}$ containing the graph $\Gamma$. The subgroupoid generated by $\Gamma$ then has as objects the set $\mathrm{Ob}(\Gamma)$ and as morphisms all the identity morphisms together with all well-defined compositions in $\cat{G}$
    \[
    a_n \cdots a_1
    \]
    such that $a_i \in \Gamma$ or $a_i^{-1} \in \Gamma$.
\end{remark}

\begin{definition}
    Let $\Gamma$ be a graph in a groupoid $\cat{G}$. We say that $\cat{G}$ is \textbf{free} on $\Gamma$ if $\mathrm{Ob}(\Gamma) = \mathrm{Ob}(\cat{G})$ and the following universal property is satisfied: For any groupoid $\cat{H}$ and any graph morphism $\varphi : \Gamma \rightarrow \cat{H}$ there exists a unique extension to a morphism of groupoids making the following diagram commutative.
    \begin{center}
        \begin{tikzcd}
\Gamma \arrow[d, hook] \arrow[rd, "\varphi"] &         \\
\cat{G} \arrow[r, "\exists !", dashed]       & \cat{H}
\end{tikzcd}
    \end{center}
\end{definition}
\newpage

\begin{proposition}[\cite{Brown}, 8.2.1]
    Let $\Gamma$ be a graph in a groupoid $\cat{G}$. The following are equivalent:
    \begin{enumerate}
        \item $\cat{G}$ is free on $\Gamma$.
        \item $\Gamma$ generates $\cat{G}$ and the non-identity elements of $\cat{G}$ can be written uniquely as products
        \[
        a_n^{\varepsilon_n} \cdots a_1^{\varepsilon_1},
        \]
        such that $a_i \in \Gamma$, $\varepsilon_i = \pm 1$ and for no $i$ it is true that both $a_i = a_{i + 1}$ and $\varepsilon_i = -\varepsilon_{i+1}$, i.e. the above product is reduced.
    \end{enumerate}
\end{proposition}

Now that we have spelled out what it means for a groupoid to be freely generated we can define the \textit{Dwyer--Kan loop groupoid} of a simplicial set.

\begin{definition}\label{Definition: Dwyer-Kan loop groupoid}
    Let $K$ be a simplicial set. Then the \textbf{Dwyer--Kan loop groupoid of} $K$ is the simplicial groupoid $\cat{G}K$ defined as follows: For $n \geq 0$ we define
    \[
    (\cat{G}K)_n:= F \left( K_{n+1} - s_0(K_n) \rightrightarrows K_0 \right)
    \]
    as the free groupoid on the graph given by the quiver
    \begin{center}
        \begin{tikzcd}
K_{n+1} - s_0(K_n) \arrow[r, "t"', shift right] \arrow[r, "s", shift left] & K_0,
\end{tikzcd}
    \end{center}
    where the source and target maps are defined as follows:
    \begin{align*}
        s = d_1 \circ d_2 \circ \cdots \circ d_{n+1} \\
        t = d_0 \circ d_2 \circ \cdots \circ d_{n+1}.
    \end{align*}
    The simplicial structure maps are now defined via generators, i.e. for $0 \leq i \leq n$
    \begin{align*}
        s_i : (\cat{G}K)_n &\longrightarrow (\cat{G}K)_{n+1} \\
         \sigma \in K_{n+1} - s^K_0(K_n) &\longmapsto s^K_{i+1}(\sigma),
    \end{align*}
    using the simplicial identity $s^K_{i+1}s^K_0 = s^K_0s^K_i $, and
    \begin{align*}
        d_i : (\cat{G}K)_n &\longrightarrow (\cat{G}K)_{n-1} \\
        \sigma \in K_{n+1} - s_0(K_n) &\longmapsto d^K_{i+1}(\sigma) \in K_n \hspace{4mm} \text{for } 1 < i \leq n \\
        \sigma \in K_{n+1} - s_0(K_n) &\longmapsto d_1^K(\sigma)  \cdot \left( d^K_{0}(\sigma) \right)^{-1}  \in K_n \hspace{4mm} \text{for } i = 0.
    \end{align*}
\end{definition}
\begin{remark}
    Notice that the graph $\Gamma_{n+1}$ given by the quiver above has as objects $\mathrm{Ob}(\Gamma_{n+1}) = K_0$ the vertices of $K$ and for all $x,y \in K_0$ the set of morphisms is given by the pullback
    \begin{center}
        \begin{tikzcd}
{\Gamma_{n+1}(x,y)} \arrow[d] \arrow[r] & K_{n+1} - s_0(K_n) \arrow[d, "{(s,t)}"] \\
* \arrow[r, "{(x,y)}"]              & K_0 \times K_0
\end{tikzcd}
    \end{center}
    We then have that the simplicial loop groupoid is degree-wise given by the free groupoid on this graph
    \[
     (\cat{G}K)_n:= F \left( \Gamma_{n+1} \right).
    \]
\end{remark}

Observe that the presentation of the simplicial groupoid $\cat{G}(K)$ differs slightly from the original definition in \cite{Dwyer-Kan} which further contains a typo in the definition of the zeroth face map $d_0 \colon \cat{G}(K)_n \rightarrow \cat{G}(K)_{n-1}$ as observed and corrected by Ehlers \cite{Ehlers}. It is then immediate that Definition \ref{Definition: Dwyer-Kan loop groupoid} is equivalent to the definition provided by \cite{Ehlers}. \\

In the next step, we wish to relate the simplicial loop groupoid $\cat{G}(K)$ to the simplicial loop group $G(K)$ introduced by Kan for connected simplicial sets $K$ in \cite{Kan}. Thus consider $(K,x_0)$ a connected simplicial set with a base point. Then the loop groupoid $\cat{G}K$ is connected and by example \ref{Example: Retraction of simplicial groupoid} we have that the inclusion
\[
i : \cat{G}K(x_0,x_0) \rightarrow \cat{G}K
\]
is a retraction. The retraction however depends on a certain choice of paths, which we specify below.

\begin{definition}\label{Definition: Dwyer-Kan simplicial loop group}
    Let $(K,x_0)$ be a pointed, connected simplicial set. The simplicial group
    \[
    \mathbb{G}(K,x_0) := \cat{G}K(x_0,x_0)
    \]
    is called the \textbf{Dwyer--Kan simplicial loop group}. Its abelianization is denoted by $\mathbb{A}(K,x_0)$. Usually, the dependence of the basepoint is omitted to simplify notation.
\end{definition}

The choices of paths to obtain the retraction can be formalized as follows:

\begin{definition}[\cite{Kan}] \label{definition maximal tree}
Let $(K,x_0)$ be as above. A \textbf{maximal tree} $T \subset K$ is a connected simplicial subset that contains all the vertices in $K$ but no closed loops. To make this definition precise we first define an $n$\textbf{-loop of length} $k$ to be a sequence
\[
(\sigma_1, ..., \sigma_{2k})
\]
of $(n+1)$-simplices in $K$ such that:
\begin{align*}
    &d_{n+1}(\sigma_{2j-1}) = d_{n+1}(\sigma_{2j}) \hspace{4mm} &\text{ for all } 1 \leq j \leq k \\
    &d_0 \left( \cdots d_n(\sigma_{2j})  \cdots \right) = d_0 \left( \cdots d_n(\sigma_{2j +1})  \cdots \right) \hspace{4mm} &\text{ for all } 1 \leq j < k \\
    &d_0 \left( \cdots d_n(\sigma_{1})  \cdots \right) = d_0 \left( \cdots d_n(\sigma_{2k}) \cdots \right) = x_0
\end{align*}
By a reduced loop we mean an $n$-loop of length $k$ such that
\[
\sigma_s \neq \sigma_{s+1} \hspace{4mm} \text{ for all } 1 \leq s < 2k.
\]
A subcomplex $T \subset K$ containing the vertex $x_0$ is said to be a \textbf{tree} if contains no loops of length $>0$. A \textbf{maximal tree} is a tree $T$ which contains all the vertices, i.e. $T_0 = K_0$.
\end{definition}

\begin{lemma}[\cite{Kan}, Lemma 9.1]
    A connected simplicial set with basepoint $(K,x_0)$ contains a maximal tree.
\end{lemma}

\begin{example}
    Let us look at the definition of loops to get a better understanding. Let first $n=0$, then a $0$-loop of length $k$ is given by a sequence
    \[
    (\sigma_1, ..., \sigma_{2k})
    \]
    of $1$-simplices in $K$ such that
    \begin{align*}
         &d_{1}(\sigma_{2j-1}) = d_{1}(\sigma_{2j}) \hspace{4mm} \text{ for all } 1 \leq j \leq k \\
    &d_0 (\sigma_{2j}) = d_0(\sigma_{2j +1}) \hspace{4mm} \text{ for all } 1 \leq j < k \\
    &d_0(\sigma_{1})= d_0(\sigma_{2k}) = x_0,
    \end{align*}
    \begin{center}
        \begin{tikzpicture}[x=0.75pt,y=0.75pt,yscale=-1,xscale=1]

\draw    (100,110) -- (148.04,100.39) ;
\draw [shift={(150,100)}, rotate = 168.69] [color={rgb, 255:red, 0; green, 0; blue, 0 }  ][line width=0.75]    (10.93,-3.29) .. controls (6.95,-1.4) and (3.31,-0.3) .. (0,0) .. controls (3.31,0.3) and (6.95,1.4) .. (10.93,3.29)   ;
\draw    (200,110) -- (151.96,100.39) ;
\draw [shift={(150,100)}, rotate = 11.31] [color={rgb, 255:red, 0; green, 0; blue, 0 }  ][line width=0.75]    (10.93,-3.29) .. controls (6.95,-1.4) and (3.31,-0.3) .. (0,0) .. controls (3.31,0.3) and (6.95,1.4) .. (10.93,3.29)   ;
\draw    (100,110) -- (66.63,133.84) ;
\draw [shift={(65,135)}, rotate = 324.46] [color={rgb, 255:red, 0; green, 0; blue, 0 }  ][line width=0.75]    (10.93,-3.29) .. controls (6.95,-1.4) and (3.31,-0.3) .. (0,0) .. controls (3.31,0.3) and (6.95,1.4) .. (10.93,3.29)   ;
\draw    (50,185) -- (64.43,136.92) ;
\draw [shift={(65,135)}, rotate = 106.7] [color={rgb, 255:red, 0; green, 0; blue, 0 }  ][line width=0.75]    (10.93,-3.29) .. controls (6.95,-1.4) and (3.31,-0.3) .. (0,0) .. controls (3.31,0.3) and (6.95,1.4) .. (10.93,3.29)   ;
\draw    (200,110) -- (228.7,143.48) ;
\draw [shift={(230,145)}, rotate = 229.4] [color={rgb, 255:red, 0; green, 0; blue, 0 }  ][line width=0.75]    (10.93,-3.29) .. controls (6.95,-1.4) and (3.31,-0.3) .. (0,0) .. controls (3.31,0.3) and (6.95,1.4) .. (10.93,3.29)   ;
\draw    (245,195) -- (230.57,146.92) ;
\draw [shift={(230,145)}, rotate = 73.3] [color={rgb, 255:red, 0; green, 0; blue, 0 }  ][line width=0.75]    (10.93,-3.29) .. controls (6.95,-1.4) and (3.31,-0.3) .. (0,0) .. controls (3.31,0.3) and (6.95,1.4) .. (10.93,3.29)   ;

\draw (111,85.4) node [anchor=north west][inner sep=0.75pt]    {$\sigma _{1}$};
\draw (176,82.4) node [anchor=north west][inner sep=0.75pt]    {$\sigma _{2k}$};
\draw (66,97.4) node [anchor=north west][inner sep=0.75pt]    {$\sigma _{2}$};
\draw (36,147.4) node [anchor=north west][inner sep=0.75pt]    {$\sigma _{3}$};
\draw (220,110.4) node [anchor=north west][inner sep=0.75pt]    {$\sigma _{2k-1}$};
\draw (246,160.4) node [anchor=north west][inner sep=0.75pt]    {$\sigma _{2k-2}$};
\draw (141,77.4) node [anchor=north west][inner sep=0.75pt]    {$x_{0}$};

\end{tikzpicture}
    \end{center}
    which looks just like a loop in the simplicial set $K$ based at $x_0$. In particular, we notice that for any $y \in K$ there exists a unique sequence of simplices in the tree $T$ connecting $y$ to $x_0$. Indeed, if there were more, then we could create a loop of length $>0$ inside the tree $T$ which gives a contradiction. 
\end{example}

\begin{construction}\label{Construction: Retract given by maximal Tree}
    Let $(K,x_0)$ be a connected and pointed simplicial set. Choose $T$ a maximal tree in $K$. We now want to define the retraction:
    \[
    r_T : \cat{G}K \rightarrow \mathbb{G}(K,x_0)
    \]
    from the simplicial loop groupoid to the loop group. Let $y$ be any vertex in $K$. Then the unique sequence of $1$-simplices $(\sigma_1, ... , \sigma_k)$ in $T$ connecting $y$ and $x_0$ represents a morphism in the free groupoid $(\cat{G}K)_0$ from $y$ to $x_0$. We denote this morphism by
    \[
    \omega_y : y \rightarrow x_0
    \]
    Then we define the morphism $r_T$ of groupoids first on the objects as the constant map onto $x_0$:
    \[
    r_T : \mathrm{Ob}(\cat{G}K) = K_0 \rightarrow \left\{ x_0 \right\} , \hspace{3mm} y \mapsto x_0
    \]
    and on the morphisms by conjugation
    \begin{align*}
        r_T : (\cat{G}K)(y,z) \rightarrow \mathbb{G}(K,x_0) \\
        \alpha : y \rightarrow z \mapsto \omega_z \alpha \omega^{-1}_y.
    \end{align*}
\end{construction}

Considering the simplicial loop group as a retract of the free simplicial groupoid $\cat{G}K$, we can find generators for the free simplicial group $\mathbb{G}(K,x_0)$ using the following lemma.

\begin{lemma}[\cite{Brown}, 8.2.3] \label{Lemma: Automorphism Group Free Groupoid}
For $\cat{G} = F(\Gamma)$ a connected, free groupoid each automorphism group $\cat{G}(x,x)$ of $\cat{G}$ is free. A choice of tree groupoid $\cat{T}$ in $\cat{G}$ such that $\cat{T} \cap \Gamma$ generates $\cat{T}$ allows to specify a presentation of $\cat{G}(x,x)$ via generators.
\end{lemma}
\begin{proof}(Sketch)
    Let $\Gamma$ be a graph freely generating $\cat{G}$ and let $x \in \cat{G}$ be an object. Similarly as one can show that any connected pointed simplicial set contains a maximal tree, we can choose $\cat{T}$ a tree groupoid that is wide in $\cat{G}$ and such that $\cat{T} \cap \Gamma$ generates $\cat{T}$. For every object $y$ of $\cat{G}$ there is now a unique morphism $\tau_y \in \cat{T}(x,y)$. Like in Example \ref{Example: Retraction of simplicial groupoid}, this defines via conjugation a retraction $r \colon \cat{G} \rightarrow \cat{G}(x,x)$. Denote now by $\Lambda$ the wide subgraph of $\Gamma$ whose edges are those which are not contained in $\cat{T}$. Denote by $\cat{H}$ the free subgroupoid of $\cat{G}$ generated by $\Lambda$. Then it follows that $\cat{G}$ is given by the free product $\cat{T} * \cat{H}$ and therefore $\cat{G}(x,x)$ is the free group on $r[\Lambda]$. Here $r[\Lambda]$ denotes the set given by the union of the sets $r\left( \Lambda(y,z)\right)$ for all objects $y,z$ in $\Lambda$.
\end{proof}

\begin{corollary} \label{Corollary: Dwyer-Kan group as a free group}
Given a pointed and connected simplicial set $(K,x_0)$ together with a choice of a maximal tree $T$ the simplicial group $\mathbb{G}(K,x_0)$ is degreewise generated as follows. For $n \geq 0$ the group $\mathbb{G}(K,x_0)_n$ is given by
    \begin{itemize}
        \item one generator $\sigma$ for every $\sigma \in K_{n+1}$
        \item one relation $s_0(\tau) = e_n$ for every $\tau \in K_n$
        \item one relation $\rho = e_n$ for every $(n+1)$-simplex $\rho \in T_{n+1}$.
    \end{itemize}
    or equivalently
    \[
    \mathbb{G}(K,x_0)_n = F \left( K_{n+1} - \left( s_0(K_n) \cup T_{n+1} \right) \right)
    \]
    with face and degeneracy maps the same ones as we have seen in the loop groupoid construction. In particular, $\mathbb{G}(K,x_0)$ agrees with Kan's simplicial loop group \normalfont{\cite[p.293]{Kan}}.
\end{corollary}
\begin{proof}
    The presentation of $\mathbb{G}(K,x_0)$ via generators and relations follows from a level-wise application of Lemma \ref{Lemma: Automorphism Group Free Groupoid} to the free groupoids $\cat{G}(K)_n$ for all $[n] \in \Delta$. Observe that the \textit{tree groupoid} used in the proof of said Lemma is in this case simply the free groupoid on the maximal tree $T$. Under this identification, the level-wise retractions agree precisely with the retraction of Construction \ref{Construction: Retract given by maximal Tree}. In particular, it follows that level-wise $ \mathbb{G}(K,x_0)_n$ is freely generated by $K_{n+1} - \left( s_0(K_n) \cup T_{n+1} \right)$. This presentation of $\mathbb{G}(K,x_0)$ now agrees with the definition given in \cite{Kan}.
\end{proof}

\appto{\bibsetup}{\raggedright}
\printbibliography

\end{document}